\documentclass[12pt]{amsart}

\usepackage{dsfont}
\usepackage{amsxtra}
\usepackage{amsmath}
\usepackage{amscd}
\usepackage{amssymb}
\usepackage{amsfonts}
\usepackage[all]{xy}
\usepackage{mathrsfs}
\usepackage{amsthm}
\usepackage{color}
\usepackage{upgreek }
\usepackage{tikz}
\usepackage{tikz-cd}
\usepackage{enumerate}
\usepackage{hyperref}
\usepackage{ae}
\usepackage{enumitem}
\usepackage{blkarray}
\usepackage{multirow}
\usepackage{mathtools}
\usepackage{graphicx}

\usetikzlibrary{matrix,arrows,decorations.pathmorphing,automata, matrix, positioning, calc, shapes.multipart}

\numberwithin{itemcounter}{subsection}
\theoremstyle{plain}

\newtheorem*{thmA}{Theorem A}
\newtheorem*{thmB}{Corollary B}
\newtheorem*{thmC}{Theorem C}

\newtheorem{theorem}{Theorem}[section]
\newtheorem{lemma}[theorem]{Lemma}
\newtheorem{definition-lemma}[theorem]{Definition-Lemma}

\newtheorem{proposition}[theorem]{Proposition}

\newtheorem{corollary}[theorem]{Corollary}
\theoremstyle{definition}
\newtheorem{definition}[theorem]{Definition}
\theoremstyle{remark}
\newtheorem{remark}[theorem]{Remark}
\newtheorem{example}[theorem]{Example}
\numberwithin{equation}{section}

\def\bbC{\mathbb{C}}

\def\bbF{\mathbb{F}}

\def\bbH{\mathbb{H}}

\def\bbN{\mathbb{N}}
\def\bbO{\mathbb{O}}

\def\bbQ{\mathbb{Q}}
\def\bbR{\mathbb{R}}

\def\bbZ{\mathbb{Z}}

\def\scrA{\mathscr{A}}

\def\scrC{\mathscr{C}}

\def\scrH{\mathscr{H}}

\def\scrI{{I}}

\def\scrP{\mathscr{P}}

\def\frakg{\mathfrak{g}}

\def\fraks{\mathfrak{s}}
\def\frakl{\mathfrak{l}}

\def\frakU{\mathfrak{U}}

\def\calC{\mathcal{C}}

\def\calE{\mathcal{E}}
\def\calF{\mathcal{F}}

\def\calH{\mathcal{H}}

\def\calJ{\mathcal{J}}

\def\calL{\mathcal{L}}

\def\calO{\mathcal{O}}

\def\calS{\mathcal{S}}

\def\calV{\mathcal{V}}

\def\frakg{\mathfrak{g}}

\def\frakl{\mathfrak{l}}

\def\bfL{\mathbf{L}}

\def\bfB{\mathbf{B}}

\def\bfF{\mathbf{F}}
\def\bfG{\mathbf{G}}

\def\bfL{\mathbf{L}}

\def\bfO{\mathbf{O}}
\def\bfP{\mathbf{P}}

\def\bfT{\mathbf{T}}
\def\bfU{\mathbf{U}}
\def\bfV{\mathbf{V}}

\def\geqs{\geqslant}
\def\leqs{\leqslant}

\def\simto{\overset{\sim}\to}

\def\bfSp{{\mathbf{Sp}}}

\def\bfSO{{\mathbf{SO}}}
\def\O{\operatorname{O}\nolimits}
\def\GL{\operatorname{GL}\nolimits}
\def\GU{\operatorname{GU}\nolimits}
\def\SO{\operatorname{SO}\nolimits}
\def\Sp{\operatorname{Sp}\nolimits}
\def\CSp{\operatorname{CSp}\nolimits}

\def\CO{\operatorname{CO}\nolimits}

\def\Irr{{\operatorname{Irr}\nolimits}}

\def\sp{{\operatorname{sp}\nolimits}}

\def\GL{\operatorname{GL}\nolimits}
\def\GU{\operatorname{GU}\nolimits}
\def\diag{\operatorname{diag}\nolimits}

\def\X{\operatorname{X}\nolimits}

\def\Z{\operatorname{Z}\nolimits}

\def\up{{\operatorname{\,up}\nolimits}}

\def\Hom{\operatorname{Hom}\nolimits}

\def\End{\operatorname{End}\nolimits}
\def\Res{\operatorname{Res}\nolimits}
\def\res{\operatorname{res}\nolimits}
\def\Ind{\operatorname{Ind}\nolimits}

\def\id{\operatorname{id}\nolimits}

\def\wt{\operatorname{wt}\nolimits}

\def\mod{\operatorname{-mod}\nolimits}

\def\Heis{\mathcal{H}eis}

\def\RX{\mathrm{RX}}
\def\LX{\mathrm{LX}}

\def\Lcross{\mathrm{LT}}
\def\MixLcross{\mathrm{LH}}
\def\Rcross{\mathrm{RT}}
\def\MixRcross{\mathrm{RH}}
\def\Proj{\mathrm{Proj}}

\def\up{\uparrow}
\def\down{\downarrow}

\def\Heis{\mathcal{H}eis}

\def\Hom{\operatorname{Hom}}

\def\wt{{\operatorname{wt}}}
\def\End{{\operatorname{End}}}

\def\Z{{\mathbb Z}}

\def\dt{{\color{white}\bullet}\!\!\!\circ}

\newcommand{\sgn}{{\operatorname{det}}}
\newcommand{\sign}{{\operatorname{sgn}}}

\newcommand*{\K}{\mathcal{K}}

\newcommand*{\tuple}[1]{\boldsymbol{#1}}

\makeatletter
\newcommand{\unit}{\mathds{1}}

\def\anticlock{\begin{tikzpicture}[baseline=-.9mm]
\filldraw[white] (0,0) circle (1.72mm);
\draw[-] (0,-0.18) to[out=180,in=-102] (-.178,0.02);
\draw[-] (-0.18,0) to[out=90,in=180] (0,0.18);
\draw[-] (0.18,0) to[out=-90,in=0] (0,-0.18);
\draw[<-] (0,0.18) to[out=0,in=90] (0.18,0);
\end{tikzpicture}\,}

\def\clock{\begin{tikzpicture}[baseline=-.9mm]
\filldraw[white] (0,0) circle (1.72mm);
\draw[-] (0,-0.18) to[out=180,in=-90] (-.18,0);
\draw[->] (-0.18,0) to[out=90,in=180] (0,0.18);
\draw[-] (-0.02,0.178) to[out=12,in=90] (0.18,0);
\draw[-] (0.18,0) to[out=-90,in=0] (0,-0.18);
\end{tikzpicture}\,}

\newcommand{\darkg}{\color{green!70!black}}
\tikzset{darkg/.style={green!70!black}}

\def\dott{{\color{white}\bullet}\!\!\!\circ}

\def\dott{{\color{white}\bullet}\!\!\!\circ}
\definecolor{darkblue}{HTML}{000000}
\def\red#1{{\color{red} #1}}
\def\blue#1{{\color{blue} #1}}
\def\up{\uparrow}
\def\down{\downarrow}

\allowdisplaybreaks 
\begin{document}
\title{Categorical action for finite classical groups and its applications: characteristic 0}

\author{Pengcheng Li, Peng Shan and Jiping Zhang}
	\address{Yau Mathematical Sciences Center, Tsinghua University, Beijing 100084, China}
\email{pcli@tsinghua.edu.cn}

\address{Yau Mathematical Sciences Center, Department of Mathematical Sciences, Tsinghua University, Beijing 100084, China}
\email{pengshan@tsinghua.edu.cn}

\address{School of Mathematical Sciences, Peking University,
  Beijing 100871, China}
\email{jzhang@pku.edu.cn}

\thanks{}



\begin{abstract} In this paper, we construct a categorical double quantum Heisenberg
action on the representation category of finite classical groups $\O_{2n+1}(q)$, $\Sp_{2n}(q)$ and $\O^{\pm}_{2n}(q)$ with $q$ odd.
Over a field of characteristic zero or characteristic $\ell$ with $\ell\nmid q(q-1)$,
we deduce a categorical action of a Kac-Moody algebra
$\mathfrak{s}\mathfrak{l}'_{I_+}\oplus\mathfrak{s}\mathfrak{l}'_{I_-}$
on the representation category of finite classical groups.
  We show that the colored weight  functions $\bbO^+(u)(-)$, $\bbO^-(v)(-)$
    and uniform projection can distinguish all  irreducible characters of finite classical
     groups. In particular,  the colored weight functions are  complete invariants of
     quadratic unipotent characters. We also show that using the theta correspondence
     and extra symmetries of categorical double quantum Heisenberg action,
      the Kac-Moody action on the Grothendieck group of the whole category
      can be determined explicitly.

\end{abstract}

\maketitle

\pagestyle{myheadings}
%%\markboth{}{}
\markboth{Categorical actions for finite classical groups}{P. Li, P. Shan and J. Zhang}
\section*{Introduction}\label{intro}

\subsection{Overview}

Let $\bfG$ be a connected reductive algebraic group defined over a finite field $\bbF_q$. The finite
group $G=\bfG(\bbF_q)$ of its $\bbF_q$-rational points is a finite reductive group.
When the center of $\bfG$ is connected, Lusztig in his  book \cite{Lu84}
classified all irreducible characters of finite reductive groups by using
 Deligne--Lusztig theory, Kazhdan--Lusztig polynomials,
perverse sheaves, cell theory for Hecke algebras, etc. Lusztig \cite{Lu88} extended this
 to the disconnected center case via the so-called regular embedding from $G$ to a larger group
 $\widetilde{G}$ with connected center. The classification is given in terms of
 restriction from irreducible characters of $\widetilde{G}$ to $G$, which is multiplicity-free. However,
 there is no natural way to distinguish all the components of this restriction.
 % For example, we can not get a complete classification of irreducible characters of alternating groups $A_n$ only via restriction from irreducible characters of symmetric groups $S_n$.
So the classification of irreducible characters of the finite reductive group with disconnected center is not complete in some sense.   In this paper,
 we focus on the representation theory of finite classical groups $\O_{2n+1}(q),$
 $\Sp_{2n}(q)$ and $\O^{\pm}_{2n}(q).$  The connected components of the latter two groups have disconnected center. One of the aims of this article is to give a complete classification of the irreducible characters of those groups by using  categorification methods.
\smallskip

Categorification has played a prominent role in the recent development of representation theory.
The pioneering work of Chuang--Rouquier \cite{CR} introduced the notion of an $\fraks\frakl_2$-categorification (a.k.a. a categorical $\fraks\frakl_2$-action), and solved Brou\'e's abelian defect group conjecture  \cite{broue1990} for symmetric groups $S_n$ and finite general linear groups $\GL_n(q)$ using this new tool. Later, Khovanov and Lauda \cite{KL09} and Rouquier \cite{R08} independently extended this notion to arbitrary symmetrizable Kac-Moody algebra $\mathfrak{g}$ by introducing a remarkable new $2$-category $\frakU(\frakg)$ in terms of generators and relations. Their $2$-categories were originally defined in different forms, it was shown by Brundan \cite{B16} that they are actually equivalent. However, it is difficult to construct direct actions of these $2$-categories on representation categories, as the invertibility of certain $2$-morphisms is impossible to check in general. But, when $\frakg$ is of type $A$ or affine type $A$, Rouquier (based on previous work of Chuang--Rouquier \cite{CR}) showed that one can avoid such difficulties using a control from the Grothendieck group \cite[Theorem 5.23]{R08}.
In particular, the construction of \cite{CR} really yields a categorical action of $\fraks\frakl'_I$ on the unipotent blocks of $\GL_n(q)$.

\smallskip

For the finite classical groups $\GU_n(q),$ $\SO_{2n+1}(q)$ and $\Sp_{2n}(q),$
Dudas, Varagnolo and Vasserot \cite{DVV,DVV2} constructed an $\fraks\frakl'_I$-categorification on unipotent blocks using a similar approach and proved that
Brou\'e's conjecture is true
for the unipotent blocks of these groups at linear primes.
Inspired by the reduction theorem of Bonnaf\'{e}-Dat-Rouquier \cite[Theorem 7.7]{BDR17},
% which reduces many modular representation problems of all blocks of finite reductive groups to the isolated blocks of their Levi subgroups,
Liu and two of the authors \cite{LLZ} constructed a categorical action of a Kac-Moody algebra $\fraks\frakl'_I\oplus \fraks\frakl'_{I'}$ on isolated blocks of $\SO_{2n+1}(q)$ which generalizes the previous works \cite{CR, DVV, DVV2} and confirmed
Brou\'e's conjecture for all the blocks of $\SO_{2n+1}(q)$ at linear primes.

\smallskip

However, all these works rely on
the result of Fong-Srinivasan \cite{FS82,FS89} about
the block theory of classical groups with connected center. Due to the nonconnectedness of the center, the block distribution of characters of $\Sp_{2n}(q)$ and $\O_{2n}^{\pm}(q)$ is not completely determined. This is an essential difficulty to generalize the above construction to these groups, since one needs to know the coincidence of the weight space decomposition and the block decomposition in order to get the desired categorical action from Rouquier's approach.

\smallskip
We overcome the above difficulty using a new method of categorical type $A$ or affine type $A$ Kac-Moody algebras via Heisenberg categories. This method is due to Brundan, Savage and Webster \cite{BSW2}. The Heisenberg category $\Heis_m(z)$ of central charge $m\in \bbZ$ is a monoidal category constructed from degenerate affine Hecke algebras or affine Hecke algebras  and it comes in two forms, degenerate or quantum, depending on the parameter $z$. It is called degenerate if $z=0$ and quantum if $z\neq 0$. They are introduced respectively by Brundan \cite{B18} and Brundan, Savage and Webster \cite{BSW1}, generalizing previous works of Khovanov \cite{K} in the degenerate case $m=-1$, Licata and Savage \cite{LS} in the quantum case $m=-1$ and Mackaay and Savage \cite{MS18} in the degenerate case $m<0$.

\smallskip

In our situation, to construct the desired Kac-Moody action on representations of $\O_{2n+1}(q),$ $\Sp_{2n}(q)$ and $\O_{2n}^{\pm}(q)$, we consider
a ``double version'' of the quantum Heisenberg  category, which is a symmetric product of two quantum  Heisenberg categories. We prove that this category acts on the category of representations of these groups, and using similar ideas as in \cite{BSW2}, we deduce a categorical action of Kac-Moody algebra $\fraks\frakl'_I\oplus \fraks\frakl'_{I'}$, which is a ``double version'' of a $\fraks\frakl'_I$-categorification and has been observed in the previous work \cite{LLZ} by Liu and two of the authors. Note that our construction is valid over any coefficient ring $R$ for representations of $G$. Moreover, it applies also to all the previous situations as well, and provides a uniform way to construct categorical $\fraks\frakl'_I\oplus \fraks\frakl'_{I'}$-actions on the whole category of representations of  finite classical groups $\O_{2n+1}(q),$
 $\Sp_{2n}(q)$ and $\O^{\pm}_{2n}(q)$.

\smallskip

Furthermore, the categorical double quantum Heisenberg action gives rise to some new invariants  $\bbO^+(u)(-)$ and $\bbO^-(v)(-)$ on the representation category. In this paper, we use them together with the uniform projection to give a complete classification of irreducible modules of
finite classical groups $\Sp_{2n}(q)$ and $\O^{\pm}_{2n}(q)$ in characteristic zero, hence completing the classification of Lusztig.
\smallskip

Now, let us discuss the main content of the paper.

\smallskip

\subsection{Categorical actions}

Throughout this paper,
we always assume that $\ell$ is a prime, $K$ is a field of characteristic 0, $\calO$ is a complete discrete valuation ring and $k$ is a field of characteristic $\ell$. Let $R=K,$ $\calO $ or $k$ such that $q(q-1)$ is invertible in $R$ and $q^{1/2},\sqrt{-1}\in R$.
Let $V$ be a finite-dimensional symplectic or orthogonal space over the field $\bbF_q$ of
odd characteristic $p$.
Let $G_n$ be one of the families of finite classical groups among $\O_{2n+1}(q),\Sp_{2n}(q)$, $\O_{2n}^{+}(q)$ and $\O_{2n}^{-}(q)$ such that $G_n$ is an isometry group of symplectic or orthogonal space $V$.
Using the tower of natural inclusions of groups
$$\cdots \subset G_{n-1}\subset G_n \subset  G_{n+1} \subset \cdots,$$
one can form the abelian categories
$$RG_\bullet\mod:=\bigoplus_{n\in \bbN^*}RG_n\mod$$
of representations
of all the  groups $G_n$ with $n$ running over
 $\bbN^*=\mathbb{Z}_{> 0}$ in the
 $\O_{2n}^{-}(q)$ case  and  $\bbN^*=\mathbb{Z}_{\geqslant 0}$ in all the other cases.
Based on the construction of the representation datum by Liu and two of the authors in \cite{LLZ}, we proved the following theorem.

\begin{thmA}
Let $R$ be one of $K$, $\mathcal O$ and $k$, then there is a strict $R$-linear monoidal functor
$$\Psi:\,\Heis_{-2}(z^+,t^+)\odot \Heis_{-2}(z^-,t^-)\to \mathcal{E}nd(RG_\bullet\mod).$$
\end{thmA}

 Here $\Heis_{-2}(z^+,t^+)\odot \Heis_{-2}(z^-,t^-)$ is a symmetric product of two quantum  Heisenberg categories $\Heis_{-2}(z^+,t^+)$ and $\Heis_{-2}(z^-,t^-)$, which is called  \emph{double quantum Heisenberg category} in this article (see \S \ref{sec:doublequantumtoKac}), and
$\mathcal{E}nd( RG_\bullet\mod)$ is the strict monoidal category consisting of
endofunctors and natural transformations of $RG_\bullet\mod$. The parameters $z^\pm$ and $t^\pm$ are defined at the beginning of \S \ref{Chap:doubleHeisonfiniteclass}.
\smallskip
%Let $E^+=\bigoplus_{i\in I_+}E^+_i$ and  $F^+=\bigoplus_{i\in I_+}F^+_i$ (resp. $E^-=\bigoplus_{i'\in I_-}E^-_{i'}$ and $F^-=\bigoplus_{i'\in I_-}F^-_{i'}$) be the decomposition of the functors into generalized
%$i$-eigenspaces for $X^+$ (resp. generalized
%$i'$-eigenspaces for $X^-$). Then $$\{[E^+_i], [F^+_i],[E^-_{i'}], [F^-_{i'}]\}_{i\in I_+,i'\in I_-}$$ act as the Chevalley generators of $\frakg=\fraks\frakl_{I_+}\oplus \fraks\frakl_{I_-}$ on
%the Grothendieck group $[\scrQU_R]$ of $\scrQU_R$, where $\frakg=\fraks\frakl_{I_+}\oplus \fraks\frakl_{I_-}$ is a direct sum of two Kac-Moody Lie algebra of type $A$.

In \cite{HL},  Huang, Shen and the first author proved that the categorical action of a
symmetric product
$\Heis_{-2}(z^+,t^+)\odot \Heis_{-2}(z^-,t^-)$ over a field
 can be strengthened to a categorical action of a Kac-Moody algebra $\frakg$ with disjoint union quivers
 of type $A$  by generalizing the corresponding results in \cite{BSW2,LLZ}, where $\frakg$  is uniquely determined by this categorical action (see \S \ref{sec:doublequantumtoKac}).
 Using Lusztig's Jordan decomposition and Howlett--Lehrer theory, we can determine the associated  Kac-Moody  algebra $\frakg=\fraks\frakl'_{I_+}\oplus \fraks\frakl'_{I_-}$ for the categorical double quantum Heisenberg action in Theorem A, which is a direct sum of two Kac-Moody  algebras of type $A$, where $$I_+=q^\bbZ\sqcup -q^\bbZ~~\text{and}~~ I_-=q^\bbZ\sqcup -q^\bbZ$$ as subsets of $R$ (see Proposition \ref{Prop:I+-}).
 The construction in \cite{HL} endows $RG\mod$ with a structure of $\frakU(\frakg)$-categorification, so we have the following corollary.

\begin{thmB}
Let $R=K$ or $k$, then there is a strict $R$-linear monoidal functor
$$\Psi':\,\frakU(\fraks\frakl'_{I_+}\oplus \fraks\frakl'_{I_-})\to \mathcal{E}nd(RG_\bullet\mod).$$

\end{thmB}

In fact, the above construction also applies to $\GL_n(q)$ and $\GU_n(q)$ by a similar argument.
In particular, for $R=k$, we get this categorical action without using the basic set properties for unipotent or quadratic unipotent modules of finite reductive groups \cite{FS82, Ge93,GH91} and the classification of blocks of finite classical groups by Fong and Srinivasan \cite{FS82, FS89}.
%This categorification result would be very useful for the modular representation theory of finite classical groups.
%The reduction theorem of
%Bonnaf\'{e}-Dat-Rouquier \cite[Theorem 7.7]{BDR17} reduces many conjectures to isolated blocks cases.

\smallskip

The Kac-Moody categorification is also helpful for other problems in modular representation theory of finite classical groups, for example,  modular Harish-Chandra theory (see \cite{DVV,N}) and  calculating the decomposition numbers (see \cite{DN,DN2}).
\smallskip

\subsection{Applications to classifications of irreducible modules in characteristic zero}
One of the advantages of double quantum Heisenberg category is that there are two central
elements $\bbO^+(u)\in \End(\unit)((u^{-1}))$ and $\bbO^-(v)\in \End(\unit)
((v^{-1}))$, where $\unit$ is the identity functor, and $u$ and $v$ are two indeterminates (see \S\ref{sec:KactoHeis}). Let $R=K$ or $k$ be an algebraically closed field. For each irreducible module
$L\in RG_\bullet\mod$, we consider the action of $\bbO^+(u)$ and $\bbO^-(v)$ on $\unit (L)=L$. By  Schur's lemma, $\bbO^+(u)(L)\in R((u^{-1}))$ and $\bbO^-(v)(L)\in R((v^{-1})).$ So we get the following functions, and  we call them \emph{colored weight functions}
$$\bbO^+(u)(-):\, \Irr(RG_\bullet\mod)\to R((u^{-1})),\quad \bbO^-(v)(-):\, \Irr(RG_\bullet\mod)\to R((v^{-1})).$$
They contain all the information on the weight space decomposition of the Kac-Moody
categorification.
%Later, in \cite{LPZ}, we will study the  reduction modulo $\ell$ of  colored weight functions
% $\bbO^+(u)(\bullet)$ and $\bbO^-(v)(\bullet)$ and  give a complete classification of blocks of $\Sp_{2n}(q)$ and $\O^{\pm}_{2n}(q)$ in positive characteristic $\ell$.

From now on, let $R=K$ be an algebraically closed field of characteristic 0.    Let $\Irr(G)$ denote the set of irreducible modules of $KG$, or equivalently, the irreducible characters of $G.$ Let $\calV(G)$ be the space of  class function on $G$,
and let $\calV(G)^\sharp$ denote the subspace spanned by Deligne-Lusztig virtual characters (see \S \ref{sub:uniformprojection}).
%Note that $\calV(G)$ is an inner product space with the inner product $\langle,\rangle_G$ given by
%$$
%\langle f,g\rangle_G=\frac{1}{|G|}\sum_{x\in G}f(x)\overline{g(x)}
%$$
%for $f,g\in\calV(G)$, where $a\mapsto \overline{a}$ is a field automorphism which maps roots
%of unity to their inverses.
For $f\in\calV(G)$ the orthogonal projection of $f$ onto $\calV(G)^\sharp$
is denoted by $f^\sharp$ and called the \emph{uniform projection}.
A class function $f\in\calV(G)$ is called \emph{uniform} if $f^\sharp=f$. In fact, for $\GL_n(q)$ and $\GU_n(q)$, since all irreducible characters are uniform, we can distinguish irreducible modules of these two groups by using the uniform projection. However, this is not true for $\O_{2n+1}(q),$
 $\Sp_{2n}(q)$ and $\O^{\pm}_{2n}(q),$ see \cite{P4} for a detailed discussion.
Using extra symmetries constructed in \S\ref{Chap:extrasymmeties},
we can show that colored weight functions together with the uniform projection can distinguish all
irreducible modules which are both  $F^+$-cuspidal and  $F^-$-cuspidal (see \S\ref{sec:coloredweight}) and
then all irreducible modules (see \S\ref{sec:completeinvriants}).
In fact, we have the following theorem.

\begin{thmC} Let $G_n=\O_{2n+1}(q)$, $\Sp_{2n}(q)$ or $\O_{2n}^\pm(q)$ and $\rho,\rho'\in \Irr(G_n).$
\begin{itemize}
\item[$(a)$]
We have $\rho=\rho'$ if and only if the following hold:
\begin{itemize}
\item[$(1)$] $\rho^\sharp=\rho'^\sharp,$
\item[$(2)$] $\mathbb{O}^+(u)(\rho)=\mathbb{O}^+(u)(\rho')$  and  $\mathbb{O}^-(v)(\rho)=\mathbb{O}^-(v)(\rho').$
    \end{itemize}
\item[$(b)$] If $\rho$ and $\rho'$ are quadratic unipotent, then $\rho=\rho'$ if and only if
$\mathbb{O}^+(u)(\rho)=\mathbb{O}^+(u)(\rho')$  and  $\mathbb{O}^-(v)(\rho)=\mathbb{O}^-(v)(\rho').$

\item[$(c)$] If $\rho$ and $\rho'$ are unipotent, then $\rho=\rho'$ if and only if
$\mathbb{O}^+(u)(\rho)=\mathbb{O}^+(u)(\rho').$

\end{itemize}
\end{thmC}

 Theorem C implies that using uniform projection and the
colored weight functions $\mathbb{O}^+(u)(-)$ and $\mathbb{O}^-(v)(-)$,
we can give a complete characterisation for irreducible modules  of finite classical groups $\O_{2n+1}(q),$
 $\Sp_{2n}(q)$ and $\O^{\pm}_{2n}(q)$ in characteristic zero.  Moreover, the colored weight functions $\mathbb{O}^+(u)(-)$ and $\mathbb{O}^-(v)(-)$ are complete invariants of quadratic unipotent modules
 and the colored weight function $\mathbb{O}^+(u)(-)$
 is a complete invariant of unipotent modules.
 \smallskip

 In order to prove Theorem C, we try to compare the Grothendieck group of $KG_\bullet\mod$ with a direct sum of tensor products of Fock space representations of the Kac-Moody algebra $\fraks\frakl'_{I_+}\oplus\fraks\frakl'_{I_-}.$ Since the action is integrable, it is enough to determine the highest weight vectors which are given by the classes of irreducible modules which are both $F^+$-cuspidal and $F^-$-cuspidal. Using the categorical Kac-Moody action and the extra symmetries, we can determine these highest weights up to some signatures. See Theorem \ref{TheoremC} for more details.
\smallskip

\subsection{Theta correspondence}
In the last part of the paper, we study an interesting connection between theta correspondence and the categorical Kac-Moody actions on finite classical groups.
\smallskip

Let us introduce the theta correspondence over finite fields now.
Let $(G,G')$ be a reductive dual pair over a finite field $\bbF_q$ of odd characteristic.
Restricting the  Weil character (\cite{Ger}) to $(G,G')$ with respect to
a nontrivial additive character $\psi$ of $\bbF_q$,
we have the \emph{Weil character} $\omega^\psi_{G,G'}$,
which has the non-negative integral decomposition
\[
\omega^\psi_{G,G'}=\sum_{\rho\in\Irr(G), \rho'\in \Irr(G')}m_{\rho,\rho'}\rho\otimes\rho'.
\]
We say that $(\rho,\rho')$ \emph{occurs} in the \emph{theta correspondence} (or \emph{Howe correspondence}) if $m_{\rho,\rho'}\neq 0$,
i.e., there is a relation
\begin{align*}
\Theta^\psi_{G,G'}=\{\,(\rho,\rho')\in\Irr(G)\times\Irr(G')\mid m_{\rho,\rho'}\neq 0\,\}
\end{align*}
between $\Irr(G)$ and $\Irr(G')$.
The finite theta correspondence has been studied in \cite{H, S, AM, AMR, P1, P2, P3, P4, P5, LW, MQZ}.
%In \cite{P2,P3}, Pan described the correspondence explicitly for symplectic-even orthogonal dual pair in terms of
%Lusztig parametrization, which confirms a conjecture by  Aubert, Michel and Rouquier \cite{AMR}.
%Pan in his another paper \cite{P5} used the compatibility of the Lusztig's Jordan decomposition and the Howe correspondence to prove the generalized preservation principle.
%In \cite{MQZ}, Ma, Qiu and Zou study the generic Hecke modules arising from theta
%correspondence between certain Harish-Chandra series and  provide new proofs
%of the generalized preservation principle and the conjecture by Aubert, Michel and Rouquier.
\smallskip

We use the theta correspondence to remove the ambiguity of signatures in Theorem \ref{TheoremC}, hence we obtain a complete description of the action of Kac-Moody algebra $\fraks\frakl'_{I_+}\oplus \fraks\frakl'_{I_-}$ on the Grothendieck group of $KG_\bullet\mod$ (see \S \ref{chp:thetaforclassical}).
The main idea is the following.
First, we note that theta cuspidal (as defined in \cite[\S 1.8]{MQZ}) is equivalent to being $F^+$-cuspidal or $F^-$-cuspidal;
 see Lemma \ref{Lem:thetacuspidal}. Then
 extra symmetries constructed in \S \ref{Chap:extrasymmeties} reduce this problem to theta cuspidal modules.
By the work of Pan \cite{P3}, together with Adams and Moy \cite{AM}, the theta correspondence for the first occurrence of theta cuspidal modules can be determined up to a twist of the sign character.
Finally, the signatures in colored weight functions of theta cuspidal modules can be obtained using a formula from the work of  Ma, Qiu and Zou \cite{MQZ}.

\smallskip

This provides a new approach, which is  of characteristic 0,  to determine the Kac-Moody action on the Grothendieck group of representation of finite classical groups, see  Theorem \ref{Thm:weightfunctions} and Theorem \ref{thm:final}.
 Note that all the previous works, such as
\cite{DVV, DVV2, LLZ} for unipotent or quadratic unipotent modules, use a different approach which relies on the Brauer tree theory of finite classical groups \cite{FS90} in positive characteristic.

\smallskip

It's worth mentioning that the categorical Kac-Moody action and the theta correspondence for finite classical groups are deeply connected. In fact, by \cite{DVV2}, the $\fraks\frakl'_{I_+}$-action on the Grothendieck group of unipotent modules of finite symplectic  groups was known, then one can recover the symplectic-even orthogonal theta correspondence for unipotent modules by \cite[Theorem 1.3]{MQZ}, see Remark \ref{connections}.

\smallskip

\subsection{Future work}
In our next paper \cite{LPZ}, we will use the tools developed here to study modular representations of $\Sp_{2n}(q)$ and $\O^{\pm}_{2n}(q)$. More precisely, using the reduction modulo $\ell$ of colored weight functions $\bbO^+(u)(-)$ and $\bbO^-(v)(-)$, we will complete the determination of the $\ell$-block distribution of characters of $\Sp_{2n}(q)$ and $\O^{\pm}_{2n}(q)$. Furthermore, based on the Kac-Moody categorification in Theorem B, we will prove that Brou\'e's conjecture is also true
for the remaining finite classical groups $\Sp_{2n}(q)$ and $\O^{\pm}_{2n}(q)$ at linear primes.

\subsection{Organization of the paper}
The paper is organized as follows.
 In Section \ref{Chap:pre}, we collect basic concepts and notations, which
 include quantum Heisenberg category,
the symmetric product of monoidal categories, double
quantum Heisenberg category, etc.
In Section \ref{Chap:doubleHeisonfiniteclass}, we construct a categorical double
quantum Heisenberg action on  $\bigoplus_{n\in \bbN^*}RG_n\mod$
for groups $G_n=\O_{2n+1}(q)$, $\Sp_{2n}(q)$ and $\O_{2n}^{\pm}(q)$ with $R\in (K,\calO,k)$ which is Theorem A
 and  a categorical Kac-Moody action  on $\bigoplus_{n\in \bbN^*}RG_n\mod$
 with $R=K$ or $k$.
In Section  \ref{Chap:extrasymmeties}, we study the
extra symmetries of the above categorical double
quantum Heisenberg action and categorical Kac-Moody action.
In Section \ref{Chap:completeinv}, we recall some basic facts on the classifications of irreducible characters of finite classical groups by Lusztig and prove Theorem B.  We also show that colored weight functions $\bbO^+(u)(-)$, $\bbO^-(v)(-)$
and uniform projection can distinguish all irreducible modules, and prove Theorem C.
In Section \ref{chp:thetaforclassical}, we introduce the theta correspondence and
 then we show that using theta correspondence we can determine the  Kac-Moody action on the Grothendieck group $[KG_\bullet\mod]$.
 \smallskip

While we were writing this paper, Sergio David C\'{i}a Luvecce informed us that he
also constructed a Kac-Moody categorification for the unipotent blocks of $\O^\pm_{2n}(q)$
by a different method.

\section{Preliminaries}\label{Chap:pre}

In this section, we collect some preliminaries about
quantum Heisenberg categories, symmetric product of monoidal categories and double quantum Heisenberg categories.
\smallskip
\subsection{Algebras and categories\label{sec:rings-cat}}
An \emph{$R$-category} $\scrC$ is an additive category enriched over the tensor
category of $R$-modules. %We shall assume that all functors on $\scrC$ are $R$-linear, and
As usual, we write $EF$ or $E\circ F$ for a composition of functors $E$ and $F$,
and $\psi\circ\phi$ for a composition of
morphisms of functors (or natural transformations) $\psi$ and $\phi$.
Also, we denote by $\unit=\unit_\scrC$ the identity functor on $\scrC$,
and by  $1_F$ or $\id_F$  the identity element in the endomorphism
ring $\End(F)$.
\smallskip

Let $\scrC$ be an abelian $R$-category.
We denote by $[\scrC]$ the complexified Grothendieck group of $\scrC$,
 by $[M]$ the  class of  an object $M$ of $\scrC$ in $[\scrC]$,
 and by $[F]$ the linear map on $[\scrC]$ induced by an exact endofunctor $F$ of $\scrC$.
\smallskip

If the Hom spaces of $\scrC$ are finitely generated over $R$, then
the category $\scrC$ is called Hom-finite. In this case, we set $\scrH(M)=
\End_\scrC(M)^\text{op}$ for an object $M\in\scrC$,
so that $\scrH(M)$ is an $R$-algebra which is finitely generated as an $R$-module.
An $R$-category $\scrC$ is a
{\em locally finite abelian category} if
$\scrC$ is abelian, all objects are of finite length, and Hom-finite.
\smallskip

Now let $A$ be an $R$-algebra that is free and finitely generated over $R$.
Let $\scrC=A\mod$. We write $\Irr(\scrC)$ or  $\Irr(A)$ for
the set of isomorphism classes of simple objects of $\scrC$.
%If there is a ring homomorphism from $R$ to $S$ then
%we have $SA=S\otimes_R A$ and the $S$-category $S\scrC=SA\mod$.
%Given another $R$-category $\scrC'$ as above and an exact ($R$-linear)
%functor
%$F:\scrC\to\scrC'$, then $F$ is represented by a projective object $P\in\scrC$, i.e., $F= \Hom_{\scrC}(P,\bullet):\scrC\to \scrC'.$
%We set $SF=\Hom_{S\scrC}(SP,\bullet):S\scrC\to S\scrC'$.
%
%\smallskip
%
%For a finite group $G$,
%the group ring of $G$ over $R$ is denoted by $RG$.
% If  $R$ is not a field, an $RG$-module which is free as an $R$-module
% will be called an \emph{$RG$-lattice}.

\smallskip
We will use the string calculus for strict monoidal categories
and strict 2-categories
as explained in \cite[Chapter 2]{TV}.
Let $\scrA$ be a strict $R$-linear monoidal category. A (strict)
module category over $\scrA$
is a $R$-linear category $\scrC$ together with a strict $R$-linear monoidal functor
$\Phi:\scrA \rightarrow \mathcal{E}nd(\scrC)$, where
$\mathcal{E}nd(\scrC)$ denotes the strict $R$-linear monoidal
category with objects that are $R$-linear endofunctors of $\scrC$ and morphisms
that are natural transformations.
We usually suppress the monoidal functor $\Phi$, using the same notation
$f:E \rightarrow F$ both for
a morphism in $\scrA$ and for
the natural transformation
between endofunctors of $\scrC$ that is its image under $\Phi$.
The evaluation $f_V:EV \rightarrow FV$
of this natural transformation
on an object $V \in \scrC$ will be represented diagrammatically by drawing a
 line labeled by
$V$ on the right-hand side of the usual string diagram
for $f$:
$$
\mathord{
\begin{tikzpicture}[baseline = 0]
	\draw[-,darkg,thick] (0.08,-.3) to (0.08,.3);
	\draw[-] (-.4,-.3) to (-.4,-.14);
	\draw[-] (-.4,.3) to (-.4,.14);
      \draw (-.4,0) circle (4pt);
   \node at (-.4,0) {$\scriptstyle{f}$};
   \node at (-.4,-.47) {$\scriptstyle{E}$};
   \node at (.08,-.47) {$\darkg\scriptstyle{V}$};
   \node at (-.4,.47) {$\scriptstyle{F}$};
\end{tikzpicture}
}\:.
$$
This line represents the identity endomorphism of the object $V$.
\smallskip

\subsection{Quantum Heisenberg category}\label{sec:quantumHeis}\hfill\\

%In \cite{BSW1}, Brundan-Savage-Webster  defined a monoidal category $\Heis_k(z,t)$, called the quantum Heisenberg category, which is a quantum analogy of the degenerate Heisenberg category defined in \cite{B18}. In this paper, we focus on the representation theory of finite classical groups, so we only use the quantum Heisenberg category. For the convenience of readers,
We recall one of the definitions of the quantum Heisenberg category in \cite{BSW1}. Let $q^{1/2}\in R$, $t\in R,$
 $z=q^{1/2}-q^{-1/2}$ and $k\in \bbZ.$
\begin{definition}[\cite{BSW1}] The {\em quantum Heisenberg category} $\Heis_k(z,t)$ is the
strict $R$-linear monoidal category
generated by objects
$F=\up$ and $E=\down$
and the following
morphisms:
\begin{align}\label{qHgens}
\mathord{
\begin{tikzpicture}[baseline = 0]
	\draw[->] (0.08,-.3) to (0.08,.4);
      \node at (0.08,0.05) {$\dott$};
\end{tikzpicture}
}
&:F\rightarrow F,
&\mathord{
\begin{tikzpicture}[baseline = 1mm]
	\draw[<-] (0.4,0.4) to[out=-90, in=0] (0.1,0);
	\draw[-] (0.1,0) to[out = 180, in = -90] (-0.2,0.4);
\end{tikzpicture}
}&:\unit\rightarrow E\otimes F
\:,
&\mathord{
\begin{tikzpicture}[baseline = 1mm]
	\draw[<-] (0.4,0) to[out=90, in=0] (0.1,0.4);
	\draw[-] (0.1,0.4) to[out = 180, in = 90] (-0.2,0);
\end{tikzpicture}
}&:F\otimes E\rightarrow\unit\:,\\
\mathord{
\begin{tikzpicture}[baseline = 0]
	\draw[->] (0.28,-.3) to (-0.28,.4);
	\draw[-,white,line width=4pt] (-0.28,-.3) to (0.28,.4);
	\draw[->] (-0.28,-.3) to (0.28,.4);
\end{tikzpicture}
}&:F\otimes F \rightarrow F\otimes F
\:,&
\mathord{
\begin{tikzpicture}[baseline = 1mm]
	\draw[-] (0.4,0.4) to[out=-90, in=0] (0.1,0);
	\draw[->] (0.1,0) to[out = 180, in = -90] (-0.2,0.4);
\end{tikzpicture}
}&:\unit\rightarrow F\otimes E
\:,
&
\mathord{
\begin{tikzpicture}[baseline = 1mm]
	\draw[-] (0.4,0) to[out=90, in=0] (0.1,0.4);
	\draw[->] (0.1,0.4) to[out = 180, in = 90] (-0.2,0);
\end{tikzpicture}
}&:E \otimes F\rightarrow\unit\label{mango}.
\end{align}
The generators $\begin{tikzpicture}[baseline = -1mm]
	\draw[->] (0.08,-.2) to (0.08,.2);
      \node at (0.08,0) {$\dott$};
\end{tikzpicture}$ and
$\begin{tikzpicture}[baseline = -1mm]
	\draw[->] (0.2,-.2) to (-0.2,.2);
	\draw[-,white,line width=4pt] (-0.2,-.2) to (0.2,.2);
	\draw[->] (-0.2,-.2) to (0.2,.2);
\end{tikzpicture}$ are required to be invertible.
The invertibility of the dot
means that now it makes sense to label dots by an arbitrary integer,
rather than just by $n \in \bbN$.
We denote the inverse of $\begin{tikzpicture}[baseline = -1mm]
	\draw[->] (0.2,-.2) to (-0.2,.2);
	\draw[-,white,line width=4pt] (-0.2,-.2) to (0.2,.2);
	\draw[->] (-0.2,-.2) to (0.2,.2);
\end{tikzpicture}$
by
\begin{align*}
\mathord{
\begin{tikzpicture}[baseline = -.5mm]
	\draw[->] (-0.28,-.3) to (0.28,.4);
	\draw[-,line width=4pt,white] (0.28,-.3) to (-0.28,.4);
	\draw[->] (0.28,-.3) to (-0.28,.4);
\end{tikzpicture}
}&:F\otimes F\rightarrow F \otimes F
\:.
\end{align*}
%Thus, we have
%\begin{align}
%\mathord{
%\begin{tikzpicture}[baseline = -1mm]
%	\draw[-] (0.28,-.6) to[out=90,in=-90] (-0.28,0);
%	\draw[->] (-0.28,0) to[out=90,in=-90] (0.28,.6);
%	\draw[-,line width=4pt,white] (-0.28,-.6) to[out=90,in=-90] (0.28,0);
%	\draw[-] (-0.28,-.6) to[out=90,in=-90] (0.28,0);
%	\draw[-,line width=4pt,white] (0.28,0) to[out=90,in=-90] (-0.28,.6);
%	\draw[->] (0.28,0) to[out=90,in=-90] (-0.28,.6);
%\end{tikzpicture}
%}&=
%\mathord{
%\begin{tikzpicture}[baseline = -1mm]
%	\draw[->] (0.18,-.6) to (0.18,.6);
%	\draw[->] (-0.18,-.6) to (-0.18,.6);
%\end{tikzpicture}
%}
%=
%\mathord{
%\begin{tikzpicture}[baseline = -1mm]
%	\draw[->] (0.28,0) to[out=90,in=-90] (-0.28,.6);
%	\draw[-,line width=4pt,white] (-0.28,0) to[out=90,in=-90] (0.28,.6);
%	\draw[->] (-0.28,0) to[out=90,in=-90] (0.28,.6);
%	\draw[-] (-0.28,-.6) to[out=90,in=-90] (0.28,0);
%	\draw[-,line width=4pt,white] (0.28,-.6) to[out=90,in=-90] (-0.28,0);
%	\draw[-] (0.28,-.6) to[out=90,in=-90] (-0.28,0);
%\end{tikzpicture}
%}\:.
%\end{align}
We also introduce the sideways crossings, both positive and negative,
\begin{align*}
\mathord{
\begin{tikzpicture}[baseline = -.5mm]
	\draw[->] (-0.28,-.3) to (0.28,.4);
	\draw[line width=4pt,white,-] (0.28,-.3) to (-0.28,.4);
	\draw[<-] (0.28,-.3) to (-0.28,.4);
\end{tikzpicture}
}&:=
\mathord{
\begin{tikzpicture}[baseline = 0]
	\draw[->] (0.3,-.5) to (-0.3,.5);
	\draw[line width=4pt,-,white] (-0.2,-.2) to (0.2,.3);
	\draw[-] (-0.2,-.2) to (0.2,.3);
        \draw[-] (0.2,.3) to[out=50,in=180] (0.5,.5);
        \draw[->] (0.5,.5) to[out=0,in=90] (0.8,-.5);
        \draw[-] (-0.2,-.2) to[out=230,in=0] (-0.6,-.5);
        \draw[-] (-0.6,-.5) to[out=180,in=-90] (-0.85,.5);
\end{tikzpicture}
}\:,
&\mathord{
\begin{tikzpicture}[baseline = -.5mm]
	\draw[<-] (0.28,-.3) to (-0.28,.4);
	\draw[line width=4pt,white,-] (-0.28,-.3) to (0.28,.4);
	\draw[->] (-0.28,-.3) to (0.28,.4);
\end{tikzpicture}
}&:=
\mathord{
\begin{tikzpicture}[baseline = 0]
	\draw[-] (-0.2,-.2) to (0.2,.3);
	\draw[-,line width=4pt,white] (0.3,-.5) to (-0.3,.5);
	\draw[->] (0.3,-.5) to (-0.3,.5);
        \draw[-] (0.2,.3) to[out=50,in=180] (0.5,.5);
        \draw[->] (0.5,.5) to[out=0,in=90] (0.8,-.5);
        \draw[-] (-0.2,-.2) to[out=230,in=0] (-0.6,-.5);
        \draw[-] (-0.6,-.5) to[out=180,in=-90] (-0.85,.5);
\end{tikzpicture}
}\:,\\
\mathord{
\begin{tikzpicture}[baseline = -.5mm]
	\draw[<-] (-0.28,-.3) to (0.28,.4);
	\draw[line width=4pt,white,-] (0.28,-.3) to (-0.28,.4);
	\draw[->] (0.28,-.3) to (-0.28,.4);
\end{tikzpicture}
}&:=
\mathord{
\begin{tikzpicture}[baseline = 0]
	\draw[-] (-0.2,.2) to (0.2,-.3);
	\draw[-,line width=4pt,white] (0.3,.5) to (-0.3,-.5);
	\draw[<-] (0.3,.5) to (-0.3,-.5);
        \draw[-] (0.2,-.3) to[out=130,in=180] (0.5,-.5);
        \draw[-] (0.5,-.5) to[out=0,in=270] (0.8,.5);
        \draw[-] (-0.2,.2) to[out=130,in=0] (-0.5,.5);
        \draw[->] (-0.5,.5) to[out=180,in=-270] (-0.8,-.5);
\end{tikzpicture}
}\:,&
\mathord{
\begin{tikzpicture}[baseline = -.5mm]
	\draw[->] (0.28,-.3) to (-0.28,.4);
	\draw[line width=4pt,white,-] (-0.28,-.3) to (0.28,.4);
	\draw[<-] (-0.28,-.3) to (0.28,.4);
\end{tikzpicture}
}&:=
\mathord{
\begin{tikzpicture}[baseline = 0]
	\draw[<-] (0.3,.5) to (-0.3,-.5);
	\draw[-,line width=4pt,white] (-0.2,.2) to (0.2,-.3);
	\draw[-] (-0.2,.2) to (0.2,-.3);
        \draw[-] (0.2,-.3) to[out=130,in=180] (0.5,-.5);
        \draw[-] (0.5,-.5) to[out=0,in=270] (0.8,.5);
        \draw[-] (-0.2,.2) to[out=130,in=0] (-0.5,.5);
        \draw[->] (-0.5,.5) to[out=180,in=-270] (-0.8,-.5);
\end{tikzpicture}
}\:,
\end{align*}
and the $(+)$-bubbles
\begin{align*}
\mathord{
\begin{tikzpicture}[baseline = 1.25mm]
  \draw[->] (0.2,0.2) to[out=90,in=0] (0,.4);
  \draw[-] (0,0.4) to[out=180,in=90] (-.2,0.2);
\draw[-] (-.2,0.2) to[out=-90,in=180] (0,0);
  \draw[-] (0,0) to[out=0,in=-90] (0.2,0.2);
   \node at (0.5,0.2) {$\scriptstyle{n-k}$};
   \node at (0,0.2) {$+$};
\end{tikzpicture}}
&:=
\left\{\begin{array}{ll}
\mathord{
\begin{tikzpicture}[baseline = 1.25mm]
  \draw[->] (0.2,0.2) to[out=90,in=0] (0,.4);
  \draw[-] (0,0.4) to[out=180,in=90] (-.2,0.2);
\draw[-] (-.2,0.2) to[out=-90,in=180] (0,0);
  \draw[-] (0,0) to[out=0,in=-90] (0.2,0.2);
   \node at (0.2,0.2) {$\dott$};
   \node at (0.6,0.2) {$\scriptstyle{n-k}$};
\end{tikzpicture}}
\hspace{51.8mm}&\text{if $k<n$,}\\
t^{n+1} z^{n-1}\det\left(
\!\mathord{
\begin{tikzpicture}[baseline = 1.25mm]
  \draw[<-] (0,0.4) to[out=180,in=90] (-.2,0.2);
  \draw[-] (0.2,0.2) to[out=90,in=0] (0,.4);
 \draw[-] (-.2,0.2) to[out=-90,in=180] (0,0);
  \draw[-] (0,0) to[out=0,in=-90] (0.2,0.2);
   \node at (-0.2,0.2) {$\dott$};
   \node at (-0.95,0.2) {$\scriptstyle{r-s+k+1}$};
\end{tikzpicture}
}\,
\right)_{r,s=1,\dots,n},&\text{if $k \geq n > 0$,}\\
tz^{-1} 1_\unit&\text{if $k \geq n=0$,}\\
0&\text{if $k \geq n < 0$,}
\end{array}\right.\\
\mathord{
\begin{tikzpicture}[baseline = 1.25mm]
  \draw[<-] (0,0.4) to[out=180,in=90] (-.2,0.2);
  \draw[-] (0.2,0.2) to[out=90,in=0] (0,.4);
 \draw[-] (-.2,0.2) to[out=-90,in=180] (0,0);
  \draw[-] (0,0) to[out=0,in=-90] (0.2,0.2);
   \node at (0,0.2) {$+$};
   \node at (-0.48,0.2) {$\scriptstyle{n+k}\,\,\,$};
\end{tikzpicture}
}&:=
\left\{
\begin{array}{ll}
\mathord{
\begin{tikzpicture}[baseline = 1.25mm]
  \draw[<-] (0,0.4) to[out=180,in=90] (-.2,0.2);
  \draw[-] (0.2,0.2) to[out=90,in=0] (0,.4);
 \draw[-] (-.2,0.2) to[out=-90,in=180] (0,0);
  \draw[-] (0,0) to[out=0,in=-90] (0.2,0.2);
   \node at (-0.2,0.2) {$\dott$};
   \node at (-0.6,0.2) {$\scriptstyle{n+k}\,\,$};
\end{tikzpicture}
}&\text{if $-k < n$,}\\
(-1)^{n+1} t^{-n-1} z^{n-1}\det\left(\:\mathord{
\begin{tikzpicture}[baseline = 1.25mm]
  \draw[->] (0.2,0.2) to[out=90,in=0] (0,.4);
  \draw[-] (0,0.4) to[out=180,in=90] (-.2,.2);
\draw[-] (-.2,0.2) to[out=-90,in=180] (0,0);
  \draw[-] (0,0) to[out=0,in=-90] (0.2,0.2);
   \node at (0.2,0.2) {$\dott$};
   \node at (0.95,0.2) {$\scriptstyle{r-s-k+1}$};
\end{tikzpicture}
}\right)_{r,s=1,\dots,n}&\text{if $-k\geq n > 0$,}\\
-t^{-1}z^{-1} 1_\unit
&\text{if $-k \geq n=0$,}\\
0&\text{if $-k\geq n < 0$.}
\end{array}\right.
\end{align*}
 Here $\sgn$ denotes the determinant of the matrix.
The other defining relations are as follows:\begin{itemize}

\item affine Hecke relations
\begin{align}\label{AHH1}
\mathord{
\begin{tikzpicture}[baseline = -.5mm]
	\draw[->,thin] (0.28,-.3) to (-0.28,.4);
      \node at (0.165,-0.15) {$\dott$};
	\draw[line width=4pt,white,-] (-0.28,-.3) to (0.28,.4);
	\draw[thin,->] (-0.28,-.3) to (0.28,.4);
\end{tikzpicture}
}\,&=\mathord{
\begin{tikzpicture}[baseline = -.5mm]
	\draw[thin,->] (-0.28,-.3) to (0.28,.4);
	\draw[-,line width=4pt,white] (0.28,-.3) to (-0.28,.4);
	\draw[->,thin] (0.28,-.3) to (-0.28,.4);
      \node at (-0.14,0.23) {$\dott$};
\end{tikzpicture}
}
\:,&
\mathord{
\begin{tikzpicture}[baseline = -.5mm]
	\draw[thin,->] (-0.28,-.3) to (0.28,.4);
      \node at (-0.16,-0.15) {$\dott$};
	\draw[-,line width=4pt,white] (0.28,-.3) to (-0.28,.4);
	\draw[->,thin] (0.28,-.3) to (-0.28,.4);
\end{tikzpicture}
}\,=&
\mathord{
\begin{tikzpicture}[baseline = -.5mm]
	\draw[->,thin] (0.28,-.3) to (-0.28,.4);
	\draw[line width=4pt,white,-] (-0.28,-.3) to (0.28,.4);
	\draw[thin,->] (-0.28,-.3) to (0.28,.4);
      \node at (0.145,0.23) {$\dott$};
\end{tikzpicture}
}\:,&
\mathord{
\begin{tikzpicture}[baseline = -.5mm]
	\draw[->,thin] (0.28,-.3) to (-0.28,.4);
	\draw[line width=4pt,white,-] (-0.28,-.3) to (0.28,.4);
	\draw[thin,->] (-0.28,-.3) to (0.28,.4);
\end{tikzpicture}
}-\mathord{
\begin{tikzpicture}[baseline = -.5mm]
	\draw[thin,->] (-0.28,-.3) to (0.28,.4);
	\draw[line width=4pt,white,-] (0.28,-.3) to (-0.28,.4);
	\draw[->,thin] (0.28,-.3) to (-0.28,.4);
\end{tikzpicture}
}=&
z\:\mathord{
\begin{tikzpicture}[baseline = -.5mm]
	\draw[->,thin] (0.18,-.3) to (0.18,.4);
	\draw[->,thin] (-0.18,-.3) to (-0.18,.4);
\end{tikzpicture}
}\:,\end{align}

\begin{align}\label{AHH2}
\mathord{
\begin{tikzpicture}[baseline = -1mm]
	\draw[-,thin] (0.28,-.6) to[out=90,in=-90] (-0.28,0);
	\draw[->,thin] (-0.28,0) to[out=90,in=-90] (0.28,.6);
	\draw[-,line width=4pt,white] (-0.28,-.6) to[out=90,in=-90] (0.28,0);
	\draw[-,thin] (-0.28,-.6) to[out=90,in=-90] (0.28,0);
	\draw[-,line width=4pt,white] (0.28,0) to[out=90,in=-90] (-0.28,.6);
	\draw[->,thin] (0.28,0) to[out=90,in=-90] (-0.28,.6);
\end{tikzpicture}
}\,=
\mathord{
\begin{tikzpicture}[baseline = -1mm]
	\draw[->,thin] (0.18,-.6) to (0.18,.6);
	\draw[->,thin] (-0.18,-.6) to (-0.18,.6);
\end{tikzpicture}
}
=&
\mathord{
\begin{tikzpicture}[baseline = -1mm]
	\draw[->,thin] (0.28,0) to[out=90,in=-90] (-0.28,.6);
	\draw[-,line width=4pt,white] (-0.28,0) to[out=90,in=-90] (0.28,.6);
	\draw[->,thin] (-0.28,0) to[out=90,in=-90] (0.28,.6);
	\draw[-,thin] (-0.28,-.6) to[out=90,in=-90] (0.28,0);
	\draw[-,line width=4pt,white] (0.28,-.6) to[out=90,in=-90] (-0.28,0);
	\draw[-,thin] (0.28,-.6) to[out=90,in=-90] (-0.28,0);
\end{tikzpicture}
}\:,&
\mathord{
\begin{tikzpicture}[baseline = -1mm]
	\draw[->,thin] (0.45,-.6) to (-0.45,.6);
        \draw[-,thin] (0,-.6) to[out=90,in=-90] (-.45,0);
        \draw[-,line width=4pt,white] (-0.45,0) to[out=90,in=-90] (0,0.6);
        \draw[->,thin] (-0.45,0) to[out=90,in=-90] (0,0.6);
	\draw[-,line width=4pt,white] (0.45,.6) to (-0.45,-.6);
	\draw[<-,thin] (0.45,.6) to (-0.45,-.6);
\end{tikzpicture}
}
=&
\mathord{
\begin{tikzpicture}[baseline = -1mm]
	\draw[->,thin] (0.45,-.6) to (-0.45,.6);
        \draw[-,line width=4pt,white] (0,-.6) to[out=90,in=-90] (.45,0);
        \draw[-,thin] (0,-.6) to[out=90,in=-90] (.45,0);
        \draw[->,thin] (0.45,0) to[out=90,in=-90] (0,0.6);
	\draw[-,line width=4pt,white] (0.45,.6) to (-0.45,-.6);
	\draw[<-,thin] (0.45,.6) to (-0.45,-.6);
\end{tikzpicture}\,,
}\end{align}

\item biadjoint relations

\begin{align}\label{adj}
\mathord{
\begin{tikzpicture}[baseline = -.8mm]
  \draw[->] (0.3,0) to (0.3,.4);
	\draw[-] (0.3,0) to[out=-90, in=0] (0.1,-0.4);
	\draw[-] (0.1,-0.4) to[out = 180, in = -90] (-0.1,0);
	\draw[-] (-0.1,0) to[out=90, in=0] (-0.3,0.4);
	\draw[-] (-0.3,0.4) to[out = 180, in =90] (-0.5,0);
  \draw[-] (-0.5,0) to (-0.5,-.4);
\end{tikzpicture}
}
=
\mathord{\begin{tikzpicture}[baseline=-.8mm]
  \draw[->] (0,-0.4) to (0,.4);
\end{tikzpicture}
}\:,\quad\quad
\mathord{
\begin{tikzpicture}[baseline = -.8mm]
  \draw[->] (0.3,0) to (0.3,-.4);
	\draw[-] (0.3,0) to[out=90, in=0] (0.1,0.4);
	\draw[-] (0.1,0.4) to[out = 180, in = 90] (-0.1,0);
	\draw[-] (-0.1,0) to[out=-90, in=0] (-0.3,-0.4);
	\draw[-] (-0.3,-0.4) to[out = 180, in =-90] (-0.5,0);
  \draw[-] (-0.5,0) to (-0.5,.4);
\end{tikzpicture}
}
=
\mathord{\begin{tikzpicture}[baseline=-.8mm]
  \draw[<-] (0,-0.4) to (0,.4);
\end{tikzpicture}
}
\:,\quad\quad
\mathord{
\begin{tikzpicture}[baseline = -.8mm]
  \draw[-] (0.3,0) to (0.3,-.4);
	\draw[-] (0.3,0) to[out=90, in=0] (0.1,0.4);
	\draw[-] (0.1,0.4) to[out = 180, in = 90] (-0.1,0);
	\draw[-] (-0.1,0) to[out=-90, in=0] (-0.3,-0.4);
	\draw[-] (-0.3,-0.4) to[out = 180, in =-90] (-0.5,0);
  \draw[->] (-0.5,0) to (-0.5,.4);
\end{tikzpicture}
}
=
\mathord{\begin{tikzpicture}[baseline=-.8mm]
  \draw[->] (0,-0.4) to (0,.4);
\end{tikzpicture}
}\:,\quad\quad
\mathord{
\begin{tikzpicture}[baseline = -.8mm]
  \draw[-] (0.3,0) to (0.3,.4);
	\draw[-] (0.3,0) to[out=-90, in=0] (0.1,-0.4);
	\draw[-] (0.1,-0.4) to[out = 180, in = -90] (-0.1,0);
	\draw[-] (-0.1,0) to[out=90, in=0] (-0.3,0.4);
	\draw[-] (-0.3,0.4) to[out = 180, in =90] (-0.5,0);
  \draw[->] (-0.5,0) to (-0.5,-.4);
\end{tikzpicture}
}
=
\mathord{\begin{tikzpicture}[baseline=-.8mm]
  \draw[<-] (0,-0.4) to (0,.4);
\end{tikzpicture}\:,
}\end{align}

\item bubble relations
\begin{align}
\mathord{
\begin{tikzpicture}[baseline = 1.25mm]
  \draw[<-] (0,0.4) to[out=180,in=90] (-.2,0.2);
  \draw[-] (0.2,0.2) to[out=90,in=0] (0,.4);
 \draw[-] (-.2,0.2) to[out=-90,in=180] (0,0);
  \draw[-] (0,0) to[out=0,in=-90] (0.2,0.2);
   \node at (-0.2,0.2) {$\dott$};
   \node at (-0.6,0.2) {$\scriptstyle{n+k}$};
\end{tikzpicture}
}&=
{\textstyle\frac{\delta_{n,-k} t - \delta_{n,0} t^{-1}}{z}}
%(\delta_{n,-k} t z^{-1} - \delta_{n,0} t^{-1}z^{-1})
1_\unit
\:\text{if $-k \leq n\leq 0$,}
&
%\hspace{7.8mm}
\mathord{
\begin{tikzpicture}[baseline = 1.25mm]
  \draw[->] (0.2,0.2) to[out=90,in=0] (0,.4);
  \draw[-] (0,0.4) to[out=180,in=90] (-.2,0.2);
\draw[-] (-.2,0.2) to[out=-90,in=180] (0,0);
  \draw[-] (0,0) to[out=0,in=-90] (0.2,0.2);
   \node at (0.2,0.2) {$\dott$};
   \node at (0.6,0.2) {$\scriptstyle{n-k}$};
\end{tikzpicture}
}\!&=
{\textstyle\frac{\delta_{n,0} t - \delta_{n,k} t^{-1}}{z}}
1_\unit
\:\:\:\text{if $k \leq n \leq 0$,}
\end{align}

\item Mackey formula relations

\begin{align}
\mathord{
\begin{tikzpicture}[baseline = -.9mm]
	\draw[<-] (-0.28,-.6) to[out=90,in=-90] (0.28,0);
	\draw[-,white,line width=4pt] (0.28,-.6) to[out=90,in=-90] (-0.28,0);
	\draw[-] (0.28,-.6) to[out=90,in=-90] (-0.28,0);
	\draw[->] (-0.28,0) to[out=90,in=-90] (0.28,.6);
	\draw[-,line width=4pt,white] (0.28,0) to[out=90,in=-90] (-0.28,.6);
	\draw[-] (0.28,0) to[out=90,in=-90] (-0.28,.6);
\end{tikzpicture}
}
&=
\mathord{
\begin{tikzpicture}[baseline = -.9mm]
	\draw[->] (0.08,-.6) to (0.08,.6);
	\draw[<-] (-0.28,-.6) to (-0.28,.6);
\end{tikzpicture}
}
+tz
\mathord{
\begin{tikzpicture}[baseline=-.9mm]
	\draw[<-] (0.3,0.6) to[out=-90, in=0] (0,.1);
	\draw[-] (0,.1) to[out = 180, in = -90] (-0.3,0.6);
	\draw[-] (0.3,-.6) to[out=90, in=0] (0,-0.1);
	\draw[->] (0,-0.1) to[out = 180, in = 90] (-0.3,-.6);
\end{tikzpicture}}
\!+z^2\!\sum_{r,s > 0}
\!\!\!\mathord{
\begin{tikzpicture}[baseline = 1mm]
  \draw[<-] (0,0.4) to[out=180,in=90] (-.2,0.2);
  \draw[-] (0.2,0.2) to[out=90,in=0] (0,.4);
 \draw[-] (-.2,0.2) to[out=-90,in=180] (0,0);
  \draw[-] (0,0) to[out=0,in=-90] (0.2,0.2);
   \node at (0,0.2) {$+$};
   \node at (-.57,0.2) {$\scriptstyle{-r-s}$};
\end{tikzpicture}
}
\mathord{
\begin{tikzpicture}[baseline=-.9mm]
	\draw[<-] (0.3,0.6) to[out=-90, in=0] (0,.1);
	\draw[-] (0,.1) to[out = 180, in = -90] (-0.3,0.6);
      \node at (0.44,-0.3) {$\scriptstyle{s}$};
	\draw[-] (0.3,-.6) to[out=90, in=0] (0,-0.1);
	\draw[->] (0,-0.1) to[out = 180, in = 90] (-0.3,-.6);
   \node at (0.27,0.3) {$\dott$};
      \node at (0.27,-0.3) {$\dott$};
   \node at (.43,.3) {$\scriptstyle{r}$};
\end{tikzpicture}}\:,&
\mathord{
\begin{tikzpicture}[baseline = -.9mm]
	\draw[->] (0.28,0) to[out=90,in=-90] (-0.28,.6);
	\draw[<-] (0.28,-.6) to[out=90,in=-90] (-0.28,0);
	\draw[-,line width=4pt,white] (-0.28,0) to[out=90,in=-90] (0.28,.6);
	\draw[-] (-0.28,0) to[out=90,in=-90] (0.28,.6);
	\draw[-,line width=4pt,white] (-0.28,-.6) to[out=90,in=-90] (0.28,0);
	\draw[-] (-0.28,-.6) to[out=90,in=-90] (0.28,0);
\end{tikzpicture}
}
&=\mathord{
\begin{tikzpicture}[baseline = -0.9mm]
	\draw[<-] (0.08,-.6) to (0.08,.6);
	\draw[->] (-0.28,-.6) to (-0.28,.6);
\end{tikzpicture}
}
-t^{-1}z
\mathord{
\begin{tikzpicture}[baseline=-0.9mm]
	\draw[-] (0.3,0.6) to[out=-90, in=0] (0,0.1);
	\draw[->] (0,0.1) to[out = 180, in = -90] (-0.3,0.6);
	\draw[<-] (0.3,-.6) to[out=90, in=0] (0,-0.1);
	\draw[-] (0,-0.1) to[out = 180, in = 90] (-0.3,-.6);
\end{tikzpicture}}
\!+z^2\!\sum_{r,s > 0}\!\!\!
\mathord{
\begin{tikzpicture}[baseline=-0.9mm]
	\draw[-] (0.3,0.6) to[out=-90, in=0] (0,0.1);
	\draw[->] (0,0.1) to[out = 180, in = -90] (-0.3,0.6);
      \node at (-0.4,0.3) {$\scriptstyle{r}$};
      \node at (-0.25,0.3) {$\dott$};
	\draw[<-] (0.3,-.6) to[out=90, in=0] (0,-0.1);
	\draw[-] (0,-0.1) to[out = 180, in = 90] (-0.3,-.6);
   \node at (-0.27,-0.4) {$\dott$};
   \node at (-.45,-.35) {$\scriptstyle{s}$};
\end{tikzpicture}}
\mathord{
\begin{tikzpicture}[baseline = 1mm]
  \draw[->] (0.2,0.2) to[out=90,in=0] (0,.4);
  \draw[-] (0,0.4) to[out=180,in=90] (-.2,0.2);
\draw[-] (-.2,0.2) to[out=-90,in=180] (0,0);
  \draw[-] (0,0) to[out=0,in=-90] (0.2,0.2);
%   \node at (0.2,0.2) {$\dott$};
   \node at (0,0.2) {$+$};
   \node at (.55,0.2) {$\scriptstyle{-r-s}$};
\end{tikzpicture}
}.\label{limb}\end{align}

\item
curl relations
\begin{align}\label{saturday}
\mathord{
\begin{tikzpicture}[baseline = -0.5mm]
	\draw[<-] (0,0.6) to (0,0.3);
	\draw[-] (0.3,-0.2) to [out=0,in=-90](.5,0);
	\draw[-] (0.5,0) to [out=90,in=0](.3,0.2);
	\draw[-] (0.3,.2) to [out=180,in=90](0,-0.3);
	\draw[-] (0,-0.3) to (0,-0.6);
	\draw[-,line width=4pt,white] (0,0.3) to [out=-90,in=180] (.3,-0.2);
	\draw[-] (0,0.3) to [out=-90,in=180] (.3,-0.2);
\end{tikzpicture}
}&=\delta_{k,0}
t^{-1}\:\mathord{
\begin{tikzpicture}[baseline = -0.5mm]
	\draw[<-] (0,0.6) to (0,-0.6);
\end{tikzpicture}
}
\:\:\text{if $k \geq 0$,}
&\mathord{
\begin{tikzpicture}[baseline = -0.5mm]
	\draw[<-] (0,0.6) to (0,0.3);
	\draw[-] (-0.3,-0.2) to [out=180,in=-90](-.5,0);
	\draw[-] (-0.5,0) to [out=90,in=180](-.3,0.2);
	\draw[-] (-0.3,.2) to [out=0,in=90](0,-0.3);
	\draw[-] (0,-0.3) to (0,-0.6);
	\draw[-,line width=4pt,white] (0,0.3) to [out=-90,in=0] (-.3,-0.2);
	\draw[-] (0,0.3) to [out=-90,in=0] (-.3,-0.2);
\end{tikzpicture}
}&=
\delta_{k,0}
t\:\mathord{
\begin{tikzpicture}[baseline = -0.5mm]
	\draw[<-] (0,0.6) to (0,-0.6);
\end{tikzpicture}
}
\:\:\quad\text{if $k \leq 0$}.\end{align}

\end{itemize}

Moreover, the quantum Heisenberg category is strictly
pivotal, i.e.,
the following relations hold:
\begin{align*}
\mathord{
\begin{tikzpicture}[baseline = -.5mm]
	\draw[<-] (0.08,-.3) to (0.08,.4);
      \node at (0.08,0.05) {$\dott$};
\end{tikzpicture}
}:=
\mathord{
\begin{tikzpicture}[baseline = -.5mm]
  \draw[->] (0.3,0) to (0.3,-.4);
	\draw[-] (0.3,0) to[out=90, in=0] (0.1,0.4);
	\draw[-] (0.1,0.4) to[out = 180, in = 90] (-0.1,0);
	\draw[-] (-0.1,0) to[out=-90, in=0] (-0.3,-0.4);
	\draw[-] (-0.3,-0.4) to[out = 180, in =-90] (-0.5,0);
  \draw[-] (-0.5,0) to (-0.5,.4);
   \node at (-0.1,0) {$\dott$};
\end{tikzpicture}
}&=\:
\mathord{
\begin{tikzpicture}[baseline = -.5mm]
  \draw[-] (0.3,0) to (0.3,.4);
	\draw[-] (0.3,0) to[out=-90, in=0] (0.1,-0.4);
	\draw[-] (0.1,-0.4) to[out = 180, in = -90] (-0.1,0);
	\draw[-] (-0.1,0) to[out=90, in=0] (-0.3,0.4);
	\draw[-] (-0.3,0.4) to[out = 180, in =90] (-0.5,0);
  \draw[->] (-0.5,0) to (-0.5,-.4);
   \node at (-0.1,0) {$\dott$};
\end{tikzpicture}
}\:,
&\mathord{
\begin{tikzpicture}[baseline = 0]
	\draw[<-] (0.28,-.3) to (-0.28,.4);\draw[line width=4pt,white,-] (-0.28,-.3) to (0.28,.4);
	\draw[<-] (-0.28,-.3) to (0.28,.4);
\end{tikzpicture}
}
:=
\mathord{
\begin{tikzpicture}[baseline = 0]
	\draw[-] (-0.2,-.2) to (0.2,.3);
        \draw[-] (0.2,.3) to[out=50,in=180] (0.5,.5);
        \draw[->] (0.5,.5) to[out=0,in=90] (0.9,-.5);
        \draw[-] (-0.2,-.2) to[out=230,in=0] (-0.6,-.5);
        \draw[-] (-0.6,-.5) to[out=180,in=-90] (-0.9,.5);
        \draw[-,white,line width=4pt] (0.3,-.5) to (-0.3,.5);
        \draw[<-] (0.3,-.5) to (-0.3,.5);
\end{tikzpicture}
}&=
\mathord{
\begin{tikzpicture}[baseline = 0]
        \draw[->] (0.3,.5) to (-0.3,-.5);
        \draw[-,white,line width=4pt] (-0.2,.2) to (0.2,-.3);
	    \draw[-] (-0.2,.2) to (0.2,-.3);
        \draw[-] (0.2,-.3) to[out=130,in=180] (0.5,-.5);
        \draw[-] (0.5,-.5) to[out=0,in=270] (0.9,.5);
        \draw[-] (-0.2,.2) to[out=130,in=0] (-0.6,.5);
        \draw[->] (-0.6,.5) to[out=180,in=-270] (-0.9,-.5);
\end{tikzpicture}
}\:.
\end{align*}
\end{definition}

\medskip

\subsection{From Heisenberg categorification to Kac-Moody categorification}\label{sec:KactoHeis}\hfill\\

In this section, we will recall some notations and results in \cite{BSW2}.

%The importance of the Heisenberg category is that it affords a new way to construct Kac-Moody categorification of type A. Indeed,
%Brundan-Savage-Webster showed in \cite{BSW1} that  any abelian
%module category over the  Heisenberg category
%satisfying suitable finiteness conditions
%may be viewed as a 2-representation over a corresponding
%Kac-Moody category.
% This gives a way to construct Kac-Moody actions
%which is independent of Rouquier's original approach in \cite{R08} via ``control by
%$K_0$'', and it will be seen later that this method is particularly applicable to the representation theory of finite classical groups.

\begin{definition}A {\em categorical quantum Heisenberg action} on a locally finite category $\mathcal{C}$ is the data
of a strict monoidal functor $\Heis_k(z,t) \rightarrow \mathcal{E}nd( \mathcal{C})$, where
$\mathcal{E}nd( \mathcal{C})$ is the strict monoidal category consisting of
endofunctors and natural transformations.
\end{definition}

In this section, we assume that $R$ is a field.
Now we assume that there exists a categorical quantum Heisenberg action
 on a locally finite category $\mathcal{C}.$
Let $L\in \calC$ be an irreducible object. We can define $m_L(u), n_L(u) \in R[u]$ to
be the (monic) {\em minimal polynomials} of
the endomorphisms
$\mathord{\begin{tikzpicture}[baseline = -1mm]
 	\draw[->] (0.08,-.2) to (0.08,.2);
     \node at (0.08,0) {$\dott$};
 	\draw[-,darkg,thick] (0.38,.2) to (0.38,-.2);
     \node at (0.55,0) {$\darkg\scriptstyle{L}$};
\end{tikzpicture}
}$ and
$\mathord{\begin{tikzpicture}[baseline = -1mm]
 	\draw[<-] (0.08,-.2) to (0.08,.2);
     \node at (0.08,0.02) {$\dott$};
 	\draw[-,darkg,thick] (0.38,.2) to (0.38,-.2);
     \node at (0.55,0) {$\darkg\scriptstyle{L}$};
\end{tikzpicture}
}$, respectively. We use the notation
$\mathord{
\begin{tikzpicture}[baseline = -1.5]
	\draw[-] (0.08,-.15) to (0.08,.3);
	\node at (0.08,0.08) {$\dt$};
	\node at (0.32,.08) {$\scriptstyle x^n$};
\end{tikzpicture}
}$ instead of
$\mathord{
\begin{tikzpicture}[baseline = -1.5]
	\draw[-] (0.08,-.15) to (0.08,.3);
	\node at (0.08,0.08) {$\dt$};
	\node at (0.25,.08) {$\scriptstyle n$};
\end{tikzpicture}
}$ to denote a dot of multiplicity $n$ and
represent linear combinations of monomials
by labelling dots by {polynomials} in $x$ too.
Then
there are {injective} homomorphisms
\begin{align}\label{CRT1}
R[u] / (m_L(u)) &\hookrightarrow \End_{\calC} (EL),
&
R[u] / (n_L(u)) &\hookrightarrow \End_{\calC} (FL),\\
p(u) &
\mapsto \mathord{
\begin{tikzpicture}[baseline = -1mm]
 	\draw[->] (0.08,-.3) to (0.08,.4);
     \node at (0.08,0.05) {$\dott$};
     \node at (-0.3,0.05) {$\scriptstyle p(x)$};
 	\draw[-,darkg,thick] (0.45,.4) to (0.45,-.3);
     \node at (0.45,-.45) {$\darkg\scriptstyle{L}$};
\end{tikzpicture}
},
&p(u) &\mapsto \mathord{
\begin{tikzpicture}[baseline = -1mm]
 	\draw[<-] (0.08,-.3) to (0.08,.4);
     \node at (0.08,0.1) {$\dott$};
     \node at (-0.3,0.1) {$\scriptstyle p(x)$};
 	\draw[-,darkg,thick] (0.45,.4) to (0.45,-.3);
     \node at (0.45,-.45) {$\darkg\scriptstyle{L}$};
\end{tikzpicture}
}.\notag
\end{align}
We assume that all roots of $m_L(u)$ and $n_L(u)$ lie in $R.$
Also, let $\epsilon_i(L)$ and $\phi_i(L)$ denote the multiplicities
of $i \in R$ as a root of the polynomials $m_L(u)$ and $n_L(u)$,
respectively.
By the Chinese remainder theorem, we have
\begin{align}\label{CRTdef}
R[u] / (m_L(u)) &\cong \bigoplus_{i \in R} R[u] \big/ \big((u-i)^{\epsilon_i(L)}\big),&
R[u] / (n_L(u)) &\cong \bigoplus_{i \in R} R[u] \big/
                   \big((u-i)^{\phi_i(L)}\big).
\end{align}
We define $E_i$ and $F_i$ to be the direct
summands of $E$ and $F$ such that $E_i L$ and $F_i L$ are the generalized $i$-eigenspaces of $\mathord{\begin{tikzpicture}
 	\draw[<-] (0.08,-.2) to (0.08,.2);
     \node at (0.08,0.02) {$\dott$};\draw[-,darkg,thick] (0.38,.2) to (0.38,-.2);
     \node at (0.55,0) {$\darkg\scriptstyle{L}$};
\end{tikzpicture}}$ and
$\mathord{\begin{tikzpicture}
 	\draw[->] (0.08,-.2) to (0.08,.2);
     \node at (0.08,0.02) {$\dott$};\draw[-,darkg,thick] (0.38,.2) to (0.38,-.2);
     \node at (0.55,0) {$\darkg\scriptstyle{L}$};
\end{tikzpicture}}$ for any  $L\in \calC$, respectively.
\smallskip

Now we define the {\em spectrum} $I$ of $\calC$
to be the union of the sets of roots of the minimal polynomials
$m_L(u)$ for all irreducible $L \in \calC$.
 By the exactness of $E_i$,
the spectrum $I$ is the set of all $i \in R$ such that
$E_i$ is a non-zero endofunctor of $\mathcal{C}$.
By adjunction, it follows that $I$ is the set of all $i \in R$
such that the endofunctor $F_i$ is non-zero, hence, $I$ could also
be defined as
the union of the sets of roots of the polynomials
$n_L(u)$ for all irreducible $L \in \mathcal{C}$.
The discussion above shows that
\begin{align}
E &=
\bigoplus_{i \in I} E_i,&
F &= \bigoplus_{i \in I} F_i,
\end{align}
where each of the endofunctors
$E_i$ and $F_i$ are non-zero.
\smallskip

It is convenient to
work with all of the bubbles at once in terms of the
generating function. We define the following
generating functions:
\begin{align}\label{summer1}
\anticlock(u) &:= t^{-1}z \sum_{r\in\bbZ}
\mathord{
\begin{tikzpicture}[baseline = 1.25mm]
  \draw[-] (0,0.4) to[out=180,in=90] (-.2,0.2);
  \draw[->] (0.2,0.2) to[out=90,in=0] (0,.4);
 \draw[-] (-.2,0.2) to[out=-90,in=180] (0,0);
  \draw[-] (0,0) to[out=0,in=-90] (0.2,0.2);
   \node at (0,.21) {$+$};
   \node at (0.33,0.2) {$\scriptstyle{r}$};
\end{tikzpicture}
}\: u^{-r}
\in u^k 1_\unit + u^{k-1} \End_{\Heis_k(z,t)}(\unit)[[
u^{-1} ]]
,\\\label{summer2}
\clock(u)&:= -tz \sum_{r\in\Z}
\mathord{
\begin{tikzpicture}[baseline = 1.25mm]
  \draw[<-] (0,0.4) to[out=180,in=90] (-.2,0.2);
  \draw[-] (0.2,0.2) to[out=90,in=0] (0,.4);
 \draw[-] (-.2,0.2) to[out=-90,in=180] (0,0);
  \draw[-] (0,0) to[out=0,in=-90] (0.2,0.2);
   \node at (0,.21) {$+$};
   \node at (-0.33,0.2) {$\scriptstyle{r}$};
\end{tikzpicture}
} \:u^{-r} \in u^{-k}1_\unit +
u^{-k-1}\End_{\Heis_k(z,t)}(\unit)[[
u^{-1} ]].
\end{align}
Then they satisfy \begin{equation}
\anticlock(u)\; \clock(u)
= 1_\unit.
\end{equation}
 The evaluation $\bbO(u)(L): \unit(L) \to \unit(L)$ of this natural transformation $\bbO(u)\in \calE nd(\unit)((u^{-1}))$
on an irreducible object $L\in \calC$ is

\begin{align}\label{jjj}
\bbO(u)(L)&:=
\mathord{
\begin{tikzpicture}[baseline = -1mm]
     \node at (0.08,0) {$\scriptstyle\clock(u)$};
 	\draw[-,darkg,thick] (0.68,.2) to (0.68,-.22);
     \node at (0.68,-.37) {$\darkg\scriptstyle{L}$};
\end{tikzpicture}
}=\left(
\mathord{
\begin{tikzpicture}[baseline = -1mm]
     \node at (0.08,0) {$\scriptstyle\anticlock(u)$};
 	\draw[-,darkg,thick] (0.68,.2) to (0.68,-.22);
     \node at (0.68,-.37) {$\darkg\scriptstyle{L}$};
\end{tikzpicture}
}\right)^{-1}\in \End_{\calC}(L)((u^{-1})).
\end{align}
For an irreducible object $L \in \calC$,
it was showed in \cite[Lemma 4.17]{BSW2} that
$$
\bbO(u)(L) = m_L(u)/n_L(u).
$$
Moreover, the constant terms of
the polynomials $m_L(u)$ and $n_L(u)$ satisfy $t^2 = m_L(0) / n_L(0)$.
\smallskip

It was shown in \cite[Lemma 4.6]{BSW2},
 $i \in I$ if and only if $qi \in I$.
Let $1\neq q \in R^\times$ and  $\scrI$ be a (possibly infinite) subset of $R^\times$.
To the pair $(\scrI,q)$ we  associate a (not necessarily connected)
quiver $\scrI(q)$ (also denoted by $\scrI$) as follows:
\begin{itemize}[leftmargin=8mm]
  \item  the vertex set is $\scrI$, and
  \item there is an arrow $i\to i\cdot q$
if and only if $i, i\cdot q\in \scrI$.
\end{itemize}
Now assume that $\scrI$ is stable under the multiplication by $q$ and $q^{-1}$.
If $q$ is a primitive $e$-th root of unity
then the quiver $\scrI(q)$ is the disjoint union of cyclic quivers  of type $A_{e-1}^{(1)}$,
 while if  $q$ is not a root of unity then  $\scrI(q)$
is the disjoint union of quivers of type $A_\infty$.

\smallskip

The quiver $\scrI(q)$ defines a symmetric generalized Cartan matrix
$A = (a_{ij})_{i,j\in \scrI}$ with

$$
\begin{cases}
a_{ii}= 2 &   \\
a_{ij} =-1 & \text{if}~i \rightarrow j~\text{or}~j \rightarrow i\\
a_{ij}=0 & \text{ otherwise}.
\end{cases}
$$ To this Cartan matrix one can
associate the  Kac-Moody algebra $\fraks\frakl_{\scrI}'$ over $\bbC$, which
has Chevalley generators $e_i,f_i$ for $i\in \scrI$, subject to the usual
relations.
For $i\in\scrI$, let $\alpha_i,$ $\alpha^\vee_i$ be the simple root and coroot
corresponding to $e_i$ and let $\Lambda_i$ be the $i$-th fundamental weight. Recall that
 $\X^\vee = \bigoplus\limits_{i\in \scrI} \bbZ \alpha_i^\vee$ and
$\X=\bigoplus\limits_{i\in \scrI}\bbZ \Lambda_i.$
\smallskip

For an irreducible object $L \in \mathcal{C}$, let
\begin{equation}
\wt(L) := \sum_{i \in I} (\phi_i(L)-\epsilon_i(L)) \Lambda_i \in \X.
\end{equation}
Then for $\lambda \in \X$ we let $\mathcal{C}_\lambda$ be the Serre
subcategory of $\mathcal{C}$ consisting of the objects $V$
such that every irreducible subquotient $L$ of $V$
satisfies $\wt(L) = \lambda$.
The point of this definition is
that irreducible objects $K, L \in \mathcal{C}$ with $\wt(K) \neq \wt(L)$
have different central characters.
Using also the general theory of blocks,
it follows that
\begin{equation}
    \textstyle
    \mathcal{C} =
        \bigoplus_{\lambda \in X} \mathcal{C}_\lambda  .
\end{equation}
We call it  the {\em weight space}
decomposition of $\mathcal{C}$.
\begin{theorem}[\cite{BSW2}]
Let $R=K$ or $k.$
Associated with $\mathcal{C}$,
there is a $2$-representation
$\Phi:\mathfrak{U}(\fraks\frakl'_{I})\rightarrow
\mathfrak{Cat}_R$
defined on objects by
$\lambda\mapsto \mathcal{C}_\lambda$,
on generating
1-morphisms by
$E_i 1_\lambda\mapsto E_i|_{\mathcal{C}_\lambda}$
and $F_i 1_\lambda\mapsto F_i|_{\mathcal{C}_\lambda}$ for $i\in  I.$
\end{theorem}
We refer readers to \cite[Theorem 4.11]{BSW2} for a more detailed and explicit explanation.
\medskip

%\subsection{Symmetric product}\label{sec:symmetricproduct}\hfill\\

\subsection{Double quantum Heisenberg category}\label{sec:doubleHeiscat}\hfill\\

In this section, we will consider the symmetric product of two quantum Heisenberg categories. 

\smallskip

Now let  us  introduce the concept of the symmetric product of two monoidal categories (see for example \cite{BSW3}).
Given a strict $R$-linear monoidal categories $\mathcal C$ and $\mathcal D$,  we can form their free product $\mathcal C\circledast \mathcal D$ as a strict $R$-linear monoidal category. This can be defined by the following universal property: the category of $R$-linear monoidal functors $\mathcal C\circledast \mathcal D\to \mathcal{B}$ for any other strict $R$-linear monoidal category $\mathcal{B}$ is the same as the category of pairs of $R$-linear monoidal functors $\mathcal C\to \mathcal{B}$ and $\mathcal D\to \mathcal{B}$.  When $\mathcal C$ and $\mathcal D$ are themselves defined by generators and relations, the free product of $\mathcal C$ and $\mathcal D$ may be constructed simply as the strict $R$-linear monoidal category defined by taking the disjoint union of the given generators and relations of $\mathcal C$ and $\mathcal D$.
The {\em symmetric product}
$\mathcal C \odot \mathcal D$ is the strict $R$-linear monoidal
category obtained from $\mathcal C \circledast \mathcal D$ by
adjoining isomorphisms $\sigma_{X,Y}:X \otimes Y \stackrel{\sim}{\rightarrow} Y \otimes
X$
such that $\sigma_{Y,X} = \sigma_{X,Y}^{-1}$
for each pair of objects $X \in \mathcal C$ and $Y \in \mathcal D$,
subject also to the relations
\begin{align*}
\sigma_{X_1 \otimes X_2, Y} &= (\sigma_{X_1,Y} \otimes 1_{X_2}) \circ
  (1_{X_1} \otimes \sigma_{X_2,Y}),&
\sigma_{X_2,Y} \circ (f \otimes 1_Y)  &= (1_Y \otimes f) \circ
                                        \sigma_{X_1,Y},\\
\sigma_{X, Y_1 \otimes Y_2} &= (1_{Y_1} \otimes \sigma_{X,Y_2}) \circ
(\sigma_{X, Y_1} \otimes 1_{Y_2}),&
\sigma_{X,Y_2} \circ (1_X \otimes g) &= (g \otimes 1_X)\circ \sigma_{X,Y_1}
\end{align*}
for all $X, X_1,X_2 \in\mathcal C, Y, Y_1,Y_2 \in \mathcal D$ and
$f\in \Hom_{\mathcal C}(X_1,X_2), g \in \Hom_{\mathcal
  D}(Y_1,Y_2)$.
  Often, the actions of a symmetric product of distinct monoidal categories on a given category demonstrate mutual commutativity.
    \smallskip

Let $\Heis_{k^+}(z^+,t^+)\odot \Heis_{k^-}(z^-,t^-)$ be the symmetric product
of $\Heis_{k^+}(z^+,t^+)$ and $\Heis_{k^-}(z^-,t^-)$.
Diagrammatically, it is convenient to use
different colors,
denoting the
symmetric product  $\red{\Heis_{k^+}(z^+,t^+)}\odot \blue{\Heis_{k^-}(z^-,t^-)}$ and using the color
 red (resp., blue) for objects   and morphisms in $\red{\Heis_{k^+}(z^+,t^+)}$ (resp., $\blue{\Heis_{k^-}(z^-,t^-)}$). For convenience, we call $\red{\Heis_{k^+}(z^+,t^+)}\odot \blue{\Heis_{k^-}(z^-,t^-)}$ \emph{double quantum Heisenberg category}.
Morphisms may then be represented by linear combinations of string diagrams
colored both blue and red.
In fact, the double quantum Heisenberg category $\red{\Heis_{k^+}(z^+,t^+)}\odot \blue{\Heis_{k^-}(z^-,t^-)}$ also has a definition by generators and relations as the following.

\begin{theorem}[\cite{HL}]\label{DoubleHeisenberg}
The  double quantum Heisenberg category $\red{\Heis_{k^+}(z^+,t^+)}\odot \blue{\Heis_{k^-}(z^-,t^-)}$ is equivalent to the strict $R$-linear monoidal category
generated by the objects
$$F^+=\red{\up},\quad E^+=\red{\down},\quad F^-=\blue{\up},\quad E^-=\blue{\down},$$
and
morphisms $\begin{tikzpicture}[baseline = -1mm]
	\draw[->,red] (0.08,-.2) to (0.08,.2);
      \node at (0.08,0) {$\dott$};
\end{tikzpicture}$,
$\begin{tikzpicture}[baseline = -1mm]
	\draw[->,red] (0.2,-.2) to (-0.2,.2);
	\draw[-,white,line width=4pt] (-0.2,-.2) to (0.2,.2);
	\draw[->,red] (-0.2,-.2) to (0.2,.2);
\end{tikzpicture}\:,$
$\begin{tikzpicture}[baseline = .75mm]
	\draw[<-,red] (0.3,0.3) to[out=-90, in=0] (0.1,0);
	\draw[-,red] (0.1,0) to[out = 180, in = -90] (-0.1,0.3);
\end{tikzpicture}\:,$
$\begin{tikzpicture}[baseline = .75mm]
	\draw[<-,red] (0.3,0) to[out=90, in=0] (0.1,0.3);
	\draw[-,red] (0.1,0.3) to[out = 180, in = 90] (-0.1,0);
\end{tikzpicture}\:,$
$\:\begin{tikzpicture}[baseline = .75mm]
	\draw[-,red] (0.3,0.3) to[out=-90, in=0] (0.1,0);
	\draw[->,red] (0.1,0) to[out = 180, in = -90] (-0.1,0.3);
\end{tikzpicture}\:,$
$\:\begin{tikzpicture}[baseline = .75mm]
	\draw[-,red] (0.3,0) to[out=90, in=0] (0.1,0.3);
	\draw[->,red] (0.1,0.3) to[out = 180, in = 90] (-0.1,0);
\end{tikzpicture}\:,$
$\begin{tikzpicture}[baseline = -1mm]
	\draw[->,blue] (0.08,-.2) to (0.08,.2);
      \node at (0.08,0) {$\dott$};
\end{tikzpicture}$,
$\begin{tikzpicture}[baseline = -1mm]
	\draw[->,blue] (0.2,-.2) to (-0.2,.2);
	\draw[-,white,line width=4pt] (-0.2,-.2) to (0.2,.2);
	\draw[->,blue] (-0.2,-.2) to (0.2,.2);
\end{tikzpicture}\:$,
$\begin{tikzpicture}[baseline = .75mm]
	\draw[<-,blue] (0.3,0.3) to[out=-90, in=0] (0.1,0);
	\draw[-,blue] (0.1,0) to[out = 180, in = -90] (-0.1,0.3);
\end{tikzpicture}\:,$
$\begin{tikzpicture}[baseline = .75mm]
	\draw[<-,blue] (0.3,0) to[out=90, in=0] (0.1,0.3);
	\draw[-,blue] (0.1,0.3) to[out = 180, in = 90] (-0.1,0);
\end{tikzpicture}\:,$
$\:\begin{tikzpicture}[baseline = .75mm]
	\draw[-,blue] (0.3,0.3) to[out=-90, in=0] (0.1,0);
	\draw[->,blue] (0.1,0) to[out = 180, in = -90] (-0.1,0.3);
\end{tikzpicture}\:,$
$\:\begin{tikzpicture}[baseline = .75mm]
	\draw[-,blue] (0.3,0) to[out=90, in=0] (0.1,0.3);
	\draw[->,blue] (0.1,0.3) to[out = 180, in = 90] (-0.1,0);
\end{tikzpicture}\:,$
$\begin{tikzpicture}[baseline = -1mm]
	\draw[->,red] (0.2,-.2) to (-0.2,.2);
	\draw[->,blue] (-0.2,-.2) to (0.2,.2);
\end{tikzpicture}\:$,
$\begin{tikzpicture}[baseline = -1mm]
	\draw[->,blue] (0.2,-.2) to (-0.2,.2);
	\draw[->,red] (-0.2,-.2) to (0.2,.2);
\end{tikzpicture}\:$
subject to the certain relations as follows.
We require that
$\begin{tikzpicture}[baseline = -1mm]
	\draw[->,red] (0.08,-.2) to (0.08,.2);
      \node at (0.08,0) {$\dott$};
\end{tikzpicture}$,
$\begin{tikzpicture}[baseline = -1mm]
	\draw[->,red] (0.2,-.2) to (-0.2,.2);
	\draw[-,white,line width=4pt] (-0.2,-.2) to (0.2,.2);
	\draw[->,red] (-0.2,-.2) to (0.2,.2);
\end{tikzpicture}\:$,
$\begin{tikzpicture}[baseline = .75mm]
	\draw[<-,red] (0.3,0.3) to[out=-90, in=0] (0.1,0);
	\draw[-,red] (0.1,0) to[out = 180, in = -90] (-0.1,0.3);
\end{tikzpicture}\:,$
$\begin{tikzpicture}[baseline = .75mm]
	\draw[<-,red] (0.3,0) to[out=90, in=0] (0.1,0.3);
	\draw[-,red] (0.1,0.3) to[out = 180, in = 90] (-0.1,0);
\end{tikzpicture}\:,$$\:\begin{tikzpicture}[baseline = .75mm]
	\draw[-,red] (0.3,0.3) to[out=-90, in=0] (0.1,0);
	\draw[->,red] (0.1,0) to[out = 180, in = -90] (-0.1,0.3);
\end{tikzpicture}\: $ and
$\:\begin{tikzpicture}[baseline = .75mm]
	\draw[-,red] (0.3,0) to[out=90, in=0] (0.1,0.3);
	\draw[->,red] (0.1,0.3) to[out = 180, in = 90] (-0.1,0);
\end{tikzpicture}\:$ satisfy all relations for $\red{\Heis_{k^+}(z^+,t^-)}$ and that
$\begin{tikzpicture}[baseline = -1mm]
	\draw[->,blue] (0.08,-.2) to (0.08,.2);
      \node at (0.08,0) {$\dott$};
\end{tikzpicture}$,
$\begin{tikzpicture}[baseline = -1mm]
	\draw[->,blue] (0.2,-.2) to (-0.2,.2);
	\draw[-,white,line width=4pt] (-0.2,-.2) to (0.2,.2);
	\draw[->,blue] (-0.2,-.2) to (0.2,.2);
\end{tikzpicture}\:$,
$\begin{tikzpicture}[baseline = .75mm]
	\draw[<-,blue] (0.3,0.3) to[out=-90, in=0] (0.1,0);
	\draw[-,blue] (0.1,0) to[out = 180, in = -90] (-0.1,0.3);
\end{tikzpicture}\:,$
$\begin{tikzpicture}[baseline = .75mm]
	\draw[<-,blue] (0.3,0) to[out=90, in=0] (0.1,0.3);
	\draw[-,blue] (0.1,0.3) to[out = 180, in = 90] (-0.1,0);
\end{tikzpicture}\:,$
$\:\begin{tikzpicture}[baseline = .75mm]
	\draw[-,blue] (0.3,0.3) to[out=-90, in=0] (0.1,0);
	\draw[->,blue] (0.1,0) to[out = 180, in = -90] (-0.1,0.3);
\end{tikzpicture}\:$ and
$\:\begin{tikzpicture}[baseline = .75mm]
	\draw[-,blue] (0.3,0) to[out=90, in=0] (0.1,0.3);
	\draw[->,blue] (0.1,0.3) to[out = 180, in = 90] (-0.1,0);
\end{tikzpicture}\:$ satisfy  all relations for $\blue{\Heis_{k^-}(z^-,t^-)}.$
In addition, we also have mixed affine-Hecke-type relations
\begin{align}\label{Mixaff1}
\mathord{
\begin{tikzpicture}[baseline = -1mm]
	\draw[<-,thin,red] (-0.45,.6) to (0.45,-.6);
	\draw[-,line width=4pt,white] (-0.45,-.6) to (0.45,.6);
	\draw[->,thin,red] (-0.45,-.6) to (0.45,.6);
        \draw[-,thin,blue] (0,-.6) to[out=90,in=-90] (-.45,0);
       \draw[->,thin,blue] (-0.45,0) to[out=90,in=-90] (0,0.6);
\end{tikzpicture}
}
=&
\mathord{
\begin{tikzpicture}[baseline = -1mm]
	\draw[<-,thin,red] (-0.45,.6) to (0.45,-.6);
	\draw[-,line width=4pt,white] (-0.45,-.6) to (0.45,.6);
	\draw[->,thin,red] (-0.45,-.6) to (0.45,.6);
        \draw[-,thin,blue] (0,-.6) to[out=90,in=-90] (.45,0);
        \draw[->,thin,blue] (0.45,0) to[out=90,in=-90] (0,0.6);
\end{tikzpicture}
}\:,&
\mathord{
\begin{tikzpicture}[baseline = -1mm]
\draw[->,thin,blue] (0.45,-.6) to (-0.45,.6);
\draw[-,line width=4pt,white] (-0.45,-.6) to (0.45,.6);
\draw[<-,thin,blue] (0.45,.6) to (-0.45,-.6);
        \draw[-,thin,red] (0,-.6) to[out=90,in=-90] (-.45,0);
       \draw[->,thin,red] (-0.45,0) to[out=90,in=-90] (0,0.6);
\end{tikzpicture}
}
=&
\mathord{
\begin{tikzpicture}[baseline = -1mm]
\draw[->,thin,blue] (0.45,-.6) to (-0.45,.6);
\draw[-,line width=4pt,white] (-0.45,-.6) to (0.45,.6);
\draw[<-,thin,blue] (0.45,.6) to (-0.45,-.6);
\draw[-,thin,red] (0,-.6) to[out=90,in=-90] (.45,0);
        \draw[->,thin,red] (0.45,0) to[out=90,in=-90] (0,0.6);
\end{tikzpicture}
}\:,
\end{align}
\begin{align}\label{Mixaff2}
\mathord{
\begin{tikzpicture}[baseline = -1mm]
	\draw[-,thin,red] (0.28,-.6) to[out=90,in=-90] (-0.28,0);
	\draw[->,thin,red] (-0.28,0) to[out=90,in=-90] (0.28,.6);
	\draw[-,thin,blue] (-0.28,-.6) to[out=90,in=-90] (0.28,0);
	\draw[->,thin,blue] (0.28,0) to[out=90,in=-90] (-0.28,.6);
\end{tikzpicture}
}\,=&
\mathord{
\begin{tikzpicture}[baseline = -1mm]
	\draw[->,thin,red] (0.18,-.6) to (0.18,.6);
	\draw[->,thin,blue] (-0.18,-.6) to (-0.18,.6);
\end{tikzpicture}
}\:,&
\mathord{
\begin{tikzpicture}[baseline = -1mm]
	\draw[->,thin,red] (0.28,0) to[out=90,in=-90] (-0.28,.6);
	\draw[->,thin,blue] (-0.28,0) to[out=90,in=-90] (0.28,.6);
	\draw[-,thin,red] (-0.28,-.6) to[out=90,in=-90] (0.28,0);
	\draw[-,thin,blue] (0.28,-.6) to[out=90,in=-90] (-0.28,0);
\end{tikzpicture}
}\,=&
\mathord{
\begin{tikzpicture}[baseline = -1mm]
	\draw[->,thin,blue] (0.18,-.6) to (0.18,.6);
	\draw[->,thin,red] (-0.18,-.6) to (-0.18,.6);
\end{tikzpicture}
}\:,\end{align}
\begin{align}\label{Mixaff3}
\mathord{
\begin{tikzpicture}[baseline = -.5mm]
	\draw[->,thin,blue] (0.28,-.3) to (-0.28,.4);
      \node at (0.165,-0.15) {$\dott$};
	\draw[thin,red,->] (-0.28,-.3) to (0.28,.4);
\end{tikzpicture}
}\,=&
\mathord{
\begin{tikzpicture}[baseline = -.5mm]
	\draw[thin,red,->] (-0.28,-.3) to (0.28,.4);
	\draw[->,thin,blue] (0.28,-.3) to (-0.28,.4);
      \node at (-0.14,0.23) {$\dott$};
\end{tikzpicture}
}
\:,&
\mathord{
\begin{tikzpicture}[baseline = -.5mm]
	\draw[thin,red,->] (-0.28,-.3) to (0.28,.4);
      \node at (-0.16,-0.15) {$\dott$};
	\draw[->,thin,blue] (0.28,-.3) to (-0.28,.4);
\end{tikzpicture}
}\,=&
\mathord{
\begin{tikzpicture}[baseline = -.5mm]
	\draw[->,thin,blue] (0.28,-.3) to (-0.28,.4);
	\draw[thin,red,->] (-0.28,-.3) to (0.28,.4);
      \node at (0.145,0.23) {$\dott$};
\end{tikzpicture}
}\:.\end{align}
We also introduce the sideways mixed crossings
\begin{align*}
\mathord{
\begin{tikzpicture}[baseline = -.5mm]
	\draw[->,red] (-0.28,-.3) to (0.28,.4);
	\draw[<-,blue] (0.28,-.3) to (-0.28,.4);
\end{tikzpicture}
}&:=
\mathord{
\begin{tikzpicture}[baseline = 0]
	\draw[->,red] (0.3,-.5) to (-0.3,.5);
	\draw[-,blue] (-0.2,-.2) to (0.2,.3);
        \draw[-,blue] (0.2,.3) to[out=50,in=180] (0.5,.5);
        \draw[->,blue] (0.5,.5) to[out=0,in=90] (0.8,-.5);
        \draw[-,blue] (-0.2,-.2) to[out=230,in=0] (-0.6,-.5);
        \draw[-,blue] (-0.6,-.5) to[out=180,in=-90] (-0.85,.5);
\end{tikzpicture}
}\:,
&\mathord{
\begin{tikzpicture}[baseline = -.5mm]
	\draw[<-,red] (0.28,-.3) to (-0.28,.4);
	\draw[->,blue] (-0.28,-.3) to (0.28,.4);
\end{tikzpicture}
}&:=
\mathord{
\begin{tikzpicture}[baseline = 0]
	\draw[-,red] (-0.2,-.2) to (0.2,.3);
	\draw[->,blue] (0.3,-.5) to (-0.3,.5);
        \draw[-,red] (0.2,.3) to[out=50,in=180] (0.5,.5);
        \draw[->,red] (0.5,.5) to[out=0,in=90] (0.8,-.5);
        \draw[-,red] (-0.2,-.2) to[out=230,in=0] (-0.6,-.5);
        \draw[-,red] (-0.6,-.5) to[out=180,in=-90] (-0.85,.5);
\end{tikzpicture}
}\:,\\
\mathord{
\begin{tikzpicture}[baseline = -.5mm]
	\draw[<-,red] (-0.28,-.3) to (0.28,.4);
	\draw[->,blue] (0.28,-.3) to (-0.28,.4);
\end{tikzpicture}
}&:=
\mathord{
\begin{tikzpicture}[baseline = 0]
	\draw[-,red] (-0.2,.2) to (0.2,-.3);
	\draw[<-,blue] (0.3,.5) to (-0.3,-.5);
        \draw[-,red] (0.2,-.3) to[out=130,in=180] (0.5,-.5);
        \draw[-,red] (0.5,-.5) to[out=0,in=270] (0.8,.5);
        \draw[-,red] (-0.2,.2) to[out=130,in=0] (-0.5,.5);
        \draw[->,red] (-0.5,.5) to[out=180,in=-270] (-0.8,-.5);
\end{tikzpicture}
}\:,&
\mathord{
\begin{tikzpicture}[baseline = -.5mm]
	\draw[->,red] (0.28,-.3) to (-0.28,.4);
	\draw[<-,blue] (-0.28,-.3) to (0.28,.4);
\end{tikzpicture}
}&:=
\mathord{
\begin{tikzpicture}[baseline = 0]
	\draw[<-,red] (0.3,.5) to (-0.3,-.5);
	\draw[-,blue] (-0.2,.2) to (0.2,-.3);
        \draw[-,blue] (0.2,-.3) to[out=130,in=180] (0.5,-.5);
        \draw[-,blue] (0.5,-.5) to[out=0,in=270] (0.8,.5);
        \draw[-,blue] (-0.2,.2) to[out=130,in=0] (-0.5,.5);
        \draw[->,blue] (-0.5,.5) to[out=180,in=-270] (-0.8,-.5);
\end{tikzpicture}
}\:,
\end{align*} then we require mixed Mackey formula relations \begin{align}\label{Mixamackey1}
\mathord{
\begin{tikzpicture}[baseline = -1mm]
	\draw[-,thin,red] (0.28,-.6) to[out=90,in=-90] (-0.28,0);
	\draw[->,thin,red] (-0.28,0) to[out=90,in=-90] (0.28,.6);
	\draw[<-,thin,blue] (-0.28,-.6) to[out=90,in=-90] (0.28,0);
	\draw[-,thin,blue] (0.28,0) to[out=90,in=-90] (-0.28,.6);
\end{tikzpicture}
}\,=&
\mathord{
\begin{tikzpicture}[baseline = -1mm]
	\draw[->,thin,red] (0.18,-.6) to (0.18,.6);
	\draw[<-,thin,blue] (-0.18,-.6) to (-0.18,.6);
\end{tikzpicture}
}\:,&
\mathord{
\begin{tikzpicture}[baseline = -1mm]
	\draw[->,thin,red] (0.28,0) to[out=90,in=-90] (-0.28,.6);
	\draw[-,thin,blue] (-0.28,0) to[out=90,in=-90] (0.28,.6);
	\draw[-,thin,red] (-0.28,-.6) to[out=90,in=-90] (0.28,0);
	\draw[<-,thin,blue] (0.28,-.6) to[out=90,in=-90] (-0.28,0);
\end{tikzpicture}
}\,=&
\mathord{
\begin{tikzpicture}[baseline = -1mm]
	\draw[<-,thin,blue] (0.18,-.6) to (0.18,.6);
	\draw[->,thin,red] (-0.18,-.6) to (-0.18,.6);
\end{tikzpicture}
}\:,\end{align}
\begin{align}\label{Mixamackey2}
\mathord{
\begin{tikzpicture}[baseline = -1mm]
	\draw[-,thin,blue] (0.28,-.6) to[out=90,in=-90] (-0.28,0);
	\draw[->,thin,blue] (-0.28,0) to[out=90,in=-90] (0.28,.6);
	\draw[<-,thin,red] (-0.28,-.6) to[out=90,in=-90] (0.28,0);
	\draw[-,thin,red] (0.28,0) to[out=90,in=-90] (-0.28,.6);
\end{tikzpicture}
}\,=&
\mathord{
\begin{tikzpicture}[baseline = -1mm]
	\draw[->,thin,blue] (0.18,-.6) to (0.18,.6);
	\draw[<-,thin,red] (-0.18,-.6) to (-0.18,.6);
\end{tikzpicture}
}\:,&
\mathord{
\begin{tikzpicture}[baseline = -1mm]
	\draw[->,thin,blue] (0.28,0) to[out=90,in=-90] (-0.28,.6);
	\draw[-,thin,red] (-0.28,0) to[out=90,in=-90] (0.28,.6);
	\draw[-,thin,blue] (-0.28,-.6) to[out=90,in=-90] (0.28,0);
	\draw[<-,thin,red] (0.28,-.6) to[out=90,in=-90] (-0.28,0);
\end{tikzpicture}
}\,=&
\mathord{
\begin{tikzpicture}[baseline = -1mm]
	\draw[<-,thin,red] (0.18,-.6) to (0.18,.6);
	\draw[->,thin,blue] (-0.18,-.6) to (-0.18,.6);
\end{tikzpicture}
}\:.\end{align}
\end{theorem}

\medskip

\subsection{From Double quantum Heisenberg   to $\mathfrak{s}\mathfrak{l}'_{I_+}\oplus\mathfrak{s}\mathfrak{l}'_{I_-}$-categorification}
\label{sec:doublequantumtoKac}\hfill\\

Let $R$ be a field and $\calC$ is a locally finite $R$-linear category. Similarly, a {\em categorical double quantum Heisenberg action} on a locally finite category $\mathcal{C}$ is the data
of a strict monoidal functor $\red{\Heis_{k^+}(z^+,t^+)}\odot \blue{\Heis_{k^-}(z^-,t^-)} \rightarrow \mathcal{E}nd( \mathcal{C})$.
Now we assume that there exists a double  quantum Heisenberg action
 on a locally finite category $\mathcal{C}.$
Similarly as in \S\ref{sec:KactoHeis}, for any irreducible object $L\in \calC,$
we can define $m^+_L(u) \in R[u]$ (resp. $ n^+_L(u)\in R[u]$, $m^-_L(v)\in R[v]$ and $n^-_L(v) \in R[v]$) to
be the monic {\em minimal polynomials} of
the endomorphisms
$\mathord{\begin{tikzpicture}[baseline = -1mm]
 	\draw[->,red] (0.08,-.2) to (0.08,.2);
     \node at (0.08,0) {$\dott$};
 	\draw[-,darkg,thick] (0.38,.2) to (0.38,-.2);
     \node at (0.55,0) {$\darkg\scriptstyle{L}$};
\end{tikzpicture}
}$ (resp.
$\mathord{\begin{tikzpicture}[baseline = -1mm]
 	\draw[<-,red] (0.08,-.2) to (0.08,.2);
     \node at (0.08,0.02) {$\dott$};
 	\draw[-,darkg,thick] (0.38,.2) to (0.38,-.2);
     \node at (0.55,0) {$\darkg\scriptstyle{L}$};
\end{tikzpicture}
},$  $\mathord{\begin{tikzpicture}[baseline = -1mm]
 	\draw[->,blue] (0.08,-.2) to (0.08,.2);
     \node at (0.08,0) {$\dott$};
 	\draw[-,darkg,thick] (0.38,.2) to (0.38,-.2);
     \node at (0.55,0) {$\darkg\scriptstyle{L}$};
\end{tikzpicture}
}$ and
$\mathord{\begin{tikzpicture}[baseline = -1mm]
 	\draw[<-,blue] (0.08,-.2) to (0.08,.2);
     \node at (0.08,0.02) {$\dott$};
 	\draw[-,darkg,thick] (0.38,.2) to (0.38,-.2);
     \node at (0.55,0) {$\darkg\scriptstyle{L}$};
\end{tikzpicture}
}$).
\smallskip

We define the set $I_+$ (resp. $I_-$)
to be the union of the sets of roots of the minimal polynomials
$m^+_L(u)$ (resp. $m^-_L(v)$) for all irreducible objects $L \in \calC$.
 By adjunction, it follows that $I_+$ (resp. $I_-$)  could also
be defined as
the union of the sets of roots of the polynomials
$n^+_L(u)$ (resp., $n^-_L(v)$) for all irreducible objects $L \in \mathcal{C}$.
Then we have the following decomposition:
\begin{align}
E^+ &=
\bigoplus_{i \in I_+} E^+_i,&
F^+ &= \bigoplus_{i \in I_+} F^+_i,&
E^- &=
\bigoplus_{i' \in I_-} E^-_{i'},&
F^- &= \bigoplus_{i' \in I_-} F^-_{i'},
\end{align}
such that $\mathord{\begin{tikzpicture}[baseline = -1mm]
 	\draw[<-,red] (0.08,-.2) to (0.08,.2);
     \node at (0.08,0.02) {$\dott$};
\end{tikzpicture}}-i$, $\mathord{\begin{tikzpicture}[baseline = -1mm]
 	\draw[->,red] (0.08,-.2) to (0.08,.2);
     \node at (0.08,0.02) {$\dott$};
\end{tikzpicture}}-i$, $\mathord{\begin{tikzpicture}[baseline = -1mm]
 	\draw[<-,blue] (0.08,-.2) to (0.08,.2);
     \node at (0.08,0.02) {$\dott$};
\end{tikzpicture}}-i'$
, $\mathord{\begin{tikzpicture}[baseline = -1mm]
 	\draw[->,blue] (0.08,-.2) to (0.08,.2);
     \node at (0.08,0.02) {$\dott$};
\end{tikzpicture}}-i'$ are  locally nilpotent on $E^+_i$, $F^+_i$,  $E^-_{i'}$, $F^-_{i'}$, respectively.
Also, let $\epsilon^+_i(L)$ (resp. $\phi^+_i(L)$, $\epsilon^-_{i'}(L)$ and $\phi^-_{i'}(L)$) denote the multiplicities
of $i \in I_+$ (resp. $i' \in I_-$) as a root of the polynomials $m^+_L(u)$ (resp. $n^+_L(u)$, $m^-_L(v)$ and $n^-_L(v)$).

\smallskip
For simplicity of notation, we write $\red\calH\blue\calH:=\red{\Heis_{k^+}(z^+,t^+)}\odot \blue{\Heis_{k^-}(z^-,t^-)}.$ We will use red color (resp. blue color) to denote the same shape diagram in $\red{\Heis_{k^+}(z^+,t^+)}$ (resp. $\blue{\Heis_{k^-}(z^-,t^-)}$).
For example, we define
\begin{align}
%\red\anticlock(u) &:= \frac{z^+}{t^+} \sum_{r\in\Z}
%\mathord{
%\begin{tikzpicture}[baseline = 1.25mm]
%  \draw[-,red] (0,0.4) to[out=180,in=90] (-.2,0.2);
%  \draw[->,red] (0.2,0.2) to[out=90,in=0] (0,.4);
% \draw[-,red] (-.2,0.2) to[out=-90,in=180] (0,0);
%  \draw[-,red] (0,0) to[out=0,in=-90] (0.2,0.2);
%   \node at (0,.21) {$+$};
%   \node at (0.33,0.2) {$\scriptstyle{r}$};
%\end{tikzpicture}
%}\: u^{-r}
%\in u^{k_+} 1_\unit + u^{k_+-1} \End_{\red\calH\blue\calH}(\unit)[[
%u^{-1} ]]
%,\\
\red\clock(u)&:= -t^+z^+ \sum_{r\in\Z}
\mathord{
\begin{tikzpicture}[baseline = 1.25mm]
  \draw[<-,red] (0,0.4) to[out=180,in=90] (-.2,0.2);
  \draw[-,red] (0.2,0.2) to[out=90,in=0] (0,.4);
 \draw[-,red] (-.2,0.2) to[out=-90,in=180] (0,0);
  \draw[-,red] (0,0) to[out=0,in=-90] (0.2,0.2);
   \node at (0,.21) {$+$};
   \node at (-0.33,0.2) {$\scriptstyle{r}$};
\end{tikzpicture}
} \:u^{-r} \in u^{-k_+}1_\unit +
u^{-k_+-1}\End_{\red\calH\blue\calH}(\unit)[[
u^{-1} ]],
\\
%\blue\anticlock(v) &:= \frac{z^-}{t^{-}} \sum_{r\in\Z}
%\mathord{
%\begin{tikzpicture}[baseline = 1.25mm]
%  \draw[-,blue] (0,0.4) to[out=180,in=90] (-.2,0.2);
%  \draw[->,blue] (0.2,0.2) to[out=90,in=0] (0,.4);
% \draw[-,blue] (-.2,0.2) to[out=-90,in=180] (0,0);
%  \draw[-,blue] (0,0) to[out=0,in=-90] (0.2,0.2);
%   \node at (0,.21) {$+$};
%   \node at (0.33,0.2) {$\scriptstyle{r}$};
%\end{tikzpicture}
%}\: v^{-r}
%\in v^{k_-} 1_\unit + v^{k_--1} \End_{\red\calH\blue\calH}(\unit)[[
%v^{-1} ]]
%,\\
\blue\clock(v)&:= -t^-z^- \sum_{r\in\Z}
\mathord{
\begin{tikzpicture}[baseline = 1.25mm]
  \draw[<-,blue] (0,0.4) to[out=180,in=90] (-.2,0.2);
  \draw[-,blue] (0.2,0.2) to[out=90,in=0] (0,.4);
 \draw[-,blue] (-.2,0.2) to[out=-90,in=180] (0,0);
  \draw[-,blue] (0,0) to[out=0,in=-90] (0.2,0.2);
   \node at (0,.21) {$+$};
   \node at (-0.33,0.2) {$\scriptstyle{r}$};
\end{tikzpicture}
} \:v^{-r} \in v^{-k_-}1_\unit +
v^{-k_--1}\End_{\red\calH\blue\calH}(\unit)[[
v^{-1} ]],
\end{align}
and the others are similar.
Then they satisfy \begin{equation}
\red\anticlock(u)\; \red\clock(u)
= 1_\unit,\quad \blue\anticlock(v)\; \blue\clock(v)
= 1_\unit.
\end{equation} For each irreducible object
$L\in \calC$, we consider the action of $\bbO^+(u)$ and $\bbO^-(v)$ on $\unit (L)=L$, then by Schur lemma
\begin{align}
\bbO^+(u)(L)&:=
\mathord{
\begin{tikzpicture}[baseline = -1mm]
     \node at (0.08,0) {$\scriptstyle\red\clock(u)$};
 	\draw[-,darkg,thick] (0.68,.2) to (0.68,-.22);
     \node at (0.68,-.37) {$\darkg\scriptstyle{L}$};
\end{tikzpicture}
},&
\quad
\bbO^-(v)(L)&:=
\mathord{
\begin{tikzpicture}[baseline = -1mm]
     \node at (0.08,0) {$\scriptstyle\blue\clock(v)$};
 	\draw[-,darkg,thick] (0.68,.2) to (0.68,-.22);
     \node at (0.68,-.37) {$\darkg\scriptstyle{L}$};
\end{tikzpicture}
}
\end{align}
define  functions from $\Irr(\calC)$ to $R((u^{-1}))$ and $R((v^{-1}))$, respectively.
We also have the following
\begin{align}
\bbO^+(u)(L) = m^+_L(u)/n^+_L(u)~ \text{and}~ \bbO^-(v)(L) = m^-_L(v)/n^-_L(v).
\end{align}

Let $I_+(q)$ and $I_-(q)$ be the quivers  defined as in \S \ref{sec:KactoHeis},
 and write
$\K_q:=I_+(q)\sqcup I_-(q)$ for the quiver that is the disjoint union of $I_+(q)$ and $I_-(q)$.
Notice that there is no arrow between $I_+(q)$ and $I_-(q)$. In particular,
$\K_q$ is not connected. We may denote by
$\fraks\frakl'_{\K_q}=\fraks\frakl'_{I_+}\oplus\fraks\frakl'_{I_-}$ the corresponding Kac-Moody algebra which is a direct sum of $\fraks\frakl'_{I_+}$ and $\fraks\frakl'_{I_-}.$
We call the Kac-Moody algebra $\fraks\frakl'_{\K_q}=\fraks\frakl'_{I_+}\oplus\fraks\frakl'_{I_-}$ is \emph{associated to the categorical action of $\red{\Heis_{k^+}(z^+,t^+)}\odot\blue{ \Heis_{k^-}(z^-,t^-)}$ on  $\calC$}.
\smallskip

Then the  weight lattice of $\fraks\frakl'_{\K_q}$ is $X=X_+ \times X_-$,
and fundamental weights are  $\{\Lambda^+_i:=\Lambda_i\times 0\:|\:i \in I_+\}\sqcup \{\Lambda^-_{i'}:=0\times \Lambda_{i'}\:|\:i' \in I_-\}$, etc.
For an irreducible object $L \in \mathcal{C}$, let
\begin{equation}\label{chickens}
\wt(L) := \sum_{i \in I_+} (\phi^+_i(L)-\epsilon^+_i(L)) \Lambda^+_i +\sum_{i' \in I_-} (\phi^-_{i'}(L)-\epsilon^-_{i'}(L)) \Lambda^-_{i'} \in X.
\end{equation}
Then for $\lambda \in X$ we let $\mathcal{C}_\lambda$ be the Serre
subcategory of $\mathcal{C}$ consisting of the objects $V$
such that every irreducible subquotient $L$ of $V$
satisfies $\wt(L) = \lambda$.
The point of this definition is
that irreducible objects $L, M \in \mathcal{C}$ with $\wt(L) \neq \wt(M)$
have different central characters.
Using also the general theory of blocks,
it is shown in \cite{HL} that
\begin{equation}\label{one}
    \textstyle
    \mathcal{C} =
        \bigoplus_{\lambda \in X} \mathcal{C}_\lambda  .
\end{equation}
We call it the {\em weight space}
decomposition of $\mathcal{C}$.

\begin{theorem}[\cite{HL}]\label{Thm:HL}Let $R$ be a field.
Associated to $\mathcal{C}$,
there is a $2$-representation
$\Phi:\mathfrak{U}(\fraks\frakl'_{ K_q})\rightarrow
\mathfrak{Cat}_R$
defined on objects by
$\lambda\mapsto \mathcal{C}_\lambda$,
on generating
1-morphisms by
$E_i 1_\lambda\mapsto E^+_i|_{\mathcal{C}_\lambda}$
and $F_i 1_\lambda\mapsto F^+_i|_{\mathcal{C}_\lambda}$ for $i\in  I_+,$ and $E_j 1_\lambda\mapsto E^-_j|_{\mathcal{C}_\lambda}$
and $F_j 1_\lambda\mapsto F^-_j|_{\mathcal{C}_\lambda}$ for $j\in I_-.$
\end{theorem}
\medskip

\section{Double quantum Heisenberg categorification for finite classical groups}\label{Chap:doubleHeisonfiniteclass}\hfill\\

Throughout this section, we always assume that $q$ is a power of an odd prime, $\ell$ is a prime with $\ell\nmid q(q-1)$, $K$ is a field of characteristic 0, $\calO$ is a complete discrete valuation ring and $k$ is a field of characteristic $\ell$. Let $R=K,$ $\calO $ or $k$.
We fix a square root $q^{1/2}$ of $q$ and a square root $\sqrt{-1}$ of $-1$,
and we assume that $q^{1/2}$ and $\sqrt{-1}$ lie in $R$.
Let $\bfG_n$ be one of $\mathbf{O}_{2n+1}$,
 $\bfSp_{2n}$ or $\mathbf{O}_{2n}$,
so that $G_n$ is one of  $\O_{2n+1}(q)$, $\Sp_{2n}(q)$ or $\O_{2n}^{\pm}(q)$.
As is conventional, we set $\bfG_0=G_0=\{1\}$,  except in the $\bfO_1$ case where
$\bfG_0=G_0=\{\pm 1\}.$ Let
 $\bbN^*=\mathbb{Z}_{> 0}$ in the
 $\O_{2n}^{-}(q)$ case  and  $\bbN^*=\mathbb{Z}_{\geqslant 0}$ in all the other cases.
Define $$RG_\bullet\mod:=\bigoplus_{n\in \bbN^*}R G _n\mod.$$
We also fix the following constants
\begin{align}\label{constant}z^+&=z^-=q^{1/2}-q^{-1/2},\\t^+&=\begin{cases}\label{t+}
  \sqrt{-1}q^{-1/2} &
    \text{if $G_n = \O_{2n+1}(q)$ or $\Sp_{2n}(q)$ }; \\
\sqrt{-1} & \text{if $G_n=\O^{\pm}_{2n}(q)$},
\end{cases}\\
t^-&=\begin{cases}\label{t-}
  \sqrt{-1}q^{-1/2} &
    \text{if $G_n = \O_{2n+1}(q)$};  \\
\sqrt{-1} & \text{if $G_n=\Sp_{2n}(q)$ or $\O^{\pm}_{2n}(q)$}.
\end{cases}
\end{align}
\smallskip

\subsection{Categorical double quantum Heisenberg action}\hfill\\

In this section, we will construct  a
strict monoidal functor from a double  quantum Heisenberg category to $\calE nd(RG_\bullet\mod)$.
\smallskip
%\begin{center}
%\begin{table}
%\begin{tabular}{cccc}
%  \hline
%  % after \\: \hline or \cline{col1-col2} \cline{col3-col4} ...
%  &$\O_{2\bullet+1}(q)$&$\Sp_{2\bullet}(q)$ &$\O^{\pm}_{2\bullet}(q)$\\\hline
%  $t^+$&$\sqrt{-1}q^{-1/2}$&
%$\sqrt{-1}q^{-1/2}$&
%$\sqrt{-1}$\\\hline
%  $t^-$&$\sqrt{-1}q^{-1/2}$&
%$\sqrt{\zeta(-1)}q^{-1/2}$&
%$\sqrt{-1}$\\\hline
%\end{tabular}
%\end{table}
%\end{center}

\subsubsection{Construction of classical groups} \label{sec:construction}\hfill\\
%We assume throughout this paper that $q$ is odd and $\ell$ is an odd prime not dividing $q$.
For $n\geqs 1$,
let $J_n$  be the $n\times n$ matrix with entry $1$ in $(i,n-i+1)$ for $1\leqs i\leqs n$
and zero elsewhere, and denote $\widetilde J_{2n}=\left(\begin{array}{cc}
  0  & J_n \\
-J_n & 0   \\
\end{array}\right)$.

\smallskip

(1) The odd-dimensional (full) orthogonal group $\bfO_{2n+1}$
is
  $$\bfO_{2n+1} = \O_{2n+1}(\overline\bbF_q)= \{g \in \bfG\bfL_{2n+1} \, \mid
  \, {}^tg  J_{2n+1} g =   J_{2n+1}\}.$$
%whose connected component
%containing the identity is the odd-dimensional special orthogonal group
%  $\bfSO_{2n+1} = \mathrm{SO}_{2n+1}(\overline\bbF_q) =\bfO_{2n+1}\cap \bfS\bfL_{2n+1}(\overline\bbF_q).$
  The corresponding finite groups of fixed points under $F$
  are  the \emph{finite  orthogonal group} $\O_{2n+1}(q)$,
where $F$ is the Frobenius endomorphism of $\bfO_{2n+1}$ induced  by
  the standard Frobenius map $F_q$ on $\bfG\bfL_{2n+1}$
  raising all entries of a matrix to their $q$-th powers.

\vspace{1ex}

(2)
The symplectic group $\bfSp_{2n}$ is
  $$ \bfSp_{2n} = \Sp_{2n}(\overline\bbF_q)= \{g \in \bfG\bfL_{2n} \, \mid \, {}^tg \widetilde J_{2n} g =  \widetilde J_{2n}\},$$
  and the \emph{finite symplectic group} $\Sp_{2n}(q)$ is  the finite group $(\bfSp_{2n})^F$ of fixed points
of $\bfSp_{2n}$ under the Frobenius endomorphism $F$ of $\bfSp_{2n}$
induced by the standard Frobenius map $F_q$ on $\bfG\bfL_{2n}$.

\vspace{1ex}

(3) The even-dimensional (full) orthogonal group $\bfO_{2n}$ is
  $$\bfO_{2n}= \mathrm{O}_{2n}(\overline\bbF_q) = \{g \in \bfG\bfL_{2n} \, \mid
  \, {}^tg  J_{2n} g =   J_{2n}\}.$$

\begin{itemize}[leftmargin=8mm]
  \item[$(3a)$] The \emph{finite orthogonal group}  $\O^{+}_{2n}(q) $ is the group $(\bfO_{2n})^F$ of fixed points
of $\bfO_{2n}$ under the Frobenius endomorphism $F$ of $\bfO_{2n}$
   induced by the standard Frobenius map $F_q$ on $\bfG\bfL_{2n}$.
  \item[$(3b)$] The \emph{finite non-split  orthogonal group}  $\O^{-}_{2n}(q)$ is
 the group $(\bfO_{2n})^{F'}$ of fixed points  of $\bfO_{2n}$ under $F'$
  where $F':=F_q\circ \sigma$ and  $\sigma(g)=sgs^{-1}$
for $g\in \bfG\bfL_{2n}$ with  $s=\ \diag(\id_{n-1},\left(\begin{array}{cc} 0&1 \\ 1& 0 \\ \end{array}\right),\id_{{n-1}})\in \bfG\bfL_{2n}$ the permutation matrix
interchanging the $n$-th and $(n+1)$-th rows.
\end{itemize}
The groups $\mathbf{SO}_{2n+1} = \mathbf{O}_{2n+1} \cap \mathbf{SL}_{2n+1}$ and $\mathbf{SO}_{2n} = \mathbf{O}_{2n+1} \cap \mathbf{SL}_{2n+1}$ are known as special orthogonal groups, while their corresponding finite counterparts are referred to as finite special orthogonal groups, denoted $\SO_{2n+1}(q)$ and $\SO^\pm_{2n}(q)$.
\smallskip

\subsubsection{Functors}\label{subsec:functors}
%Recall that $q$ is an odd prime number and $R$ is any commutative domain in which $q(q-1)$ is invertible.

In this subsection, we will recall the construction of functors defined in \cite[\S4]{LLZ}. For the sake of simplicity, we have simplified the subscripts for the definitions of functors in \cite[\S4]{LLZ}.
\smallskip

We first define the embedding of certain (possibly disconnected) Levi subgroups into $\bfG_{n}$.
Given $r,m, m_1,m_2,\dots,m_t\in\bbN$ such that $m=\sum_im_i$ and $n=r+m$,
we always embed the group $\bfG_{r}\times\bfG\bfL_{m_1}\times\dots\times
\bfG\bfL_{m_t}$  into $\bfG_{n}$ via the map:
\begin{align*}
\bfG_{r}\times\bfG\bfL_{m_1}\times\dots\times \bfG\bfL_{m_t} &\hookrightarrow
\bfG_{n}\\
B \times A_1\times\dots\times A_t&\mapsto {\rm diag}(A_t,\ldots,A_1,B,A_1',\ldots,A_t')
%\left(\begin{array}{ccccccc}
% A_t & & & & & & \\
% & \ddots & & & & & \\
% & & A_1 & & & & \\
% & & & B & & & \\
% & & & & A_1' & & \\
% & & & & & \ddots & \\
% & & & & & & A_t' \\
%\end{array}\right)
\end{align*}
where
$A_i'=J_{m_i}A_i^{-tr}J_{m_i}$.
The image of the above map, which is a Levi subgroup of $\bfG_{n}$, will be denoted by
$\bfL_{r,m_1,\dots,m_t}$. The corresponding Levi subgroup of $G_{n}$
will be denoted by $L_{r,m_1,\dots,m_t}$.
In particular, the  finite group corresponding to $\bfL_{r,1^m}$
  is  $L_{r,1^m}\simeq G_r \times \GL_1(q)^m$.  
  %We abbreviate $\bfL_r=\bfL_{r,1}$ and $L_r=L_{r,1}$.
\smallskip

The subgroup of  upper-triangular matrices $\bfB_{r+1}$ in $\bfG_{r+1}$ is a  Borel subgroup.
Let $\bfP_{r}$ be the  parabolic subgroup  of $\bfG_{r+1}$ containing $\bfB_{r+1}$ with the Levi complement $\bfL_{r,1}.$
We write $P_r=\bfP_r^F$ (or $\bfP_r^{F'}$ in the case of $\O_{2r}^{-}(q)$).
Let $\bfV_r$ be the unipotent radical of $\bfP_{r}$ and $V_r = \bfV_r^F$ (or $\bfV_r^{F'}$ in the case of $\O_{2r}^{-}(q)$),
and $U_{r}\subset G_{r+1}$ be the subgroup given by
\begin{align}\label{U}
  U_{r}=\bbF^\times_{q}\ltimes V_{r}=
  \begin{pmatrix}*&*&\cdots&\cdots&*\\
	&1&&(0)&\vdots\\
	&&\ddots&&\vdots\\
	&(0)&&1&*\\
	&&&&*
  \end{pmatrix},
\end{align}
where we identify $\bbF^\times_{q}$ with $\{{\rm diag}(t,1,\ldots,1,t^{-1})\in U_{r}\mid t\in \bbF^\times_{q}\}.$ So we have $P_r=G_r\ltimes U_r$ and $P_r=L_{r,1}\ltimes V_r.$

\begin{remark}\label{Vr}
More precisely, $\bfV_r$ consists of all matrices of the form
 $$
  \left(\begin{array}{ccc} 1&-v_1^{tr}J_{2r+1} & -\frac{ v_1^{tr}J_{2r+1}v_1}{2} \\ & \id_{\bfG_r} & v_1\\ &&1 \\ \end{array}\right),
 \: \left(\begin{array}{ccc}1&v_2^{tr}\widetilde{J}_{2r} & z \\ & \id_{\bfG_r} & v_2\\ &&1 \\ \end{array}\right),
 \:\left(\begin{array}{ccc} 1&-v^{tr}J_{2r} & -\frac{v_3^{tr}J_{2r}v_3}{2}\\ & \id_{\bfG_r} & v_3\\ &&1\\ \end{array}\right)$$
 where $z\in \overline{\bbF}_{q},$ $v_1\in \overline{\bbF}_{q}^{2r+1}$ and $v_2 ,v_3\in \overline{\bbF}_{q}^{2r},$  and $\bfG_r=\bfO_{2r+1}, \bfSp_{2r}$ or  $\bfO_{2r}$, respectively.
In particular, the order $|V_r|$ of $V_r=\bfV_r^F$ (or $\bfV_r^{F'}$ in the case of $\O_{2r}^{-}(q)$) is equal to $q^{2r+1}, q^{2r+1}$ or $ q^{2r}$ when $G_r=\O_{2r+1}(q), \Sp_{2r}(q)$ or  $\O_{2r}^{\pm}(q)$, respectively.
\end{remark}

Since $\bbF^\times_{q}$ is a cyclic group of  order $q-1$ and $q$ is odd,
we denote by $\zeta$  the unique character of order 2 (Legendre symbol for $\bbF^\times_{q}$) and view it as a character of $U_r$ via the surjection $U_r\twoheadrightarrow \bbF^\times_{q}$.
Now we define two idempotents of $R P_r:$
$$e^+_r:=e_{U_r}=\frac{1}{|U_r|}\sum_{u\in U_r}u,$$  and $$e^-_r:=\frac{1}{|U_r|}\sum_{u\in U_r}\zeta(u^{-1})u.$$

\smallskip

The  idempotents defined above satisfy the following:
\begin{itemize}
\item for any $u\in U_r, ue^+_r=e^+_r u=e^+_r$ and $ue^-_r=e^-_r u=\zeta(u) e^-_r;$
\item for any $g\in G_r, ge^+_r=e^+_r g= e^+_r$ and $ge^-_r=e^-_r g= e^-_r.$
\end{itemize}
\smallskip

From now on, we identify  $e^\pm_r$ with its image in $R G _n$
under the embedding $\iota_{r,n}: G_r\to G_n.$
The $(R G _{r+1},R G _{r})$-bimodules $R G _{r+1}\cdot e^+_r$ and
  $R G _{r+1}\cdot e^-_r$ are both left $R G _{r+1}$-projective
  and right $R G _{r}$-projective, so we can define
the following bi-adjoint functors $(F_{r,r+1}^+,E_{r+1,r}^+)$ and $(F_{r,r+1}^-,E_{r+1,r}^-)$(see \cite[\S4]{LLZ}, see also \cite[\S 6.2]{DVV2}):
$$\begin{array}{rcrclcl}
 F^+_{r,r+1} &=&R G _{r+1}\cdot  e^+_r\otimes_{R G _r} -    &:&\,R G _r\mod&\to &R G _{r+1}\mod, \\
   E^+_{r+1,r}&=& e^+_r \cdot   R G _{r+1}\otimes_{R G _{r+1}}-   &:&\,R G _{r+1}\mod&\to& R G _r\mod,\\
  F^-_{r,r+1}&=&R G _{r+1}\cdot  e^-_r  \otimes_{R G _r}-     &:&\,R G _r\mod &\to& R G _{r+1}\mod,\\
E^-_{r+1,r}  &=& e^-_r \cdot  R G _{r+1}\otimes_{R G _{r+1}}-   &:&\,R G _{r+1}\mod&\to &R G _r\mod.
\end{array}
$$

For $M\in R G _r\mod$, we have
  \begin{align}\label{formulaofF}
  	F^+_{r,r+1}(M)\cong R_{L_{r,1}}^{G_{r+1}}(M\boxtimes R _1)~\mbox{and}~F^-_{r,r+1}(M)\cong R_{L_{r,1}}^{G_{r+1}}(M\boxtimes R _\zeta),
  	\end{align}
   where $R_{L_{r,1}}^{G_{r+1}}$ is the Harish-Chandra induction functor, $R _1$ and $R _\zeta$ are the $R \bbF^\times_{q}$-modules
   affording the trivial character  of $\bbF^\times_{q}$ and
   the Legendre symbol $\zeta$, respectively.
\smallskip

Now we set
$$F^+=\bigoplus_{r\in\bbN^*}F^+_{r,r+1},\qquad
E^+=\bigoplus_{r\in\bbN^*}E^+_{r+1,r},$$
$$F^-=\bigoplus_{r\in\bbN^*}F^-_{r,r+1},\qquad
E^-=\bigoplus_{r\in\bbN^*}E^-_{r+1,r}.$$

Then these functors define endo-functors of category $RG_\bullet\mod.$
Diagrammatically, it is convenient to use the color red for the functors $E^+$, $F^+$ and use the color blue for $E^-$, $F^-.$
\smallskip

\subsubsection{Adjunctions}\label{sec:adjunction}

Since $(E^+,F^+)$ and $(E^-,F^-)$ are bi-adjoint pairs, there exist four adjunction maps for each pair. Note that they are unique up to a scalar.
Now we give explicit constructions of these adjunction maps by using the corresponding bimodule maps as follows.
\begin{itemize}
\item[$(a)$]
For an element $a\in G_{r+1}$, we set $$\Proj^\pm_{P_r}(e^\pm_r a e^\pm_r )=\begin{cases}
 e^\pm_r a  e^\pm_r \,\,&\text{if}\,\,a\in P_{r},\\
0\,\,&\text{if} \,\,a\notin P_{r}.
\end{cases}$$
Extending by $R $-linearity, this defines  $(R G _{r},R G _{r})$-bimodule
maps $$\varepsilon^\pm_L: e^\pm_r\cdot  R G _{r+1}\cdot  e^\pm_r \to  e^\pm_r\cdot
 P_{r} \cdot e^\pm_r \cong R G _r.$$
The second  $(R G _r, R G_r )$-bimodules isomorphism $ e^\pm_r \cdot RP_{r} \cdot e^\pm_r \cong  R G_r $ is determined by sending $ e^\pm_r $ to $1\in  R G_r .$

\item[$(b)$]

Let $G_{r}=\coprod\limits_{i=1}^{s}P_{r-1}g_i$ be a decomposition of $G_{r}
$ into left $P_{r-1}$ cosets. The elements
$$Z^\pm_r=\sum\limits_{i=1}^s(g_i^{-1}e^\pm_{r-1}\otimes e^\pm_{r-1}g_i)\in R G _{r}\cdot e^\pm_{r-1} \otimes_{R G _{r-1}} e^\pm_{r-1}\cdot   R G_r $$
do not depend on the choice of representatives $\{g_i\}_{i=1}^{s}$ and satisfies $aZ^\pm_r=Z^\pm_ra$ for any $a\in G_{r}$.
\smallskip

 Then we define $( R G_r , R G_r )$-bimodule maps as follows:
 $$\eta^\pm_L: R G_r \to  R G_r \cdot  e^\pm_{r-1} \otimes_{R G _{r-1}} e^\pm_{r-1}\cdot  R G_r ,\quad a\mapsto aZ^\pm_r=Z^\pm_ra$$
for any $a\in  R G_r .$

\item[$(c)$]

We define by $\varepsilon^\pm_R$ the following $( R G_r , R G_r )$-bimodule maps:
 $$\varepsilon^\pm_R: R G_r \cdot e^\pm_{r-1}\otimes_{ R G_{r-1} } e^\pm_{r-1}\cdot   R G_r \to  R G_r ,\quad g  e^\pm_{r-1} \otimes  e^\pm_{r-1} h\mapsto -t^\pm z^\pm g e^\pm_{r-1} h,$$
for any $g,h\in  R G_r $. Note that $t^\pm$ and $z^\pm$ are constants defined in (\ref{constant}), (\ref{t+}) and (\ref{t-}).
\item[$(d)$] Finally we define by $\eta^\pm_R$  the following $( R G_r , R G_r )$-bimodule maps:
$$\eta^\pm_R:  R G_r \to  e^\pm_r\cdot  R G _{r+1} \cdot  e^\pm_r ,\quad g\mapsto  -(t^\pm z^\pm)^{-1}ge^\pm_r $$
for all $g\in  R G_r .$
\end{itemize}
\smallskip

\subsubsection{Lifts of simple reflections} \label{sub:lift-refl}

Let $\mathbf{G}_n$ be one of $\mathbf{O}_{2n+1}$, $\bfSp_{2n}$ or $\bfO_{2n}$.
Let $\bfT$ be the split maximal torus consisting of diagonal matrices of the connected component $\mathbf{G}_n^\circ$
of $\mathbf{G}_n$.  The  group $W_n:=N_{\mathbf{G}_{n}}(\bfT)/C_{\mathbf{G}_{n}}(\bfT)$ is a Weyl group of
type $B_n.$
The numbering of the simple reflections of $W_n$ will be taken
the following convention:
\begin{align}\label{B}
\begin{split}
\begin{tikzpicture}[scale=.4]
\draw[thick] (0 cm,0) circle (.3cm);
\node [below] at (0 cm,-.5cm) {$s_1$};
\draw[thick] (0.3 cm,-0.15cm) -- +(1.9 cm,0);
\draw[thick] (0.3 cm,0.15cm) -- +(1.9 cm,0);
\draw[thick] (2.5 cm,0) circle (.3cm);
\node [below] at (2.5 cm,-.5cm) {$s_2$};
\draw[thick] (2.8 cm,0) -- +(1.9 cm,0);
\draw[thick] (5 cm,0) circle (.3cm);
\draw[thick] (5.3 cm,0) -- +(1.9 cm,0);
\draw[thick] (7.5 cm,0) circle (.3cm);
\draw[dashed,thick] (7.8 cm,0) -- +(4.4 cm,0);
\draw[thick] (12.5 cm,0) circle (.3cm);
\node [below] at (12.5 cm,-.5cm) {$s_{n-1}$};
\draw[thick] (12.8 cm,0) -- +(1.9 cm,0);
\draw[thick] (15 cm,0) circle (.3cm);
\node [below] at (14.8 cm,-.5cm) {$s_n$};
\end{tikzpicture}
\end{split}
\end{align}

When $\bfG_n=\bfO_{2n+1},$ we have
$\bfT = \{{\rm diag}(t_n,
\ldots,t_1,1,t_1^{-1},\ldots,t_n^{-1})\mid t_i\in \overline{\bbF}_q^{\times}~\text{for~all}\\  1\leqs i\leqs n\}.$ It is easy to see that  $\widehat{\bfT}=C_{\bfG_n}(\bfT) = \{{\rm diag}(t_n,
\ldots,t_1,\pm1,t_1^{-1},\ldots,t_n^{-1})\mid t_i\in \overline{\bbF}_q^{\times}~\text{for~all}~1\leqs i\leqs n\}$ and  $W_n\cong N_{\bfG^{\circ}_{n}}(\bfT)/\bfT$ is  a Weyl group of
type $B_n$ on which $F$ acts trivially.
For $i \neq 1$, the simple
reflection $s_i$ can be lifted to $N_{\bfG_n}(\bfT)$ as the signed permutation matrix
$\dot{s}_i=\ \diag(\id_{n-i},\left(\begin{array}{cc} 0&1 \\ -1& 0 \\ \end{array}\right),\id_{\bfG_{i-2}},\left(\begin{array}{cc} 0&  -1 \\1& 0 \\ \end{array}\right),\id_{n-i})$ and $s_1$
is lifted to be the signed permutation matrix
%the action of $s_i$ on $\mathrm{diag}(t_n,\ldots,t_1,1,t_1^{-1},\ldots,t_n^{-1})$
%interchanges $t_{i-1}$ and $t_i$, whereas $s_1$ interchanges $t_1$ and $t_1^{-1}$.
%So the simple reflection $s_i (i>1)$ can be lifted to $N_{\bfSO_{2n+1}}(\bfT)$ as the signed permutation matrix
%$\diag\big(\id_{i-1},\left(\begin{array}{cc} 0&1 \\ -1& 0 \\ \end{array}\right),\id_{2(n-i)-1},\left(\begin{array}{cc} 0&  -1 \\1& 0 \\ \end{array}\right),\id_{i-1}\big)$ and $s_1$ as the matrix
%we choose that $s_i, 1\leqs i\leqs n$ are the  same as above, and
\begin{equation*}
\dot{s}_1=\ \begin{blockarray}{(ccc|ccc|ccc)}
         \BAmulticolumn{3}{c|}{\multirow{2}{*}{$\id_{n-1}$}}&&&&&&\\
         &&&&&&&&&\\
         \cline{1-9}
        &&&0&0&-1&&&&\\
        &&&0&1&0&&&&\\
        &&&-1&0&0&&&&\\
        \cline{1-9}
        &&&&&&\BAmulticolumn{3}{c}{\multirow{2}{*}{$\id_{n-1}$}}\\
        &&&&&&&&&\\
    \end{blockarray}\,\, .
\end{equation*}

When $\bfG_n=\bfSp_{2n}$, we have $\bfT = \{ \mathrm{diag}(t_n,
\ldots,t_1,t_1^{-1},\ldots,t_n^{-1})\mid t_i\in \overline{\bbF}_q^{\times}$ for all $1\leqs i\leqs n\}$.
Similarly,
the simple reflection $s_i (i>1)$ can be lifted to $N_{\bfG_n}(\bfT)$ as the signed permutation matrix
$\dot{s}_i=\ \diag(\id_{n-i},\left(\begin{array}{cc} 0&1 \\ -1& 0 \\ \end{array}\right),\id_{\bfG_{i-2}},\left(\begin{array}{cc} 0&  -1 \\1& 0 \\ \end{array}\right),\id_{n-i})$ and $s_1$ as the signed permutation matrix
\begin{equation*}
\dot{s}_1=\ \begin{blockarray}{(cccc|cc|cccc)}
         \BAmulticolumn{4}{c|}{\multirow{2}{*}{$\id_{n-1}$}}&&&&&\\
         &&&&&&&&&\\
         \cline{1-10}
        &&&&0&-1&&&&\\
        &&&&1&0&&&&\\
        \cline{1-10}
        &&&&&&\BAmulticolumn{3}{c}{\multirow{2}{*}{$\id_{n-1}$}}\\
        &&&&&&&&&
    \end{blockarray}\ .
\end{equation*}

When $\bfG_{n}=\bfO_{2n}$, we have $\bfT = \{ \mathrm{diag}(t_n,
\ldots,t_1,t_1^{-1},\ldots,t_n^{-1})\mid t_i\in \overline{\bbF}_q^{\times}$ for all $1\leqs i\leqs n\}$.

For $i \neq 1$, the simple
reflection $s_i$ can be lifted to $N_{\bfG_n}(\bfT)$ as the signed permutation matrix
$\dot{s}_i=\ \diag(\id_{n-i},\left(\begin{array}{cc} 0&1 \\ -1& 0 \\ \end{array}\right),\id_{\bfG_{i-2}},\left(\begin{array}{cc} 0&  -1 \\1& 0 \\ \end{array}\right),\id_{n-i})$ and $s_1$ as the signed permutation matrix
\begin{equation*}
\dot{s}_1=\ \begin{blockarray}{(cccc|cc|cccc)}
         \BAmulticolumn{4}{c|}{\multirow{2}{*}{$\id_{n-1}$}}&&&&&\\
         &&&&&&&&&\\
         \cline{1-10}
        &&&&0&-1&&&&\\
        &&&&-1&0&&&&\\
        \cline{1-10}
        &&&&&&\BAmulticolumn{3}{c}{\multirow{2}{*}{$\id_{n-1}$}}\\
        &&&&&&&&&
    \end{blockarray}\ .
\end{equation*}

For the next section,
we now define
$x_1:=s_1$ and
 $$x_{k}:=s_{k}s_{k-1} \cdots s_2 x_1 s_2^{-1} \cdots s^{-1}_{k-1} s^{-1}_{k}\,\,\text{for}\,\, 2\leqs k\leqs n.$$
Correspondingly, we define their lifts by $\dot{x}_1=\dot{s}_1$ and  $$\dot{x}_{k}:=\dot{s}_{k}\dot{s}_{k-1} \cdots \dot{s}_2 \dot{x}_1 \dot{s}_2^{-1} \cdots \dot{s}^{-1}_{k-1} \dot{s}^{-1}_{k}\,\,\text{for}\,\, 2\leqs k\leqs n.$$
\smallskip

\subsubsection{Natural transformations} \label{sub:naturaltransf}
In this subsection, we will recall the construction of some natural transformations appearing in \cite[\S 4.2.2]{LLZ}, see also \cite[\S 6.2]{DVV2}.
\smallskip

An endomorphism of $F^\pm_{r,r+1}$  can be represented by an $(R G _{r+1}, R G_r )$-bimodule endomorphism of $ R G _{r+1} \cdot e^\pm_r$ , or equivalently, by right multiplication by an element of $ e^\pm_r\cdot
R G _{r+1} \cdot  e^\pm_r $  centralizing $G_r$.
It is easy to show that $e_{r+1}^{\epsilon}e_{r}^{\epsilon'}=e_{r}^{\epsilon'}e_{r+1}^{\epsilon}$ for any $\epsilon, \epsilon'\in \{\pm\},$ so $e_{r+1}^{\epsilon}e_{r}^{\epsilon'}$ is an idempotent.
The functors $F^+_{r+1,r+2}F^+_{r,r+1} ,$  $ F^-_{r+1,r+2}F^-_{r,r+1},$ $ F^+_{r+1,r+2}F^-_{r,r+1}$ and $ F^-_{r+1,r+2}F^+_{r,r+1} $
 are
represented by the $( R G_{r+2} , R G_r )$-bimodules
$$ R G _{r+2}\cdot e^+_{r+1}e^+_r ,\ \   R G_{r+2} \cdot e^-_{r+1} e^-_r ,\ \   R G_{r+2} \cdot e^+_{r+1}e^-_r ,\ \   R G_{r+2} \cdot e^-_{r+1}e^+_r ,$$
 respectively. An endomorphism of those functors can be represented  by right multiplication by an element of
 $$e^+_{r+1}e^+_r\cdot  R G_{r+2} \cdot  e^+_{r+1}e^+_r  ,\ \  e^-_{r+1}e^-_r \cdot  R G_{r+2}  \cdot e^-_{r+1}e^-_r, $$$$e^+_{r+1}e^-_r  \cdot   R G_{r+2} \cdot  e^+_{r+1} e^-_r, \ \ e^-_{r+1}e^+_r\cdot  R G_{r+2} \cdot  e^-_{r+1} e^+_r$$
centralizing $ R G_r $, respectively.
Let $\dot{x}_{r+1}$ and $\dot{s}_{r+2}$ be as defined in \S \ref{sub:lift-refl}.
It is easy to see that both $\dot{x}_{r+1}\in G_{r+1}$ and  $\dot{s}_{r+2}\in G_{r+2}$ centralize $G_r$.
In fact,
$\dot{x}_{r+1}$
 is exactly the matrix
  $$
  \left(\begin{array}{ccc} & & -1 \\ & \id_{G_r} & \\ -1 \\ \end{array}\right),
 \quad \left(\begin{array}{ccc} & & -1 \\ & \id_{G_r} & \\\, 1 \,\\ \end{array}\right)
  \quad\text{or} \quad \left(\begin{array}{ccc} & & -1 \\ & \id_{G_r} & \\ -1 \\ \end{array}\right)$$
when $G_n=\O_{2n+1}(q), \Sp_{2n}(q)$ or  $\O_{2n}^{\pm}(q)$, respectively.

Thus, the right multiplication by the elements
\begin{align*}
& X^+_{r+1}= q^r  e^+_r  \dot{x}_{r+1} \,  e^+_r ,\qquad
X^-_{r+1}= \beta q^r  e^-_r  \dot{x}_{r+1} \,  e^-_r ,
\end{align*}
define  natural transformations of the functors $F^+_{r}$ and $F^-_{r}$, respectively,
where \begin{align}\beta=\begin{cases}\label{for:normalization}
  \sqrt{\zeta(-1)}q^{-1/2}
    \quad\text{if $G_n =\Sp_{2n}(q)$}; \\
\zeta(2) \quad \text{if $G_n=\O_{2n+1}(q)$ or $\O^{\pm}_{2n}(q)$}
\end{cases}
\end{align}
which is a normalization.
Similarly, the right multiplication by the elements
\begin{align*}
& T^+_{r+2}= q^{1/2}    e^+_{r+1}e^+_r  \dot{s}_{r+2}\,  e^+_{r+1}e^+_r \,,\ \ H_{r+2}=q^{1/2} e^+_{r+1} e^-_r \dot{s}_{r+2}\, e^-_{r+1}e^+_r
 \,,\\
&T^-_{r+2}=q^{1/2} e^-_{r+1}e^-_r  \dot{s}_{r+2}\, e^-_{r+1}e^-_r \,,\ \ H'_{r+2}=q^{1/2}   e^-_{r+1} e^+_r \dot{s}_{r+2}\, e^+_{r+1}e^-_r \,,
\end{align*}
define natural transformations in 
\begin{align*}
&\End(F^+_{r+1,r+2}F^+_{r,r+1})\,,\Hom(F^+_{r+1,r+2}F^-_{r,r+1},F^-_{r+1,r+2}F^+_{r,r+1})\,,\\ &\End(F^-_{r+1,r+2}F^-_{r,r+1} )\,, \Hom(F^-_{r+1,r+2} F^+_{r,r+1},F^+_{r+1,r+2}F^-_{r,r+1})
\,,
\end{align*}
respectively.

\begin{remark}
Note that in this article, we use  slightly different notations and choose a different normalization with \cite{LLZ} to define these natural transformations.
\end{remark}
We set
$$
X^+=\bigoplus_{r\in \bbN^*} X^+_{r+1} \,,\qquad
X^-=\bigoplus_{r\in \bbN^*}X^-_{r+1}\,,\qquad T^+=\bigoplus_{r\in \bbN^*}T^+_{r+2}\,
,$$
$$
T^-=\bigoplus_{r\in \bbN^*}T^-_{r+2}\,,\qquad
H=\bigoplus_{r\in \bbN^*}H_{r+2}\,,\qquad
H'=\bigoplus_{r\in \bbN^*}H'_{r+2}\,.$$

\smallskip

\subsubsection{Categorical double quantum Heisenberg action}

The following is the main theorem of this section.
\begin{theorem}\label{Thm:doubleheisenberg}
There is a unique strict $R$-linear monoidal functor
$$
\Psi:\red{\Heis_{-2}(z^+,t^+)}\odot\blue{ \Heis_{-2}(z^-,t^-)} \rightarrow \mathcal{E}nd
\left(R G_\bullet \mod\right)
$$ such that $\Psi$
sends the generating objects $\red{\up},\red{\down},\blue{\up}$ and $\blue{\down}$ to endo-functors $F^+,E^+,F^-$ and $E^-$ of $ RG_\bullet\mod $, respectively, i.e.,
$$\Psi(\red{\up})=F^+,\quad\Psi(\red{\down})=E^+,\quad\Psi(\blue{\up})=F^-,\quad\Psi(\blue{\down})=E^-,$$
and $\Psi$ sends the generating morphisms  to the corresponding natural transformations
in  $ RG_\bullet\mod $ as follows:
\begin{itemize}
\item
$\Psi(\begin{tikzpicture}[baseline = -1mm]
	\draw[->,red] (0.08,-.2) to (0.08,.2);
      \node at (0.08,0) {$\dott$};
\end{tikzpicture})=X^+\::\;F^+ \Rightarrow F^+\,,\quad\Psi(
\begin{tikzpicture}[baseline = -1mm]
	\draw[->,red] (0.2,-.2) to (-0.2,.2);
	\draw[-,white,line width=4pt] (-0.2,-.2) to (0.2,.2);
	\draw[->,red] (-0.2,-.2) to (0.2,.2);
\end{tikzpicture})=T^+\::\; (F^+)^2 \Rightarrow (F^+)^2\,,$
\item
$\Psi(\begin{tikzpicture}[baseline = .75mm]
	\draw[<-,red] (0.3,0) to[out=90, in=0] (0.1,0.3);
	\draw[-,red] (0.1,0.3) to[out = 180, in = 90] (-0.1,0);
\end{tikzpicture})=\varepsilon^+_R\::\;F^+\circ  E^+ \Rightarrow \unit
\,,\quad\Psi(\begin{tikzpicture}[baseline = .75mm]
	\draw[<-,red] (0.3,0.3) to[out=-90, in=0] (0.1,0);
	\draw[-,red] (0.1,0) to[out = 180, in = -90] (-0.1,0.3);
\end{tikzpicture})=\eta^+_R\::\;\unit\Rightarrow E^+ \circ F^+\,,$
\item
$\Psi(\begin{tikzpicture}[baseline = .75mm]
	\draw[-,red](0.1,0.3) to[out=0, in=90](0.3,0);
	\draw[<-,red] (-0.1,0)to[out =90 , in = 180](0.1,0.3);
\end{tikzpicture})=\varepsilon^+_L\::\;E^+\circ  F^+ \Rightarrow \unit
\,,\quad\Psi(\begin{tikzpicture}[baseline = .75mm]
	\draw[-,red] (0.1,0)to[out=0, in=-90](0.3,0.3);
	\draw[<-,red] (-0.1,0.3)to[out =-90, in =180] (0.1,0);
\end{tikzpicture})=\eta^+_L\::\;\unit\Rightarrow F^+\circ  E^+\,,$
\item
$\Psi(\begin{tikzpicture}[baseline = -1mm]
	\draw[->,blue] (0.08,-.2) to (0.08,.2);
      \node at (0.08,0) {$\dott$};
\end{tikzpicture})=X^-\::\;F^- \Rightarrow F^-\,,\quad \Psi(\begin{tikzpicture}[baseline = -1mm]
	\draw[->,blue] (0.2,-.2) to (-0.2,.2);
	\draw[-,white,line width=4pt] (-0.2,-.2) to (0.2,.2);
	\draw[->,blue] (-0.2,-.2) to (0.2,.2);
\end{tikzpicture})=T^-\::\; (F^-)^2 \Rightarrow (F^-)^2\,,$
\item
$\Psi(\begin{tikzpicture}[baseline = .75mm]
	\draw[<-,blue] (0.3,0) to[out=90, in=0] (0.1,0.3);
	\draw[-,blue] (0.1,0.3) to[out = 180, in = 90] (-0.1,0);
\end{tikzpicture})=\varepsilon^-_R\::\;F^-\circ  E^- \Rightarrow \unit
\,,\quad\Psi(\begin{tikzpicture}[baseline = .75mm]
	\draw[<-,blue] (0.3,0.3) to[out=-90, in=0] (0.1,0);
	\draw[-,blue] (0.1,0) to[out = 180, in = -90] (-0.1,0.3);
\end{tikzpicture})=\eta^-_R\::\;\unit\Rightarrow E^- \circ  F^-\,,$
\item
$\Psi(\begin{tikzpicture}[baseline = .75mm]
	\draw[-,blue](0.1,0.3) to[out=0, in=90](0.3,0);
	\draw[<-,blue] (-0.1,0)to[out =90 , in = 180](0.1,0.3);
\end{tikzpicture})=\varepsilon^-_L\::\;E^-\circ  F^- \Rightarrow \unit
\,,\quad\Psi(\begin{tikzpicture}[baseline = .75mm]
	\draw[-,blue] (0.1,0)to[out=0, in=-90](0.3,0.3);
	\draw[<-,blue] (-0.1,0.3)to[out =-90, in =180] (0.1,0);
\end{tikzpicture})=\eta^-_L\::\;\unit\Rightarrow F^-\circ  E^-\,,$
\item
$\Psi(\begin{tikzpicture}[baseline = -1mm]
	\draw[->,blue] (0.2,-.2) to (-0.2,.2);
	\draw[->,red] (-0.2,-.2) to (0.2,.2);
\end{tikzpicture})=H\::\;F^+\circ  F^- \Rightarrow F^-\circ  F^+,\quad
\Psi(
\begin{tikzpicture}[baseline = -1mm]
	\draw[->,red] (0.2,-.2) to (-0.2,.2);
	\draw[->,blue] (-0.2,-.2) to (0.2,.2);
\end{tikzpicture})=H'\::\; F^-\circ  F^+ \Rightarrow F^+\circ  F^-.$
\end{itemize}
\end{theorem}

\smallskip
To prove this theorem, we need to check all relations in Theorem \ref{DoubleHeisenberg}. We will check all of them in the next section.
In fact when $k^\pm=-2$, the bubble relations have simpler forms as the following:
\begin{align}\label{redbubble}
\mathord{\begin{tikzpicture}[baseline = -1mm]
  \draw[-,thin,red] (0,0.2) to[out=180,in=90] (-.2,0);
  \draw[->,thin,red] (0.2,0) to[out=90,in=0] (0,.2);
 \draw[-,thin,red] (-.2,0) to[out=-90,in=180] (0,-0.2);
  \draw[-,thin,red] (0,-0.2) to[out=0,in=-90] (0.2,0);
      \node at (0.2,0) {$\dott$};
      \node at (0.4,0) {$\color{darkblue}\scriptstyle a$};
\end{tikzpicture}
}=-\frac{1}{t^+z^+},0,\frac{t^+}{z^+},~\text{ when $a=0,1,2,$ respectively;}\\
\mathord{\begin{tikzpicture}[baseline = -1mm]
  \draw[-,thin,blue] (0,0.2) to[out=180,in=90] (-.2,0);
  \draw[->,thin,blue] (0.2,0) to[out=90,in=0] (0,.2);
 \draw[-,thin,blue] (-.2,0) to[out=-90,in=180] (0,-0.2);
  \draw[-,thin,blue] (0,-0.2) to[out=0,in=-90] (0.2,0);
      \node at (0.2,0) {$\dott$};
      \node at (0.4,0) {$\color{darkblue}\scriptstyle a$};
\end{tikzpicture}
}=-\frac{1}{t^-z^-},0,\frac{t^-}{z^-},~\text{ when $a=0,1,2$, respectively,}\label{bluebubble}
\end{align}
Mackey formula relations for $k^\pm=-2$ have simpler forms as the following (see \cite[(4.14)]{BSW1}):
\begin{align}\label{MMackey1}
\mathord{
\begin{tikzpicture}[baseline = -.9mm]
	\draw[<-,red] (-0.28,-.6) to[out=90,in=-90] (0.28,0);
	\draw[-,white,line width=4pt] (0.28,-.6) to[out=90,in=-90] (-0.28,0);
	\draw[-,red] (0.28,-.6) to[out=90,in=-90] (-0.28,0);
	\draw[->,red] (-0.28,0) to[out=90,in=-90] (0.28,.6);
	\draw[-,line width=4pt,white] (0.28,0) to[out=90,in=-90] (-0.28,.6);
	\draw[-,red] (0.28,0) to[out=90,in=-90] (-0.28,.6);
\end{tikzpicture}
}
&=
\mathord{
\begin{tikzpicture}[baseline = -.9mm]
	\draw[->,red] (0.08,-.6) to (0.08,.6);
	\draw[<-,red] (-0.28,-.6) to (-0.28,.6);
\end{tikzpicture}
}
+t^+z^+
\mathord{
\begin{tikzpicture}[baseline=-.9mm]
	\draw[<-,red] (0.3,0.6) to[out=-90, in=0] (0,.1);
	\draw[-,red] (0,.1) to[out = 180, in = -90] (-0.3,0.6);
	\draw[-,red] (0.3,-.6) to[out=90, in=0] (0,-0.1);
	\draw[->,red] (0,-0.1) to[out = 180, in = 90] (-0.3,-.6);
\end{tikzpicture}}
\!-\frac{z^+}{t^+}
\mathord{
\begin{tikzpicture}[baseline=-.9mm]
	\draw[<-,red] (0.3,0.6) to[out=-90, in=0] (0,.1);
	\draw[-,red] (0,.1) to[out = 180, in = -90] (-0.3,0.6);
      \node at (0.44,-0.3) {$$};
	\draw[-,red] (0.3,-.6) to[out=90, in=0] (0,-0.1);
	\draw[->,red] (0,-0.1) to[out = 180, in = 90] (-0.3,-.6);
   \node at (0.27,0.3) {$\dott$};
      \node at (0.27,-0.3) {$\dott$};
   \node at (.43,.3) {$$};
\end{tikzpicture}}\:,&
\mathord{
\begin{tikzpicture}[baseline = 0mm]
	\draw[-,thin,red] (-0.28,0) to[out=90,in=-90] (0.28,.6);
	\draw[-,thin,red] (-0.28,-.6) to[out=90,in=-90] (0.28,0);
	\draw[-,line width=4pt,white] (0.28,0) to[out=90,in=-90] (-0.28,.6);
	\draw[-,line width=4pt,white] (0.28,-.6) to[out=90,in=-90] (-0.28,0);
	\draw[->,thin,red] (0.28,0) to[out=90,in=-90] (-0.28,.6);
	\draw[<-,thin,red] (0.28,-.6) to[out=90,in=-90] (-0.28,0);
\end{tikzpicture}
}
=\mathord{
\begin{tikzpicture}[baseline = 0]
	\draw[<-,thin,red] (0.08,-.6) to (0.08,.6);
	\draw[->,thin,red] (-0.28,-.6) to (-0.28,.6);
\end{tikzpicture}
}
\,;
\end{align}

\begin{align}\label{MMackey2}
\mathord{
\begin{tikzpicture}[baseline = -.9mm]
	\draw[<-,blue] (-0.28,-.6) to[out=90,in=-90] (0.28,0);
	\draw[-,white,line width=4pt] (0.28,-.6) to[out=90,in=-90] (-0.28,0);
	\draw[-,blue] (0.28,-.6) to[out=90,in=-90] (-0.28,0);
	\draw[->,blue] (-0.28,0) to[out=90,in=-90] (0.28,.6);
	\draw[-,line width=4pt,white] (0.28,0) to[out=90,in=-90] (-0.28,.6);
	\draw[-,blue] (0.28,0) to[out=90,in=-90] (-0.28,.6);
\end{tikzpicture}
}
&=
\mathord{
\begin{tikzpicture}[baseline = -.9mm]
	\draw[->,blue] (0.08,-.6) to (0.08,.6);
	\draw[<-,blue] (-0.28,-.6) to (-0.28,.6);
\end{tikzpicture}
}
+t^-z^-
\mathord{
\begin{tikzpicture}[baseline=-.9mm]
	\draw[<-,blue] (0.3,0.6) to[out=-90, in=0] (0,.1);
	\draw[-,blue] (0,.1) to[out = 180, in = -90] (-0.3,0.6);
	\draw[-,blue] (0.3,-.6) to[out=90, in=0] (0,-0.1);
	\draw[->,blue] (0,-0.1) to[out = 180, in = 90] (-0.3,-.6);
\end{tikzpicture}}
\!-\frac{z^-}{t^-}
\mathord{
\begin{tikzpicture}[baseline=-.9mm]
	\draw[<-,blue] (0.3,0.6) to[out=-90, in=0] (0,.1);
	\draw[-,blue] (0,.1) to[out = 180, in = -90] (-0.3,0.6);
      \node at (0.44,-0.3) {$$};
	\draw[-,blue] (0.3,-.6) to[out=90, in=0] (0,-0.1);
	\draw[->,blue] (0,-0.1) to[out = 180, in = 90] (-0.3,-.6);
   \node at (0.27,0.3) {$\dott$};
      \node at (0.27,-0.3) {$\dott$};
   \node at (.43,.3) {$$};
\end{tikzpicture}}\:,&
\mathord{
\begin{tikzpicture}[baseline = 0mm]
	\draw[-,thin,blue] (-0.28,0) to[out=90,in=-90] (0.28,.6);
	\draw[-,thin,blue] (-0.28,-.6) to[out=90,in=-90] (0.28,0);
	\draw[-,line width=4pt,white] (0.28,0) to[out=90,in=-90] (-0.28,.6);
	\draw[-,line width=4pt,white] (0.28,-.6) to[out=90,in=-90] (-0.28,0);
	\draw[->,thin,blue] (0.28,0) to[out=90,in=-90] (-0.28,.6);
	\draw[<-,thin,blue] (0.28,-.6) to[out=90,in=-90] (-0.28,0);
\end{tikzpicture}
}
=\mathord{
\begin{tikzpicture}[baseline = 0]
	\draw[<-,thin,blue] (0.08,-.6) to (0.08,.6);
	\draw[->,thin,blue] (-0.28,-.6) to (-0.28,.6);
\end{tikzpicture}
}
\,,
\end{align}
and curl relations have simpler forms:
 \begin{align}\label{k=2curl}
 \mathord{
\begin{tikzpicture}[baseline = -0.5mm]
	\draw[<-,red] (0,0.6) to (0,0.3);
	\draw[-,red] (-0.3,-0.2) to [out=180,in=-90](-.5,0);
	\draw[-,red] (-0.5,0) to [out=90,in=180](-.3,0.2);
	\draw[-,red] (-0.3,.2) to [out=0,in=90](0,-0.3);
	\draw[-,red] (0,-0.3) to (0,-0.6);
	\draw[-,line width=4pt,white] (0,0.3) to [out=-90,in=0] (-.3,-0.2);
	\draw[-,red] (0,0.3) to [out=-90,in=0] (-.3,-0.2);
\end{tikzpicture}
}=
0 ~\text{ and }~\mathord{
\begin{tikzpicture}[baseline = -0.5mm]
	\draw[<-,blue] (0,0.6) to (0,0.3);
	\draw[-,blue] (-0.3,-0.2) to [out=180,in=-90](-.5,0);
	\draw[-,blue] (-0.5,0) to [out=90,in=180](-.3,0.2);
	\draw[-,blue] (-0.3,.2) to [out=0,in=90](0,-0.3);
	\draw[-,blue] (0,-0.3) to (0,-0.6);
	\draw[-,line width=4pt,white] (0,0.3) to [out=-90,in=0] (-.3,-0.2);
	\draw[-, blue] (0,0.3) to [out=-90,in=0] (-.3,-0.2);
\end{tikzpicture}
}=0.
\end{align}

\subsubsection{Categorical Kac-Moody action}
Let $\scrI_+$ (resp. $\scrI_-$) be the set of the generalized eigenvalues of $X^+$ on $F^+$ (resp. $X^-$ on $F^-$).
 We have the decompositions
$$E^+=\bigoplus_{i\in I_+} E^+_i,\ \ \ \
F^+=\bigoplus_{i\in I_+} F^+_i,\ \ \ \
E^-=\bigoplus_{i'\in I_-} E^-_{i'},\ \ \ \
F^-=\bigoplus_{i'\in I_-} F^-_{i'}$$
such that $X^+-i$ is locally nilpotent on $E^+_i$ and $F^+_i$ and $X^--i'$ is locally nilpotent on $E^-_{i'}$ and $F^-_{i'}$,
 respectively. As a direct consequence of Theorem \ref{Thm:doubleheisenberg} and
 Theorem \ref{Thm:HL}, we have the following theorem. The set $I_+$ and $I_-$ will be determined in Proposition \ref{Prop:I+-}.
 \begin{theorem}\label{cor:kacmoody}
 Let $R=K$ or $k$ and assume  $\ell\nmid q(q-1).$
 Then there exists a $\frakU(\fraks\frakl'_{I_+}\oplus\fraks\frakl'_{I_-})$-categorification on $RG_\bullet\mod$, i.e., there is a strict $R$-linear monoidal functor
$$\Psi':\,\frakU(\fraks\frakl'_{I_+}\oplus \fraks\frakl'_{I_-})\to \calE nd(RG_\bullet\mod).$$
 \end{theorem}

\medskip

\subsection{The proof of the main theorem} \label{sub:proofofthm1}\hfill\\

In this section, we start to check all the relations in Theorem \ref{Thm:doubleheisenberg}. Most relations are checked by explicit computations and the most interesting one is to check the invertibility of $X^+$ and $X^-$, which is given in \S \ref{sub:invetible}.
\smallskip

\subsubsection{Double cosets decomposition}\label{sub:doublecosets}

\begin{lemma}\label{doublecoset}
\begin{itemize}
\item[$(a)$] The set $D_{r,r}$ of the distinguished representatives for double coset $W_{r}\backslash W_{r+1}/W_r$ is $\{1,x_{r+1},s_{r+1}\}.$
\item[$(b)$] The double coset $P_{r}\backslash G_{r+1}/P_{r}$ are exactly $\{P_{r},P_{r}\dot{x}_{r+1} P_{r},P_{r} \dot{s}_{r+1} P_{r}\}.$
    \item[$(c)$]  $G_r\cap {^{s_{r+1}}P}_{r}=P_{r-1}.$
    \item[$(d)$] $^{s_{r+1}}U_{r-1}\leq  U_r.$
\end{itemize}
\end{lemma}
\begin{proof}
They are easy to check by  direct computations.
\end{proof}

\begin{lemma}\label{Lem:T}
We have $$e^+_r \dot{s}_{r+1}
 e^+_r = e^+_r  e^+_{r-1} \dot{s}_{r+1} e^+_r  e^+_{r-1}\,,\quad  e^-_r \dot{s}_{r+1}
 e^-_r = e^-_r  e^-_{r-1} \dot{s}_{r+1} e^-_r  e^-_{r-1}\,,$$
$$e^+_r \dot{s}_{r+1}
 e^-_r = e^+_r  e^-_{r-1} \dot{s}_{r+1} e^-_r  e^+_{r-1}\,,\quad e^-_r \dot{s}_{r+1}
 e^+_r = e^-_r  e^+_{r-1} \dot{s}_{r+1} e^+_r  e^-_{r-1}.$$
\end{lemma}

\begin{proof}We only prove $e^+_r \dot{s}_{r+1}
 e^-_r = e^+_r  e^-_{r-1} \dot{s}_{r+1} e^-_r  e^+_{r-1}$. The others are similar.
By definition we have $ e^+_r =\frac{1}{|U_r|}\sum\limits_{g\in U_{r}}g$ and $e^-_r =\frac{1}{|U_r|}\sum\limits_{g\in U_{r}}\zeta(g^{-1})g.$
On the other hand, by Lemma \ref{doublecoset}$(d)$, we have $^{s_{r+1}}U_{r-1}\subset U_{r}$, so $^{s_{r+1}}{ e^\pm_{r-1} }\cdot  e^\pm_r = e^\pm_r\cdot ^{s_{r+1}}{ e^\pm_{r-1} } = e^\pm_r .$
By moving $e^-_{r-1}$ to the right, we get
\begin{align*}
 e^+_r  e^-_{r-1} \dot{s}_{r+1} e^-_r  e^+_{r-1}
=& e^+_r \dot{s}_{r+1}{^{{s}_{r+1}} e^-_{r-1} } e^-_r  e^+_{r-1} \\
=& e^+_r \dot{s}_{r+1} e^-_r  e^+_{r-1}
\end{align*}
 since $^{s_{r+1}}{ e^-_{r-1} }\cdot e^-_r= e^-_r.$
By moving $e^+_{r-1}$ to the left, we get
\begin{align*}e^+_r \dot{s}_{r+1} e^-_r  e^+_{r-1}
 =&e^+_r (^{{s}_{r+1}} e^+_{r-1} \dot{s}_{r+1}) e^-_r \\
=& e^+_r \dot{s}_{r+1} e^-_r
\end{align*}
where the last equation is deduced from $e^+_r\cdot (^{s_{r+1}}{ e^+_{r-1} }) = e^+_r $. So the lemma follows.
\end{proof}

\begin{proposition}\label{mkd}
\begin{itemize}
\item[$(a)$] As an $( R G_r , R G_r )$-bimodule, $ e^\pm_r \cdot  R G_{r+1}
\cdot e^\pm_r $ is generated by $ e^\pm_r \,, X^\pm_{r+1} $ and $T^\pm_{r+1}.$
Moreover, We have the following $( R G_r , R G_r )$-bimodule isomorphism
\begin{align} e^\pm_r\cdot  R G_{r+1} \cdot e^\pm_r \cong  R G_r \cdot e^\pm_r \oplus  R G_r  \cdot X^\pm_{r+1}  \oplus  R G_r \cdot T^\pm_{r+1}\cdot  R G_r .
\end{align}
\item[$(b)$]As  $( R G_r , R G_r )$-bimodules, $ e^+_r \cdot
R G_{r+1} \cdot  e^-_r $ is generated by $H_{r+1}$ and $ e^-_r
\cdot  R G_{r+1}  \cdot e^+_r $ is generated by $H'_{r+1},$
equivalently, we have the following $( R G_r , R G_r )$-bimodule isomorphisms
\begin{align} e^+_r\cdot   R G_{r+1} \cdot  e^-_r \cong   R G_r \cdot H_{r+1}\cdot  R G_r
\end{align}
and
\begin{align} e^-_r \cdot  R G_{r+1} \cdot  e^+_r \cong   R G_r \cdot H'_{r+1}\cdot  R G_r .\end{align}
%$$\Lambda G _{n}H'_{n+1,n-1}\Lambda G _{n}\cong \Lambda G _{n}e'_{n,n-1}\otimes_{\Lambda G _{n-1}}e_{n,n-1}\Lambda G _n.$$
\end{itemize}
\end{proposition}
\begin{proof}

 $(a)$ By Lemma \ref{doublecoset}$(b)$, we have $$G_{r+1}=P_{r}\sqcup P_{r}\dot{x}_{r+1}P_{r}\sqcup P_{r}\dot{s}_{r+1}P_{r}.$$
It is obviously $(L_{r,1},L_{r,1})$-equivariant. So we have the following $(R G_r,R G_r)$-bimodule decomposition $$ e^\pm_r\cdot   R G_{r+1} \cdot  e^\pm_r = e^\pm_r\cdot  RP_{r} \cdot e^\pm_r \oplus  e^\pm_r\cdot  R(P_{r}\dot{x}_{r+1}P_{r})\cdot  e^\pm_r \oplus e^\pm_r\cdot  R(P_{r}\dot{s}_{r+1}P_{r}) \cdot e^\pm_r .$$

Since $ e^\pm_r $ commutes with $G_r$ and $\dot{x}_{r+1}$ commutes with $G_r$, we have $$ e^\pm_r \cdot RP_{r}\cdot  e^\pm_r =RP_{r}\cdot  e^\pm_r = e^\pm_r\cdot  RP_{r}= R G_r \cdot  e^\pm_r $$ and $$ e^\pm_r\cdot  R(P_{r}\dot{x}_{r+1}P_{r})\cdot  e^\pm_r = R G_r  \cdot  e^\pm_r \dot{x}_{r+1}  e^\pm_r = R G_r  \cdot (c e^\pm_r \dot{x}_{r+1}  e^\pm_r )= R G_r \cdot  X^\pm_{r+1}$$
where $c $ is any non-zero constant.
\smallskip

Now we prove $ e^\pm_r \cdot R(P_{r}\dot{s}_{r+1}P_{r})\cdot  e^\pm_r =
 R G_r \cdot T^\pm_{r+1}\cdot  R G_r .$
Note that  $$ e^\pm_r\cdot  R(P_{r}\dot{s}_{r+1}P_{r})\cdot  e^\pm_r = R G_r  \cdot e^\pm_r \dot{s}_{r+1}
 e^\pm_r \cdot  R G_r , $$
which is equal to $ R G_r  \cdot e^\pm_r e^\pm_{r-1}\dot{s}_{r+1}
 e^\pm_re^\pm_{r-1} \cdot  R G_r = R G_r \cdot T^\pm_{r+1}\cdot  R G_r$ by Lemma \ref{Lem:T}.
 \smallskip

$(b)$ Similarly, we have the following $(G_r,G_r)$-bimodule decomposition $$ e^+_r \cdot  R G_{r+1} \cdot  e^-_r = e^+_r\cdot  RP_{r}\cdot  e^-_r \oplus e^+_r\cdot  R(P_{r}\dot{x}_{r+1}P_{r})\cdot  e^-_r \oplus e^+_r\cdot  R(P_{r}\dot{s}_{r+1}P_{r})\cdot  e^-_r .$$

Note that $ e^+_r $ commutes with $G_r$ and $\dot{x}_{r+1}$ commutes with $G_r$, so we have
$$e^+_r\cdot  R P_{r}\cdot  e^-_r =R P_r\cdot  e^+_r  e^-_r =0$$
where the last equation follows from the orthogonality of $ e^+_r $ and $ e^-_r .$
Similarly, we have  $$ e^+_r \cdot R (P_{r}\dot{x}_{r+1}P_{r})\cdot e^-_r = R G_r  \cdot e^+_r \dot{x}_{r+1} e^-_r = R G_r \cdot \dot{x}_{r+1}  e^+_r  e^-_r =0.$$
So we have
$ e^+_r \cdot  R G_{r+1} \cdot  e^-_r =  e^+_r\cdot  R(P_{r}\dot{s}_{r+1}P_{r})\cdot  e^-_r ,$ and by a similar argument as in $(a)$, we get
$$ e^+_r\cdot  R(P_{r}\dot{s}_{r+1}P_{r})\cdot  e^-_r= R G_r \cdot  e^+_r   \dot{s}_{r+1} e^-_r   \cdot  R G_r
= R G_r \cdot H_{r+1}\cdot  R G_r \,.
$$

The same proof works for $e^-_r \cdot  R G_{r+1} \cdot  e^+_r \cong   R G_r \cdot H'_{r+1}\cdot  R G_r.$
\end{proof}

\begin{remark}Note that the maps $\varepsilon^\pm_L:  e^\pm_r  \cdot R G_{r+1}\cdot   e^\pm_r \to   R G_{r+1} $ defined in \S 1.1.2 coincide with the projections to the first
direct summand.
\end{remark}
\smallskip

\subsubsection{Affine-Hecke-type relations}

\begin{proposition}\label{prop:12relations}

The natural transformations  $X^+=\begin{tikzpicture}[baseline = -1mm]
	\draw[->,red] (0.08,-.2) to (0.08,.2);
      \node at (0.08,0) {$\dott$};
\end{tikzpicture}$,
$T^+=\begin{tikzpicture}[baseline = -1mm]
	\draw[->,red] (0.2,-.2) to (-0.2,.2);
	\draw[-,white,line width=4pt] (-0.2,-.2) to (0.2,.2);
	\draw[->,red] (-0.2,-.2) to (0.2,.2);
\end{tikzpicture}\:,$
$X^-=\begin{tikzpicture}[baseline = -1mm]
	\draw[->,blue] (0.08,-.2) to (0.08,.2);
      \node at (0.08,0) {$\dott$};
\end{tikzpicture}$,
$T^-=\begin{tikzpicture}[baseline = -1mm]
	\draw[->,blue] (0.2,-.2) to (-0.2,.2);
	\draw[-,white,line width=4pt] (-0.2,-.2) to (0.2,.2);
	\draw[->,blue] (-0.2,-.2) to (0.2,.2);
\end{tikzpicture}\:$,
$H=\begin{tikzpicture}[baseline = -1mm]
	\draw[->,blue] (0.2,-.2) to (-0.2,.2);
	\draw[->,red] (-0.2,-.2) to (0.2,.2);
\end{tikzpicture}\:$,
$H'=\begin{tikzpicture}[baseline = -1mm]
	\draw[->,red] (0.2,-.2) to (-0.2,.2);
	\draw[->,blue] (-0.2,-.2) to (0.2,.2);
\end{tikzpicture}\:$
satisfy affine Hecke relations $(\ref{AHH1})$ and $(\ref{AHH2})$ for both red color and blue color and mixed affine-Hecke-type relations $(\ref{Mixaff1})$--$(\ref{Mixaff3})$.

\end{proposition}
\begin{proof} See the proof of \cite[Proposition 4.2]{LLZ}.
\end{proof}
\smallskip

\subsubsection{Adjunctions}
\begin{proposition}
\label{prop:adj}The  maps $\varepsilon^\pm_R$, $\eta^\pm_R$, $\varepsilon^\pm_L$, $\eta^\pm_L$ defined above make $(E^\pm,F^\pm)$ into a biadjoint pair, i.e., the natural transformations satisfy the  relations $(\ref{adj})$ for both red color and blue color.
\end{proposition}
\begin{proof}
Note that the functors $F^+,F^-$ are  special cases of parabolic induction functors and  the functors $E^+,E^-$ are  special cases of parabolic restriction functors defined in \cite[\S 7]{LS}. The proof is deduced from  \cite[Proposition 7.5]{LS}.
\end{proof}
\smallskip

\subsubsection{Rightward crossings}
Now we introduce the rightward crossings:
\begin{align*}
\Rcross^+:=\mathord{
\begin{tikzpicture}[baseline = -.5mm]
	\draw[thin,red,->] (-0.28,-.3) to (0.28,.4);
	\draw[line width=4pt,white,-] (0.28,-.3) to (-0.28,.4);
	\draw[<-,thin,red] (0.28,-.3) to (-0.28,.4);
\end{tikzpicture}
}:=
\mathord{
\begin{tikzpicture}[baseline = 0]
	\draw[->,thin,red] (0.3,-.5) to (-0.3,.5);
	\draw[line width=4pt,-,white] (-0.2,-.2) to (0.2,.3);
	\draw[-,thin,red] (-0.2,-.2) to (0.2,.3);
        \draw[-,thin,red] (0.2,.3) to[out=50,in=180] (0.5,.5);
        \draw[->,thin,red] (0.5,.5) to[out=0,in=90] (0.8,-.5);
        \draw[-,thin,red] (-0.2,-.2) to[out=230,in=0] (-0.6,-.5);
        \draw[-,thin,red] (-0.6,-.5) to[out=180,in=-90] (-0.85,.5);
\end{tikzpicture}
}\:\quad\text{ and}\quad
&
\Rcross^-:=\mathord{
\begin{tikzpicture}[baseline = -.5mm]
	\draw[thin,blue,->] (-0.28,-.3) to (0.28,.4);
	\draw[line width=4pt,white,-] (0.28,-.3) to (-0.28,.4);
	\draw[<-,thin,blue] (0.28,-.3) to (-0.28,.4);
\end{tikzpicture}
}:=
\mathord{
\begin{tikzpicture}[baseline = 0]
	\draw[->,thin,blue] (0.3,-.5) to (-0.3,.5);
	\draw[line width=4pt,-,white] (-0.2,-.2) to (0.2,.3);
	\draw[-,thin,blue] (-0.2,-.2) to (0.2,.3);
        \draw[-,thin,blue] (0.2,.3) to[out=50,in=180] (0.5,.5);
        \draw[->,thin,blue] (0.5,.5) to[out=0,in=90] (0.8,-.5);
        \draw[-,thin,blue] (-0.2,-.2) to[out=230,in=0] (-0.6,-.5);
        \draw[-,thin,blue] (-0.6,-.5) to[out=180,in=-90] (-0.85,.5);
\end{tikzpicture}
}\:.
\end{align*}
By the definition, $\Rcross^\pm=(1_{E^\pm}\otimes1_{F^\pm}\otimes\varepsilon^\pm_R)\circ(1_{E^\pm}\otimes T^\pm\otimes 1_{E^\pm})\circ(\eta^\pm_R\otimes 1_{F^\pm}\otimes 1_{E^\pm}).$
We also define the rightward mixed crossings:
\begin{align*}
\MixRcross:=\mathord{
\begin{tikzpicture}[baseline = -.5mm]
	\draw[thin,blue,->] (-0.28,-.3) to (0.28,.4);
	\draw[<-,thin,red] (0.28,-.3) to (-0.28,.4);
\end{tikzpicture}
}:=
\mathord{
\begin{tikzpicture}[baseline = 0]
	\draw[->,thin,blue] (0.3,-.5) to (-0.3,.5);
	\draw[-,thin,red] (-0.2,-.2) to (0.2,.3);
        \draw[-,thin,red] (0.2,.3) to[out=50,in=180] (0.5,.5);
        \draw[->,thin,red] (0.5,.5) to[out=0,in=90] (0.8,-.5);
        \draw[-,thin,red] (-0.2,-.2) to[out=230,in=0] (-0.6,-.5);
        \draw[-,thin,red] (-0.6,-.5) to[out=180,in=-90] (-0.85,.5);
\end{tikzpicture}
}\:
 \quad\text{ and}\quad
&
\MixRcross':=\mathord{
\begin{tikzpicture}[baseline = -.5mm]
	\draw[thin,red,->] (-0.28,-.3) to (0.28,.4);
	\draw[<-,thin,blue] (0.28,-.3) to (-0.28,.4);
\end{tikzpicture}
}:=
\mathord{
\begin{tikzpicture}[baseline = 0]
	\draw[->,thin,red] (0.3,-.5) to (-0.3,.5);
	\draw[-,thin,blue] (-0.2,-.2) to (0.2,.3);
        \draw[-,thin,blue] (0.2,.3) to[out=50,in=180] (0.5,.5);
        \draw[->,thin,blue] (0.5,.5) to[out=0,in=90] (0.8,-.5);
        \draw[-,thin,blue] (-0.2,-.2) to[out=230,in=0] (-0.6,-.5);
        \draw[-,thin,blue] (-0.6,-.5) to[out=180,in=-90] (-0.85,.5);
\end{tikzpicture}
}\:.
\end{align*}
By the definition, $\MixRcross=(1_{E^+}\otimes1_{F^-}\otimes \varepsilon^+_R)\circ(1_{E^+}\otimes H\otimes 1_{E^+})\circ(\eta^+_R\otimes 1_{F^-}\otimes 1_{E^+})$
and  $\MixRcross'=(1_{E^-}\otimes1_{F^+}\otimes \varepsilon^-_R)\circ(1_{E^-}\otimes H'\otimes 1_{E^-})\circ(\eta^-_R\otimes 1_{F^+}\otimes 1_{E^-})$.
\begin{lemma}\label{lcross}
\begin{itemize}
\item[$(a)$]
 The natural transformations $\Rcross^\pm: F^\pm_{r-1,r}E^\pm_{r,r-1}\Rightarrow E^\pm_{r+1,r}F^\pm_{r,r+1}$ are given by  the $( R G_r , R G_r )$-bimodule map $$ R G_r  \cdot e^\pm_{r-1} \otimes _{ R G_{r-1} } e^\pm_{r-1}\cdot  R G_r \to   e^\pm_r \cdot R G_{r+1} \cdot e^\pm_r ,$$ $$a e^\pm_{r-1} \otimes  e^\pm_{r-1} b\mapsto  e^\pm_r a T_{r+1}b e^\pm_r $$
 for any $a,b\in G_r.$

 \item[$(b)$]
  $\MixRcross: F^+_{r-1,r}E^-_{r,r-1}\Rightarrow E^-_{r+1,r}F^+_{r,r+1}$ is given by  the $( R G_r , R G_r )$-bimodule map $$ R G_r \cdot e^+_{r-1} \otimes _{ R G_{r-1} } e^-_{r-1} \cdot R G_r \to  e^-_r\cdot  R G_{r+1} \cdot e^+_r ,$$ $$a e^+_{r-1} \otimes  e^-_{r-1} b\mapsto  e^-_r a H_{r+1}b e^+_r $$
and $\MixRcross': F^-_{r-1,r}E^+_{r, r-1}\Rightarrow E^+_{r+1,r}F^-_{r,r+1} $ is given by the $( R G_r , R G_r )$-bimodule map $$ R G_r \cdot e^-_{r-1} \otimes _{ R G_{r-1} } e^+_{r-1} \cdot  R G_r \to  e^+_r \cdot R G_{r+1} \cdot e^-_r $$$$a e^-_{r-1} \otimes  e^+_{r-1} b\mapsto  e^+_r a H'_{r+1}b e^-_r $$
 for any $a,b\in G_r.$
 \end{itemize}
\end{lemma}
\begin{proof}

%We prove $(a)$. Note that $\Rcross_r$ is the composition of the following $(G_r,G_r)$-bimodule maps:
%\begin{align*}
%& R G_r  e^+_{r-1} \otimes _{ R G_{r-1} } e^+_{r-1}  R G_r \cong  R G_r \otimes_{ R G_r } R G_r  e^+_{r-1} \otimes _{ R G_{r-1} } e^+_{r-1}  R G_r  \\
%&\stackrel{ \Rcup\otimes \id\otimes \id}{\longrightarrow} e^+_r  R G_{r+1} \otimes_{ R G_{r+1} } R G_{r+1}  e^+_r \otimes_{ R G_r } R G_r  e^+_{r-1} \otimes _{ R G_{r-1} } e^+_{r-1}  R G_r \\
%&\stackrel{ \id\otimes \Ucross\otimes \id}{\longrightarrow} e^+_r  R G_{r+1} \otimes_{ R G_{r+1} } R G_{r+1}  e^+_r \otimes_{ R G_r } R G_r  e^+_{r-1} \otimes _{ R G_{r-1} } e^+_{r-1}  R G_r \\&\stackrel{ \id\otimes\id\otimes \Rcap}{\longrightarrow} e^+_r  R G_{r+1} \otimes_{ R G_{r+1} } R G_{r+1}  e^+_r \otimes_{ R G_r } R G_r \cong e^+_r  R G_{r+1}  e^+_r ,
%\end{align*}
%We compute the image of $ e^+_{r-1} \otimes  e^+_{r-1} $:
%\begin{align*}
%& e^+_{r-1} \otimes  e^+_{r-1}  = 1\otimes  e^+_{r-1} \otimes  e^+_{r-1} \\&\mapsto e^+_r \otimes  e^+_r \otimes  e^+_{r-1} \otimes  e^+_{r-1}  \\
%&\mapsto
% e^+_r \otimes  e^+_r \otimes  e^+_{r-1} T_{r+1,r-1}\otimes  e^+_{r-1} \\
%&\mapsto
% e^+_r \otimes  e^+_r  T_{r+1,r-1}\otimes  e^+_{r-1} =  e^+_r  T_{r+1,r-1} e^+_r .
%\end{align*}
We only prove $(a)$ and the proof of $(b)$ is similar. As an $(R G_{r}, R G_r)$-bimodule, the module $ R G_r  \cdot e^\pm_{r-1} \otimes _{ R G_{r-1} } e^\pm_{r-1}\cdot  R G_r $ is generated by the unique  element $ e^\pm_{r-1} \otimes  e^\pm_{r-1} $, so we only need to compute the image of it. Note that $\Rcross^\pm_r$ is the composition of the following bimodule maps:
\begin{align*}
& R G_r\cdot  e^\pm_{r-1} \otimes _{ R G_{r-1} } e^\pm_{r-1} \cdot R G_r \cong  R G_r  \otimes_{ R G_r } R G_r \cdot  e^\pm_{r-1} \otimes _{ R G_{r-1} } e^\pm_{r-1} \cdot  R G_r \\&\stackrel{\eta^\pm_R\otimes 1\otimes 1}{\longrightarrow} e^\pm_r\cdot
 R G_{r+1} \otimes_{ R G_{r+1} } R G_{r+1}\cdot  e^\pm_r\otimes_{ R G_r }
 R G_r \cdot e^-_{r-1} \otimes _{ R G_{r-1} } e^\pm_{r-1}\cdot  R G_r  \\&\stackrel{ 1\otimes T^\pm\otimes 1}{\longrightarrow} e^\pm_r\cdot
 R G_{r+1} \otimes_{ R G_{r+1} } R G_{r+1}\cdot  e^\pm_r\otimes_{ R G_r }
 R G_r \cdot e^\pm_{r-1} \otimes _{ R G_{r-1} } e^\pm_{r-1}\cdot  R G_r   \\&\stackrel{ 1\otimes1\otimes \varepsilon^\pm_R}{\longrightarrow} e^\pm_r \cdot R G_{r+1} \otimes_{ R G_{r+1} } R G_{r+1} \cdot e^\pm_r \otimes_{ R G_r } R G_r
\cong  e^\pm_r\cdot  R G_{r+1} \cdot e^\pm_r .
\end{align*}
We compute the image of $ e^\pm_{r-1} \otimes  e^\pm_{r-1} $ as follows:
\begin{align*}
& e^\pm_{r-1} \otimes  e^\pm_{r-1}  \xrightarrow{\thicksim} 1\otimes   e^\pm_{r-1} \otimes e^\pm_{r-1}\\
&\mapsto -(t^\pm z^\pm)^{-1} e^\pm_r \otimes  e^\pm_r\otimes e^\pm_{r-1} \otimes  e^\pm_{r-1} \\
&\mapsto-(t^\pm z^\pm)^{-1}
 e^\pm_r\otimes T^\pm_{r+1}\otimes e^\pm_{r-1} \otimes  e^\pm_{r-1} \\
&\mapsto
  e^\pm_r \otimes  T^\pm_{r+1}\otimes  e^\pm_{r-1} \xrightarrow{\thicksim}   e^\pm_r  T^\pm_{r+1} e^\pm_{r-1} =T^\pm_{r+1}.
\end{align*}
\end{proof}
\smallskip
\subsubsection{Leftward crossings}
Now we introduce the leftward crossings:
\begin{align*}
\Lcross^+:=\mathord{
\begin{tikzpicture}[baseline = -.5mm]
	\draw[<-,red] (-0.28,-.3) to (0.28,.4);
	\draw[line width=4pt,white,-] (0.28,-.3) to (-0.28,.4);
	\draw[->,red] (0.28,-.3) to (-0.28,.4);
\end{tikzpicture}
}:=
\mathord{
\begin{tikzpicture}[baseline = 0]
	\draw[-,red] (-0.2,.2) to (0.2,-.3);
	\draw[-,line width=4pt,white] (0.3,.5) to (-0.3,-.5);
	\draw[<-,red] (0.3,.5) to (-0.3,-.5);
        \draw[-,red] (0.2,-.3) to[out=130,in=180] (0.5,-.5);
        \draw[-,red] (0.5,-.5) to[out=0,in=270] (0.8,.5);
        \draw[-,red] (-0.2,.2) to[out=130,in=0] (-0.5,.5);
        \draw[->,red] (-0.5,.5) to[out=180,in=-270] (-0.8,-.5);
\end{tikzpicture}
}\:
\quad\text{ and}\quad
&\Lcross^-:=\mathord{
\begin{tikzpicture}[baseline = -.5mm]
	\draw[<-,blue] (-0.28,-.3) to (0.28,.4);
	\draw[line width=4pt,white,-] (0.28,-.3) to (-0.28,.4);
	\draw[->,blue] (0.28,-.3) to (-0.28,.4);
\end{tikzpicture}
}:=
\mathord{
\begin{tikzpicture}[baseline = 0]
	\draw[-,blue] (-0.2,.2) to (0.2,-.3);
	\draw[-,line width=4pt,white] (0.3,.5) to (-0.3,-.5);
	\draw[<-,blue] (0.3,.5) to (-0.3,-.5);
        \draw[-,blue] (0.2,-.3) to[out=130,in=180] (0.5,-.5);
        \draw[-,blue] (0.5,-.5) to[out=0,in=270] (0.8,.5);
        \draw[-,blue] (-0.2,.2) to[out=130,in=0] (-0.5,.5);
        \draw[->,blue] (-0.5,.5) to[out=180,in=-270] (-0.8,-.5);
\end{tikzpicture}
}\:.
\end{align*}
By the definition, $\Lcross^\pm=(\varepsilon^\pm_L\otimes 1_{F^\pm}\otimes 1_{E^\pm} )\circ(1_{E^\pm}\otimes T^\pm \otimes 1_{E^\pm})\circ(1_{E^\pm}\otimes 1_{F^\pm}\otimes\eta^\pm_L).$
We also define the leftward mixed crossings:
$$
\MixLcross:=\mathord{
\begin{tikzpicture}[baseline = 0]
	\draw[->,red] (0.28,-.3) to (-0.28,.4);
	\draw[<-,blue] (-0.28,-.3) to (0.28,.4);
\end{tikzpicture}
}
:=
\mathord{
\begin{tikzpicture}[baseline = 0]
	\draw[<-, red] (0.3,.5) to (-0.3,-.5);
	\draw[-,blue] (-0.2,.2) to (0.2,-.3);
        \draw[-,blue] (0.2,-.3) to[out=130,in=180] (0.5,-.5);
        \draw[-,blue] (0.5,-.5) to[out=0,in=270] (0.9,.5);
        \draw[-,blue] (-0.2,.2) to[out=130,in=0] (-0.6,.5);
        \draw[->,blue] (-0.6,.5) to[out=180,in=-270] (-0.9,-.5);
\end{tikzpicture}
}\quad\text{and}\quad
\MixLcross':=\mathord{
\begin{tikzpicture}[baseline = 0]
	\draw[->,blue] (0.28,-.3) to (-0.28,.4);
	\draw[<-,red] (-0.28,-.3) to (0.28,.4);
\end{tikzpicture}
}
:=
\mathord{
\begin{tikzpicture}[baseline = 0]
	\draw[<-, blue] (0.3,.5) to (-0.3,-.5);
	\draw[-,red] (-0.2,.2) to (0.2,-.3);
        \draw[-,red] (0.2,-.3) to[out=130,in=180] (0.5,-.5);
        \draw[-,red] (0.5,-.5) to[out=0,in=270] (0.9,.5);
        \draw[-,red] (-0.2,.2) to[out=130,in=0] (-0.6,.5);
        \draw[->,red] (-0.6,.5) to[out=180,in=-270] (-0.9,-.5);
\end{tikzpicture}
}\,.$$
By the definition, $\MixLcross=(\varepsilon^-_L\otimes 1_{F^+}\otimes 1_{E^-} )\circ(1_{E^-}\otimes H \otimes 1_{E^-})\circ(1_{E^-}\otimes 1_{F^+}\otimes\eta^-_L)$ and $\MixLcross'=(\varepsilon^+_L\otimes 1_{F^-}\otimes 1_{E^+} )\circ(1_{E^+}\otimes H' \otimes 1_{E^+})\circ(1_{E^+}\otimes 1_{F^-}\otimes\eta^+_L).$

\begin{lemma}\label{Lem:Rcrossing}
\begin{itemize}
\item[$(a)$]
 The natural transformation $\Lcross^\pm:E^\pm_{r+1,r}F^\pm_{r,r+1}\Rightarrow F^\pm_{r-1,r}
 E^\pm_{r,r-1}$ is given by the $( R G_r , R G_r )$-bimodule map $$ e^\pm_r \cdot  R G_{r+1} \cdot e^\pm_r \to  R G_r  \cdot e^\pm_{r-1} \otimes _{ R G_{r-1} } e^\pm_{r-1}\cdot  R G_r $$
 $$ e^\pm_r \mapsto 0,$$
$$   X^\pm_{r+1}  \mapsto 0,$$
$$T^\pm_{r+1}  \mapsto  e^\pm_{r-1} \otimes  e^\pm_{r-1} .$$

\item[$(b)$]
 The natural transformation $\MixLcross:E^-_{r,r-1}F^+_{r,r+1}\Rightarrow F^+_{r-1,r}E^-_{r, r-1}$ is given by the $( R G_r , R G_r )$-bimodule map $$ e^-_r\cdot   R G_{r+1}  \cdot e^+_r \to  R G_r \cdot  e^+_{r-1} \otimes _{ R G_{r-1} } e^-_{r-1}\cdot   R G_r ,$$
$$H'_{r+1}\mapsto  e^+_{r-1} \otimes  e^-_{r-1},$$
and the natural transformation  $\MixLcross':E^+_{r+1,r}F^-_{r,r+1}\Rightarrow F^-_{r,r-1}E^+_{r-1,r} $ is given by the $( R G_r , R G_r )$-bimodule map $$ e^+_r \cdot  R G_{r+1} \cdot  e^-_r \to  R G_r  \cdot e^-_{r-1} \otimes _{ R G_{r-1} } e^+_{r-1}\cdot   R G_r ,$$
$$H_{r+1}\mapsto  e^-_{r-1} \otimes  e^+_{r-1} .$$
\end{itemize}
\end{lemma}
\begin{proof}Let $G_{r}=\coprod\limits_{i=1}^{s}P_{r-1}g_i$ be a decomposition of $G_{r}
$ into left $P_{r-1}$ cosets.

$(a)$
$\Rcross^\pm$ is the composition of the following bimodule maps:
\begin{align*}& e^\pm_r \cdot R G_{r+1} \cdot e^\pm_r \cong e^\pm_r \cdot R G_{r+1} \cdot e^\pm_r\otimes_{ R G_r }  R G_r \\
&\stackrel{1\otimes 1\otimes \eta^\pm_L}{\longrightarrow} e^\pm_r\cdot  R G_{r+1} \cdot e^\pm_r\otimes_{ R G_r }  R G_r \cdot e^\pm_{r-1}
\otimes_{ R G_{r-1} } e^\pm_{r-1}\cdot  R G_r \\&\stackrel{ 1\otimes T^\pm\otimes 1}{\longrightarrow}e^\pm_r \cdot R G_{r+1} \cdot e^\pm_r\otimes_{ R G_r }  R G_r \cdot e^\pm_{r-1}
\otimes_{ R G_{r-1} } e^\pm_{r-1}\cdot  R G_r \\
 &\stackrel{ \varepsilon^\pm_L\otimes 1\otimes 1}{\longrightarrow}  R G_r \otimes_{ R G_r } R G_r  \cdot e^\pm_{r-1} \otimes _{ R G_{r-1} }e^\pm_{r-1} \cdot  R G_r
 \cong  R G_r  \cdot e^\pm_{r-1} \otimes _{ R G_{r-1} } e^\pm_{r-1} \cdot  R G_r .
\end{align*}

We compute the image of $ y\in  e^\pm_r \cdot R G_{r+1} \cdot e^\pm_r$:
\begin{align*}& y  \xrightarrow{\thicksim} y\otimes 1\\
&\mapsto y \otimes Z^\pm_r = y\otimes (\sum\limits_{i=1}^sg_i^{-1} e^\pm_{r-1} \otimes e^\pm_{r-1} g_i)\\
& \mapsto
y \otimes\sum\limits_{i=1}^s(g_i^{-1}T^\pm_{r+1} \otimes e^\pm_{r-1} g_i )\\&= \sum\limits_{i=1}^s(y g_i^{-1}T^\pm_{r+1} \otimes e^\pm_{r-1} g_i )
\\
&\mapsto\sum\limits_{i=1}^s(\Proj^{\pm}_{P_r}(y g^{-1}_i T^\pm_{r+1} ) \otimes  e^\pm_{r-1} g_i)\,.
\end{align*}

Let $y=e^\pm_r.$ For any $g_i\in  G_r,$ $g^{-1}_i\dot{s}_{r+1}\notin P_r$, we have $\Proj^\pm_{P_r}( e^\pm_r g^{-1}_i\dot{s}_{r+1} e^\pm_r )=0$.
Hence $ e^\pm_r \mapsto 0.$
\smallskip

 Let $y=X^\pm_{r+1}.$
For any $h\in G_r,$
we have
\begin{align*}
&X^\pm_{r+1}h T^\pm_{r+1}=hX^\pm_{r+1} T^\pm_{r+1}\\
&=hT^\pm_{r+1}X^\pm_{r}T^\pm_{r+1}T^\pm_{r+1}\quad\text{(by Prop.\ref{prop:12relations} $(\ref{AHH1})$ and $(\ref{AHH2})$)}\\
&=hT^\pm_{r+1}X^\pm_{r}(z^\pm T^\pm_{r+1}+e^\pm_re^\pm_{r-1})\quad\text{(by Prop.\ref{prop:12relations} $(\ref{AHH1})$)}\\
&=z^\pm hX^\pm_{r+1}+hT^\pm_{r+1}X^\pm_{r}\quad\text{(by Prop.\ref{prop:12relations} $(\ref{AHH1})$ and $(\ref{AHH2})$)}\\
&=z^\pm hX^\pm_{r+1}+q^{r-1-1/2}h e^\pm_r
\dot{s}_{r+1}\dot{x}_{r} e^\pm_r \notin RP_r.
\end{align*}
So
$\Proj^\pm_{G_r}(X^\pm_{r+1}g^{-1}_iT^\pm_{r+1})=0$ for all $g_i,$ and hence
$X^\pm_{r+1}\mapsto 0.$

\smallskip

 Let $y=T^\pm_{r+1}.$
For any $h\in G_r,$ we have
\begin{align*}&T^\pm_{r+1}hT^\pm_{r+1}\\
&=q e^\pm_r \dot{s}_{r+1} e^\pm_r h e^\pm_r
\dot{s}_{r+1} e^\pm_r\quad\text{(by Lemma \ref{Lem:T})}\\
&=q{(^{s_{r+1}}h)} e^\pm_r \dot{s}_{r+1} e^\pm_r  e^\pm_r \dot{s}_{r+1} e^\pm_r \\
&={^{{s}_{r+1}}h}T^\pm_{r+1}T^\pm_{r+1}\\
&={^{s_{r+1}}h}(z^\pm T^\pm_{r+1}+ e^\pm_r  e^\pm_{r-1} )\quad\text{(by Prop.\ref{prop:12relations} $(\ref{AHH1})$ and $(\ref{AHH2})$)}\\
&=z^\pm e^\pm_r (^{s_{r+1}}h)T^\pm_{r+1}+(^{s_{r+1}}h) e^\pm_r  e^\pm_{r-1} \\
&=q^{\frac{1}{2}}z^{\pm} e^\pm_r \dot{s}_{r+1}h e^\pm_r +(^{s_{r+1}}h) e^\pm_r
e^\pm_{r-1}\quad\text{(by Lemma \ref{Lem:T})}.
\end{align*}
 Note that $q^{\frac{1}{2}}z^\pm e^\pm_r \dot{s}_{r+1}h e^\pm_r \notin RP_r,$ hence
$$\Proj^\pm_{P_r}(T^\pm_{r+1}g^{-1}_iT^\pm_{r+1})=\begin{cases}
 (^{s_{r+1}}g_i)^{-1} e^\pm_r  e^\pm_{r-1} \,\,&\text{if}\,\, (^{s_{r+1}}g_i)^{-1}\in P_r;\\
0\,\,&\text{if}\,\,  (^{s_{r+1}}g_i)^{-1}\notin P_r.
\end{cases}$$
However, by Lemma \ref{doublecoset} (c), $G_r\cap {^{s_{r+1}}P_r}=P_{r-1}.$ So if $(^{s_{r+1}}g_i)^{-1}\in P_r$, we must have $g^{-1}_i\in P_{r-1}$ and we can write $$g^{-1}_i= v\times l\in P_{r-1}=U_{r-1}\rtimes G_{r-1}$$ where $l\in G_r$ and $v\in U_{r-1}.$
Since $v\in U_{r-1}$, by Lemma \ref{doublecoset} (d), we have ${^{s_{r+1}}v}\in U_{r}.$ It is easy to show that ${^{s_{r+1}}v} e^+_r = e^+_r ,$  ${^{s_{r+1}}v} e^-_r =\zeta(v) e^-_r ,$ $v^{-1} e^+_{r-1} = e^+_{r-1} $ and $v^{-1} e^-_{r-1} =\zeta(v) e^-_{r-1} .$
There is only one nonzero term since $g^{-1}_i\in P_{r-1}$ and the nonzero term is exactly
\begin{align*}& \Proj^{\pm}_{P_r}((^{s_{r+1}}g_i)^{-1} e^\pm_r  e^\pm_{r-1} ) \otimes e^\pm_{r-1} g_i   \\&=   \Proj^{\pm}_{P_r}({{^{s_{r+1}}v}l} e^\pm_r e^\pm_{r-1} )\otimes e^\pm_{r-1} l^{-1} v^{-1}\\&= e^\pm_{r-1} l\otimes l^{-1} e^\pm_{r-1} = e^\pm_{r-1} \otimes   e^\pm_{r-1} .
\end{align*}

 $(b)$ We only prove for $\rm{LH}$. Note that $\rm{LH}$ is the composition of the following bimodule maps:
\begin{align*}& e^-_r \cdot R G_{r+1} \cdot e^+_r \cong e^-_r \cdot R G_{r+1} \cdot e^+_r\otimes_{ R G_r }  R G_r \\
&\stackrel{1\otimes 1\otimes \eta^-_L}{\longrightarrow} e^-_r\cdot  R G_{r+1} \cdot e^+_r\otimes_{ R G_r }  R G_r \cdot e^-_{r-1}
\otimes_{ R G_{r-1} } e^-_{r-1}\cdot  R G_r \\&\stackrel{ 1\otimes H\otimes 1}{\longrightarrow}e^-_r \cdot R G_{r+1} \cdot e^-_r\otimes_{ R G_r }  R G_r \cdot e^+_{r-1}
\otimes_{ R G_{r-1} } e^-_{r-1}\cdot  R G_r \\
 &\stackrel{ \varepsilon^-_L\otimes 1\otimes 1}{\longrightarrow}  R G_r \otimes_{ R G_r } R G_r  \cdot e^+_{r-1} \otimes _{ R G_{r-1} }e^-_{r-1} \cdot  R G_r
 \cong  R G_r  \cdot e^+_{r-1} \otimes _{ R G_{r-1} } e^-_{r-1} \cdot  R G_r .
\end{align*}

 By Proposition \ref{mkd}, as an $(R G_{r}, R G_r)$-bimodule, the bimodule $  e^-_r\cdot   R G_{r+1}  \cdot e^+_r $ is generated by the unique  element $H'_{r+1}$, so we only need to compute the image of $H'_{r+1}$. Similarly to $(a)$, we have
$$H'_{r+1}\mapsto\sum\limits_{i=1}^s(\Proj^+_{P_r}(H'_{r+1}g^{-1}_iH_{r+1})
\otimes e^-_{r-1}g_i )$$
For any $h\in G_r,$ we have
\begin{align*}&H'_{r+1}hH_{r+1}\\
&=q e^-_r \dot{s}_{r+1} e^+_r h e^+_r
\dot{s}_{r+1} e^-_r\quad\text{(by Lemma \ref{Lem:T})} \\
&=q{(^{s_{r+1}}h)} e^-_r \dot{s}_{r+1} e^+_r  e^+_r \dot{s}_{r+1} e^-_r \\
&={^{{s}_{r+1}}h}H'_{r+1}H_{r+1}\\
&={^{s_{r+1}}h} e^+_r  e^-_{r-1} \quad\text{(by Prop. \ref{prop:12relations} (\ref{Mixaff2}))}.
\end{align*}
Hence
$$\Proj^+_{P_r}(H'_{r+1}g^{-1}_iH_{r+1})=\begin{cases}
 (^{s_{r+1}}g_i)^{-1} e^+_r  e^-_{r-1} \,\,&\text{if}\,\, (^{s_{r+1}}g_i)^{-1}\in P_r;\\
0\,\,&\text{if}\,\,  (^{s_{r+1}}g_i)^{-1}\notin P_r.
\end{cases}$$

Similarly to $(a)$, if $(^{s_{r+1}}g_i)^{-1}\in P_r$, we must have $g^{-1}_i\in P_{r-1}$ and we can write $$g^{-1}_i= v\times l\in P_{r-1}=U_{r-1}\rtimes G_{r-1}$$ where $l\in G_r$ and $v\in U_{r-1}.$
Since $v\in U_{r-1}$, by Lemma \ref{doublecoset} (d), we have ${^{s_{r+1}}v}\in U_{r}.$ It is easy to show that ${^{s_{r+1}}v} e^+_r = e^+_r $ and $v^{-1} e^+_{r-1} = e^+_{r-1} .$
Hence there is only one nonzero term and the nonzero term is exactly
\begin{align*}& \Proj^-_{P_r}((^{s_{r+1}}g_i)^{-1} e^-_{r}e^+_{r-1} ) \otimes e^-_{r-1} g_i \\&= \Proj^-_{P_r}({{^{s_{r+1}}v}l} e^-_r  e^+_{r-1} )   \otimes  e^-_{r-1}  l^{-1} v^{-1}\\&\xrightarrow{\thicksim} l e^+_{r-1}
\otimes  e^-_{r-1} l^{-1}= e^+_{r-1} \otimes  e^-_{r-1} .
\end{align*}
\end{proof}
\smallskip

\subsubsection{Mackey formula relations}

We define
$$
\RX^+:=\mathord{
\begin{tikzpicture}[baseline = -0.9mm]
	\draw[<-,thin,red] (0.4,0.2) to[out=-90, in=0] (0.1,-.2);
	\draw[-,thin,red] (0.1,-.2) to[out = 180, in = -90] (-0.2,0.2);
      \node at (0.38,0) {$\dott$};
\end{tikzpicture}
}\quad\text{
and}\quad
\RX^-:=\mathord{
\begin{tikzpicture}[baseline = -0.9mm]
	\draw[<-,thin,blue] (0.4,0.2) to[out=-90, in=0] (0.1,-.2);
	\draw[-,thin,blue] (0.1,-.2) to[out = 180, in = -90] (-0.2,0.2);
      \node at (0.38,0) {$\dott$};
\end{tikzpicture}.
}$$
By the definition, $\RX^\pm=(1_{E^\pm}\otimes X^\pm )\circ\eta^\pm_R .$
We also define
$$
\LX^+:=\mathord{
\begin{tikzpicture}[baseline = 1mm]
	\draw[-,thin,red] (0.4,0) to[out=90, in=0] (0.1,0.4);
	\draw[->,thin,red] (0.1,0.4) to[out = 180, in = 90] (-0.2,0);
      \node at (0.35,0.2) {$\dott$};
\end{tikzpicture}
}\quad\text{
and}
\quad
\LX^-:=\mathord{
\begin{tikzpicture}[baseline = 1mm]
	\draw[-,thin,blue] (0.4,0) to[out=90, in=0] (0.1,0.4);
	\draw[->,thin,blue] (0.1,0.4) to[out = 180, in = 90] (-0.2,0);
       \node at (0.35,0.2) {$\dott$};
\end{tikzpicture}
}.
$$
By the definition, $\LX^\pm=\varepsilon^\pm_L\circ(1_{E^\pm}\otimes X^\pm ).$

\begin{lemma}\label{lrdots}
\begin{itemize}
\item[$(a)$]
 The natural transformation $\LX^\pm$ is  given by the $( R G_r , R G_r )$-bimodule map
$$ e^\pm_r \cdot  R G_{r+1}  \cdot e^\pm_r  \to  R G_r ,$$
 $$ e^\pm_r \mapsto 0,$$
$$ X^\pm_{r+1} \,\mapsto\,-(t^\pm)^{2},$$
$$T^\pm_{r+1}\mapsto 0.$$

\item[$(b)$]
The natural transformation $\RX^\pm$ is  given by  $( R G_r , R G_r )$-bimodule map
$$ R G_r  \to  e^\pm_r \cdot  R G_{r+1} \cdot  e^\pm_r ,$$
$$1\mapsto  -({t^\pm z^\pm})^{-1} X^\pm_{r+1} .$$

\end{itemize}
\end{lemma}
\begin{proof}
We only prove $(a)$ for $\rm{LX}^-$ and $(b)$ for $\rm{RX}^-$.  The proof of others is similar.
\smallskip

$(a)$ Note that $\LX^-$ is the composition of the following bimodule maps:
$$ e^-_r \cdot  R G_{r+1} \cdot  e^-_r \stackrel{1\otimes X^- }{\longrightarrow} e^-_r \cdot R G_{r+1} \cdot e^-_r \stackrel{\varepsilon^-_L}{\longrightarrow}  R G_r .$$
The image of $ e^-_r $ is
 $$ e^-_r \mapsto X^-_{r+1}\mapsto \Proj^-_{P_r}( X^-_{r+1})=0.$$
 The image of $X^-_{r+1}$ is
$$X^-_{r+1}\mapsto (X^-_{r+1})^2\mapsto \Proj^-_{P_r}((X^-_{r+1})^2).$$
We have $(X^-_{r+1})^2
=\beta^2q^{2r} e^-_r \dot{x}_{r+1} e^-_r \dot{x}_{r+1} e^-_r $ and $ e^-_r =\sum\limits_{g\in U_r}\zeta(g)g$. Then, by the proof of \cite[Prop.10.5.3]{Car}, we have that $v_1\dot{x}_{r+1}v_2\dot{x}_{r+1}v_3\in P_r$ if and only $v_2=1$ where $v_1,v_2,v_3\in V_r.$ Hence,
\begin{align*}
(X^-_{r+1})^2
&=\beta^2q^{2r} e^-_r \dot{x}_{r+1} e^-_r \dot{x}_{r+1} e^-_r \\
&=\beta^2q^{2r} e^-_r \dot{x}_{r+1}e_{V_r}\dot{x}_{r+1} e^-_r \\
&=\beta^2q^{2r} e^-_r \dot{x}_{r+1}\frac{1}{|V_r|}(1+\sum\limits_{v\in V_{r}-\{1\}}{v})\dot{x}_{r+1} e^-_r .
\end{align*}
For any $v\in V_{r}-\{1\}$, we have $ e^-_r \dot{x}_{r+1}v\dot{x}_{r+1} e^-_r \notin RP_r$ and for $v=1$, we have $ e^-_r \dot{x}_{r+1}v\dot{x}_{r+1} e^-_r =
 e^-_r (\dot{x}_{r+1})^2 e^-_r .$
Note that $$(\dot{x}_{r+1})^2=\begin{cases}
 \id_{G_{r+1}}  &
    \text{if $G_r = \O_{2r+1}(q)$ or $\O^{\pm}_{2r}(q)$ }; \\
\diag(-1,\id_{G_r},-1) & \text{if $G_r=\Sp_{2r}(q)$},
\end{cases}\,$$
so
$$ e^-_r (\dot{x}_{r+1})^2 e^-_r =\begin{cases}
 e^-_r    &
    \text{if $G_r = \O_{2r+1}(q)$ or $\O^{\pm}_{2r}(q)$ }; \\
\zeta(-1) e^-_r  & \text{if $G_r=\Sp_{2r}(q)$}.
\end{cases}\,$$
Hence
$$X^-_{r+1}\mapsto \Proj^-_{P_r}((X^-_{r+1})^2)=\Proj^-_{P_r}
(\frac{\beta^2q^{2r}}{|V_r|} e^-_r (\dot{x}_{r+1})^2 e^-_r ) \xrightarrow{\thicksim} -(t^-)^{-2}.$$

The image of $T^-_{r+1}$ is
$$T^-_{r+1}\mapsto T^-_{r+1}X^-_{r+1}\mapsto \Proj^-_{P_r}( T^-_{r+1}X^-_{r+1}).$$
By a similar argument as in the proof of Lemma \ref{Lem:Rcrossing} (b), we have the equality $\Proj^-_{P_r}( T^-_{r+1}X^-_{r+1})=0,$ hence $T^-_{r+1}\mapsto 0.$
\smallskip

$(b)$ Note that $\RX^-$ is the composition of the following bimodule maps:
$$ R G_r  \stackrel{ \eta^-_R}{\longrightarrow} e^-_r \cdot R G_{r+1} \cdot e^-_r \stackrel{1\otimes X^-}{\longrightarrow}  e^-_r \cdot R G_{r+1} \cdot e^-_r ,$$
$$1\mapsto  -({t^- z^-})^{-1} e^-_r \mapsto -({t^- z^-})^{-1}X^-_{r+1}.$$
\end{proof}
Now we can state the main proposition in this section.
\begin{proposition}\label{prop:Mackey}
\begin{itemize}
\item[$(a)$]
The natural transformation\begin{align*}
\begin{array}{rl}
\left[\!\!\!\!\!
\begin{array}{l}\,\,\,\mathord{
\begin{tikzpicture}[baseline = 0]
	\draw[<-,thin,red] (-0.28,-.3) to (0.28,.4);
	\draw[-,line width=4pt,white] (0.28,-.3) to (-0.28,.4);
	\draw[->,thin,red] (0.28,-.3) to (-0.28,.4);
\end{tikzpicture}
}\\\,\,\,\mathord{
\begin{tikzpicture}[baseline = 1mm]
	\draw[-,thin,red] (0.4,0) to[out=90, in=0] (0.1,0.4);
	\draw[->,thin,red] (0.1,0.4) to[out = 180, in = 90] (-0.2,0);
\end{tikzpicture}
}
\\
\,\,\,\mathord{
\begin{tikzpicture}[baseline = 1mm]
	\draw[-,thin,red] (0.4,0) to[out=90, in=0] (0.1,0.4);
	\draw[->,thin,red] (0.1,0.4) to[out = 180, in = 90] (-0.2,0);
      \node at (0.35,0.2) {$\dott$};
\end{tikzpicture}
}
\end{array}
\right]
:
E^+F^+ \Rightarrow
F^+E^+ \oplus \unit^{\oplus 2}
\end{array}
\end{align*}

is an isomorphism of functors. It has a two-sided inverse
\begin{align*}
\label{invrel}
\begin{array}{rl}
\left[
\:\mathord{
\begin{tikzpicture}[baseline = 0]
	\draw[->,thin,red] (-0.28,-.3) to (0.28,.4);
	\draw[-,line width=4pt,white] (0.28,-.3) to (-0.28,.4);
	\draw[<-,thin,red] (0.28,-.3) to (-0.28,.4);
\end{tikzpicture}
}
\:\:\:\:-t^+z^+
\mathord{
\begin{tikzpicture}[baseline = -0.9mm]
	\draw[<-,thin,red] (0.4,0.2) to[out=-90, in=0] (0.1,-.2);
	\draw[-,thin,red] (0.1,-.2) to[out = 180, in = -90] (-0.2,0.2);
\end{tikzpicture}
}
\:\:\:\:\:\frac{z^+}{t^+}
\mathord{
\begin{tikzpicture}[baseline = -0.9mm]
	\draw[<-,thin,red] (0.4,0.2) to[out=-90, in=0] (0.1,-.2);
	\draw[-,thin,red] (0.1,-.2) to[out = 180, in = -90] (-0.2,0.2);
      \node at (0.38,0) {$\dott$};
\end{tikzpicture}
}
\right]
:F^+E^+\oplus
\unit^{\oplus 2}
\Rightarrow E^+F^+.
&
\end{array}
\end{align*}

\item[$(b)$]
The natural transformation\begin{align*}
\begin{array}{rl}
\left[\!\!\!\!\!
\begin{array}{l}\,\,\,\mathord{
\begin{tikzpicture}[baseline = 0]
	\draw[<-,thin,blue] (-0.28,-.3) to (0.28,.4);
	\draw[-,line width=4pt,white] (0.28,-.3) to (-0.28,.4);
	\draw[->,thin,blue] (0.28,-.3) to (-0.28,.4);
\end{tikzpicture}
}\\\,\,\,\mathord{
\begin{tikzpicture}[baseline = 1mm]
	\draw[-,thin,blue] (0.4,0) to[out=90, in=0] (0.1,0.4);
	\draw[->,thin,blue] (0.1,0.4) to[out = 180, in = 90] (-0.2,0);
\end{tikzpicture}
}
\\
\,\,\,\mathord{
\begin{tikzpicture}[baseline = 1mm]
	\draw[-,thin,blue] (0.4,0) to[out=90, in=0] (0.1,0.4);
	\draw[->,thin,blue] (0.1,0.4) to[out = 180, in = 90] (-0.2,0);
      \node at (0.35,0.2) {$\dott$};
\end{tikzpicture}
}
\end{array}
\right]
:
E^-F^- \Rightarrow
F^-E^- \oplus \unit^{\oplus 2}
\end{array}
\end{align*}

is an isomorphism of functors. It has a two-sided inverse
\begin{align*}
\begin{array}{rl}
\left[
\:\mathord{
\begin{tikzpicture}[baseline = 0]
	\draw[->,thin,blue] (-0.28,-.3) to (0.28,.4);
	\draw[-,line width=4pt,white] (0.28,-.3) to (-0.28,.4);
	\draw[<-,thin,blue] (0.28,-.3) to (-0.28,.4);
\end{tikzpicture}
}
\:\:\:\:-t^-z^-
\mathord{
\begin{tikzpicture}[baseline = -0.9mm]
	\draw[<-,thin,blue] (0.4,0.2) to[out=-90, in=0] (0.1,-.2);
	\draw[-,thin,blue] (0.1,-.2) to[out = 180, in = -90] (-0.2,0.2);
\end{tikzpicture}
}
\:\:\:\:\:\frac{z^-}{t^-}
\mathord{
\begin{tikzpicture}[baseline = -0.9mm]
	\draw[<-,thin,blue] (0.4,0.2) to[out=-90, in=0] (0.1,-.2);
	\draw[-,thin,blue] (0.1,-.2) to[out = 180, in = -90] (-0.2,0.2);
      \node at (0.38,0) {$\dott$};
\end{tikzpicture}
}
\right]
:F^-E^-\oplus
\unit ^{\oplus 2}
\Rightarrow
 E^-F^-.
\end{array}
\end{align*}

\item[$(c)$]
The natural transformations
\begin{align*}
\begin{array}{rl}
\mathord{
\begin{tikzpicture}[baseline = 0]
	\draw[->,thin,blue] (-0.28,-.3) to (0.28,.4);
	\draw[<-,thin,red] (0.28,-.3) to (-0.28,.4);
\end{tikzpicture}
}
:F^- E^+
\Rightarrow
 E^+  F^-
\end{array}
~\text{and}~
\begin{array}{rl}
\mathord{
\begin{tikzpicture}[baseline = 0]
	\draw[->,blue] (0.28,-.3) to (-0.28,.4);
	\draw[<-,red] (-0.28,-.3) to (0.28,.4);
\end{tikzpicture}
}
:E^+ F^-
\Rightarrow
 F^- E^+
\end{array}
\end{align*}
are two-sided inverse.

\item[$(d)$]
The natural transformations
\begin{align*}
\begin{array}{rl}
\mathord{
\begin{tikzpicture}[baseline = 0]
	\draw[->,thin,red] (-0.28,-.3) to (0.28,.4);
	\draw[<-,thin,blue] (0.28,-.3) to (-0.28,.4);
\end{tikzpicture}
}
:F^+ E^-
\Rightarrow
 E^- F^+
\end{array}
~\text{and}~
\begin{array}{rl}
\mathord{
\begin{tikzpicture}[baseline = 0]
	\draw[->,red] (0.28,-.3) to (-0.28,.4);
	\draw[<-,blue] (-0.28,-.3) to (0.28,.4);
\end{tikzpicture}
}
:E^- F^+
\Rightarrow
 F^+ E^-
\end{array}
\end{align*}
are two-sided inverse.
\end{itemize}
\end{proposition}
\begin{proof}
 The statements follow from Lemmas \ref{lcross}, \ref{Lem:Rcrossing}, \ref{lrdots} and a straightforward computation on generators.
\end{proof}
\smallskip

\subsubsection{Curl relations}

\begin{proposition}\label{prop:curl}
The curl relations (\ref{k=2curl}) hold.
\end{proposition}
\begin{proof}
We decompose them as $(\varepsilon_L^\pm\otimes 1_{F^\pm})\circ (1_{E^\pm}\otimes T^\pm)\circ(\eta_R^\pm\otimes 1_{F^\pm}).$
Then they are the compositions of the following bimodule maps:
\begin{align*}
& R G_{r+1} \cdot e^\pm_{r} \cong R G_{r+1} \otimes_{R G_{r+1}} R G_{r+1} \cdot e^\pm_{r}  \\&\stackrel{\eta^\pm_R\otimes 1}{\longrightarrow} e^\pm_{r+1}
\cdot R G_{r+2} \cdot e^\pm_{r+1} \otimes_{ R G_{r+1} } R G_{r+1}\cdot  e^\pm_r \\&\stackrel{ 1\otimes T^\pm}{\longrightarrow} e^\pm_{r+1}
\cdot R G_{r+2} \cdot e^\pm_{r+1} \otimes_{ R G_{r+1} } R G_{r+1}\cdot  e^\pm_r   \\&\stackrel{\varepsilon^\pm_L \otimes1}{\longrightarrow}  R G_{r+1} \otimes_{R G_{r+1}} R G_{r+1} \cdot e^\pm_{r}
\cong    R G_{r+1} \cdot e^\pm_{r} .
\end{align*}
We compute the image of $ e^\pm_{r}$ as follows:
\begin{align*}
&  e^\pm_{r}   \xrightarrow{\thicksim} 1\otimes    e^\pm_{r}\\
&\mapsto  -(t^\pm z^\pm)^{-1}e^\pm_{r+1} \otimes  e^\pm_r  \\
&\mapsto
  -(t^\pm z^\pm)^{-1}T^\pm_{r+2}\otimes e^\pm_{r} \\
&\mapsto -(t^\pm z^\pm)^{-1} \Proj^\pm_{P_{r+1}}({T^\pm_{r+2}})=0.
\end{align*}

\end{proof}
\smallskip

\subsubsection{Invertibility of $X^+$ and $X^-$}\label{sub:invetible}
In this section, we prove the invertibility  of natural transformations $X^+$ and $X^-$.

\begin{definition}Let $R=K$ or $k$ and  $M\in  \calC=RG_\bullet\mod $. We say that $M$ is $F^+$-cuspidal if $E^+(M)=0$ and $M$ is $F^-$-cuspidal if $E^-(M)=0.$
\end{definition}

\begin{proposition}\label{Prop:dim2}Let $R=K$ or $k$ be an algebraically closed field. Let $\epsilon\in \{\pm\}$ and $M$ be an irreducible $F^\epsilon$-cuspidal module.
\begin{itemize}
\item[$(1)$] We have $$\dim_{R}(\End_\mathcal{C}(F^\epsilon(M)))=2$$ and  the evaluation map $\phi_{F^\epsilon}\::\End(F^\epsilon)\to \End_\mathcal{C}(F^\epsilon(M))$ is surjective, where $\End(F^\epsilon)$ is the endomorphism ring of $F^\epsilon$ in the double quantum Heisenberg category.
\item[$(2)$] The image $X^\epsilon(M)$ of $X^\epsilon$ under $\phi_{F^\epsilon}$ acts on $F^\epsilon(M)$ satisfing
$$(X^\epsilon(M))^2=\gamma\cdot  X^\epsilon(M)-(t^\epsilon)^2\cdot \id_{F^\epsilon}(M)$$  for some $\gamma\in R$ and hence $X^\epsilon(M)$ is invertible.
\end{itemize}

\end{proposition}

\begin{proof}
$(a)$ By Mackey formula relations (see Proposition \ref{prop:Mackey}), we have $ E^\epsilon F^\epsilon\cong F^\epsilon E^\epsilon\oplus \unit^{\oplus2}.$ So
\begin{align*}
\End_\mathcal{C}(F^\epsilon(M))&=\Hom_\mathcal{C}(E^\epsilon  F^\epsilon(M),M)\quad\text{(since $F^\epsilon$ and $E^\epsilon$ are biadjoint)}\\
&=\Hom_\mathcal{C}(F^\epsilon E^\epsilon(M),M)\oplus \Hom_\mathcal{C}(M^{\oplus2},M)\\
&=\Hom_\mathcal{C}(M^{\oplus 2},M)\quad\text{(since $M$ is $F^\epsilon$-cuspidal)}.
\end{align*}
Since $M$ is irreducible, $\dim_{R}(\End_\mathcal{C}(F^\epsilon(M)))=2$ follows.
\smallskip

To show the surjectivity, it suffices to show that $\id_{F^\epsilon}(M)$ and $X^\epsilon(M)$ are linearly independent. Let $\up\in\{\red{\up},\blue{\up}\}.$ Assume that $$\:a\:\mathord{
\begin{tikzpicture}[baseline = -1mm]
 	\draw[->] (0.08,-.3) to (0.08,.4);
 	\draw[-,darkg,thick] (0.45,.4) to (0.45,-.3);
     \node at (0.45,-.45) {$\darkg\scriptstyle{M}$};
\end{tikzpicture}
}\:+\:b\: \mathord{
\begin{tikzpicture}[baseline = -1mm]
 	\draw[->] (0.08,-.3) to (0.08,.4);
    \node at (0.08,0.1) {$\dott$};
 	\draw[-,darkg,thick] (0.45,.4) to (0.45,-.3);
     \node at (0.45,-.45) {$\darkg\scriptstyle{M}$};
\end{tikzpicture}
}=0.$$
Composing with $\begin{tikzpicture}[baseline = .75mm]
	\draw[<-] (0.3,0.3) to[out=-90, in=0] (0.1,0);
	\draw[-] (0.1,0) to[out = 180, in = -90] (-0.1,0.3);
\end{tikzpicture}\:$ and
$\begin{tikzpicture}[baseline = 1mm]
	\draw[-] (0.3,0) to[out=90, in=0] (0.1,0.3);
	\draw[->] (0.1,0.3) to[out = 180, in = 90] (-0.1,0);
\end{tikzpicture}\:$, we get
$$
a\:\mathord{\begin{tikzpicture}[baseline = -1mm]
  \draw[-,thin] (0,0.2) to[out=180,in=90] (-.2,0);
  \draw[->,thin] (0.2,0) to[out=90,in=0] (0,.2);
 \draw[-,thin] (-.2,0) to[out=-90,in=180] (0,-0.2);
  \draw[-,thin] (0,-0.2) to[out=0,in=-90] (0.2,0);
  \draw[-,darkg,thick] (0.45,.4) to (0.45,-.3);
     \node at (0.45,-.45) {$\darkg\scriptstyle{M}$};
\end{tikzpicture}
}\:+\:b\: \mathord{\begin{tikzpicture}[baseline = -1mm]
  \draw[-,thin] (0,0.2) to[out=180,in=90] (-.2,0);
  \draw[->,thin] (0.2,0) to[out=90,in=0] (0,.2);
 \draw[-,thin] (-.2,0) to[out=-90,in=180] (0,-0.2);
  \draw[-,thin] (0,-0.2) to[out=0,in=-90] (0.2,0);
      \node at (0.2,0) {$\dott$};\draw[-,darkg,thick] (0.45,.4) to (0.45,-.3);
     \node at (0.45,-.45) {$\darkg\scriptstyle{M}$};
\end{tikzpicture}
}=0.
$$
Since $\mathord{\begin{tikzpicture}[baseline = -1mm]
  \draw[-,thin] (0,0.2) to[out=180,in=90] (-.2,0);
  \draw[->,thin] (0.2,0) to[out=90,in=0] (0,.2);
 \draw[-,thin] (-.2,0) to[out=-90,in=180] (0,-0.2);
  \draw[-,thin] (0,-0.2) to[out=0,in=-90] (0.2,0);
\end{tikzpicture}
}=-\frac{1}{t^\epsilon  z^\epsilon}\neq 0$ and $\mathord{\begin{tikzpicture}[baseline = -1mm]
  \draw[-,thin] (0,0.2) to[out=180,in=90] (-.2,0);
  \draw[->,thin] (0.2,0) to[out=90,in=0] (0,.2);
 \draw[-,thin] (-.2,0) to[out=-90,in=180] (0,-0.2);
  \draw[-,thin] (0,-0.2) to[out=0,in=-90] (0.2,0);
      \node at (0.2,0) {$\dott$};
\end{tikzpicture}
}=0$ (see (\ref{redbubble}) and (\ref{bluebubble})), we get $a=0.$
\smallskip

Composing with $\begin{tikzpicture}[baseline = .75mm]
	\draw[<-] (0.3,0.3) to[out=-90, in=0] (0.1,0);
	\draw[-] (0.1,0) to[out = 180, in = -90] (-0.1,0.3);
\end{tikzpicture}\:$ and
$\begin{tikzpicture}[baseline = 1mm]
	\draw[-] (0.3,0) to[out=90, in=0] (0.1,0.3);
	\draw[->] (0.1,0.3) to[out = 180, in = 90] (-0.1,0);
\node at (0.26,0.18) {$\dott$};
\end{tikzpicture}\:$, we get
$$
a\:\begin{tikzpicture}[baseline = -1mm]
  \draw[-,thin] (0,0.2) to[out=180,in=90] (-.2,0);
  \draw[->,thin] (0.2,0) to[out=90,in=0] (0,.2);
 \draw[-,thin] (-.2,0) to[out=-90,in=180] (0,-0.2);
  \draw[-,thin] (0,-0.2) to[out=0,in=-90] (0.2,0);
      \node at (0.2,0) {$\dott$};
 	\draw[-,darkg,thick] (0.45,.4) to (0.45,-.3);
     \node at (0.45,-.45) {$\darkg\scriptstyle{M}$};
\end{tikzpicture}
\:+\:b\: \begin{tikzpicture}[baseline = -1mm]
  \draw[-,thin] (0,0.2) to[out=180,in=90] (-.2,0);
  \draw[->,thin] (0.2,0) to[out=90,in=0] (0,.2);
 \draw[-,thin] (-.2,0) to[out=-90,in=180] (0,-0.2);
  \draw[-,thin] (0,-0.2) to[out=0,in=-90] (0.2,0);
      \node at (0.2,0) {$\dott$};
      \node at (0.4,0) {$\color{darkblue}\scriptstyle 2$};
 	\draw[-,darkg,thick] (0.55,.4) to (0.55,-.3);
     \node at (0.5,-.45) {$\darkg\scriptstyle{M}$};
\end{tikzpicture}
=0.$$
Since $\mathord{\begin{tikzpicture}[baseline = -1mm]
  \draw[-,thin] (0,0.2) to[out=180,in=90] (-.2,0);
  \draw[->,thin] (0.2,0) to[out=90,in=0] (0,.2);
 \draw[-,thin] (-.2,0) to[out=-90,in=180] (0,-0.2);
  \draw[-,thin] (0,-0.2) to[out=0,in=-90] (0.2,0);
  \node at (0.4,0) {$\color{darkblue}\scriptstyle 2$};
      \node at (0.2,0) {$\dott$};
\end{tikzpicture}
}=\frac{ t^\epsilon }{ z^\epsilon}\neq 0$ and $\mathord{\begin{tikzpicture}[baseline = -1mm]
  \draw[-,thin] (0,0.2) to[out=180,in=90] (-.2,0);
  \draw[->,thin] (0.2,0) to[out=90,in=0] (0,.2);
 \draw[-,thin] (-.2,0) to[out=-90,in=180] (0,-0.2);
  \draw[-,thin] (0,-0.2) to[out=0,in=-90] (0.2,0);
      \node at (0.2,0) {$\dott$};
\end{tikzpicture}
}=0$ (see also (\ref{redbubble}) and (\ref{bluebubble})), we get $b=0.$
Hence $\id_{F^\epsilon}(M)$ and $X^\epsilon(M)$ are linearly independent.
\smallskip

$(b)$  By $(a)$, we know that $\End_{\mathcal{C}}(F^\epsilon(M))$ is of dimension 2 and has a basis $\id_{F^\epsilon}(M)$ and $X^\epsilon(M)$. Since $(X^\epsilon)^2(M)\in \End_{\mathcal{C}}(F^\epsilon(M)),$ there are constants $c,d\in R$ such that $(X^\epsilon)^2(M)=c\cdot \id_{F^\epsilon}(M) +d\cdot X^\epsilon(M).$ Diagrammatically, let $\up\in\{\red{\up},\blue{\up}\}$, then
$$\mathord{
\begin{tikzpicture}[baseline = -1mm]
 	\draw[->] (0.08,-.3) to (0.08,.4);
     \node at (0.08,0.0) {$\dott$};
    \node at (0.08,0.2) {$\dott$};
 	\draw[-,darkg,thick] (0.45,.4) to (0.45,-.3);
     \node at (0.45,-.45) {$\darkg\scriptstyle{M}$};
\end{tikzpicture}
}= \:c\:\mathord{
\begin{tikzpicture}[baseline = -1mm]
 	\draw[->] (0.08,-.3) to (0.08,.4);
 	\draw[-,darkg,thick] (0.45,.4) to (0.45,-.3);
     \node at (0.45,-.45) {$\darkg\scriptstyle{M}$};
\end{tikzpicture}
}\:+\:d\: \mathord{
\begin{tikzpicture}[baseline = -1mm]
 	\draw[->] (0.08,-.3) to (0.08,.4);
    \node at (0.08,0.1) {$\dott$};
 	\draw[-,darkg,thick] (0.45,.4) to (0.45,-.3);
     \node at (0.45,-.45) {$\darkg\scriptstyle{M}$};
\end{tikzpicture}
}.
$$

Composing  with $\begin{tikzpicture}[baseline = .75mm]
	\draw[<-] (0.3,0.3) to[out=-90, in=0] (0.1,0);
	\draw[-] (0.1,0) to[out = 180, in = -90] (-0.1,0.3);
\end{tikzpicture}\:$ and
$\begin{tikzpicture}[baseline = 1mm]
	\draw[-] (0.3,0) to[out=90, in=0] (0.1,0.3);
	\draw[->] (0.1,0.3) to[out = 180, in = 90] (-0.1,0);
\end{tikzpicture}\:$, we get
$$
\mathord{
\begin{tikzpicture}[baseline = -1mm]
  \draw[-,thin] (0,0.2) to[out=180,in=90] (-.2,0);
  \draw[->,thin] (0.2,0) to[out=90,in=0] (0,.2);
 \draw[-,thin] (-.2,0) to[out=-90,in=180] (0,-0.2);
  \draw[-,thin] (0,-0.2) to[out=0,in=-90] (0.2,0);
      \node at (0.2,0) {$\dott$};
      \node at (0.4,0) {$\color{darkblue}\scriptstyle 2$};
 	\draw[-,darkg,thick] (0.55,.4) to (0.55,-.3);
     \node at (0.45,-.45) {$\darkg\scriptstyle{M}$};
\end{tikzpicture}
}= \:c\:\mathord{
\begin{tikzpicture}[baseline = -1mm]
  \draw[-,thin] (0,0.2) to[out=180,in=90] (-.2,0);
  \draw[->,thin] (0.2,0) to[out=90,in=0] (0,.2);
 \draw[-,thin] (-.2,0) to[out=-90,in=180] (0,-0.2);
  \draw[-,thin] (0,-0.2) to[out=0,in=-90] (0.2,0);
 	\draw[-,darkg,thick] (0.45,.4) to (0.45,-.3);
     \node at (0.45,-.45) {$\darkg\scriptstyle{M}$};
\end{tikzpicture}
}\:+\:d\: \mathord{
\begin{tikzpicture}[baseline = -1mm]
  \draw[-,thin] (0,0.2) to[out=180,in=90] (-.2,0);
  \draw[->,thin] (0.2,0) to[out=90,in=0] (0,.2);
 \draw[-,thin] (-.2,0) to[out=-90,in=180] (0,-0.2);
  \draw[-,thin] (0,-0.2) to[out=0,in=-90] (0.2,0);
      \node at (0.2,0) {$\dott$};
 	\draw[-,darkg,thick] (0.45,.4) to (0.45,-.3);
     \node at (0.45,-.45) {$\darkg\scriptstyle{M}$};
\end{tikzpicture}
}.
$$
Since $ \mathord{
\begin{tikzpicture}[baseline = -1mm]
  \draw[-,thin] (0,0.2) to[out=180,in=90] (-.2,0);
  \draw[->,thin] (0.2,0) to[out=90,in=0] (0,.2);
 \draw[-,thin] (-.2,0) to[out=-90,in=180] (0,-0.2);
  \draw[-,thin] (0,-0.2) to[out=0,in=-90] (0.2,0);
      \node at (0.2,0) {$\dott$};
\end{tikzpicture}
}=0$, we have
 $c=\mathord{
\begin{tikzpicture}[baseline = -1mm]
  \draw[-,thin] (0,0.2) to[out=180,in=90] (-.2,0);
  \draw[->,thin] (0.2,0) to[out=90,in=0] (0,.2);
 \draw[-,thin] (-.2,0) to[out=-90,in=180] (0,-0.2);
  \draw[-,thin] (0,-0.2) to[out=0,in=-90] (0.2,0);
      \node at (0.2,0) {$\dott$};
      \node at (0.4,0) {$\color{darkblue}\scriptstyle 2$};
\end{tikzpicture}
}/\mathord{
\begin{tikzpicture}[baseline = -1mm]
  \draw[-,thin] (0,0.2) to[out=180,in=90] (-.2,0);
  \draw[->,thin] (0.2,0) to[out=90,in=0] (0,.2);
 \draw[-,thin] (-.2,0) to[out=-90,in=180] (0,-0.2);
  \draw[-,thin] (0,-0.2) to[out=0,in=-90] (0.2,0);
\end{tikzpicture}
}=-(t^\epsilon)^2\neq 0$, i.e., $(X^\epsilon)^2(M)=\gamma\cdot X^\epsilon(M)-(t^\epsilon)^2\cdot  \id_{F^\epsilon}(M)$. In particular, $X^\epsilon(M)$ is invertible.
\end{proof}

\begin{proposition}\label{prop:invertible}
$X^+\in \End(F^+)$ and $X^-\in \End(F^-)$ are invertible.
\end{proposition}
\begin{proof}

Step 1: Reduce the problem to fields. When $R=\mathcal{O}$, the functor $F^\epsilon_{r, r+1}$ is represented by the $(\mathcal{O}G_{r+1},\mathcal{O}G_r)$-bimodule $\mathcal{O}G_{r+1}\cdot e^\epsilon_r $ which is a free $\mathcal{O}$-module. The action of $X^\epsilon$ on $F^\epsilon$ is given by right multiplication by $X^\epsilon_{r+1}$ on $ R G_{r+1} \cdot e^-_r $. Hence, to show that $ X^\epsilon$ is invertible, it suffices to show that $X^\epsilon_k\cong X^\epsilon\otimes_{\mathcal{O}}k$ is invertible in $\End(F^\epsilon\otimes_{\mathcal{O}}k)$, where $k$ is the residue field of $\calO.$
So now we assume that $R=K$ or $k.$ Moreover, we can assume that $R=K$ or $k$ is an algebraically closed field.
\smallskip

Step 2: Reduce this problem to $F^\epsilon$-cuspidal modules. To show that $X^\epsilon$ is invertible, it suffices to show that $X^\epsilon(M)$ is invertible in $\End_{\mathcal{C}}(F^\epsilon(M))$ for any module $M\in  RG_\bullet\mod .$ By the exactness of functor $F^\epsilon$, we only need to show that $X^\epsilon(M)$ is invertible in $\End_{\mathcal{C}}(F^\epsilon(M))$ for any irreducible module $M.$ From now on, we assume that $M$ is irreducible. If $E^\epsilon(M)\neq 0,$ than there exists a module $N\in  RG_\bullet\mod $ such that $M$ is a composition factor of $F^\epsilon(N).$
By induction, we can assume that $X^\epsilon(N)$ is invertible in $\End_{\mathcal{C}}(F^\epsilon(N))$. Since $X^\epsilon$ and $T^\epsilon$ satisfy affine Hecke relation, by \cite[Lemma 4.7]{Go}, we deduce that the the eigenvalues of $X^\epsilon(F^\epsilon(N))$ on $(F^\epsilon)^2(N)$  are $q^{\pm 1}$ times the eigenvalues of $X^\epsilon(N)$ on $F^\epsilon(N)$, hence $X^\epsilon(F^\epsilon(N))$ is  invertible, and so is $X^\epsilon(M).$
\smallskip

Step 3: Now assume that $M$ is an $F^\epsilon$-cuspidal irreducible module, then we need to show the invertibility of $X^\epsilon(M),$ which follows from Proposition \ref{Prop:dim2} $(b)$.
\end{proof}

\smallskip

\subsubsection{Proof of Theorem \ref{Thm:doubleheisenberg}}\label{sub:proof}

\begin{proof} The theorem follows by previous propositions. In fact, the affine Hecke relations (\ref{AHH1}) and (\ref{AHH2}) (for both red and blue colors) are checked in Proposition \ref{prop:12relations}; the biadjoint relations (\ref{adj})  are checked in Proposition \ref{prop:adj}; the bubble relations  and Mackey formula relations $(\ref{redbubble})$--$(\ref{MMackey2})$ are the corollaries of Proposition \ref{prop:Mackey} $(a)$ and $(b)$; the curl relations (\ref{k=2curl}) are checked in Proposition \ref{prop:curl}; the mixed affine-Hecke-type relations (\ref{Mixaff1})--(\ref{Mixaff3}) are checked in Proposition \ref{prop:12relations}; the mixed Mackey formula relations (\ref{Mixamackey1}) and (\ref{Mixamackey2}) are checked in Proposition \ref{prop:Mackey} $(c)$ and $(d);$ the invertibility of $X^+$ and $X^-$ is showed in Proposition \ref{prop:invertible}.
\end{proof}

%\begin{remark}
%In fact, by using Rouquier's original approach in \cite{R08} via ``control by
%$K_0$'',  Dudas, Varagnolo and Vasserot \cite{DVV2}  proved that there exists a  $\frakU(\fraks\frakl'_{I_+})$-categorification on
% the unipotent modules of finite classical groups $\SO_{2n+1}(q)$ and $\Sp_{2n}(q)$ and  Liu and two of the authors \cite{LLZ}  proved that there exists a
% $\frakU(\fraks\frakl'_{I_+}\oplus\fraks\frakl'_{I_-})$-categorification on
% the quadratic unipotent modules of $\SO_{2n+1}(q)$.
% \smallskip
%
% Note that our method in this article gives a $\frakU(\fraks\frakl'_{I_+}\oplus\fraks\frakl'_{I_-})$-categorification on the whole category $RG_\bullet\mod$ for finite classical groups $\O_{2n+1}(q),$ $\Sp_{2n}(q)$ and $\O_{2n}^\pm(q)$  which generalizes previous results of \cite{DVV2} and \cite{LLZ}.
%\end{remark}
\medskip

\section{Extra symmetries of categorification }\label{Chap:extrasymmeties}
 In this section, we investigate functors arising from the spinor  character, the
sign character and diagonal automorphism and their properties.
\smallskip

\subsection{Spinor functor for $\O_{2n+1}(q)$ and $\O^{\pm}_{2n}(q)$}\label{sub:spin}\hfill\\

In this section we let $G_n=\O_{2n+1}(q)$ or $\O^{\pm}_{2n}(q)$ with $q$ odd.
Let $\sp$ be the \emph{spinor character} of $G_n$
 (see \cite[\S 1.6]{Wa}). Let $R_{\sp}$ be the 1-dimensional $R G _n$-module affording $\sp$.
When there is a need to avoid ambiguity,
we use the notation $\sp_n$ with subscript $n$ for the character $\sp$ of $G_n$.
We define the functor $$\mathrm{Spin}_n:R G _n\mod\to R G _n\mod$$
by $$M\mapsto R_{\sp_n}\otimes_R M.$$
 We define a $(RG_n, RG_n)$-bimodule  $RG_n \cdot \mathfrak{sp}_n$ as follows: 
\begin{enumerate}
	\item   As a left $RG_n$-module, $RG_n \cdot \mathfrak{sp}_n$ is isomorphic to the group ring $RG_n$ itself, where $\mathfrak{sp}_n$ serves as a formal generator;
	
	\item The right $RG_n$-module structure is defined through the twisted action
	\[
	(x \cdot \mathfrak{sp}_n)\cdot h \coloneqq \sp_n(h)xh \cdot \mathfrak{sp}_n
	\]
	for all $ h \in G_n$ and $x  \cdot \mathfrak{sp}_n\in RG_n\cdot \mathfrak{sp}_n$.
\end{enumerate}
Then the functor $\mathrm{Spin}_n$
is represented by  the $(R G _n,R G _n)$-bimodule
 $R G _n\cdot\mathfrak{sp}_n.$ 
Let $$\mathrm{Spin}:=\mathord{
\begin{tikzpicture}[baseline = 0]
	\draw[-,gray, thick, dashed] (0.08,-.2) to (0.08,.4);
\end{tikzpicture}
}:=\bigoplus\limits_{n\in \bbN^*}\mathrm{Spin}_n:\bigoplus\limits_{n\in \bbN^*}R G_{n}
\mod\to \bigoplus\limits_{n\in \bbN^*}R G_{n}\mod.$$

\smallskip

 We define $\dot{\Phi}=\mathord{
\begin{tikzpicture}[baseline = -.6mm]
	\draw[-,blue] (-0.28,-.3) to (0,0);
    \draw[->,red] (0,0) to (0.28,.3);
	\draw[-,gray, thick, dashed] (0.28,-.3) to (-0.28,.3);
\end{tikzpicture}
}\::F^-\circ \mathrm{Spin} \to
\mathrm{Spin}\circ F^+\,$ and $\dot{\Phi}'=\mathord{
\begin{tikzpicture}[baseline = -.6mm]
	\draw[-,red] (-0.28,-.3) to (0,0);
\draw[->,blue] (0,0) to (0.28,.3);
	\draw[-,gray, thick, dashed] (0.28,-.3) to (-0.28,.3);
\end{tikzpicture}
}\::F^+\circ \mathrm{Spin}\to
\mathrm{Spin}\circ F^-\,$ by the following $(R G_{n+1}, RG_n)$-bimodule maps:

$$\begin{array}{rcl}
\dot{\Phi}\::R G_{n+1}\cdot e^-_{n}\otimes_{R G_n}R G_n\cdot\mathfrak{sp}_n
&\to& R G_{n+1}\cdot\mathfrak{sp}_{n+1}\otimes_{R G_{n+1}}R G_{n+1}\cdot
e^+_{n}\\
ge^-_{n}\otimes h\cdot\mathfrak{sp}_n &\mapsto & g\cdot
\mathfrak{sp}_{n+1}\otimes \sp_n(h)h e^+_{n}
\end{array}
$$
and
$$\begin{array}{rcl}
\dot{\Phi}'\:: R G_{n+1}\cdot e^+_{n}\otimes_{R G_n}R G_n\cdot\mathfrak{sp}_n &\to& R G_{n+1}
\cdot\mathfrak{sp}_{n+1}\otimes_{R G_{n+1}}R G_{n+1}\cdot
e^-_{n}\\
ge^+_{n}\otimes h\cdot\mathfrak{sp}_n&\mapsto& g\cdot
\mathfrak{sp}_{n+1}\otimes \sp_n(h)h e^-_{n},
\end{array}
$$ respectively,
for all $g\in  G_{n+1}$ and $h\in  G_n$.
\smallskip

Similarly, we define $\dot{\Psi}=\mathord{
\begin{tikzpicture}[baseline = -.6mm]
	\draw[<-,red] (0.28,-.3) to (0,0);
    \draw[-,blue] (0,0) to (-0.28,.3);
	\draw[-,gray, thick, dashed] (-0.28,-.3) to (0.28,.3);
\end{tikzpicture}
}\::\mathrm{Spin}\circ E^+\to E^-\circ \mathrm{Spin}
\,$ and $\dot{\Psi}'=\mathord{
\begin{tikzpicture}[baseline = -.6mm]
	\draw[<-,blue] (0.28,-.3) to (0,0);
\draw[-,red] (0,0) to (-0.28,.3);
	\draw[-,gray, thick, dashed] (-0.28,-.3) to (0.28,.3);
\end{tikzpicture}
}\::\mathrm{Spin}\circ E^-\to E^+\circ \mathrm{Spin}
\,$
 via the $(RG_{n-1}, RG_n)$-bimodule maps
$$\begin{array}{rcl}
 \dot{\Psi}\::R G_{n-1}\cdot\mathfrak{sp}_{n-1}\otimes_{R G_{n-1}}e^+_{n-1}\cdot R G_{n}
&\to&e^-_{n-1}\cdot R G_{n}\otimes_{R G_n}R G_n\cdot\mathfrak{sp}_n
\\
g\cdot
\mathfrak{sp}_{n-1}\otimes e^+_{n-1}h&\mapsto &e^-_{n-1}\cdot g\otimes \sp_n(h) h\cdot\mathfrak{sp}_n
\end{array}
$$
and
$$\begin{array}{rcl}
 \dot{\Psi}'\::R G_{n-1}\cdot\mathfrak{sp}_{n-1}\otimes_{R G_{n-1}}e^-_{n-1}\cdot R G_{n}
&\to&e^+_{n-1}\cdot R G_{n}\otimes_{R G_n}R G_n\cdot\mathfrak{sp}_n
\\
g\cdot
\mathfrak{sp}_{n-1}\otimes e^-_{n-1} h&\mapsto &e^+_{n-1}g\otimes \sp_n(h) h\cdot\mathfrak{sp}_n
\end{array}
$$ respectively,
for all $g\in  G_{n-1}$ and $h\in  G_n$.
\smallskip

It is easy to see that $\mathrm{Spin}^2\cong \unit.$
By  \cite[Lemma 4.7]{LLZ},  $\dot{\Phi}$ (resp. $\dot{\Phi}',\dot{\Psi} $ or $\dot{\Psi}'$) is an isomorphism.
We define its inverse  $\dot{\Phi}^{-1}=\mathord{
\begin{tikzpicture}[baseline = -.6mm]
	\draw[-,red] (0.28,-.3) to (0,0);
    \draw[->,blue] (0,0) to (-0.28,.3);
	\draw[-,gray, thick, dashed] (-0.28,-.3) to (0.28,.3);
\end{tikzpicture}
}$ (resp.  $\dot{\Phi}'^{-1}=\mathord{
\begin{tikzpicture}[baseline = -.6mm]
	\draw[-,blue] (0.28,-.3) to (0,0);
\draw[->,red] (0,0) to (-0.28,.3);
	\draw[-,gray, thick, dashed] (-0.28,-.3) to (0.28,.3);
\end{tikzpicture}
}\,,$ $\dot{\Psi}^{-1}=\mathord{
\begin{tikzpicture}[baseline = -.6mm]
	\draw[<-,blue] (-0.28,-.3) to (0,0);
    \draw[-,red] (0,0) to (0.28,.3);
	\draw[-,gray, thick, dashed] (0.28,-.3) to (-0.28,.3);
\end{tikzpicture}
}\,$ or $\dot{\Psi}'^{-1}=\mathord{
\begin{tikzpicture}[baseline = -.6mm]
	\draw[<-,red] (-0.28,-.3) to (0,0);
\draw[-,blue] (0,0) to (0.28,.3);
	\draw[-,gray, thick, dashed] (0.28,-.3) to (-0.28,.3);
\end{tikzpicture}
}$).

\begin{proposition}\label{prop:spinor}
We have the following commutative diagrams:
\begin{align}
\mathord{
\begin{tikzpicture}[baseline = 0mm]
	\draw[-,thin,red] (0.3,-.6)[out=90,in=-45] to (-.03,-.29);
\draw[-,thin,blue] (-.03,-.29)[out=-45,in=-90] to (-.3,0);
\draw[-,thin,blue] (-.3,0)[out=90,in=45] to (-.03,.31);
\draw[->,thin,red] (-.03,.31) [out=45,in=-90]to (.3,.6);
	\draw[-,gray, thick, dashed] (-0.3,-.6) to[out=90,in=-90] (0.3,0);
	\draw[-,gray, thick, dashed] (0.3,0) to[out=90,in=-90] (-0.3,.6);
\end{tikzpicture}
}\,=
\mathord{
\begin{tikzpicture}[baseline = -1mm]
	\draw[->,thin,red] (0.18,-.6) to (0.18,.6);
	\draw[-,gray, thick, dashed] (-0.18,-.6) to (-0.18,.6);
\end{tikzpicture}
}\:, \quad
\mathord{
\begin{tikzpicture}[baseline = 0mm]
	\draw[-,thin,blue] (0.3,-.6)[out=90,in=-45] to (-.03,-.29);
\draw[-,thin,red] (-.03,-.29)[out=-45,in=-90] to (-.3,0);
\draw[-,thin,red] (-.3,0)[out=90,in=45] to (-.03,.31);
\draw[->,thin,blue] (-.03,.31) [out=45,in=-90]to (.3,.6);
	\draw[-,gray, thick, dashed] (-0.3,-.6) to[out=90,in=-90] (0.3,0);
	\draw[-,gray, thick, dashed] (0.3,0) to[out=90,in=-90] (-0.3,.6);
\end{tikzpicture}
}\,=
\mathord{
\begin{tikzpicture}[baseline = -1mm]
	\draw[->,thin,blue] (0.18,-.6) to (0.18,.6);
	\draw[-,gray, thick, dashed] (-0.18,-.6) to (-0.18,.6);
\end{tikzpicture}
}\:, \quad
\mathord{
\begin{tikzpicture}[baseline = 0mm]
	\draw[<-,thin,red] (0.3,-.6)[out=90,in=-45] to (-.03,-.29);
\draw[-,thin,blue] (-.03,-.29)[out=-45,in=-90] to (-.3,0);
\draw[-,thin,blue] (-.3,0)[out=90,in=45] to (-.03,.31);
\draw[-,thin,red] (-.03,.31) [out=45,in=-90]to (.3,.6);
	\draw[-,gray, thick, dashed] (-0.3,-.6) to[out=90,in=-90] (0.3,0);
	\draw[-,gray, thick, dashed] (0.3,0) to[out=90,in=-90] (-0.3,.6);
\end{tikzpicture}
}\,=
\mathord{
\begin{tikzpicture}[baseline = -1mm]
	\draw[<-,thin,red] (0.18,-.6) to (0.18,.6);
	\draw[-,gray, thick, dashed] (-0.18,-.6) to (-0.18,.6);
\end{tikzpicture}
}\:,\quad
\mathord{
\begin{tikzpicture}[baseline = 0mm]
	\draw[<-,thin,blue] (0.3,-.6)[out=90,in=-45] to (-.03,-0.29);
\draw[-,thin,red] (-.03,-.29)[out=-45,in=-90] to (-.3,0);
\draw[-,thin,red] (-.3,0)[out=90,in=45] to (-.03,0.31);
\draw[-,thin,blue] (-.03,.31) [out=45,in=-90]to (.3,.6);
	\draw[-,gray, thick, dashed] (-0.3,-.6) to[out=90,in=-90] (0.3,0);
	\draw[-,gray, thick, dashed] (0.3,0) to[out=90,in=-90] (-0.3,.6);
\end{tikzpicture}
}\,=
\mathord{
\begin{tikzpicture}[baseline = -1mm]
	\draw[<-,thin,blue] (0.18,-.6) to (0.18,.6);
	\draw[-,gray, thick, dashed] (-0.18,-.6) to (-0.18,.6);
\end{tikzpicture}
}\:,
\end{align}

\begin{align}
\mathord{
\begin{tikzpicture}[baseline = 0mm]
	\draw[-,thin,red] (0.28,-.3) to (0,0);
\draw[->,thin,blue] (0,0) to (-0.28,.3);
      \node at (0.165,-0.18) {$\dott$};
	\draw[gray, thick, dashed,-] (-0.28,-.3) to (0.28,.3);
\end{tikzpicture}
}&=\,\mathord{
\begin{tikzpicture}[baseline = 0mm]
	\draw[gray, thick, dashed,-] (-0.28,-.3) to (0.28,.3);
	\draw[-,thin,red] (0.28,-.3) to (0,0);
\draw[->,thin,blue] (0,0) to (-0.28,.3);
      \node at (-0.14,0.13) {$\dott$};
\end{tikzpicture}
}
\:,&
\mathord{
\begin{tikzpicture}[baseline = 0mm]
	\draw[-,thin,blue] (0.28,-.3) to (0,0);
\draw[->,thin,red] (0,0) to (-0.28,.3);
      \node at (0.165,-0.18) {$\dott$};
	\draw[gray, thick, dashed,-] (-0.28,-.3) to (0.28,.3);
\end{tikzpicture}
}&=\,\mathord{
\begin{tikzpicture}[baseline = 0mm]
	\draw[gray, thick, dashed,-] (-0.28,-.3) to (0.28,.3);
	\draw[-,thin,blue] (0.28,-.3) to (0,0);
\draw[->,thin,red] (0,0) to (-0.28,.3);
      \node at (-0.14,0.13) {$\dott$};
\end{tikzpicture}
}
\:,
\end{align}

\begin{align}
\mathord{
\begin{tikzpicture}[baseline = -1mm]
	\draw[-,thin,red] (0.45,-.6) to (0,0);
\draw[->,thin,blue] (0,0) to (-0.45,.6);
        \draw[-,thin,red] (0,-.6) to[out=90,in=-45] (-.25,-.3);
        \draw[-,thin,blue] (-.25,-.3) to[out=-45,in=-90] (-.45,0);
        \draw[-,line width=4pt,white] (-0.45,0) to[out=90,in=-90] (0,0.6);
        \draw[->,thin,blue] (-0.45,0) to[out=90,in=-90] (0,0.6);
	\draw[-,thick,gray,dashed] (0.45,.6) to (-0.45,-.6);
\end{tikzpicture}
}
\,=
\mathord{
\begin{tikzpicture}[baseline = -1mm]
	\draw[-,thin,red] (0.45,-.6) to (0,0);
\draw[->,thin,blue] (0,0) to (-0.45,.6);
        \draw[-,line width=4pt,white] (0,-.6) to[out=90,in=-90] (.45,0);
        \draw[-,thin,red] (0,-.6) to[out=90,in=-90] (.45,0);
        \draw[-,thin,red] (0.45,0) to[out=90,in=-45] (.21,0.3);
        \draw[->,thin,blue] (0.21,.3) to[out=-45,in=-90] (0,0.6);
	\draw[-,thick,gray,dashed] (0.45,.6) to (-0.45,-.6);
\end{tikzpicture}
}\:,\quad
\mathord{
\begin{tikzpicture}[baseline = -1mm]
	\draw[-,thin,blue] (0.45,-.6) to (0,0);
\draw[->,thin,red] (0,0) to (-0.45,.6);
        \draw[-,thin,blue] (0,-.6) to[out=90,in=-45] (-.25,-.3);
        \draw[-,thin,red] (-.25,-.3) to[out=-45,in=-90] (-.45,0);
        \draw[-,line width=4pt,white] (-0.45,0) to[out=90,in=-90] (0,0.6);
        \draw[->,thin,red] (-0.45,0) to[out=90,in=-90] (0,0.6);
	\draw[-,thick,gray,dashed] (0.45,.6) to (-0.45,-.6);
\end{tikzpicture}
}
\,=
\mathord{
\begin{tikzpicture}[baseline = -1mm]
	\draw[-,thin,blue] (0.45,-.6) to (0,0);
\draw[->,thin,red] (0,0) to (-0.45,.6);
        \draw[-,line width=4pt,white] (0,-.6) to[out=90,in=-90] (.45,0);
        \draw[-,thin,blue] (0,-.6) to[out=90,in=-90] (.45,0);
        \draw[-,thin,blue] (0.45,0) to[out=90,in=-45] (.21,0.3);
        \draw[->,thin,red] (0.21,.3) to[out=-45,in=-90] (0,0.6);
	\draw[-,thick,gray,dashed] (0.45,.6) to (-0.45,-.6);
\end{tikzpicture}
}\:,\quad
\mathord{
\begin{tikzpicture}[baseline = -1mm]
	\draw[-,thick,gray,dashed] (0.45,.6) to (-0.45,-.6);
	\draw[-,blue] (0.45,-.6) to (0,0);
\draw[->,red] (0,0) to (-0.45,.6);
        \draw[-,red] (0,-.6) to[out=90,in=-45] (-.25,-.3);
        \draw[-,blue] (-.25,-.3) to[out=-45,in=-90] (-.45,0);
        \draw[->,blue] (-0.45,0) to[out=90,in=-90] (0,0.6);
\end{tikzpicture}
}
\,=
\mathord{
\begin{tikzpicture}[baseline = -1mm]
	\draw[-,thick,gray,dashed] (0.45,.6) to (-0.45,-.6);
	\draw[-,blue] (0.45,-.6) to (0,0);
\draw[->,red] (0,0) to (-0.45,.6);
        \draw[-,red] (0,-.6) to[out=90,in=-90] (.45,0);
        \draw[-,red] (0.45,0) to[out=90,in=-45] (.21,0.3);
        \draw[->,blue] (0.21,.3) to[out=-45,in=-90] (0,0.6);
\end{tikzpicture}
}\:,\quad \mathord{
\begin{tikzpicture}[baseline = -1mm]
	\draw[-,thick,gray,dashed] (0.45,.6) to (-0.45,-.6);
	\draw[-,red] (0.45,-.6) to (0,0);
\draw[->,blue] (0,0) to (-0.45,.6);
        \draw[-,blue] (0,-.6) to[out=90,in=-45] (-.25,-.3);
        \draw[-,red] (-.25,-.3) to[out=-45,in=-90] (-.45,0);
        \draw[->,red] (-0.45,0) to[out=90,in=-90] (0,0.6);
\end{tikzpicture}
}
\,=
\mathord{
\begin{tikzpicture}[baseline = -1mm]
	\draw[-,thick,gray,dashed] (0.45,.6) to (-0.45,-.6);
	\draw[-,red] (0.45,-.6) to (0,0);
\draw[->,blue] (0,0) to (-0.45,.6);
        \draw[-,blue] (0,-.6) to[out=90,in=-90] (.45,0);
        \draw[-,blue] (0.45,0) to[out=90,in=-45] (.21,0.3);
        \draw[->,red] (0.21,.3) to[out=-45,in=-90] (0,0.6);
\end{tikzpicture}
}\:,
\end{align}

\begin{align}
\mathord{
\begin{tikzpicture}[baseline = 0]
\draw[-,gray,dashed, thick](.6,.4) to (.1,-.3);
	\draw[<-,red] (0.6,-0.3) to[out=140, in=-60] (.33,0.03);
     \draw[-,blue] (.33,0.03) to[out=-60, in=0] (0.1,0.2);
	\draw[-,blue] (0.1,0.2) to[out = -180, in = 90] (-0.2,-0.3);
\end{tikzpicture}
}\,=
\mathord{
\begin{tikzpicture}[baseline = 0]
\draw[-,gray,dashed, thick](-.5,.4) to (0,-.3);
	\draw[<-,red] (0.3,-0.3) to[out=90, in=0] (0,0.2);
	\draw[-,red] (0,0.2) to[out = -180, in = 60] (-0.25,0.05);
\draw[-,blue] (-0.25,0.05) to[out = 60, in = 40] (-0.5,-0.3);
\end{tikzpicture}
}\:,\quad
\mathord{
\begin{tikzpicture}[baseline = 0]
\draw[-,gray,dashed, thick](.6,.4) to (.1,-.3);
	\draw[<-,blue] (0.6,-0.3) to[out=140, in=-60] (.33,0.03);
     \draw[-,red] (.33,0.03) to[out=-60, in=0] (0.1,0.2);
	\draw[-,red] (0.1,0.2) to[out = -180, in = 90] (-0.2,-0.3);
\end{tikzpicture}
}\,=
\mathord{
\begin{tikzpicture}[baseline = 0]
\draw[-,gray,dashed, thick](-.5,.4) to (0,-.3);
	\draw[<-,blue] (0.3,-0.3) to[out=90, in=0] (0,0.2);
	\draw[-,blue] (0,0.2) to[out = -180, in = 60] (-0.25,0.05);
\draw[-,red] (-0.25,0.05) to[out = 60, in = 40] (-0.5,-0.3);
\end{tikzpicture}
}\:,\quad
\mathord{
\begin{tikzpicture}[baseline = 0]
\draw[-,gray,dashed,thick](-.5,-.3) to (0,.4);
	\draw[<-,red] (0.3,0.4) to[out=-90, in=0] (0,-0.1);
	\draw[-,red] (0,-0.1) to[out = 180, in = -60] (-0.24,0.11);
    \draw[-,blue] (-.24,0.11) to[out =-60, in = -40] (-0.5,0.4);
\end{tikzpicture}
}\,=
\mathord{
\begin{tikzpicture}[baseline = 0]
\draw[-,gray,dashed,thick](.6,-.3) to (.1,.4);
	\draw[<-,red] (0.6,0.4) to[out=-140, in=60] (0.31,0.06);
\draw[-,blue] (0.31,0.06) to[out=60, in=0] (0.1,-0.1);
	\draw[-,blue] (0.1,-0.1) to[out = 180, in = -90] (-0.2,0.4);
\end{tikzpicture}
}
\:,\quad
\mathord{
\begin{tikzpicture}[baseline = 0]
\draw[-,gray,dashed,thick](-.5,-.3) to (0,.4);
	\draw[<-,blue] (0.3,0.4) to[out=-90, in=0] (0,-0.1);
	\draw[-,blue] (0,-0.1) to[out = 180, in = -60] (-0.24,0.11);
    \draw[-,red] (-.24,0.11) to[out =-60, in = -40] (-0.5,0.4);
\end{tikzpicture}
}\,=
\mathord{
\begin{tikzpicture}[baseline = 0]
\draw[-,gray,dashed,thick](.6,-.3) to (.1,.4);
	\draw[<-,blue] (0.6,0.4) to[out=-140, in=60] (0.31,0.06);
\draw[-,red] (0.31,0.06) to[out=60, in=0] (0.1,-0.1);
	\draw[-,red] (0.1,-0.1) to[out = 180, in = -90] (-0.2,0.4);
\end{tikzpicture}
}
\:,
\end{align}

\begin{align}
\mathord{
\begin{tikzpicture}[baseline = 0]
\draw[-,gray,dashed, thick](.6,.4) to (.1,-.3);
	\draw[-,red] (0.6,-0.3) to[out=140, in=-60] (.33,0.03);
     \draw[-,blue] (.33,0.03) to[out=-60, in=0] (0.1,0.2);
	\draw[->,blue] (0.1,0.2) to[out = -180, in = 90] (-0.2,-0.3);
\end{tikzpicture}
}\,=
\mathord{
\begin{tikzpicture}[baseline = 0]
\draw[-,gray,dashed, thick](-.5,.4) to (0,-.3);
	\draw[-,red] (0.3,-0.3) to[out=90, in=0] (0,0.2);
	\draw[-,red] (0,0.2) to[out = -180, in = 60] (-0.25,0.05);
\draw[->,blue] (-0.25,0.05) to[out = 60, in = 40] (-0.5,-0.3);
\end{tikzpicture}
}\:,\quad
\mathord{
\begin{tikzpicture}[baseline = 0]
\draw[-,gray,dashed, thick](.6,.4) to (.1,-.3);
	\draw[-,blue] (0.6,-0.3) to[out=140, in=-60] (.33,0.03);
     \draw[-,red] (.33,0.03) to[out=-60, in=0] (0.1,0.2);
	\draw[->,red] (0.1,0.2) to[out = -180, in = 90] (-0.2,-0.3);
\end{tikzpicture}
}\,=
\mathord{
\begin{tikzpicture}[baseline = 0]
\draw[-,gray,dashed, thick](-.5,.4) to (0,-.3);
	\draw[-,blue] (0.3,-0.3) to[out=90, in=0] (0,0.2);
	\draw[-,blue] (0,0.2) to[out = -180, in = 60] (-0.25,0.05);
\draw[->,red] (-0.25,0.05) to[out = 60, in = 40] (-0.5,-0.3);
\end{tikzpicture}
}\:,\quad
\mathord{
\begin{tikzpicture}[baseline = 0]
\draw[-,gray,dashed,thick](-.5,-.3) to (0,.4);
	\draw[-,red] (0.3,0.4) to[out=-90, in=0] (0,-0.1);
	\draw[-,red] (0,-0.1) to[out = 180, in = -60] (-0.24,0.11);
    \draw[->,blue] (-.24,0.11) to[out =-60, in = -40] (-0.5,0.4);
\end{tikzpicture}
}\,=
\mathord{
\begin{tikzpicture}[baseline = 0]
\draw[-,gray,dashed,thick](.6,-.3) to (.1,.4);
	\draw[-,red] (0.6,0.4) to[out=-140, in=60] (0.31,0.06);
\draw[-,blue] (0.31,0.06) to[out=60, in=0] (0.1,-0.1);
	\draw[->,blue] (0.1,-0.1) to[out = 180, in = -90] (-0.2,0.4);
\end{tikzpicture}
}
\:,\quad
\mathord{
\begin{tikzpicture}[baseline = 0]
\draw[-,gray,dashed,thick](-.5,-.3) to (0,.4);
	\draw[-,blue] (0.3,0.4) to[out=-90, in=0] (0,-0.1);
	\draw[-,blue] (0,-0.1) to[out = 180, in = -60] (-0.24,0.11);
    \draw[->,red] (-.24,0.11) to[out =-60, in = -40] (-0.5,0.4);
\end{tikzpicture}
}\,=
\mathord{
\begin{tikzpicture}[baseline = 0]
\draw[-,gray,dashed,thick](.6,-.3) to (.1,.4);
	\draw[-,blue] (0.6,0.4) to[out=-140, in=60] (0.31,0.06);
\draw[-,red] (0.31,0.06) to[out=60, in=0] (0.1,-0.1);
	\draw[->,red] (0.1,-0.1) to[out = 180, in = -90] (-0.2,0.4);
\end{tikzpicture}
}
\:.
\end{align}
%
%only investigate the relationship between  Spin functor of $\widetilde{G}_n$ and the representation datum of $\bigoplus_{n\geqslant 0}R\rm{O}_{2n+1}(q)\mod$. A similar proposition for the relationship between  Spin functor of $G_n$ and the representation datum of $\scrQU^{\SO}_R$ will be given in \S \ref{sub:explicit-repdatumSO} after the construction of the representation datum of $\scrQU^{\SO}_R$.

\end{proposition}

\begin{proof}The proof is similar to the proof of \cite[Proposition 4.8]{LLZ}.
\end{proof}
Then we have the following direct corollary.
\begin{corollary}
Let  $\scrI_+$ and $ \scrI_-$ be the sets of the generalized eigenvalues of $X^+$ on $F^+$
 and of $X^-$ on $F^-$, respectively.
Then $\scrI_+= \scrI_-$ (as subsets of $R$) and we have  the following isomorphisms for all $i\in \scrI_+=I_-$:
$$\mathrm{Spin}\circ F^+_i\circ \mathrm{Spin}\cong F^-_i\quad \text{and}\,\,
\quad \mathrm{Spin}\circ E^+_i\circ \mathrm{Spin}\cong E^-_i.$$
 \end{corollary}

\smallskip

 \subsection{Determinant functor of  $\O_{2n+1}(q)$ and $\O^\pm_{2n}(q)$}\label{sub:det}\hfill\\

In this section, we let $G_n=\O_{2n+1}(q)$ or $\O^{\pm}_{2n}(q)$ with $q$ odd.
Denote by $\sgn$ or $\sgn_n$ be the \emph{determinant  character} of $G_n$ the unique nontrivial character of order 2 whose kernel coincides with the special orthogonal group  $\SO_{2n+1}(q)$ or $\SO^{\pm}_{2n}(q)$, respectively.
Let $R_{\mathrm{det}_n}$ be the corresponding 1-dimensional module.
Similar to the functor ${\rm Spin}_n$, we define a functor $$\mathrm{Det}_n:
R G _n\mod\to R G _n\mod$$
by $$M\mapsto R_{{\sgn_n}}\otimes_{R} M.$$
Similar to $RG_n \cdot \mathfrak{sp}_n$,  we define a $(RG_n, RG_n)$-bimodule  $RG_n \cdot \mathfrak{det}_n$ as follows: 
\begin{enumerate}
	\item   As a left $RG_n$-module, $RG_n \cdot \mathfrak{det}_n$ is isomorphic to the group ring $RG_n$ itself, where $\mathfrak{det}_n$ serves as a formal generator;
	
	\item The right $RG_n$-module structure is defined through the twisted action
	\[
	(x \cdot \mathfrak{det}_n)h \coloneqq \sgn_n(h)xh \cdot \mathfrak{det}_n
	\]
	for all $ h \in G_n$ and $x  \cdot \mathfrak{det}_n\in RG_n\cdot \mathfrak{det}_n$.
\end{enumerate}
Then the functor $\mathrm{Det}_n$
is represented by  the $(R G _n,R G _n)$-bimodule
$R G _n\cdot\mathfrak{det}_n.$

Let $$\mathrm{Det}:=\mathord{
\begin{tikzpicture}[baseline = 0]
	\draw[-,green, thick, dashed] (0.08,-.2) to (0.08,.4);
\end{tikzpicture}
}:=\bigoplus\limits_{n\geqs 0}\mathrm{Det}_n:
\bigoplus\limits_{n\in \bbN^*}R G_{n}\mod\to \bigoplus\limits_{n\in \bbN^*}R G_{n}\mod.$$
We define $\Phi^+=\mathord{
\begin{tikzpicture}[baseline = -.6mm]
	\draw[->,red] (-0.28,-.3) to (0.28,.3);
	\draw[-,green, thick, dashed] (0.28,-.3) to (-0.28,.3);
\end{tikzpicture}
}\::F^+\circ \mathrm{Det} \to
\mathrm{Det}\circ F^+\,$ and $\Phi^-=\mathord{
\begin{tikzpicture}[baseline = -.6mm]
	\draw[->,blue] (-0.28,-.3) to (0.28,.3);
	\draw[-,green, thick, dashed] (0.28,-.3) to (-0.28,.3);
\end{tikzpicture}
}\::F^-\circ \mathrm{Det}\to
\mathrm{Det}\circ F^-\,$ by the following $(R G_{n+1}, RG_n)$-bimodule maps:
$$\begin{array}{rcl}
\Phi^\epsilon\::R G_{n+1}\cdot e^\epsilon_{n}\otimes_{R G_n}R G_n\cdot\mathfrak{det}_n
&\to& R G_{n+1}\cdot\mathfrak{det}_{n+1}\otimes_{R G_{n+1}}R G_{n+1}\cdot
e^\epsilon_{n}\\
ge^\epsilon_{n}\otimes h\cdot\mathfrak{det}_n &\mapsto & g\cdot\mathfrak{det}_{n+1}\otimes
{\sgn}_n(h)h e^\epsilon_{n}
\end{array}
$$
for all $g\in  G_{n+1}$ and $h\in  G_n$ and $\epsilon\in \{\pm\}$.
\smallskip

Similarly, we define $\Psi^+=\mathord{
\begin{tikzpicture}[baseline = -.6mm]
	\draw[<-,red] (0.28,-.3)  to (-0.28,.3);
	\draw[-,green, thick, dashed] (-0.28,-.3) to (0.28,.3);
\end{tikzpicture}
}\::\mathrm{Det}\circ E^+\to E^+\circ \mathrm{Det}
\,$ and $\Psi^-=\mathord{
\begin{tikzpicture}[baseline = -.6mm]
	\draw[<-,blue] (0.28,-.3)  to (-0.28,.3);
	\draw[-,green, thick, dashed] (-0.28,-.3) to (0.28,.3);
\end{tikzpicture}
}\::\mathrm{Det}\circ E^-\to E^-\circ \mathrm{Det}
\,$ via the $(R G_{n-1}, RG_n)$-bimodule maps:
$$\begin{array}{rcl}
\Psi^\epsilon\::R G_{n-1}\cdot\mathfrak{det}_{n-1}\otimes_{R G_{n}}e^\epsilon_{n-1}\cdot R G_{n}
&\to& e^\epsilon_{n-1}\cdot R G_{n}\otimes_{R G_n}R G_n\cdot\mathfrak{det}_n\\
g\cdot\mathfrak{det}_{n-1}\otimes
e^\epsilon_{n}h &\mapsto &e^\epsilon_{n-1}g\otimes {\sgn}_n(h)h\cdot\mathfrak{det}_n
\end{array}
$$
for all $g\in  G_{n-1}$ and $h\in  G_n$ and $\epsilon\in \{\pm\}$.
It is easy to see that $\mathrm{Det}^2\cong \unit.$
Similarly,  $\Phi^+$ (resp. $\Phi^-,\Psi^+ $ or $\Psi^-$) is an isomorphisms.
We define its inverse  $(\Phi^+)^{-1}=\mathord{
\begin{tikzpicture}[baseline = -.6mm]
	\draw[->,red] (0.28,-.3) to  (-0.28,.3);
	\draw[-,green, thick, dashed] (-0.28,-.3) to (0.28,.3);
\end{tikzpicture}
}$ (resp.  $(\Phi^-)^{-1}=\mathord{
\begin{tikzpicture}[baseline = -.6mm]
	\draw[->,blue] (0.28,-.3) to (-0.28,.3);
	\draw[-,green, thick, dashed] (-0.28,-.3) to (0.28,.3);
\end{tikzpicture}
}\,,$ $(\Psi^+)^{-1}=\mathord{
\begin{tikzpicture}[baseline = -.6mm]
	\draw[<-,red] (-0.28,-.3) to (0.28,.3);
	\draw[-,green, thick, dashed] (0.28,-.3) to (-0.28,.3);
\end{tikzpicture}
}\,$ or $(\Psi^-)^{-1}=\mathord{
\begin{tikzpicture}[baseline = -.6mm]
	\draw[<-,blue] (-0.28,-.3) to (0.28,.3);
	\draw[-,green, thick, dashed] (0.28,-.3) to (-0.28,.3);
\end{tikzpicture}
}$).
\smallskip

Also, we have the following analogous results of Proposition \ref{prop:spinor} for $\mathrm{Det}$.
\begin{proposition}\label{prop:det}
We have the following commutative diagrams:
\begin{align}\label{change}
\mathord{
\begin{tikzpicture}[baseline = -1mm]
	\draw[-,thin,red] (0.28,-.6) to[out=90,in=-90] (-0.28,0);
	\draw[->,thin,red] (-0.28,0) to[out=90,in=-90] (0.28,.6);
	\draw[-,thick,green,dashed] (-0.28,-.6) to[out=90,in=-90] (0.28,0);
	\draw[-,thick,green,dashed] (0.28,0) to[out=90,in=-90] (-0.28,.6);
\end{tikzpicture}
}\,=
\mathord{
\begin{tikzpicture}[baseline = -1mm]
	\draw[->,thin,red] (0.18,-.6) to (0.18,.6);
	\draw[-,thick,green,dashed] (-0.18,-.6) to (-0.18,.6);
\end{tikzpicture}
}\:,\quad
\mathord{
\begin{tikzpicture}[baseline = -1mm]
	\draw[-,thick,green,dashed] (0.28,0) to[out=90,in=-90] (-0.28,.6);
	\draw[->,thin,blue] (-0.28,0) to[out=90,in=-90] (0.28,.6);
	\draw[-,thick,green,dashed] (-0.28,-.6) to[out=90,in=-90] (0.28,0);
	\draw[-,thin,blue] (0.28,-.6) to[out=90,in=-90] (-0.28,0);
\end{tikzpicture}
}\,=
\mathord{
\begin{tikzpicture}[baseline = -1mm]
	\draw[->,thin,blue] (0.18,-.6) to (0.18,.6);
	\draw[-,thick,green,dashed] (-0.18,-.6) to (-0.18,.6);
\end{tikzpicture}
}\:,\quad
\mathord{
\begin{tikzpicture}[baseline = -1mm]
	\draw[<-,thin,red] (0.28,-.6) to[out=90,in=-90] (-0.28,0);
	\draw[-,thin,red] (-0.28,0) to[out=90,in=-90] (0.28,.6);
	\draw[-,thick,green,dashed] (-0.28,-.6) to[out=90,in=-90] (0.28,0);
	\draw[-,thick,green,dashed] (0.28,0) to[out=90,in=-90] (-0.28,.6);
\end{tikzpicture}
}\,=
\mathord{
\begin{tikzpicture}[baseline = -1mm]
	\draw[<-,thin,red] (0.18,-.6) to (0.18,.6);
	\draw[-,thick,green,dashed] (-0.18,-.6) to (-0.18,.6);
\end{tikzpicture}
}\:,\quad
\mathord{
\begin{tikzpicture}[baseline = -1mm]
	\draw[-,thick,green,dashed] (0.28,0) to[out=90,in=-90] (-0.28,.6);
	\draw[-,thin,blue] (-0.28,0) to[out=90,in=-90] (0.28,.6);
	\draw[-,thick,green,dashed] (-0.28,-.6) to[out=90,in=-90] (0.28,0);
	\draw[<-,thin,blue] (0.28,-.6) to[out=90,in=-90] (-0.28,0);
\end{tikzpicture}
}\,=
\mathord{
\begin{tikzpicture}[baseline = -1mm]
	\draw[<-,thin,blue] (0.18,-.6) to (0.18,.6);
	\draw[-,thick,green,dashed] (-0.18,-.6) to (-0.18,.6);
\end{tikzpicture}
}\:,
\end{align}

\begin{align}\label{dotslide11}
\mathord{
\begin{tikzpicture}[baseline = -.5mm]
	\draw[->,thin,red] (0.28,-.3) to (-0.28,.4);
      \node at (0.165,-0.15) {$\dott$};
	\draw[-,thick,green,dashed] (-0.28,-.3) to (0.28,.4);
\end{tikzpicture}
}&=-\mathord{
\begin{tikzpicture}[baseline = -.5mm]
	\draw[-,thick,green,dashed] (-0.28,-.3) to (0.28,.4);
	\draw[->,thin,red] (0.28,-.3) to (-0.28,.4);
      \node at (-0.14,0.23) {$\dott$};
\end{tikzpicture}
},
&
\mathord{
\begin{tikzpicture}[baseline = -.5mm]
	\draw[->,thin,blue] (0.28,-.3) to (-0.28,.4);
      \node at (0.165,-0.15) {$\dott$};
	\draw[-,thick,green,dashed] (-0.28,-.3) to (0.28,.4);
\end{tikzpicture}
}&=-\mathord{
\begin{tikzpicture}[baseline = -.5mm]
	\draw[-,thick,green,dashed] (-0.28,-.3) to (0.28,.4);
	\draw[->,thin,blue] (0.28,-.3) to (-0.28,.4);
      \node at (-0.14,0.23) {$\dott$};
\end{tikzpicture}
}
\:,
\end{align}

\begin{align}\label{slide1}
\mathord{
\begin{tikzpicture}[baseline = -1mm]
	\draw[->,thin,red] (0.45,-.6) to (-0.45,.6);
        \draw[-,thin,red] (0,-.6) to[out=90,in=-90] (-.45,0);
        \draw[-,line width=4pt,white] (-0.45,0) to[out=90,in=-90] (0,0.6);
        \draw[->,thin,red] (-0.45,0) to[out=90,in=-90] (0,0.6);
	\draw[-,thick,green,dashed] (0.45,.6) to (-0.45,-.6);
\end{tikzpicture}
}\,
=
\mathord{
\begin{tikzpicture}[baseline = -1mm]
	\draw[->,thin,red] (0.45,-.6) to (-0.45,.6);
        \draw[-,line width=4pt,white] (0,-.6) to[out=90,in=-90] (.45,0);
        \draw[-,thin,red] (0,-.6) to[out=90,in=-90] (.45,0);
        \draw[->,thin,red] (0.45,0) to[out=90,in=-90] (0,0.6);
	\draw[-,thick,green,dashed] (0.45,.6) to (-0.45,-.6);
\end{tikzpicture}
}\:,\quad
\mathord{
\begin{tikzpicture}[baseline = -1mm]
	\draw[->,thin,blue] (0.45,-.6) to (-0.45,.6);
        \draw[-,thin,blue] (0,-.6) to[out=90,in=-90] (-.45,0);
        \draw[-,line width=4pt,white] (-0.45,0) to[out=90,in=-90] (0,0.6);
        \draw[->,thin,blue] (-0.45,0) to[out=90,in=-90] (0,0.6);
	\draw[-,thick,green,dashed] (0.45,.6) to (-0.45,-.6);
\end{tikzpicture}
}\,
=
\mathord{
\begin{tikzpicture}[baseline = -1mm]
	\draw[->,thin,blue] (0.45,-.6) to (-0.45,.6);
        \draw[-,line width=4pt,white] (0,-.6) to[out=90,in=-90] (.45,0);
        \draw[-,thin,blue] (0,-.6) to[out=90,in=-90] (.45,0);
        \draw[->,thin,blue] (0.45,0) to[out=90,in=-90] (0,0.6);
	\draw[-,thick,green,dashed] (0.45,.6) to (-0.45,-.6);
\end{tikzpicture}
}\:,\quad
\mathord{
\begin{tikzpicture}[baseline = -1mm]
	\draw[-,thick,green,dashed] (0.45,.6) to (-0.45,-.6);
	\draw[->,blue] (0.45,-.6) to (-0.45,.6);
        \draw[-,red] (0,-.6) to[out=90,in=-90] (-.45,0);
        \draw[->,red] (-0.45,0) to[out=90,in=-90] (0,0.6);
\end{tikzpicture}
}\,
=
\mathord{
\begin{tikzpicture}[baseline = -1mm]
	\draw[-,thick,green,dashed] (0.45,.6) to (-0.45,-.6);
	\draw[->,blue] (0.45,-.6) to (-0.45,.6);
        \draw[-,red] (0,-.6) to[out=90,in=-90] (.45,0);
        \draw[->,red] (0.45,0) to[out=90,in=-90] (0,0.6);
\end{tikzpicture}
}\:,\quad
\mathord{
\begin{tikzpicture}[baseline = -1mm]
	\draw[-,thick,green,dashed] (0.45,.6) to (-0.45,-.6);
	\draw[->,red] (0.45,-.6) to (-0.45,.6);
        \draw[-,blue] (0,-.6) to[out=90,in=-90] (-.45,0);
        \draw[->,blue] (-0.45,0) to[out=90,in=-90] (0,0.6);
\end{tikzpicture}
}
=
\mathord{
\begin{tikzpicture}[baseline = -1mm]
	\draw[-,thick,green,dashed] (0.45,.6) to (-0.45,-.6);
	\draw[->,red] (0.45,-.6) to (-0.45,.6);
        \draw[-,blue] (0,-.6) to[out=90,in=-90] (.45,0);
        \draw[->,blue] (0.45,0) to[out=90,in=-90] (0,0.6);
\end{tikzpicture}
}\:,
\end{align}

\begin{align}\label{slide2}
\mathord{
\begin{tikzpicture}[baseline = 0]
\draw[-,green,dashed, thick](.6,.4) to (.1,-.3);
	\draw[<-,red] (0.6,-0.3) to[out=140, in=0] (0.1,0.2);
	\draw[-,red] (0.1,0.2) to[out = -180, in = 90] (-0.2,-0.3);
\end{tikzpicture}
}\,=
\mathord{
\begin{tikzpicture}[baseline = 0]
\draw[-,green,dashed, thick](-.5,.4) to (0,-.3);
	\draw[<-,red] (0.3,-0.3) to[out=90, in=0] (0,0.2);
	\draw[-,red] (0,0.2) to[out = -180, in = 40] (-0.5,-0.3);
\end{tikzpicture}
}\:,\quad
\mathord{
\begin{tikzpicture}[baseline = 0]
\draw[-,green,dashed, thick](.6,.4) to (.1,-.3);
	\draw[<-,blue] (0.6,-0.3) to[out=140, in=0] (0.1,0.2);
	\draw[-,blue] (0.1,0.2) to[out = -180, in = 90] (-0.2,-0.3);
\end{tikzpicture}
}\,=
\mathord{
\begin{tikzpicture}[baseline = 0]
\draw[-,green,dashed, thick](-.5,.4) to (0,-.3);
	\draw[<-,blue] (0.3,-0.3) to[out=90, in=0] (0,0.2);
	\draw[-,blue] (0,0.2) to[out = -180, in = 40] (-0.5,-0.3);
\end{tikzpicture}
}\:,\quad
\mathord{
\begin{tikzpicture}[baseline = 0]
\draw[-,green,dashed,thick](-.5,-.3) to (0,.4);
	\draw[<-,red] (0.3,0.4) to[out=-90, in=0] (0,-0.1);
	\draw[-,red] (0,-0.1) to[out = 180, in = -40] (-0.5,0.4);
\end{tikzpicture}
}\,=
\mathord{
\begin{tikzpicture}[baseline = 0]
\draw[-,green,dashed,thick](.6,-.3) to (.1,.4);
	\draw[<-,red] (0.6,0.4) to[out=-140, in=0] (0.1,-0.1);
	\draw[-,red] (0.1,-0.1) to[out = 180, in = -90] (-0.2,0.4);
\end{tikzpicture}
}
\:,\quad
\mathord{
\begin{tikzpicture}[baseline = 0]
\draw[-,green,dashed,thick](-.5,-.3) to (0,.4);
	\draw[<-,blue] (0.3,0.4) to[out=-90, in=0] (0,-0.1);
	\draw[-,blue] (0,-0.1) to[out = 180, in = -40] (-0.5,0.4);
\end{tikzpicture}
}\,=
\mathord{
\begin{tikzpicture}[baseline = 0]
\draw[-,green,dashed,thick](.6,-.3) to (.1,.4);
	\draw[<-,blue] (0.6,0.4) to[out=-140, in=0] (0.1,-0.1);
	\draw[-,blue] (0.1,-0.1) to[out = 180, in = -90] (-0.2,0.4);
\end{tikzpicture}
}
\:,
\end{align}

\begin{align}
\mathord{
\begin{tikzpicture}[baseline = 0]
\draw[-,green,dashed, thick](.6,.4) to (.1,-.3);
	\draw[-,red] (0.6,-0.3) to[out=140, in=0] (0.1,0.2);
	\draw[->,red] (0.1,0.2) to[out = -180, in = 90] (-0.2,-0.3);
\end{tikzpicture}
}\,=
\mathord{
\begin{tikzpicture}[baseline = 0]
\draw[-,green,dashed, thick](-.5,.4) to (0,-.3);
	\draw[-,red] (0.3,-0.3) to[out=90, in=0] (0,0.2);
	\draw[->,red] (0,0.2) to[out = -180, in = 40] (-0.5,-0.3);
\end{tikzpicture}
}\:,\quad
\mathord{
\begin{tikzpicture}[baseline = 0]
\draw[-,green,dashed, thick](.6,.4) to (.1,-.3);
	\draw[-,blue] (0.6,-0.3) to[out=140, in=0] (0.1,0.2);
	\draw[->,blue] (0.1,0.2) to[out = -180, in = 90] (-0.2,-0.3);
\end{tikzpicture}
}\,=
\mathord{
\begin{tikzpicture}[baseline = 0]
\draw[-,green,dashed, thick](-.5,.4) to (0,-.3);
	\draw[-,blue] (0.3,-0.3) to[out=90, in=0] (0,0.2);
	\draw[->,blue] (0,0.2) to[out = -180, in = 40] (-0.5,-0.3);
\end{tikzpicture}
}\:,\quad
\mathord{
\begin{tikzpicture}[baseline = 0]
\draw[-,green,dashed,thick](-.5,-.3) to (0,.4);
	\draw[-,red] (0.3,0.4) to[out=-90, in=0] (0,-0.1);
	\draw[->,red] (0,-0.1) to[out = 180, in = -40] (-0.5,0.4);
\end{tikzpicture}
}\,=
\mathord{
\begin{tikzpicture}[baseline = 0]
\draw[-,green,dashed,thick](.6,-.3) to (.1,.4);
	\draw[-,red] (0.6,0.4) to[out=-140, in=0] (0.1,-0.1);
	\draw[->,red] (0.1,-0.1) to[out = 180, in = -90] (-0.2,0.4);
\end{tikzpicture}
}
\:,\quad
\mathord{
\begin{tikzpicture}[baseline = 0]
\draw[-,green,dashed,thick](-.5,-.3) to (0,.4);
	\draw[-,blue] (0.3,0.4) to[out=-90, in=0] (0,-0.1);
	\draw[->,blue] (0,-0.1) to[out = 180, in = -40] (-0.5,0.4);
\end{tikzpicture}
}\,=
\mathord{
\begin{tikzpicture}[baseline = 0]
\draw[-,green,dashed,thick](.6,-.3) to (.1,.4);
	\draw[-,blue] (0.6,0.4) to[out=-140, in=0] (0.1,-0.1);
	\draw[->,blue] (0.1,-0.1) to[out = 180, in = -90] (-0.2,0.4);
\end{tikzpicture}
}
\:.
\end{align}
\end{proposition}
\begin{proof}
We only show (\ref{dotslide11}). The others are similar.
Note that
$$(1_{\rm{Det}}\otimes X^\epsilon)\circ\Phi^\epsilon:F^\epsilon_{n,n+1}\circ \mathrm{Det}_n\to \mathrm{Det}_{n+1}\circ F^\epsilon_{n,n+1}$$
is represented by the composition of the $(R G_{n+1},R G_n)$-bimodule maps
$$
\begin{array}{ccc}
	R G_{n+1}\cdot e^\epsilon_{n}\otimes_{R G_n}R G_n\cdot\mathfrak{det}_n
	&\rightarrow & R G_{n+1}\cdot\mathfrak{det}_{n+1}\otimes_{R G_{n+1}}R G_{n+1}\cdot e^\epsilon_{n}
	\\
	ge^\epsilon_{n}\otimes h\cdot\mathfrak{det}_n&\mapsto& g\cdot\mathfrak{det}_{n+1}\otimes {\sgn(h)h} e^\epsilon_{n},
\end{array}
$$
and
\begin{align*}
	R G_{n+1}\cdot\mathfrak{det}_{n+1}\otimes_{R G_{n+1}}R G_{n+1}\cdot e^\epsilon_{n}
	&\rightarrow  R G_{n+1}\cdot\mathfrak{det}_{n+1}\otimes_{R G_{n+1}}R G_{n+1}\cdot e^\epsilon_{n}\\ g\cdot\mathfrak{det}_{n+1}\otimes {\sgn(h)h} e^\epsilon_{n}
	&\mapsto g\cdot\mathfrak{det}_{n+1}\otimes {\sgn(h)h} X^\epsilon_{n+1},
\end{align*}
for all $g\in  G_{n+1}$ and $h\in  G_n.$
Also, we note that
$$\Phi^\epsilon\circ (X^\epsilon\otimes 1_{\rm {Det}}):F^\epsilon_{n,n+1} \circ \mathrm{Det
}_n\to \mathrm{Det}_{n+1}\circ F^\epsilon_{n,n+1}$$
is represented by the composition of the $(R G_{n+1},R G_n)$-bimodule maps
$$
\begin{array}{ccc}
	R G_{n+1}\cdot e^\epsilon_{n}\otimes_{R G_n}R G_n\cdot\mathfrak{det}_n
	&\to&
	R G_{n+1}\cdot e^\epsilon_{n}\otimes_{R G_n}R G_n\cdot\mathfrak{det}_n
	\\
	ge^\epsilon_{n}\otimes h\cdot\mathfrak{det}_n &\mapsto & gX^\epsilon_{n+1}\otimes h\cdot\mathfrak{det}_n,
\end{array}
$$
and
$$
\begin{array}{ccc}
	R G_{n+1}\cdot e^\epsilon_{n}\otimes_{R G_n}R G_n\cdot\mathfrak{det}_n
	&\to &
	R G_{n+1}\cdot\mathfrak{det}_{n+1}\otimes_{R G_{n+1}}R G_{n+1}\cdot e^\epsilon_{n}\\
	gX^\epsilon_{n+1}\otimes h\cdot\mathfrak{det}_n
	&\mapsto & gX^\epsilon_{n+1}\cdot\mathfrak{det}_{n+1}\otimes {\sgn(h)h} e^\epsilon_{n}.
\end{array}
$$
\smallskip

Since $U_n$ is a subgroup of $\SO_{2n+1}(q)$ or $\SO^\pm_{2n}(q)$, we have 
$e^\epsilon_{n}\cdot\mathfrak{det}_{n+1}=\mathfrak{det}_{n+1}\cdot e^\epsilon_{n}$.
We also have
$$\dot{x}_{n+1}\cdot\mathfrak{det}_{n+1}
=\mathfrak{det}_{n+1}\cdot {\sgn(\dot{x}_{n+1})}\dot{x}_{n+1} =-\mathfrak{det}_{n+1}\cdot \dot{x}_{n+1},$$ since $\sgn(\dot{x}_{n+1})=-1.$ 
It follows that
$$gX^\epsilon_{n+1}\cdot\mathfrak{det}_{n+1}\otimes {\sgn(h)h} e_{n+1,n}=-g\cdot\mathfrak{det}_{n+1}\otimes {\sgn(h)h}\, X^\epsilon_{n+1}.$$
Hence
$(1_{\rm {Det}}\otimes X^\epsilon)\circ\Phi^\epsilon=-\Phi^\epsilon\circ (X^\epsilon \otimes  1_{\rm {Det}})$, which proves (\ref{dotslide11}).
\end{proof}
Proposition \ref{prop:det} has a straightforward consequence.

\begin{corollary} Let  $\scrI_+$ and $ \scrI_-$ be the sets of the
generalized eigenvalues of $X^+$ on $F^+$
 and of $X^-$ on $F^-$,
 respectively. Then $-\scrI_+=\scrI_+$ and $-\scrI_-=\scrI_-.$ Namely,
if $i\in \scrI_+$ (resp. $i'\in \scrI_-$) is a generalized
eigenvalue of $X^+$ on $F^+$ (resp. $X^-$ on $F^-$) then $-i$
(resp. $-i'$) is also a generalized eigenvalue of $X^+$ on $F^+$ (resp. $X^-$ on $F^-$).
Moreover, we have the following isomorphisms for all $i\in I_+$ and $i'\in I_-:$
$$\mathrm{Det}\circ F^+_i\circ\mathrm{Det}\cong F^+_{- i}~{,}~\mathrm{Det}
\circ F^-_{i'}\circ\mathrm{Det}\cong F^-_{- i'},$$
$$\mathrm{Det}\circ E^+_i\circ\mathrm{Det}\cong E^+_{- i}~{,}~\mathrm{Det}
\circ E^-_{i'}\circ\mathrm{Det}\cong E^-_{- i'}.$$
\end{corollary}
\smallskip

\subsection{Diagonal automorphism functor of $\Sp_{2n}(q)$ and $\O^{\pm}_{2n}(q)$}\label{sec:diag}\hfill\\

In this section, let $G_n=\Sp_{2n}(q)$ or $\O^{\pm}_{2n}(q)$ and
$ \widetilde{G}_n=\CSp_{2n+1}(q)$ or $\CO^{\pm}_{2n}(q)$ be the corresponding
similitude group with $q$ odd.
Let $\tau_n\in  \widetilde{G}_n$ satisfying $\zeta(\lambda(\tau_n))=-1$, where $\lambda$ is determined by $\langle hu,hv\rangle=\lambda(h)\langle u,v\rangle$ for any $h\in \widetilde{G}_n$ and $u,v$ lying in the corresponding quadratic space with quadratic from $\langle\,\,, \,\,\rangle$, then the \emph{diagonal automorphism} of $G_n$ is given by$$\diag_n:G_n\to G_n$$ such that
$\mathrm{diag}_n(g)={^{\tau_n}g}$ for any $g\in G_n$, where $^{\tau_n}g={\tau_n}g\tau^{-1}_n.$
Note that up to inner automorphisms, the diagonal automorphism is unique and of order 2. For any representation $\rho$ of $G_n$, let $\rho^c$ denote
the representation obtained from $\rho$ by conjugating $\tau_n$ in the corresponding
similitude group  (see\cite[\S 4.3,\,\S 4.10]{Wa}).

In particular, we fix a choice of $\tau_n\in  \widetilde{G}_n$ for explicit computation. Let $\delta$ be a non-square in $\bbF_q$ and let $\beta$ and $\overline{\beta}$ be the two roots of $x^2-\delta.$
Then we define
$\tau_n$ %, which centralizes $G_r$,
to  be  the matrix
 % $$ \left(\begin{array}{ccc} & &(-1)^{r+1}2\\ &-\Id_{G_r}&\\(-1)^{r+1}1/2&&\\ \end{array}\right), \quad
 $$\left(\begin{array}{cc} \id_n&   \\ & \delta\id_{n} \\ \end{array}\right),
 \quad \left(\begin{array}{cc} \id_n&  \\ & \delta\id_{n}  \\ \end{array}\right)
  \quad\text{or} \quad \left(\begin{array}{cccc}\id_{n-1} & & &\\&0&\beta& \\ &\overline{\beta}&0& \\&& &\delta\id_{n-1}\\ \end{array}\right)$$
when $G_n= \Sp_{2n}(q), \O_{2n}^{+}(q)$ or  $\O_{2n}^{-}(q)$, respectively. It is easy to check that under our choices of $\tau_n$,
$\tau_{n-1}\tau_n^{-1}$ commutes with $G_{n-1}$, i.e., $[\tau_{n-1}\tau_n^{-1},G_{n-1}]=1.$
%Let $\mathfrak{K}$ be a sub-$(R G_n,R G_n)$-module
% generated by $\mathrm{diag}_n(g)\otimes1-1\otimes g$, where $g\in  G_n.$
To define the functor $\mathrm{Diag}_n$,
we define a $(RG_n, RG_n)$-bimodule  $RG_n \cdot \mathfrak{diag}_n$ as follows: 
\begin{enumerate}
	\item   As a left $RG_n$-module, $RG_n \cdot \mathfrak{diag}_n$ is isomorphic to the group ring $RG_n$ itself, where $\mathfrak{diag}_n$ serves as a formal generator;
	
	\item The right $RG_n$-module structure is defined through the twisted action
	\[
	(x \cdot \mathfrak{diag}_n)h \coloneqq x(^{\tau^{-1}_n }h )\cdot \mathfrak{diag}_n
	\]
	for all $ h \in G_n$ and $x  \cdot \mathfrak{diag}_n\in RG_n\cdot \mathfrak{diag}_n$.
\end{enumerate}
We define the functor $$\mathrm{Diag}_n:
R G _n\mod\to R G _n\mod$$
by $$M\mapsto  RG_n \cdot \mathfrak{diag}_n  \otimes_{RG_n} M.$$ 
If we denote the representation associated with the module $M$ by $\rho$, it follows directly that the representation corresponding to the module $\mathrm{Diag}_n(M)$ is precisely $\rho^c.$
\smallskip

Let $$\mathrm{Diag}:=\mathord{
\begin{tikzpicture}[baseline = 0]
	\draw[-,brown, thick, dashed] (0.08,-.2) to (0.08,.4);
\end{tikzpicture}
}:=\bigoplus\limits_{n\in \bbN^*}\mathrm{Diag}_n:\bigoplus_{n\in \bbN^*}
R{G}_{n}\mod\to \bigoplus_{n\in \bbN^*}R{G}_{n}\mod.$$

 We define $\dot{\Phi}^+=\mathord{
\begin{tikzpicture}[baseline = -.6mm]
	\draw[->,red] (-0.28,-.3) to (0.28,.3);
	\draw[-,brown, thick, dashed] (0.28,-.3) to (-0.28,.3);
\end{tikzpicture}
}\::F^+\circ \mathrm{Diag} \to
\mathrm{Diag}\circ F^+\,$ and $\dot{\Phi}^-=\mathord{
\begin{tikzpicture}[baseline = -.6mm]
	\draw[->,blue] (-0.28,-.3) to (0.28,.3);
	\draw[-,brown, thick, dashed] (0.28,-.3) to (-0.28,.3);
\end{tikzpicture}
}\::F^-\circ \mathrm{Diag}\to
\mathrm{Diag}\circ F^-\,$ by the following $(R G_{n+1}, RG_n)$-bimodule maps:
$$\begin{array}{rcl}
\dot{\Phi}^\epsilon\::R G _{n+1}\cdot e^\epsilon_{n}\otimes_{R G _n}R G _n\cdot\mathfrak{diag}_n
&\to& R G _{n+1}\cdot\mathfrak{diag}_{n+1}\otimes_{R G _{n+1}}R G _{n+1}
\cdot e^\epsilon_{n}\\
ge^\epsilon_{n}\otimes h\cdot\mathfrak{diag}_n &\mapsto & g\cdot\mathfrak{diag}_{n+1}
\otimes {^{\tau_n}h} e^\epsilon_{n}
\end{array}
$$
for all $g\in  G_{n+1}$ and $h\in  G_n$ and $\epsilon\in \{\pm\}$.
\smallskip

Similarly,  we define $\dot{\Psi}^+=\mathord{
\begin{tikzpicture}[baseline = -.6mm]
	\draw[->,red] (-0.28,.3) to (0.28,-.3);
	\draw[-,brown, thick, dashed] (0.28,.3) to (-0.28,-.3);
\end{tikzpicture}
}\:: \mathrm{Diag}\circ E^+\to E^+\circ \mathrm{Diag}
\,$ and $\dot{\Psi}^-=\mathord{
\begin{tikzpicture}[baseline = -.6mm]
	\draw[->,blue] (-0.28,.3) to (0.28,-.3);
	\draw[-,brown, thick, dashed] (0.28,.3) to (-0.28,-.3);
\end{tikzpicture}
}\::\mathrm{Diag}\circ E^-\to
E^-\circ \mathrm{Diag}\,$ by the following $(R G_{n-1}, RG_n)$-bimodule maps:
$$\begin{array}{rcl}
\dot{\Psi}^\epsilon\::R G _{n-1}\cdot\mathfrak{diag}_{n-1}\otimes_{R G _{n-1}}
e^\epsilon_{n}\cdot R G _{n}
&\to&e^\epsilon_{n-1}\cdot R G _{n}\otimes_{R G _n}R G _n\cdot\mathfrak{diag}_n\\
g\cdot\mathfrak{diag}_{n-1}
\otimes e^\epsilon_{n}h&\mapsto &e^\epsilon_{n-1} g\otimes{^{\tau^{-1}_n}h}
\cdot\mathfrak{diag}_n
\end{array}
$$
for all $g\in  G_{n-1}$ and $h\in  G_n$ and $\epsilon\in \{\pm\}$.

\begin{lemma}\label{lem:diag}
The map $\dot{\Phi}^\epsilon$ (resp. $\dot{\Psi}^\epsilon$) is an
isomorphism between the functors $F^\epsilon\circ \mathrm{Diag}$ and
$\mathrm{Diag}\circ F^\epsilon$ (resp. $\mathrm{Diag}\circ E^\epsilon$ and $E^\epsilon\circ \mathrm{Diag}$).
\end{lemma}
\begin{proof}
We only prove the case of $\dot{\Phi}^\epsilon$.
It suffices to show that $\dot{\Phi}^\epsilon$ is
an $(R G _{n+1},R G _n)$-bimodule isomorphism for  each $n\in\bbN$.
We first prove that the map $\dot{\Phi}^\epsilon$ is well-defined. In fact, it suffices to prove that
\begin{equation}\label{Phi:1}
\dot{\Phi}^\epsilon(g k^{-1}e^\epsilon_{n}\otimes kh\cdot\mathfrak{diag}_n)=\dot{\Phi}^\epsilon(g e^\epsilon_{n}\otimes h\cdot\mathfrak{diag}_n)
\end{equation}
and \begin{equation}\label{Phi:2}
\dot{\Phi}^\epsilon( gue^\epsilon_{n}\otimes h\cdot\mathfrak{diag}_n)=\dot{\Phi}^\epsilon( \zeta_{\epsilon}(u) g e^\epsilon_{n}\otimes h\cdot\mathfrak{diag}_n)
\end{equation}
for all $k\in G_n$ and $u\in U_{n},$
where $\zeta_\epsilon=1$ if $\epsilon=+$ and $\zeta_\epsilon=\zeta$ if $\epsilon=-.$
\smallskip

For Equality (\ref{Phi:1}), we have
\begin{align*}
&\dot{\Phi}^\epsilon(g k^{-1}e^\epsilon_{n}\otimes kh\cdot\mathfrak{diag}_n)\\
&= gk^{-1}\cdot\mathfrak{diag}_{n+1}\otimes {^{\tau_n}(kh)} e^\epsilon_{n}  \\
 &= g \cdot\mathfrak{diag}_{n+1}\,{^{\tau_{n+1}}(k^{-1})}\otimes {^{\tau_n}(kh)} e^\epsilon_{n}\\
 &= g \cdot\mathfrak{diag}_{n+1}\,{^{\tau_{n}}(k^{-1})}\otimes {^{\tau_n}(kh)} e^\epsilon_{n}\quad (\text{since $[\tau_{n}\tau_{n+1}^{-1},G_{n}]=1$})\\
  &= g \cdot\mathfrak{diag}_{n+1}\,\otimes {^{\tau_{n}}(k^{-1})}\,\,{^{\tau_n}(kh)} e^\epsilon_{n}\\
  &= g\cdot\mathfrak{diag}_{n+1}\otimes {^{\tau_n}h} e^\epsilon_{n}\\
  &= \dot{\Phi}^\epsilon(g e^\epsilon_{n}\otimes h\cdot\mathfrak{diag}_n).
\end{align*}

Note that
$\tau_{n+1}$ normalize $U_n$ and $\zeta_\epsilon({^{\tau_{n+1}}u})=\zeta_\epsilon(u)$.
For Equality (\ref{Phi:2}), we
have
$$\begin{array}{rclr}
\dot{\Phi}^\epsilon( gue^\epsilon_{n}\otimes h\cdot\mathfrak{diag}_n)&=&
 gu\cdot\mathfrak{diag}_{n+1}\otimes  {^{\tau_n}h} e^\epsilon_{n} &  \\
 &=&g\cdot\mathfrak{diag}_{n+1} {^{\tau_{n+1}}u}\otimes {^{\tau_{n+1}}h} e^\epsilon_{n} &\\
  &=&g\cdot\mathfrak{diag}_{n+1}\otimes  {^{\tau_{n+1}}(uh)} e^\epsilon_{n} &\\
  &=&g\cdot\mathfrak{diag}_{n+1}\otimes  {^{\tau_{n+1}}(hu)} e^\epsilon_{n} &\\
 &=&\zeta_\epsilon({^{\tau_{n+1}}u})g\cdot\mathfrak{diag}_{n+1}\otimes {^{\tau_n}h} e^\epsilon_{n} &\\
 &=&\zeta_\epsilon(u)g\cdot\mathfrak{diag}_{n+1}\otimes {^{\tau_n}h} e^\epsilon_{n} &\\
  &=&\dot{\Phi}^\epsilon( \zeta_{\epsilon}(u) g e^\epsilon_{n}\otimes h\cdot\mathfrak{diag}_n). &
\end{array}
$$

Now we prove that the map $\dot{\Phi}^\epsilon$ is an $(R G_{n+1},R G_n)$-bimodule homomorphism.
Obviously,  $\dot{\Phi}^\epsilon$ is a left $R G_{n+1}$-module map.
The conclusion that $\dot{\Phi}^\epsilon$ is also a right $ R G_n$-module map follows from
the following:
$$\begin{array}{rcl}
\dot{\Phi}^\epsilon(g e^\epsilon_{n}\otimes h\cdot\mathfrak{diag}_n k)&=&\dot{\Phi}^\epsilon (g e^\epsilon_{n}\otimes h\,\,{^{\tau^{-1}_n}k}\cdot\mathfrak{diag}_n )\\
&=& g\cdot\mathfrak{diag}_{n+1}\otimes {^{\tau_n}( h\,\,{^{\tau^{-1}_n}k})} e^\epsilon_{n} \\
&=& g\cdot\mathfrak{diag}_{n+1}\otimes {^{\tau_n}h} e^\epsilon_{n}k \\
&=&\dot{\Phi}^\epsilon(g e^\epsilon_{n}\otimes h\cdot\mathfrak{diag}_n)k
\end{array}
$$
for any $k\in  G_n$.

\smallskip

Finally, we prove that the map $\dot{\Phi}^\epsilon$ is an isomorphism.
The surjection of $\dot{\Phi}^\epsilon$ is clear since every element
$g\cdot\mathfrak{diag}_{n+1}\otimes he^\epsilon_{n}$ has a  pre-image
$ge^\epsilon_{n}\otimes {^{\tau_n^{-1}}h}\cdot\mathfrak{diag}_{n}.$
It is easy to see that $R G_{n+1}\cdot e^\epsilon_{n}\otimes_{R G_n}R G_n\cdot\mathfrak{diag}_n$ and
$R G_{n+1}\cdot\mathfrak{diag}_{n+1}\otimes_{R G_{n+1}}R G_{n+1}\cdot e^\epsilon_{n}$
are both $R$-free with the same rank equal to that of the
bimodule $R G_{n+1}\cdot e^\epsilon_{n}$.
Since $\dot{\Phi}^\epsilon$ is an $R$-free map, we get that $\dot{\Phi}^\epsilon$ is an isomorphism.
\end{proof}

Since  $\dot{\Phi}^+$ (resp. $\dot{\Phi}^-,\dot{\Psi}^+ $ or $\dot{\Psi}^-$) is an isomorphisms.
we define its inverse  $(\dot{\Phi}^+)^{-1}=\mathord{
\begin{tikzpicture}[baseline = -.6mm]
	\draw[->,red] (0.28,-.3) to  (-0.28,.3);
	\draw[-,brown, thick, dashed] (-0.28,-.3) to (0.28,.3);
\end{tikzpicture}
}$ (resp.  $(\dot{\Phi}^-)^{-1}=\mathord{
\begin{tikzpicture}[baseline = -.6mm]
	\draw[->,blue] (0.28,-.3) to (-0.28,.3);
	\draw[-,brown, thick, dashed] (-0.28,-.3) to (0.28,.3);
\end{tikzpicture}
}\,,$ $(\dot{\Psi}^+)^{-1}=\mathord{
\begin{tikzpicture}[baseline = -.6mm]
	\draw[<-,red] (-0.28,-.3) to (0.28,.3);
	\draw[-,brown, thick, dashed] (0.28,-.3) to (-0.28,.3);
\end{tikzpicture}
}\,$ or $(\dot{\Psi}^-)^{-1}=\mathord{
\begin{tikzpicture}[baseline = -.6mm]
	\draw[<-,blue] (-0.28,-.3) to (0.28,.3);
	\draw[-,brown, thick, dashed] (0.28,-.3) to (-0.28,.3);
\end{tikzpicture}
}$).
\smallskip

 Also, we have the following analogous results of Proposition \ref{prop:spinor} for $\mathrm{Diag}$.
\begin{proposition}\label{prop:diag}
We have the following relations:
\begin{align}
\mathord{
\begin{tikzpicture}[baseline = -1mm]
	\draw[-,thin,red] (0.28,-.6) to[out=90,in=-90] (-0.28,0);
	\draw[->,thin,red] (-0.28,0) to[out=90,in=-90] (0.28,.6);
	\draw[-,thick,dashed,brown] (-0.28,-.6) to[out=90,in=-90] (0.28,0);
	\draw[-,thick,dashed,brown] (0.28,0) to[out=90,in=-90] (-0.28,.6);
\end{tikzpicture}
}\,=
\mathord{
\begin{tikzpicture}[baseline = -1mm]
	\draw[->,thin,red] (0.18,-.6) to (0.18,.6);
	\draw[-,thick,dashed,brown] (-0.18,-.6) to (-0.18,.6);
\end{tikzpicture}
}\:,\quad
\mathord{
\begin{tikzpicture}[baseline = -1mm]
	\draw[-,thick,dashed,brown] (0.28,0) to[out=90,in=-90] (-0.28,.6);
	\draw[->,thin,blue] (-0.28,0) to[out=90,in=-90] (0.28,.6);
	\draw[-,thick,dashed,brown] (-0.28,-.6) to[out=90,in=-90] (0.28,0);
	\draw[-,thin,blue] (0.28,-.6) to[out=90,in=-90] (-0.28,0);
\end{tikzpicture}
}\,=
\mathord{
\begin{tikzpicture}[baseline = -1mm]
	\draw[->,thin,blue] (0.18,-.6) to (0.18,.6);
	\draw[-,thick,dashed,brown] (-0.18,-.6) to (-0.18,.6);
\end{tikzpicture}
}\:,\quad
\mathord{
\begin{tikzpicture}[baseline = -1mm]
	\draw[<-,thin,red] (0.28,-.6) to[out=90,in=-90] (-0.28,0);
	\draw[-,thin,red] (-0.28,0) to[out=90,in=-90] (0.28,.6);
	\draw[-,thick,dashed,brown] (-0.28,-.6) to[out=90,in=-90] (0.28,0);
	\draw[-,thick,dashed,brown] (0.28,0) to[out=90,in=-90] (-0.28,.6);
\end{tikzpicture}
}\,=
\mathord{
\begin{tikzpicture}[baseline = -1mm]
	\draw[<-,thin,red] (0.18,-.6) to (0.18,.6);
	\draw[-,thick,dashed,brown] (-0.18,-.6) to (-0.18,.6);
\end{tikzpicture}
}\:,\quad
\mathord{
\begin{tikzpicture}[baseline = -1mm]
	\draw[-,thick,dashed,brown] (0.28,0) to[out=90,in=-90] (-0.28,.6);
	\draw[-,thin,blue] (-0.28,0) to[out=90,in=-90] (0.28,.6);
	\draw[-,thick,dashed,brown] (-0.28,-.6) to[out=90,in=-90] (0.28,0);
	\draw[<-,thin,blue] (0.28,-.6) to[out=90,in=-90] (-0.28,0);
\end{tikzpicture}
}\,=
\mathord{
\begin{tikzpicture}[baseline = -1mm]
	\draw[<-,thin,blue] (0.18,-.6) to (0.18,.6);
	\draw[-,thick,dashed,brown] (-0.18,-.6) to (-0.18,.6);
\end{tikzpicture}
}\:,
\end{align}
\begin{align}\label{mixedcross}
\mathord{
\begin{tikzpicture}[baseline = -.5mm]
	\draw[->,thin,red] (0.28,-.3) to (-0.28,.4);
      \node at (0.165,-0.15) {$\dott$};
	\draw[-,thick,dashed,brown] (-0.28,-.3) to (0.28,.4);
\end{tikzpicture}
}&=\mathord{
\begin{tikzpicture}[baseline = -.5mm]
	\draw[-,thick,dashed,brown] (-0.28,-.3) to (0.28,.4);
	\draw[->,thin,red] (0.28,-.3) to (-0.28,.4);
      \node at (-0.14,0.23) {$\dott$};
\end{tikzpicture}
},
&
\mathord{
\begin{tikzpicture}[baseline = -.5mm]
	\draw[->,thin,blue] (0.28,-.3) to (-0.28,.4);
      \node at (0.165,-0.15) {$\dott$};
	\draw[-,thick,dashed,brown] (-0.28,-.3) to (0.28,.4);
\end{tikzpicture}
}&=-\mathord{
\begin{tikzpicture}[baseline = -.5mm]
	\draw[-,thick,dashed,brown] (-0.28,-.3) to (0.28,.4);
	\draw[->,thin,blue] (0.28,-.3) to (-0.28,.4);
      \node at (-0.14,0.23) {$\dott$};
\end{tikzpicture}
}
\:,
\end{align}

\begin{align}
\mathord{
\begin{tikzpicture}[baseline = -1mm]
	\draw[->,thin,red] (0.45,-.6) to (-0.45,.6);
        \draw[-,thin,red] (0,-.6) to[out=90,in=-90] (-.45,0);
        \draw[-,line width=4pt,white] (-0.45,0) to[out=90,in=-90] (0,0.6);
        \draw[->,thin,red] (-0.45,0) to[out=90,in=-90] (0,0.6);
	\draw[-,thick,dashed,brown] (0.45,.6) to (-0.45,-.6);
\end{tikzpicture}
}
\,=
\mathord{
\begin{tikzpicture}[baseline = -1mm]
	\draw[->,thin,red] (0.45,-.6) to (-0.45,.6);
        \draw[-,line width=4pt,white] (0,-.6) to[out=90,in=-90] (.45,0);
        \draw[-,thin,red] (0,-.6) to[out=90,in=-90] (.45,0);
        \draw[->,thin,red] (0.45,0) to[out=90,in=-90] (0,0.6);
	\draw[-,thick,dashed,brown] (0.45,.6) to (-0.45,-.6);
\end{tikzpicture}
}\:,\quad
\mathord{
\begin{tikzpicture}[baseline = -1mm]
	\draw[->,thin,blue] (0.45,-.6) to (-0.45,.6);
        \draw[-,thin,blue] (0,-.6) to[out=90,in=-90] (-.45,0);
        \draw[-,line width=4pt,white] (-0.45,0) to[out=90,in=-90] (0,0.6);
        \draw[->,thin,blue] (-0.45,0) to[out=90,in=-90] (0,0.6);
	\draw[-,thick,dashed,brown] (0.45,.6) to (-0.45,-.6);
\end{tikzpicture}
}
\,=
\mathord{
\begin{tikzpicture}[baseline = -1mm]
	\draw[->,thin,blue] (0.45,-.6) to (-0.45,.6);
        \draw[-,line width=4pt,white] (0,-.6) to[out=90,in=-90] (.45,0);
        \draw[-,thin,blue] (0,-.6) to[out=90,in=-90] (.45,0);
        \draw[->,thin,blue] (0.45,0) to[out=90,in=-90] (0,0.6);
	\draw[-,thick,dashed,brown] (0.45,.6) to (-0.45,-.6);
\end{tikzpicture}
}\:,\quad
\mathord{
\begin{tikzpicture}[baseline = -1mm]
	\draw[-,thick,dashed,brown] (0.45,.6) to (-0.45,-.6);
	\draw[->,blue] (0.45,-.6) to (-0.45,.6);
        \draw[-,red] (0,-.6) to[out=90,in=-90] (-.45,0);
        \draw[->,red] (-0.45,0) to[out=90,in=-90] (0,0.6);
\end{tikzpicture}
}
\,=
\mathord{
\begin{tikzpicture}[baseline = -1mm]
	\draw[-,thick,dashed,brown] (0.45,.6) to (-0.45,-.6);
	\draw[->,blue] (0.45,-.6) to (-0.45,.6);
        \draw[-,red] (0,-.6) to[out=90,in=-90] (.45,0);
        \draw[->,red] (0.45,0) to[out=90,in=-90] (0,0.6);
\end{tikzpicture}
}\:,\quad \mathord{
\begin{tikzpicture}[baseline = -1mm]
	\draw[-,thick,dashed,brown] (0.45,.6) to (-0.45,-.6);
	\draw[->,red] (0.45,-.6) to (-0.45,.6);
        \draw[-,blue] (0,-.6) to[out=90,in=-90] (-.45,0);
        \draw[->,blue] (-0.45,0) to[out=90,in=-90] (0,0.6);
\end{tikzpicture}
}
\,=
\mathord{
\begin{tikzpicture}[baseline = -1mm]
	\draw[-,thick,dashed,brown] (0.45,.6) to (-0.45,-.6);
	\draw[->,red] (0.45,-.6) to (-0.45,.6);
        \draw[-,blue] (0,-.6) to[out=90,in=-90] (.45,0);
        \draw[->,blue] (0.45,0) to[out=90,in=-90] (0,0.6);
\end{tikzpicture}
}\:,
\end{align}

\begin{align}
\mathord{
\begin{tikzpicture}[baseline = 0]
\draw[-,dashed, thick,brown](.6,.4) to (.1,-.3);
	\draw[<-,red] (0.6,-0.3) to[out=140, in=0] (0.1,0.2);
	\draw[-,red] (0.1,0.2) to[out = -180, in = 90] (-0.2,-0.3);
\end{tikzpicture}
}\,=
\mathord{
\begin{tikzpicture}[baseline = 0]
\draw[-,dashed, thick,brown](-.5,.4) to (0,-.3);
	\draw[<-,red] (0.3,-0.3) to[out=90, in=0] (0,0.2);
	\draw[-,red] (0,0.2) to[out = -180, in = 40] (-0.5,-0.3);
\end{tikzpicture}
}\:,\quad
\mathord{
\begin{tikzpicture}[baseline = 0]
\draw[-,dashed, thick,brown](.6,.4) to (.1,-.3);
	\draw[<-,blue] (0.6,-0.3) to[out=140, in=0] (0.1,0.2);
	\draw[-,blue] (0.1,0.2) to[out = -180, in = 90] (-0.2,-0.3);
\end{tikzpicture}
}\,=
\mathord{
\begin{tikzpicture}[baseline = 0]
\draw[-,dashed, thick,brown](-.5,.4) to (0,-.3);
	\draw[<-,blue] (0.3,-0.3) to[out=90, in=0] (0,0.2);
	\draw[-,blue] (0,0.2) to[out = -180, in = 40] (-0.5,-0.3);
\end{tikzpicture}
}\:,\quad
\mathord{
\begin{tikzpicture}[baseline = 0]
\draw[-,dashed,thick,brown](-.5,-.3) to (0,.4);
	\draw[<-,red] (0.3,0.4) to[out=-90, in=0] (0,-0.1);
	\draw[-,red] (0,-0.1) to[out = 180, in = -40] (-0.5,0.4);
\end{tikzpicture}
}\,=
\mathord{
\begin{tikzpicture}[baseline = 0]
\draw[-,dashed,thick,brown](.6,-.3) to (.1,.4);
	\draw[<-,red] (0.6,0.4) to[out=-140, in=0] (0.1,-0.1);
	\draw[-,red] (0.1,-0.1) to[out = 180, in = -90] (-0.2,0.4);
\end{tikzpicture}
}
\:,\quad
\mathord{
\begin{tikzpicture}[baseline = 0]
\draw[-,dashed,thick,brown](-.5,-.3) to (0,.4);
	\draw[<-,blue] (0.3,0.4) to[out=-90, in=0] (0,-0.1);
	\draw[-,blue] (0,-0.1) to[out = 180, in = -40] (-0.5,0.4);
\end{tikzpicture}
}\,=
\mathord{
\begin{tikzpicture}[baseline = 0]
\draw[-,dashed,thick,brown](.6,-.3) to (.1,.4);
	\draw[<-,blue] (0.6,0.4) to[out=-140, in=0] (0.1,-0.1);
	\draw[-,blue] (0.1,-0.1) to[out = 180, in = -90] (-0.2,0.4);
\end{tikzpicture}
}
\:,
\end{align}

\begin{align}
\mathord{
\begin{tikzpicture}[baseline = 0]
\draw[-,dashed, thick,brown](.6,.4) to (.1,-.3);
	\draw[-,red] (0.6,-0.3) to[out=140, in=0] (0.1,0.2);
	\draw[->,red] (0.1,0.2) to[out = -180, in = 90] (-0.2,-0.3);
\end{tikzpicture}
}\,=
\mathord{
\begin{tikzpicture}[baseline = 0]
\draw[-,dashed, thick,brown](-.5,.4) to (0,-.3);
	\draw[-,red] (0.3,-0.3) to[out=90, in=0] (0,0.2);
	\draw[->,red] (0,0.2) to[out = -180, in = 40] (-0.5,-0.3);
\end{tikzpicture}
}\:,\quad
\mathord{
\begin{tikzpicture}[baseline = 0]
\draw[-,dashed, thick,brown](.6,.4) to (.1,-.3);
	\draw[-,blue] (0.6,-0.3) to[out=140, in=0] (0.1,0.2);
	\draw[->,blue] (0.1,0.2) to[out = -180, in = 90] (-0.2,-0.3);
\end{tikzpicture}
}\,=
\mathord{
\begin{tikzpicture}[baseline = 0]
\draw[-,dashed, thick,brown](-.5,.4) to (0,-.3);
	\draw[-,blue] (0.3,-0.3) to[out=90, in=0] (0,0.2);
	\draw[->,blue] (0,0.2) to[out = -180, in = 40] (-0.5,-0.3);
\end{tikzpicture}
}\:,\quad
\mathord{
\begin{tikzpicture}[baseline = 0]
\draw[-,dashed,thick,brown](-.5,-.3) to (0,.4);
	\draw[-,red] (0.3,0.4) to[out=-90, in=0] (0,-0.1);
	\draw[->,red] (0,-0.1) to[out = 180, in = -40] (-0.5,0.4);
\end{tikzpicture}
}\,=
\mathord{
\begin{tikzpicture}[baseline = 0]
\draw[-,dashed,thick,brown](.6,-.3) to (.1,.4);
	\draw[-,red] (0.6,0.4) to[out=-140, in=0] (0.1,-0.1);
	\draw[->,red] (0.1,-0.1) to[out = 180, in = -90] (-0.2,0.4);
\end{tikzpicture}
}
\:,\quad
\mathord{
\begin{tikzpicture}[baseline = 0]
\draw[-,dashed,thick,brown](-.5,-.3) to (0,.4);
	\draw[-,blue] (0.3,0.4) to[out=-90, in=0] (0,-0.1);
	\draw[->,blue] (0,-0.1) to[out = 180, in = -40] (-0.5,0.4);
\end{tikzpicture}
}\,=
\mathord{
\begin{tikzpicture}[baseline = 0]
\draw[-,dashed,thick,brown](.6,-.3) to (.1,.4);
	\draw[-,blue] (0.6,0.4) to[out=-140, in=0] (0.1,-0.1);
	\draw[->,blue] (0.1,-0.1) to[out = 180, in = -90] (-0.2,0.4);
\end{tikzpicture}
}
\:.
\end{align}
\end{proposition}

\begin{proof}
We only prove (\ref{mixedcross}). The others are similar.
Note that
$$(1_{\rm{Diag}}\otimes X^\epsilon)\circ\dot{\Phi}^\epsilon:F^\epsilon_{n,n+1}\circ \mathrm{Diag}_n\to \mathrm{Diag}_{n+1}\circ F^\epsilon_{n,n+1}$$
 is represented by the composition of the $(R G_{n+1},R G_n)$-bimodule maps
 $$
 \begin{array}{ccc}
R G_{n+1}\cdot e^\epsilon_{n}\otimes_{R G_n}R G_n\cdot\mathfrak{diag}_n
&\rightarrow & R G_{n+1}\cdot\mathfrak{diag}_{n+1}\otimes_{R G_{n+1}}R G_{n+1}\cdot e^\epsilon_{n}
\\
ge^\epsilon_{n}\otimes h\cdot\mathfrak{diag}_n&\mapsto& g\cdot\mathfrak{diag}_{n+1}\otimes {^{\tau_n}h} e^\epsilon_{n},
 \end{array}
$$
and
\begin{align*}
 R G_{n+1}\cdot\mathfrak{diag}_{n+1}\otimes_{R G_{n+1}}R G_{n+1}\cdot e^\epsilon_{n}
&\rightarrow  R G_{n+1}\cdot\mathfrak{diag}_{n+1}\otimes_{R G_{n+1}}R G_{n+1}\cdot e^\epsilon_{n}\\ g\cdot\mathfrak{diag}_{n+1}\otimes {^{\tau_n}h} e^\epsilon_{n}
&\mapsto g\cdot\mathfrak{diag}_{n+1}\otimes {^{\tau_n}h} X^\epsilon_{n+1},
 \end{align*}
for all $g\in  G_{n+1}$ and $h\in  G_n.$
Also, we note that
$$\dot{\Phi}^\epsilon\circ (X^\epsilon\otimes 1_{\rm {Diag}}):F^\epsilon_{n,n+1} \circ \mathrm{Diag
}_n\to \mathrm{Diag}_{n+1}\circ F^\epsilon_{n,n+1}$$
 is represented by the composition of the $(R G_{n+1},R G_n)$-bimodule maps
$$
 \begin{array}{ccc}
R G_{n+1}\cdot e^\epsilon_{n}\otimes_{R G_n}R G_n\cdot\mathfrak{diag}_n
&\to&
R G_{n+1}\cdot e^\epsilon_{n}\otimes_{R G_n}R G_n\cdot\mathfrak{diag}_n
\\
ge^\epsilon_{n}\otimes h\cdot\mathfrak{diag}_n &\mapsto & gX^\epsilon_{n+1}\otimes h\cdot\mathfrak{diag}_n,
 \end{array}
$$
and
$$
 \begin{array}{ccc}
R G_{n+1}\cdot e^\epsilon_{n}\otimes_{R G_n}R G_n\cdot\mathfrak{diag}_n
&\to &
 R G_{n+1}\cdot\mathfrak{diag}_{n+1}\otimes_{R G_{n+1}}R G_{n+1}\cdot e^\epsilon_{n}\\
 gX^\epsilon_{n+1}\otimes h\cdot\mathfrak{diag}_n
&\mapsto & gX^\epsilon_{n+1}\cdot\mathfrak{diag}_{n+1}\otimes {^{\tau_n}h} e^\epsilon_{n}.
 \end{array}
$$
\smallskip
 
We have
 \begin{align}\label{compute}
\diag_{n+1}(e^\epsilon_{n})=\frac{1}{|U_{n}|}\sum\limits_{u\in U_{n}}\zeta_\epsilon(u)\cdot {^{\tau_{n+1}}u}=\frac{1}{|U_{n}|}\sum\limits_{u\in U_{n}}\zeta_\epsilon({^{\tau_{n+1}}u})\cdot{^{\tau_{n+1}}u}=e^\epsilon_{n}.
\end{align}
Hence
$e^\epsilon_{n}\cdot\mathfrak{diag}_{n+1}=\mathfrak{diag}_{n+1}\cdot e^\epsilon_{n}$.
\smallskip

On the other hand,
 $$\dot{x}_{n+1}\cdot\mathfrak{diag}_{n+1}
=\mathfrak{diag}_{n+1}\cdot {^{\tau_{n+1}}\dot{x}_{n+1}}=\mathfrak{diag}_{n+1}\cdot \dot{x}_{n+1}t,$$
where $t=\diag(\delta,\id_{G_n},\delta^{-1}).$
\smallskip

Since
\begin{align*}te^\epsilon_{n}=
\begin{cases}1(\delta)\cdot e^\epsilon_{n}=e^\epsilon_{n}&\text{if $\epsilon=+$};\\
\zeta(\delta)\cdot e^\epsilon_{n}=-e^\epsilon_{n}&\text{if $\epsilon=-$},
\end{cases}
\end{align*}
it follows that
$$gX^\epsilon_{n+1}\cdot\mathfrak{diag}_{n+1}\otimes {^{\tau_{n}}h} e_{n+1,n}=g\cdot\mathfrak{diag}_{n+1}\otimes {^{\tau_{n}}h}\,\, (\epsilon\cdot X^\epsilon_{n+1}).$$
Hence
 $(1_{\rm {Diag}}\otimes X^\epsilon)\circ\dot{\Phi}^\epsilon=\dot{\Phi}^\epsilon\circ (\epsilon\cdot X^\epsilon \otimes  1_{\rm {Diag}})$, which proves (\ref{mixedcross}).
\end{proof}

 Then we have the following direct corollary.
\begin{corollary}
Let $I_+$ and $ \scrI_-$ be the sets of the generalized eigenvalues of $X^+$ on $F^+$ and $X^-$ on $F^-$, respectively.
Then $\scrI_-=-\scrI_-$ and we have  the following isomorphisms for
all $i\in \scrI_+$ and $j\in \scrI_-$:
$$\mathrm{Diag}\circ F^+_i\circ \mathrm{Diag}\cong F^+_i,\quad\,\,
\quad \mathrm{Diag}\circ F^-_j\circ \mathrm{Diag}\cong F^-_{-j},$$
$$\mathrm{Diag}\circ E^+_i\circ \mathrm{Diag}\cong E^+_i,\quad\,\,
\quad \mathrm{Diag}\circ E^-_j\circ \mathrm{Diag}\cong E^-_{-j}.$$

 \end{corollary}
Now we have the following extra symmetries in Table \ref{diag:extrasymm}, which is a summary of this section.
\begin{center}
\begin{table}[h]\caption{Extra symmetries}\label{diag:extrasymm}
\small{\begin{tabular}{ccc}
  \hline
  $\O_{2\bullet+1}(q)$&$\Sp_{2\bullet}(q)$ &$\O^{\pm}_{2\bullet}(q)$\\\hline
  $\mathfrak{U}(\fraks\frakl'_{I_+}\oplus\fraks\frakl'_{I_-})\rtimes ( \bbZ/2\bbZ\times \bbZ/2\bbZ)$&
$\mathfrak{U}(\fraks\frakl'_{I_+}\oplus\fraks\frakl'_{I_-})\rtimes \bbZ/2\bbZ$&
$\mathfrak{U}(\fraks\frakl'_{I_+}\oplus\fraks\frakl'_{I_-})\rtimes ( \bbZ/2\bbZ\wr \bbZ/2\bbZ)$\\\hline
  extra symmetries:&  extra symmetries: &extra symmetries:\\
  $\mathrm{Spin},\,\mathrm{Det}$ &$\mathrm{Diag}$&$\mathrm{Spin},\,\mathrm{Det},\,\mathrm{Diag}$
  \\\hline
\end{tabular}}
\end{table}
\end{center}

\section{Complete invariants of $KG_\bullet\mod$}\label{Chap:completeinv}

In this section, we only consider the representation theory in characteristic 0. So we always let $R=K$ be an algebraically closed field of characteristic 0. Let $KG_\bullet\mod:=\bigoplus_{n\in \bbN^*}K G _n\mod.$
Recall that in \S \ref{sec:doublequantumtoKac}, we define two colored central elements $\bbO^+(u):=\red\clock(u)$ and
$\bbO^-(v):=\blue\clock(v).$ For each irreducible module
$L\in KG_\bullet\mod$, we consider the action of $\bbO^+(u)$ and $\bbO^-(v)$ on $\unit (L)=L$. Then $\bbO^+(u)(L)\in K((u^{-1}))$ and $\bbO^-(v)(L)\in K((v^{-1})).$ So we get the following functions, and  we call them \emph{colored weight functions}
$$\bbO^+(u)(-):\, \Irr(KG_\bullet\mod)\to K((u^{-1})),\quad \bbO^-(v)(-):\, \Irr(KG_\bullet\mod)\to K((v^{-1})).$$

 We will show that $\bbO^+(u)(-)$, $\bbO^-(v)(-)$  and
 uniform projection
   can distinguish all irreducible modules.
   In fact, when focusing solely on quadratic unipotent modules,
   $\bbO^+(u)(-)$ and $\bbO^-(v)(-)$ serve as complete invariants for these modules.

\smallskip

\subsection{Lusztig parametrization for finite classical groups}\label{sec:Lusztigparametrization}\hfill\\

In this section, we will recall some basic facts about character theory of finite classical groups.
\smallskip

\subsubsection{Partitions and Lusztig symbols}\label{sub:symbol}
A \emph{partition} of $n$ is a non-increasing sequence of non-negative integers
$\lambda = (\lambda_1 \geqslant \lambda_2 \geqslant \cdots)$ with $\sum_i \lambda_i=n$,
to which one
associates the so-called \emph{Young diagram}
$Y(\lambda)=\{(x,y)\in\bbZ_{>0}\times\bbZ_{>0}\,\mid\,y\leqslant \lambda_x\}.$
We write $\scrP=\bigsqcup_n\scrP_n$ for the set
of all partitions, where $\scrP_n$ is the set of partitions of $n$.
For $\lambda\in \scrP$, we denote by $|\lambda|$ the \emph{weight}
of $\lambda$ and by $\lambda^*$ the partition  conjugate to $\lambda$.
An \emph{$l$-partition} of $n$ is an $l$-tuple $\tuple\lambda=(\lambda^{1},\ldots,\lambda^{l})$ of partitions whose weights add up to $n$,
and its Young diagram is the set $Y(\tuple\lambda)=\bigsqcup_{p=1}^lY(\lambda^p)\times\{p\}.$
The integer $|\tuple\lambda|=\sum_p|\lambda^{p}|$ is called the weight of the $l$-partition. We write $\scrP^l=\bigsqcup_n\scrP^l_n$ for the set of all $l$-partitions, where
$\scrP^l_n$ is the set of $l$-partitions of $n$.
We will call the \emph{charged $l$-partition} a pair in $\scrP^l\times \bbZ^l.$ In particular, if $l=1$, we will call them \emph{charged partition}.
\smallskip

A \emph{beta-set} of a partition $\lambda=(\alpha_1, \alpha_2, \cdots , \alpha_t)$ is a finite set of non-negative integers
$\{\beta_1, \beta_2, \cdots , \beta_s\}$ where $\beta_1 < \beta_2 < \cdots < \beta_s$ and $\alpha_i=\beta_{s-i+1}-s+i$ for $1 \leq i \leq s$ and $\alpha_i$
is considered to be zero for $i > t.$
Now if $d$ is a non-negative integer then the $d$-shift of a beta-set $\{\beta_1, \beta_2, \cdots, \beta_s\}$ is
$\{0, 1, \cdots, d-1\}\cup \{\beta_1+d, \beta_2+d, \cdots , \beta_s+d\}.$
Two beta-sets are said to be equivalent if one is the $d$-shift of the other for some non-
negative integer $d$ and two beta-sets give rise to the same partition if and only if they
are equivalent. There is a bijection between partitions and equivalence
classes of 	beta-sets, mapping the partition $\lambda=(\alpha_1, \alpha_2, \cdots , \alpha_t)$ to the beta-set containing $\beta=\{\alpha_t,\alpha_{t-1}+1,\cdots,\alpha_{1}+t-1\}.$ For a positive integer $d$,  an \emph{$d$-hook} of $\beta$ is a pair $(x,x+e)$ such
that $x\in \beta$ and $x+d\not\in \beta$.

\smallskip

An \emph{(ordered) Lusztig symbol} or \emph{symbol} is an ordered pair of beta-sets  $$\Lambda=\begin{pmatrix}
X\\
Y
\end{pmatrix}
=
\begin{pmatrix}
a_1,a_2,\cdots,a_{m_1}\\
b_1,b_2,\cdots,b_{m_2}
\end{pmatrix}.$$ A symbol $\lambda$ is said to be
degenerate if $X = Y$. Two symbols are said to be equivalent if one of them can be
obtained from the other by simultaneous $d$-shifts on both parts of the symbol for some
non-negative integer $d$. Let $\calS$ denote the set of equivalent
classes of symbols.  Let
\begin{align*}
&\mathrm{rank}(\Lambda)=\sum_{a_i\in X}a_i+\sum_{b_i\in Y}b_i-\left\lfloor\left(\frac{|X|+|Y|-1}{2}\right)^2\right\rfloor, \\
&\mathrm{def}(\Lambda)=|X|-|Y|,
\end{align*}
which are called \emph{rank} and \emph{defect} of $\Lambda,$ respectively.  A symbol $\Lambda$ is called \emph{cuspidal} if
$
{\rm rank}(\Lambda)=\left\lfloor\left(\frac{{\rm def}
(\Lambda)}{2}\right)^2\right\rfloor.
$
For a symbol $\Lambda=\begin{pmatrix}
X\\
Y
\end{pmatrix}$, let
$\Lambda^t=\begin{pmatrix}
Y\\
X
\end{pmatrix}$. A \emph{$d$-hook} of $\Lambda$ is a pair of integers $(x,x+d)$ which is either
a $d$-hook of $X$ or a $d$-hook of $Y$. The  symbol obtained by deleting $x$ from $X$ (resp. $Y$)
and replacing it by $x+d$ is said to be gotten from $\Lambda$ by \emph{adding the $d$-hook $(x,x+d)$}.
\smallskip

Modified from Lusztig, we define
\begin{align}\label{S}
\calS_{\O^+_{2n}} &=\{\,\Lambda\in\calS\mid {\rm rank}(\Lambda)=n,\ {\rm def}(\Lambda)\equiv 0 \,(\textrm{mod }4)\,\}, \\
\calS_{\SO_{2n+1}}=\calS_{\Sp_{2n}} &=\{\,\Lambda\in\calS\mid {\rm rank}(\Lambda)=n,\ {\rm def}(\Lambda)\equiv 1\,(\textrm{mod }4) \,\}, \\
\calS_{\O^-_{2n}} &=\{\,\Lambda\in\calS\mid {\rm rank}(\Lambda)=n,\ {\rm def}(\Lambda)\equiv 2\,(\textrm{mod }4)\,\}.
\end{align}
Note that $\Lambda\in\calS_{\O^\epsilon_{2n}}$ if and only if $\Lambda^t\in\calS_{\O^\epsilon_{2n}}$
where $\epsilon=+$ or $-$.
\smallskip

It is not difficult to see that a symbol is cuspidal if and only if it is similar to a symbol of the forms
$\binom{k,k-1,\ldots,0}{-}$ or $\binom{-}{k,k-1,\ldots,0}$ for some non-negative integer $k$.
Note that $\binom{A}{-}$ means that the second row of the symbol is the empty set.
\smallskip

The map $\Upsilon$ from symbols to bi-partitions is given by
\begin{equation}\label{bipartitionmap}
\Upsilon\colon\binom{a_1,a_2,\ldots,a_{m_1}}{b_1,b_2,\ldots,b_{m_2}}\mapsto
(\mu_1,\mu_2).
\end{equation}
where $\mu_1=(a_1-(m_1-1),a_2-(m_1-2),\ldots,a_{m_1-1}-1,a_{m_1})$ and $\mu_2=(b_1-(m_2-1),b_2-(m_2-2),
\ldots,b_{m_2}-1,b_{m_2}).$
 Let $\calS_{\delta, m}$ be all symbols $\Lambda$ such that $\mathrm{rank}(\Lambda)=m$ and $\mathrm{def}(\Lambda)=\delta,$ we can check that  $\Upsilon$ gives a bijection between $\calS_{\delta, m}$ and bi-partitions of $m$
\begin{equation}
\calS_{\delta, m}\rightarrow\scrP^2_m.
\end{equation}

\subsubsection{Deligne--Lusztig characters}\label{sub:uniformprojection}As in \cite[\S 5]{DL}, we fix once and for all a group isomorphism $\iota: \overline{\bbF}_q^{\times}\to (\bbQ/\bbZ)_{p'}$ and an injective group homomorphism $\iota': (\bbQ/\bbZ)_{p'}\to \mu(K^\times),$ where $\mu(K^\times)$ is the multiplicative group of roots of unity in $K^\times$.
First, suppose $\bfG$ is a connected classical group $\bfG\bfL_n,\bfG\bfU_n,\bfSO_{2n+1}$, $\bfSp_{2n}$, or $\bfSO^\epsilon_{2n}$ with $\epsilon=\pm$ and $F$ is a Frobenius map of $\bfG$. 
The notation $\bfSO^-_{2n}$ indicates that, in this context, $\bfG = \bfSO_{2n}$, while the Frobenius map $F$ corresponds to $F'$, as defined in Section \ref{sec:construction} (3b). Similarly, the notation $\bfO^-_{2n}$ (used later) adopts an analogous interpretation.
\smallskip

Assume that $(\bfG^*,F^*)$ is in duality with $(\bfG,F)$. Let $G=\bfG^F$ and $G^*=\bfG^{*F^*}.$ For an $F$-stable maximal torus $\bfT$ in $\bfG$ and $\theta\in \Irr(\bfT^F)$, the \emph{Deligne-Lusztig virtual characters} $R^\bfG_{\bfT}(s)$ is defined in \cite{DL}. There is a cananical bijection between $\bfG^F$-orbits of pair $(\bfT,\theta)$
 (where $\bfT \subset \bfG$ is an $F$-stable maximal torus and $\theta\in \Irr(\bfT^F)$) and $(\bfG^*)^{F^*}$-orbits of pairs $(\bfT^*,s)$ (where $\bfT^* \subset \bfG^*$ is an $F^*$-stable maximal torus and $s\in (\bfT^*)^{F^*}$). We remark that this bijection depends on the choices of $\iota$ and $\iota'.$ If the pairs $(\bfT,\theta)$ and $(\bfT^*,s)$ correspond to each other in this way, then we write $R^{\bfG}_{\bfT^*,s} = R^{\bfG}_{\bfT}(\theta).$
\smallskip

For the disconnected group $\bfG=\bfO_{2n+1}$ or $\bfO^\epsilon_{2n}$, let $\bfG^*=\bfSp_{2n}\times \bbZ/2\bbZ$ or $\bfO^\epsilon_{2n}$, respectively. In particular, ${\bfG^*}^0=\bfSp_{2n}$ or $\bfSO^\epsilon_{2n}$, respectively. We define 
$R^{\bfO_{2n+1}}_{\bfT}(\theta)$ and $R^{\bfO^\epsilon_{2n}}_{\bfT}(\theta)$  by\begin{equation}
R^{\bfO_{2n+1}}_{\bfT}(\theta)=\Ind_{\SO_{2n+1}}^{\O_{2n+1
}}R^{\bfSO_{2n+1}}_{\bfT}(\theta),
\quad
R^{\bfO^\epsilon_{2n}}_{\bfT}(\theta)=\Ind_{\SO^\epsilon_{2n}}
^{\O^\epsilon_{2n}}R^{\bfSO^\epsilon_{2n}}_{\bfT}(\theta)
,\end{equation}
respectively. Similarly, if the pairs $(\bfT,\theta)$ and $(\bfT^*,s)$ correspond to each other, then we write $R^{\bfG}_{\bfT^*,s} = R^{\bfG}_{\bfT}(\theta).$
\smallskip

We shall speak below of ``complex conjugation" in the field $K$, denoted by
$a\mapsto \overline{a}$. This is justified by noting that $K$ is abstractly isomorphic to $\bbC.$
Let $\calV(G)$ be the space of class function on $G$,
and let $\calV(G)^\sharp$ denote the subspace spanned by Deligne-Lusztig virtual characters.
Note that $\calV(G)$ is an inner product space with the inner product $\langle,\rangle_G$ given by
$$
\langle f,g\rangle_G=\frac{1}{|G|}\sum_{x\in G}f(x)\overline{g(x)}
$$
for $f,g\in\calV(G)$.
For $f\in\calV(G)$ the orthogonal projection of $f$ onto $\calV(G)^\sharp$
is denoted by $f^\sharp$ and called the \emph{uniform projection}.
A class function $f\in\calV(G)$ is called \emph{uniform} if $f^\sharp=f$.
\smallskip

For a semisimple element $s\in ({\bfG^*}^0)^{F^*}$, define
$$\calE(G,s):=\{\chi\in \Irr(G)| \langle\chi, R^{\bfG}_{\bfT^*,s}\rangle\neq 0\,\,\text{ for some $F^*$-stable $\bfT^*$ contains $s$}\}.$$
The set $\calE(G,s)$ is called a \emph{Lusztig series}, and it is known that $\Irr(G)$ is partitioned into
Lusztig series indexed by the conjugacy classes $(s)$ of rational semisimple elements $s$
in ${\bfG^*}^0$, i.e., 
$$\mathrm{Irr}(G)=\bigcup_{(s)} {\mathcal E}(G,s)$$
where $(s)$ runs over the  set of $G^*$-conjugacy classes of semisimple elements of
$({\bfG^*}^0)^{F^*}$.
The {\em unipotent characters} of $G$ are those in $\mathcal{E}(G,1)$.

\begin{proposition}[Lusztig]\label{prop:jordandec}
	
	Let $\bfG$ be one of classical groups $\bfG\bfL_n,\bfG\bfU_n,\bfSO_{2n+1}$, $\bfSp_{2n}$,  $\bfSO^\epsilon_{2n}$ or $\bfO^{\epsilon}_{2n}.$
There  is a bijection
 $\calJ_s:\mathcal{E}(G,s)\to \mathcal{E}(C_{G^*}(s)^*,1)$ satisfying the condition 
 $$\langle \rho,\epsilon_{\bfG}R^{\bfG}_{\bfT^*,s}\rangle_G=\langle \calJ_s(\rho),\epsilon_{C_{\bfG*}(s)}R^{C_{\bfG^*}(s)}_{\bfT^*}(1_{T^*})\rangle_{C_{G^*}(s)}$$
for any $F^*$ -stable maximal torus $T^*$ containing $s$, where $\epsilon_{\bfG}=(-1)^r$ with $r$ is the semi-simple $\bbF_q$-rank of $\bfG^0.$ 
Moreover, we have $$\dim(\rho)=\frac{|G|_{p'}}{|C_{G^*}(s)|_{p'}}\dim(\calJ_s(\rho)),$$
where $|G|_{p'}$ denotes greatest factor of $|G|$ not divided by $p$.
 \end{proposition}
 When $\bfG$ is connected, the above proposition can be found in \cite[Theorem 11.5.1]{DM};
 when $\bfG=\bfO^{\pm}_{2n}$, the proposition is  \cite[Proposition 1.7]{AMR}, see also \cite[\S 4]{Wa} and \cite[\S 7.5]{LS-depth0}. Such a bijection $\calJ_s$ is called a \emph{Lusztig's Jordan decomposition}. A character $\chi\in \calE(G,s)$ is called \emph{qudratic unipotent}, if $s^2=1.$
\smallskip

\subsubsection{Unipotent characters}\label{sub:unipotent}
We now review some results on the classification of the irreducible unipotent representations by Lusztig.
\smallskip

If $G$ is a general linear group $\GL_n(q)$ or a unitary group $\GU_n(q)$,
then every irreducible character is uniform (i.e., $\rho=\rho^\sharp$).
For $\lambda\in\scrP_n$, we define
\[
R_\lambda^G=\frac{1}{|S_n|}\sum_{w\in S_n}\varphi_\lambda(w)R^\bfG_{\bfT_w}(1)
\]
where $\varphi_\lambda$ denotes the irreducible character of symmetric group $S_n$ corresponding to $\lambda$ and $\bfT_w$ is the maximal torus of type $w$, see \cite[\S 15.4]{DM}.
It is known that $R_\lambda^G$ is an irreducible unipotent character (up to a sign $\epsilon_\lambda$) of $G$.
Let $\calS_{\GL_n}=\calS_{\GU_n}=\scrP_n$.
Then the map
\begin{align}\label{GLnGUn}
\calL_1\colon\calS_G\rightarrow\calE(G,1)\quad \text{ given by }
\lambda\mapsto\rho_\lambda:=\epsilon_\lambda R_\lambda^G
\end{align}
is a bijection.
\smallskip

However, if $G$ is a symplectic group or orthogonal group $\SO_{2n+1}$, $\Sp_{2n}$ or $\O^{\epsilon}_{2n}$, we have $\calV(G)\neq \calV(G)^\sharp$.
Lusztig constructs
a bijective map $$\calL_1\colon\calS_G\rightarrow\calE(G,1)$$ denoted by
$\Lambda\mapsto\rho_\Lambda$ satisfying some good properties, see \cite[Proposition 3.12]{P4}.
Such a map $\calL_1$ is called a \emph{Lusztig parametrization} of unipotent characters of $G$.
For example, under this map,  cuspidal symbols map to cuspidal modules. We always write $\bbN=\bbZ_{\geqs 0}.$
 Now we give a parametrization for unipotent cuspidal modules for finite classical groups. 
 \begin{itemize}

 \item For $G=\Sp_{2n}(q)$ or $\SO_{2n+1}(q)$, if $n=t^2+t$ for some $t\in \bbN $, then $G$ has exactly one cuspidal module; otherwise, it has none.
     We label it  by $\rho_{t}=\rho^{\Sp}_{t}$ or $\rho^{\SO}_{t}$ and the corresponding cuspidal Lusztig symbol $\Lambda_t$ with defect $(-1)^t(2t+1)$.
 \item For $G=\O_{2n+1}(q),$ if $n=t^2+t$ for some $t\in \bbN $, then $G$ has exactly two cuspidal modules, otherwise, it has none. We label them  by $\rho_{ t, \epsilon}=\rho^{\O_{\mathrm{odd}}}_{ t, \epsilon}$ for $t\in \bbN, \epsilon\in \{\pm1\}$ such that $\rho^{\O_{\mathrm{odd}}}_{t,\epsilon}\cdot \sgn=\rho^{\O_{\mathrm{odd}}}_{t,-\epsilon}$. Let the corresponding cuspidal Lusztig symbol of $\rho_{t,\epsilon}$ be $\Lambda_t\times \epsilon$ such that   $\mathrm{def}(\Lambda_t)=(-1)^t(2t+1)$.
 \item For $G=\O^\epsilon_{2n}(q),$ if $n=t^2\neq 0$ and $\epsilon=(-1)^t$ for some integer $t$, then $G$ has exactly two unipotent cuspidal modules; if $n=0$, then $G=\O^+_0(q)=\{1\}$ has exactly one unipotent cuspidal module; otherwise, it has none.
     We label them  by $\rho_{\pm t}=\rho^{\O_{\mathrm{even}}}_{\pm t}$ for $\pm t\in\bbZ$ such that $\rho^{\O_{\mathrm{even}}}_{t}\cdot \sgn=\rho^{\O_{\mathrm{even}}}_{-t}$. Let the corresponding cuspidal Lusztig symbol of $\rho_{t}$ be $\Lambda_t$ such that   $\mathrm{def}(\Lambda_t)=2t$.
 \end{itemize}
\begin{remark}
For unipotent characters, the notation ``cuspidal'' coincides with  ``$F^+$-cuspidal''. So the above is also a parametrization for unipotent $F^+$-cuspidal modules.
\end{remark}
%To describe the unipotent Harish-Chandra series in classical groups, we need the  bijection between  $\Irr(W_n)$ and Harish-Chandra series above $\rho_{\Lambda}$.
\subsubsection{Unipotent Harish-Chandra series in finite classical groups}
 Let $\rho=\rho_{t}$ (or $\rho_{t,\epsilon}$ in $\O_{2r+1}(q)$ case) be a unipotent cuspidal character of $G_{r}$, then by the above classification we know that $r=t^2+t$ if  $G=\SO_{2r+1}(q)$, $\Sp_{2r}(q)$ or $\O_{2r+1}(q)$ and $r=t^2$ if  $G=\O^\pm_{2r}(q)$ for some $t\in \bbN $ or $t\in\bbZ$.
 \smallskip

  By \cite{L77}, the Harish-Chandra series of $G_n$ with $n=r+m$ above $\rho$ consists of the unipotent characters of $G$ labelled by symbols in $\calS_{t,m}.$ It is well-known that the irreducible characters $\Irr(W_m)$ of Weyl group $W_m$ of type $B_m$ is parametrized by bi-partitions $\scrP^2_m.$
 Then the map (see (\ref{bipartitionmap}))
$$\Upsilon\:: \calS_{t,m}\rightarrow\scrP^2_m $$
define a natural bijection of Harish-Chandra series above $\rho$ and $\Irr(W_{m}).$

\subsubsection{Semisimple conjugacy classes of finite classical groups}\label{sub:conjugacyclass}

We follow mainly the notation from \cite{FS89}.
Let $V$ be a finite-dimensional symplectic or orthogonal space over the field $\bbF_q$.
We denote by $G(V)$ the group of isometries of $V$,
$G_0(V)$ the subgroups of $G(V)$ of determinant $1$, and
$\eta(V)=\pm1$ the type of $V$ if $V$ is orthogonal, i.e., $\eta(V)=\epsilon$ if $G(V)=\O^\epsilon_{2n}(q).$
For simplicity, we set $\eta(V) = 1$ if $V$ is symplectic.
We identify $1$, $-1$ with $+$, $-$ respectively when considering the type of spaces and groups.
\smallskip

We denote by  $\Irr(\bbF_{q}[x])$ the set of all monic irreducible polynomials over the field $\bbF_{q}$.
For each $\Gamma$ in $\Irr(\bbF_{q}[x])\setminus \{x\}$, we define $\Gamma^*$ be the polynomial in
$\Irr(\bbF_{q}[x])$ whose roots are the inverses of the roots of $\Gamma$.
Now, we denote by
\begin{align*}
\calF_{0}&=\left\{ ~x-1,x+1 ~\right\},\\
\calF_{1}&=\left\{ ~\Gamma\in\Irr(\bbF_{q}[x])\mid \Gamma\notin \calF_0, \Gamma\neq x,\Gamma=\Gamma^* ~\right\},\\
\calF_{2}&=\left\{~ \Gamma\Gamma^* ~|~ \Gamma\in\Irr(\bbF_{q}[x])\setminus \calF_0, \Gamma\neq x,\Gamma\ne\Gamma^* ~\right\}.
\end{align*}
Let $\calF=\calF_0\cup\calF_1\cup\calF_2$.
Given $\Gamma\in\calF$, denote by $d_\Gamma$ its degree and by $\delta_\Gamma$ its \emph{reduced degree} defined by
$$\delta_\Gamma=
\left\{ \begin{array}{ll} d_\Gamma & \text{if}\ \Gamma\in\calF_0; \\
\frac{1}{2}d_\Gamma & \text{if}\ \Gamma\in\calF_1\cup \calF_2 . \end{array} \right.$$
Since the polynomials in $\calF_1\cup \calF_2$ have even degree, $\delta_\Gamma$ is an integer.
In addition, we mention a sign $\varepsilon_\Gamma$ for $\Gamma\in\calF_1\cup \calF_2$ defined by
$$\varepsilon_\Gamma=
\left\{ \begin{array}{ll} -1 & \text{if}\ \Gamma\in\calF_1; \\
1 & \text{if}\ \Gamma\in\calF_2 . \end{array} \right.$$

Let $V$ be a vector space over $\bbF_q$ equipped with a  non-degenerate quadratic from $\langle\,\,,\,\,\rangle_V:V\times V\to \bbF_q$ and let $G(V)$ be the isometry group of $\langle\,\,,\,\,\rangle_V.$
Given a semisimple element $s\in G(V)$, there exists a unique orthogonal decomposition
\begin{equation}\label{def-pri-dec}
V=\bigoplus\limits_\Gamma V_\Gamma(s), \quad s=\prod_{\Gamma}s_\Gamma,
\end{equation}
where the $V_\Gamma(s)$ are non-degenerate subspaces of $V$, $s_\Gamma\in G(V_\Gamma(s))$, and $s_\Gamma$ has minimal polynomial $\Gamma$.
The decomposition (\ref{def-pri-dec}) is called the \emph{primary decomposition} of $s$ in $G(V)$.
Let $m_\Gamma(s)$ be the multiplicity of $\Gamma$ in $s_\Gamma$.
If $m_\Gamma(s)\ne 0$, then we say $\Gamma$ is an \emph{elementary divisor} of $s$.
Then the centralizer of $s$ in  $G(V)$ has a decomposition $C_{G(V)}(s)=\prod_{\Gamma}C_\Gamma(s)$, where $C_\Gamma(s)=C_{G(V_\Gamma(s))}(s_\Gamma)$.
Moreover, by \cite[(1.13)]{FS89},
\begin{equation*}
C_\Gamma(s)=
\left\{ \begin{array}{ll} G(V_\Gamma(s)) & \text{if}\ \Gamma\in\calF_0; \\
\GL_{m_\Gamma(s)}(\varepsilon_\Gamma q^{\delta_\Gamma}) & \text{if}\ \Gamma\in\calF_1\cup\calF_2. \end{array} \right.
\end{equation*}
Here, $\GL_m(-q)$ means $\GU_m(q)$.
\smallskip

Let $\eta_\Gamma(s)$ be the type of $V_\Gamma(s)$.
Here $\eta_\Gamma(s)=1$ for all $\Gamma\in\calF$ if $V$ is symplectic.
By \cite[(1.12)]{FS89},
the multiplicity and type functions $\Gamma\mapsto m_\Gamma(s)$, $\Gamma\mapsto \eta_\Gamma(s)$ satisfy the following relations
\begin{equation}\label{function-mult-type}
\begin{aligned}
&\mathrm{dim}V=\sum\limits_\Gamma d_\Gamma m_\Gamma(s),\\
&\eta(V)=\zeta(-1)^{m_{x-1}(s)m_{x+1}(s)}\prod_{\Gamma}\eta_\Gamma(s),\\
&\eta_\Gamma(s)=\varepsilon_\Gamma^{m_\Gamma(s)}\  \text{for} \ \Gamma\in\calF_1\cup\calF_2.
\end{aligned}
\end{equation}
Recall that $\zeta$ is the Legendre symbol for $\bbF_q^\times.$
Conversely, if $\Gamma\mapsto m_\Gamma(s)$, $\Gamma\mapsto \eta_\Gamma$ are functions from $\calF$ to $\mathbb N$, $\{\pm 1\}$ respectively satisfying (\ref{function-mult-type}),
then there exists a semisimple element $s$ of $G(V)$ with these functions as multiplicity and type functions.
Moreover, two semisimple elements $s$ and $s'$ of $G(V)$ are $G(V)$-conjugate if and only if $m_\Gamma(s) = m_\Gamma(s')$ and $\eta_\Gamma(s)=\eta_\Gamma(s')$ for all $\Gamma\in\calF$.
If $V$ is orthogonal,
a semisimple element $s$ lies in $G_0(V)$ if and only if $m_{x+1}(s)$ is even.
%If $s\in G_0(V)$, then
%$$|C_{G_0(V)}(s):\prod_{\Gamma}C_{G_0(V_\Gamma(s))}(s(\Gamma))|=1\ \text{or}\ 2,$$ and index $2$ occurs if and only if $m_{x-1}(s)$ and $m_{x+1}(s)$ are both non-zero.
For more details, see \cite[\S1]{FS89}.
\smallskip

\subsubsection{Lusztig's Jordan decomposition for finite classical groups}\label{sub:Jordan}
Let $G$ be one of $\Sp_{2n}(q),$ $\SO_{2n+1}(q)$ and $\O^\epsilon_{2n}(q).$ Then $G^*=\SO_{2n+1}(q)$, $\Sp_{2n}(q)$ and $\O^\epsilon_{2n}(q)$, respectively.
For $s\in (G^*)^0$, we define
\begin{align}
 G^{(+)}=G^{(+)}(s)
&=\begin{cases}
G(V^*_{x-1}(s))& \text{if $G=\SO_{2n+1}(q)$ or $\O^\epsilon_{2n}(q)$};\\
G_0(V^*_{x-1}(s))& \text{if $G=\Sp_{2n}(q)$};
\end{cases}\\G^{(-)}=G^{(-)}(s)
&=G(V^*_{x+1}(s));
\\
G^{(\star)}=G^{(\star)}(s)
&=\prod_{\Gamma\in \calF_{1}\cup \calF_{2}}C_{G(V^*_{\Gamma}(s))}(s_\Gamma),
\end{align}
where $V^*_{\Gamma}(s)$ is defined in the last subsection.
We know that
$$
C_{G^*}(s)\simeq  G^{(+)}\times G^{(-)}\times G^{(\star)}.
$$
and $G^{(\star)}$ is a product of general linear groups or unitary group, and
\begin{equation}\label{equ:+-}
(G^{(+)},G^{(-)})=\begin{cases}
(\Sp_{2n^{(+)}(q)},\Sp_{2n^{(-)}}(q)), & \text{if $G=\SO_{2n+1}(q)$};\\
(\SO_{2n^{(+)}+1}(q),\O^{\epsilon^{(-)}}_{2n^{(-)}}(q)), & \text{if $G=\Sp_{2n}(q)$};\\
(\O^{\epsilon^{(+)}}_{2n^{(+)}}(q),\O^{\epsilon^{(-)}}_{2n^{(-)}}(q)), & \text{if $G=\O^\epsilon_{2n}(q)$}
\end{cases}
\end{equation}
where  $n^{(+)}=\lfloor m_{x-1}(s)/2\rfloor,n^{(-)}=m_{x+1}(s)/2$ and
$\epsilon^{(+)}=\eta_{x-1}(s),\epsilon^{(-)}=\eta_{x+1}(s)\in \{\pm1\}$.
The element $s$ can be written as
\begin{equation}
s= s^{(+)}\times s^{(-)}\times s^{(\star)}
\end{equation}
where $s^{(+)}=s_{x-1}$ (resp.~$s^{(-)}=s_{x+1}$) is the part whose eigenvalues are all equal to $1$ (resp. $-1$),
and $s^{(\star)}=\prod\limits_{\Gamma\in \calF_1\cup\calF_2}s_\Gamma$ is the part whose eigenvalues do not contain $1$ or $-1$.
In particular, $s^{(\dag)}$ is in the center of $G^{(\dag)}$ for $\dag=+,-,\star$.
\smallskip

Then a  Lusztig's Jordan decomposition (see Propositon \ref{prop:jordandec})
\begin{equation}\label{Jsjordan}
\mathcal{J}_s\colon \mathcal{E}(G,s)\rightarrow\mathcal{E}(G^{(+)}\times G^{(-)}\times G^{(\star)},1)
\end{equation}
can be written as
\begin{equation}
\mathcal{J}_s(\rho)=\rho^{(+)}\otimes\rho^{(-)}\otimes\rho^{(\star)}
\end{equation}
where $\rho^{(\dag)}\in\mathcal{E}(G^{(\dag)},1)$.
\smallskip

Now suppose that $G=\O_{2n+1}(q).$  In this case, we have $G^*=\Sp_{2n}(q)\times \bbZ/2\bbZ,$ $G^0=\SO_{2n+1}(q)$ and $(G^*)^0=\Sp_{2n}(q).$ Since $\O_{2n+1}(q)\cong\SO_{2n+1}(q)\times \bbZ/2\bbZ,$ where $\bbZ/2\bbZ$ is generated by the element $-\id_{2n+1}\in \O_{2n+1}(q)$, we have $$\Irr(G)=\{\rho\cdot 1_{\bbZ/2\bbZ}, \rho\cdot \sign_{\bbZ/2\bbZ}|\rho\in \Irr(G^0)\},$$
where $\sign_{\bbZ/2\bbZ}$ is the unique non-trivial character of $\bbZ/2\bbZ.$
For $s\in (G^*)^0$, we define $$\calE(G,s)_1=\{\rho\cdot 1_{\bbZ/2\bbZ} |\rho\in \calE(G^0,s)\}\quad\text{and}\quad \calE(G,s)_{-1}=\{\rho\cdot \sign_{\bbZ/2\bbZ} |\rho\in \calE(G^0,s)\},$$ so $\calE(G,s)=\calE(G,s)_1 \sqcup \calE(G,s)_{-1}$ is the Lusztig series of $G$ corresponding to $s\in (G^*)^0.$
Writing $s=s^{(+)}\times s^{(-)}\times s^{(\star)}$, we have
$$
C_{G^{*}}(s)\cong G^{(+)}\times G^{(-)}\times G^{(\star)}\times\bbZ/2\bbZ,
$$
and the Lusztig's Joran decomposition for $G=\O_{2n+1}(q)$:
\begin{align}
\begin{split}
\calJ_s\colon \mathcal{E}(G,s)\rightarrow&\mathcal{E}(G^{(+)}\times G^{(-)}\times G^{(\star)},1)\times\Irr(\bbZ/2\bbZ) \\
\mathcal{J}_s(\rho)&=\rho^{(+)}\otimes\rho^{(-)}\otimes\rho^{(\star)}\otimes \sign^\epsilon
\end{split}
\end{align}
where $\epsilon=+$ and $\sign^+=1_{\bbZ/2\bbZ}$ if $\rho\in \calE(G,s)_1$ and $\epsilon=-$ and $\sign^{-}=\sign_{\bbZ/2\bbZ}$ if $\rho\in \calE(G,s)_{-1}.$ 
\smallskip

Recall that in section \S \ref{sec:diag}, we denote by $\rho^c$ to be the representation $\rho$ conjugated by diagonal automorphism. Additionally, recall that $\sp$ and $\sgn$ represent the spinor and determinant characters, respectively, as defined in \S \ref{sub:spin} and \S \ref{sub:det}.
Pan in \cite{P4} discussed the unicity and ambiguity of Lusztig's Jordan decompositions.
\begin{theorem}[\cite{P4}]\label{Thm:Panuniformproj}
Let $G$ be a symplectic group or an orthogonal group, and let $\rho,\rho'\in\Irr(G)$.
Then $\rho'^\sharp=\rho^\sharp$ if and only if
\begin{equation}
\rho'=\begin{cases}
\rho, \rho\cdot\sgn& \text{if\/ $G=\O_{2n+1}(q)$};\\
\rho,\rho^c & \text{if\/ $G=\Sp_{2n}(q)$};\\
\rho,\rho^c,\rho\cdot\sgn,\rho^c\cdot\sgn & \text{if\/ $G=\O^\pm_{2n}(q)$}.
\end{cases}
\end{equation}
\end{theorem}
\begin{theorem}[\cite{W,P4}]\label{Thm:Jordan} Let $G=\O_{2n+1}(q), \Sp_{2n}(q)$ or $ \O^\epsilon_{2n}(q)$, $s\in(G^*)^0$, and $\rho\in \calE(G,s).$
Then there exists a Jordan decomposition $\calJ_s\colon\calE(G,s)\rightarrow\calE(C_{G^*}(s),1)$  which is compatible with Harish-Chandra induction and satisfying the following extra symmetries.
\begin{itemize}
\item[$(1)$]
For $G=\O_{2n+1}(q)$, assume $\mathcal{J}_s(\rho)=\rho^{(+)}\otimes\rho^{(-)}\otimes\rho^{(\star)}\otimes\sign^\epsilon,$ then
$\rho\cdot\sgn\in\calE(G,s)$ and
\[
\calJ_s(\rho\cdot\sgn)=(\rho^{(+)})\otimes(\rho^{(-)})\otimes\rho^{(\star)}\otimes \sign^{-\epsilon},
\]
\[
\calJ_s(\rho\cdot\sp)=(\rho^{(-)})\otimes(\rho^{(+)})\otimes\rho^{(\star)}\otimes {\sign^{\sp(-\id_{2n+1})\cdot\epsilon}}.
\]
\item[$(2)$] For $G=\Sp_{2n}(q),$ assume $\mathcal{J}_s(\rho)=\rho^{(+)}\otimes\rho^{(-)}\otimes\rho^{(\star)},$ then  $\rho^c\in\calE(G,s)$ and
\[
\calJ_s(\rho^c)=(\rho^{(+)})\otimes(\rho^{(-)}\cdot\sgn)\otimes\rho^{(\star)}.
\]
\item[$(3)$] For $G=\O^\epsilon_{2n}(q),$ assume $\mathcal{J}_s(\rho)=\rho^{(+)}\otimes\rho^{(-)}\otimes\rho^{(\star)},$ then we have $\rho\cdot\sgn, \rho^c, \rho^c\cdot\sgn\in\calE(G,s)$ and
\[
\calJ_s(\rho\cdot\sgn)=(\rho^{(+)}\cdot\sgn)\otimes(\rho^{(-)}\cdot\sgn)\otimes\rho^{(\star)},
\]

\[
\calJ_s(\rho\cdot\sp)=(\rho^{(-)})\otimes(\rho^{(+)})\otimes\rho^{(\star)}
\]
and
\[
\calJ_s(\rho^c)=(\rho^{(+)})\otimes(\rho^{(-)}\cdot\sgn)\otimes\rho^{(\star)}.
\]
\end{itemize}
\end{theorem}

Note that a Lusztig's Jordan decomposition satisfying the properties in Theorem \ref{Thm:Jordan} is not unique (see \cite{P4}).
Henceforth, we consistently assume that Lusztig’s Jordan decomposition is compatible with Harish-Chandra induction, thereby simplifying certain arguments in  \S \ref{Sec:allcharacters}.
\smallskip
\subsubsection{ Lusztig parametrization}\label{sub:parametrization}
For $G=\SO_{2n+1}(q),\Sp_{2n}(q)$ or $\O^{\pm}_{2n}(q)$,
let $\mathcal{L}_1$ denote a bijection $\mathcal{S}_G\rightarrow\mathcal{E}(G,1)$ denoted by $\Lambda\mapsto\rho_\Lambda$.
Note that such a parametrization of unipotent characters also exists for a general linear group
or a unitary group given by (\ref{GLnGUn}) in \S\ref{sub:unipotent}.
Combining $\mathcal{L}_1$ for $G^{(+)}\times G^{(-)}\times G^{(\star)}$
and the inverse of $\calJ_s$ in (\ref{Jsjordan}):
\[
\mathcal{S}_{G^{(+)}}\times\mathcal{S}_{G^{(-)}}\times\mathcal{S}_{G^{(*)}}
\rightarrow \mathcal{E}(G^{(+)}\times G^{(-)}\times G^{(\star)},1)
\rightarrow \mathcal{E}(G,s^{(+)}\times s^{(-)}\times s^{(\star)})
\]
we obtain a bijection
\begin{align}
\begin{split}
\mathcal{L}_s\colon\mathcal{S}_{G^{(+)}}\times
\mathcal{S}_{G^{(-)}}\times\mathcal{S}_{G^{(\star)}} &\rightarrow\mathcal{E}(G,s) \\
(\Lambda_+,\Lambda_-,\lambda_\star) &\mapsto  \rho_{\Lambda_+,\Lambda_-,\lambda_\star},
\end{split}
\end{align}
where $\calS_{G^{(\star)}}=\prod\limits_{\Gamma\in \calF_1\cup\calF_2}\calS_{C_{G(V^*_{\Gamma})}(s_\Gamma)}$ and $\lambda_\star=\prod\limits_{\Gamma\in \calF_1\cup\calF_2}\lambda_{\Gamma}\in \calS_{G^{(\star)}}.$
Such a bijection $\calL_s$ is  called a \emph{Lusztig parametrization}.
\smallskip

For $G=\O_{2n+1}(q),$  we have
Lusztig parametrization:
\begin{align}
\begin{split}
\calL_s\colon\mathcal{S}_{G^{(+)}}\times\mathcal{S}_{G^{(-)}}\times\mathcal{S}_{G^{(*)}}&\times\{\pm\}\rightarrow \mathcal{E}(G,s) \\
(\Lambda_+,\Lambda_-, \lambda_\star,\epsilon) &\mapsto  \rho_{\Lambda_+,\Lambda_-,\lambda_\star,\epsilon}=\rho^G_{\Lambda_+,\Lambda_-,\lambda_\star,\epsilon}
\end{split}
\end{align}
where $\rho^G_{\Lambda_+,\Lambda_-,\lambda_\star,\epsilon}=\rho^{G^0}_{\Lambda_+,\Lambda_-,\lambda_\star}\cdot 1_{\bbZ/2\bbZ}$ if $\epsilon=+$ and $\rho^G_{\Lambda_+,\Lambda_-,\lambda_\star,\epsilon}=\rho^{G^0}_{\Lambda_+,\Lambda_-,\lambda_\star}\cdot\sign_{\bbZ/2\bbZ}$ if $\epsilon=-.$
\smallskip

\subsubsection{$F^+$-cuspidal modules and $F^-$-cuspidal modules}

Under the Lusztig parametrization,
we can give a explicit description of the action of $F^+$ and $F^-.$
For a character $\rho=\rho_{\Lambda_+,\Lambda_-,\lambda_*}\in \calE(G_n,s^{(+)}\times s^{(-)}\times s^{(*))})$ (or $\rho=\rho_{\Lambda_+,\Lambda_-,\lambda_\star,\epsilon}$ for $G_n=\O_{2n+1}(q)$ case),
by the compatibility with parabolic induction, we have
$$F^+(\rho_{\Lambda_+,\Lambda_-,\lambda_\star})=\sum_{\Theta_+}\rho_{\Theta_+,\Lambda_-,\lambda_\star}
$$
where the summand runs over all
$\Theta_+$ such that $\Theta_+$ is obtained from $\Lambda_+$ by adding a $1$-hook, and
$$F^-(\rho_{\Lambda_+,\Lambda_-,\lambda_\star})=\sum_{\Theta_-}\rho_{\Lambda_+,\Theta_-,\lambda_\star}$$
 where the summand runs over all
$\Theta_-$ such that $\Theta_-$ is obtained from $\Lambda_-$ by adding a $1$-hook (or in $\O_{2n+1}(q)$ case, \begin{align}\label{O2n+1(1)}F^+(\rho_{\Lambda_+,\Lambda_-,\lambda_\star,\epsilon})=\sum_{\Theta_+}\rho_{\Theta_+,\Lambda_-,\lambda_\star,\epsilon}
\end{align}
where the summand runs over all
$\Theta_+$ such that $\Theta_+$ is obtained from $\Lambda_+$ by adding a $1$-hook, and
\begin{align}\label{O2n+1(2)}
F^-(\rho_{\Lambda_+,\Lambda_-,\lambda_\star,\epsilon})=
\sum_{\Theta_-}\rho_{\Lambda_+,\Theta_-,\lambda_\star,\zeta(-1)\cdot\epsilon}
\end{align}
where the summand runs over all
$\Theta_-$ such that $\Theta_-$ is obtained from $\Lambda_-$ by adding a $1$-hook, see \cite[Proposition 4.12 (a2)]{LLZ}). So we deduce the following lemma.
\begin{lemma}\label{Lem:cuspidal}
Let $G=\O_{2n+1}(q)$, $\Sp_{2n}(q)$ and $\O_{2n}^\pm(q)$,
$s\in G^*$ and $\calJ_s$ be the fixed  Lusztig's Jordan decomposition in Theorem \ref{Thm:Jordan}. Then $\rho=\rho_{\Lambda_+,\Lambda_-,\lambda_\star}$ (or $\rho_{\Lambda_+,\Lambda_-,\lambda_\star,\epsilon}$) is
 $F^\pm$-cuspidal if and only if $\Lambda_\pm$ is cuspidal.
\end{lemma}
\begin{proof}Since $(E^+,F^+)$ and $(E^-,F^-)$ are adjoint pairs, so this lemma follows from the action of $F^+$ and $F^-$ described as above.
\end{proof}

Fix a semisimple element $s=s^{(+)}\times s^{(-)}\times s^{(\star)}\in (G^*)^0,$ we have $C_{G^*}(s)=G^{(+)}\times G^{(-)}\times G^{(\star )}$ (or $G^{(+)}\times G^{(-)}\times G^{(\star )}\times \bbZ/2\bbZ$).
In each Lusztig series $\calE(G,s)$, if $\rho\in \calE(G,s)$ is both $F^+$-cuspidal and $F^-$-cuspidal, so is its $\langle [\rm{Det}_{\O_{2n+1}}]\rangle$-, $\langle [\rm{Diag}_{\Sp_{2n}}]\rangle$- or  $\langle [\rm{Diag}_{\O^\pm_{2n}}], [\rm{Det}_{\O^\pm_{2n}}]\rangle$-orbit  for $G=\O_{2n+1}(q),\Sp_{2n}(q)$ or $\O^\pm_{2n}(q)$, respectively.
 We have the following description of irreducible characters which are both $F^+$-cuspidal and $F^-$-cuspidal.
 \begin{itemize}
\item For $G=\O_{2n+1}(q),$ if $n^{(+)}=t_+^2+t_+$ and
 $n^{(-)}=t_-^2+t_-$ for some $t_+,t_-\in \bbN ,$
 then $\calE(G,s)$ has characters $\rho$ which are both $F^+$-cuspidal and $F^-$-cuspidal; otherwise, it has none.
 Moreover, under Lusztig's Jordan decomposition,
 $\calJ_s(\rho)=\rho^{(+)}\otimes \rho^{(-)}\otimes \rho^{(\star)}\otimes \sign^\epsilon$
 satisfies  $\rho^{(+)}=\rho^{\Sp}_{t_+}$, $\rho^{(-)}=\rho^{\Sp}_{ t_-}$ and $\rho(-\id)=\epsilon\cdot\rho(\id)$.
     We label them by $\rho^{\O_{\mathrm{odd}}}_{t_+,t_-,\lambda_\star, \epsilon}$
     for $t_+,t_-\in\bbN , \epsilon\in \{\pm 1\}$.
     Note that this parametrization is unique and satisfies
     $\rho^{\O_{\mathrm{odd}}}_{t_+,t_-,\lambda_\star, \epsilon}
     \cdot \sgn=\rho^{\O_{\mathrm{odd}}}_{t_+,t_-,\lambda_\star, -\epsilon}$.

 \item For $G=\Sp_{2n}(q)$, if $n^{(+)}=t_+^2+t_+$ and $n^{(-)}=t_-^2$ for some $t_+,t_-\in \bbN $, then $\calE(G,s)$ has characters  $\rho$ which are both $F^+$-cuspidal and $F^-$-cuspidal; otherwise, it has none. Moreover, under Lusztig's Jordan decomposition,  $\calJ_s(\rho)=\rho^{(+)}\otimes \rho^{(-)}\otimes \rho^{(\star)}$ satisfies  $\rho^{(+)}=\rho^{\SO}_{t_+}$ and $\rho^{(-)}=\rho^{\O_{\mathrm{even}}}_{\pm t_-}$.
     We label them by $\rho^{\Sp}_{t_+, t_-,\lambda_\star}$ for $t_+\in\bbN , t_-\in \bbZ$ such
     that $(\rho^{\Sp}_{t_+,t_-,\lambda_\star})^c =\rho^{\Sp}_{t_+,-t_-,\lambda_\star}$.

 \item For $G=\O^\epsilon_{2n}(q),$ if $n^{(+)}=t_+^2, n^{(-)}=t_-^2$ and
 $\epsilon^{(+)}=(-1)^{t_{(+)}}, \epsilon^{(-)}=(-1)^{t_{(-)}}$ for
 some integer $t_+,t_-\in \bbN$, then $\calE(G,s)$  has characters $\rho$ which are both $F^+$-cuspidal and $F^-$-cuspidal; otherwise, it has none. Moreover, under Lusztig's Jordan decomposition,  $\calJ_s(\rho)=\rho^{(+)}\otimes \rho^{(-)}\otimes \rho^{(\star)}$ satisfies
 $\rho^{(+)}=\rho^{\O_{\mathrm{even}}}_{\pm t_+}$ and $\rho^{(-)}=
 \rho^{\O_\mathrm{even}}_{\pm t_-}$.
     We label them  by $\rho^{\O_{\mathrm{even}}}_{t_+,t_-,\lambda_\star}$
     for $t_+,t_-\in\bbZ$ such that $\rho^{\O_{\mathrm{even}}}_{t_+,t_-,\lambda_\star}
     \cdot \sgn=\rho^{\O_{\mathrm{even}}}_{-t_+,-t_-,\lambda_\star}$ and
     $(\rho^{\O_{\mathrm{even}}}_{t_+,t_-,\lambda_\star})^c=
     \rho^{\O_{\mathrm{even}}}_{t_+,-t_-,\lambda_\star}$.

 \end{itemize}
  Note that the parametrization of characters of $\Sp_{2n}(q)$ and $\O_{2n}^\pm(q)$, which are both $F^+$-cuspidal and $F^-$-cuspidal, is not unique.  In the next section, by using the colored weight function $\bbO^{+}(u)(-)$ and $\bbO^{-}(v)(-),$ we will give a unique parametrization.

\begin{center}
\begin{table}[h]\caption{ $F^+$-cuspidal and $F^-$-cuspidal modules}\label{diag:cuspidal}
\begin{tabular}{cccc}
  \hline
  &$\O_{2\bullet+1}(q)$&$\Sp_{2\bullet}(q)$ &$\O^{\pm}_{2\bullet}(q)$\\\hline
  &$\rho_{t_+,t_-,\lambda_\star, \varepsilon},t_\pm\in \bbN ,\varepsilon\in\{\pm\}$&
$\rho_{t_+,t_-,\lambda_\star},t_+\in\bbN ,t_-\in \bbZ$&
$\rho_{t_+,t_-,\lambda_\star}, t_\pm\in \bbZ$\\\hline
\end{tabular}
\end{table}
\end{center}

\subsection{Colored weight functions for $F^+$-cuspidal  and $F^-$-cuspidal modules}\label{sec:coloredweight}\hfill\\

In this section,  we deduce that the Kac-Moody algebra associated to the categorical action of ${\Heis_{-2}(z^+,t^+)}\odot{ \Heis_{-2}(z^-,t^-)}$ on $KG_\bullet\mod$ is $\fraks\frakl'_{I_+}\oplus \fraks\frakl'_{I_-}$ with $$I_+=q^\bbZ\sqcup-q^\bbZ~~\text{and}~~I_-=q^\bbZ\sqcup-q^\bbZ,$$ see Corollary \ref{ThmB}. 

\smallskip

\subsubsection{Some lemmas}
For any irreducible module $\rho\in KG_\bullet\mod,$
we  define $m^+_\rho(u) \in K[u]$ (resp. $ n^+_\rho(u)\in K[u]$, $m^-_\rho(v)\in K[v]$ and $n^-_\rho(v) \in K[v]$) to
be the monic  minimal polynomials of
the endomorphisms
$\mathord{\begin{tikzpicture}[baseline = -1mm]
 	\draw[->,red] (0.08,-.2) to (0.08,.2);
     \node at (0.08,0) {$\dott$};
 	\draw[-,darkg,thick] (0.38,.2) to (0.38,-.2);
     \node at (0.55,0) {$\darkg\scriptstyle{\rho}$};
\end{tikzpicture}
}$ (resp.
$\mathord{\begin{tikzpicture}[baseline = -1mm]
 	\draw[<-,red] (0.08,-.2) to (0.08,.2);
     \node at (0.08,0.02) {$\dott$};
 	\draw[-,darkg,thick] (0.38,.2) to (0.38,-.2);
     \node at (0.55,0) {$\darkg\scriptstyle{\rho}$};
\end{tikzpicture}
},$  $\mathord{\begin{tikzpicture}[baseline = -1mm]
 	\draw[->,blue] (0.08,-.2) to (0.08,.2);
     \node at (0.08,0) {$\dott$};
 	\draw[-,darkg,thick] (0.38,.2) to (0.38,-.2);
     \node at (0.55,0) {$\darkg\scriptstyle{\rho}$};
\end{tikzpicture}
}$ and
$\mathord{\begin{tikzpicture}[baseline = -1mm]
 	\draw[<-,blue] (0.08,-.2) to (0.08,.2);
     \node at (0.08,0.02) {$\dott$};
 	\draw[-,darkg,thick] (0.38,.2) to (0.38,-.2);
     \node at (0.55,0) {$\darkg\scriptstyle{\rho}$};
\end{tikzpicture}
}$).

\begin{lemma}\label{Lemma:minimalpoly}
Let $\rho\in \Irr(G_n)$ be  $F^{\epsilon}$-cuspidal, where $\epsilon\in \{\pm1\}$.
Then we have
$\bbO^{\epsilon}(u)(\rho)=m^\epsilon_\rho(u)$ the monic minimal polynomial of $X^\epsilon(\rho)$ on $F^{\varepsilon}(\rho).$
\end{lemma}
\begin{proof}
Since $\rho$ is $F^{\epsilon}$-cuspidal, $E^\epsilon(\rho)=0$ and  we have $n^\epsilon_\rho(u)=1.$
Then this lemma  follows by \cite[Lemma 4.4]{BSW2}.
\end{proof}
\begin{lemma}
Let $\rho\in \Irr(G_n)$ be $F^{\epsilon}$-cuspidal,
where $\epsilon\in \{\pm1\}$. Let $\rho_1$ and $\rho_2$
be the two irreducible components of $F^{\varepsilon}(\rho).$
Then  $\dim(\rho_1)/\dim(\rho_2)=q^c$ with $c\in \bbQ$.
Then the action of
$X^{\epsilon}(\rho)$ on $F^{\epsilon}(\rho)$ satisfies
$$X^{\epsilon}(\rho)^2=\gamma\cdot X^{\epsilon}(\rho)-(t^\epsilon)^2\cdot \id_{F^\epsilon}(\rho)$$ and $\gamma$ satisfies $\gamma^2=-\frac{(q^c-1)^2\cdot(t^\epsilon)^2}{q^c}.$

\end{lemma}
\begin{proof}By Propostion \ref{Prop:dim2}$(b)$ we know that $X^{\epsilon}(\rho)$ on $F^{\epsilon}(\rho)$ satisfies
$X^{\epsilon}(\rho)^2=\gamma\cdot X^{\varepsilon}(\rho)-(t^\epsilon)^2\cdot \id_{F^\epsilon}(\rho).$
Note that $F^{+}(\rho)\cong R_{L_{n,1}}^{G_{n+1}}(\rho\otimes 1)$ and $F^{-}(\rho)\cong R_{L_{n,1}}^{G_{n+1}}(\rho\otimes \zeta)$.
Repeat the proof of \cite[Theorem 3.18]{HL80}, then we know that if $\alpha_1$ and $\alpha_2$ are two roots of $m^\varepsilon_\rho(u)=0$, then $\alpha_1\alpha_2=(t^\epsilon)^2$ and $\alpha_1/\alpha_2=-q^c$. Hence $\alpha_1^2=-q^c\cdot (t^\epsilon)^2$, $\alpha_2^2=-q^{-c}\cdot(t^\epsilon)^2$ and $\gamma^2=(\alpha_1+\alpha_2)^2=\alpha_1^2+\alpha_2^2+2\alpha_1\alpha_2=-\frac{(q^c-1)^2\cdot(t^\epsilon)^2}{q^c}.$
\end{proof}

\begin{corollary}\label{cor:minimalpoly}
\begin{itemize}
\item[$(a)$]
If $G_n=\Sp_{2n}(q),$ then
 \begin{itemize}
\item[$(a1)$] $\bbO^+(u)(\rho_{t_+,t_-,\lambda_\star})=(u-\epsilon' q^{t_+})(u+\epsilon' q^{-1-t_+}),$
\item[$(a2)$]$\bbO^-(v)(\rho_{t_+,t_-,\lambda_\star})=
(v-\epsilon'' q^{t_-})(v+\epsilon'' q^{-t_-})$
\end{itemize}
for some $\epsilon', \epsilon'' \in \{\pm 1\}.$
\item[$(b)$]
If $G_n=\O^\pm_{2n}(q)$, then
\begin{itemize}
\item[$(b1)$] $\bbO^+(u)(\rho_{t_+,t_-,\lambda_\star})=(u-\epsilon'
    q^{t_+})(u+\epsilon' q^{-t_+}),$
\item[$(b2)$]$\bbO^-(v)(\rho_{t_+,t_-,\lambda_\star})=(v-\epsilon''q^{t_-})
(v+\epsilon'' q^{-t_-})$
\end{itemize}
for some $\epsilon', \epsilon'' \in \{\pm 1\}.$

\item[$(c)$] If $G_n=\O_{2n+1}(q)$, then
\begin{itemize}
\item[$(c1)$] $\bbO^+(u)(\rho_{t_+,t_-,\lambda_\star,\epsilon})=
    (u-\epsilon' q^{t_+})(u+ \epsilon' q^{-1-t_+}),$
\item[$(c2)$]$\bbO^-(v)(\rho_{t_+,t_-,\lambda_\star,\epsilon})=
(v-\epsilon'' q^{t_-})(v+ \epsilon''q^{-1-t_-})$
\end{itemize}
for some $\epsilon', \epsilon'' \in \{\pm 1\}.$
\end{itemize}
In all cases, the signatures $\epsilon'$ and $\epsilon''$ depend on $t_+,t_-,\lambda_\star$ and $\epsilon.$
\end{corollary}

\begin{proof}
This is just a corollary of Lusztig's Jordan decomposition since Lusztig's Jordan decomposition can be chosen such that it is compatible with Harish-Chandra induction and preserves the radio $\dim(\rho_1)/\dim(\rho_2)=\dim(\calJ_s(\rho_1))/\dim(\calJ_s(\rho_2))$, see Proposition \ref{prop:jordandec} and Theorem \ref{Thm:Jordan}. Then the problem is reduced to unipotent case. Finally, in the unipotent case, this result is established by Lusztig; see \cite[p. 464]{Car},  also \cite[\S 9]{Wa}.
\end{proof}

\subsubsection{Kac-Moody algebra $\fraks\frakl'_{I_+}\oplus \fraks\frakl'_{I_-}$}
We define the set $I_+$ (resp. $I_-$)
to be the union of the sets of roots of the minimal polynomials
$m^+_\rho(u)$ (resp. $m^-_\rho(u)$) for all irreducible modules $\rho \in KG_\bullet\mod$. Now we get the following proposition.
\begin{proposition}\label{Prop:I+-}
For $KG_\bullet\mod$, we have $I_+=q^\bbZ\sqcup \,-q^\bbZ$ and $I_-=q^\bbZ\sqcup \,-q^\bbZ.$
\end{proposition}
\begin{proof} Let us first reduce this problem to the irreducible  modules which are both $F^+$-cuspidal and $F^-$-cuspidal. Assume that $\rho\in \Irr(KG_\bullet\mod)$ is not $F^\epsilon$-cuspidal, then there exists $\sigma\in \Irr(KG_\bullet\mod)$ such that $\rho$ is a direct summand of $F^\epsilon(\sigma)$. Since $X^\epsilon$ and $T^\epsilon$ satisfying affine Hecke relation, by \cite[Lemma 4.7]{Go}, we deduce that the the eigenvalues of $X^\epsilon(F^\epsilon(\sigma))$ on $(F^\epsilon)^2(\sigma)$  are $q^{\pm 1}$ times the eigenvalues of $X^\epsilon(\sigma)$ on $F^\epsilon(\sigma)$, hence if the roots of $m^\epsilon_{\sigma}(u)$ belong to $ q^\bbZ\sqcup \,-q^\bbZ$, then so is $m^\epsilon_{\rho}(u)$.
\smallskip
If $\rho$ is both $F^+$-cuspidal and $F^-$-cuspidal, it is proved in Corollary \ref{cor:minimalpoly}.
\end{proof}

Now we can state our main theorem in this section, which is Corollary B in the Introduction.
\begin{corollary}\label{ThmB}
Let $R=K$ or $k$, then there is a strict $R$-linear monoidal functor
$$\Psi':\,\frakU(\fraks\frakl'_{I_+}\oplus \fraks\frakl'_{I_-})\to \mathcal{E}nd(RG_\bullet\mod),$$
where
$I_+=q^\bbZ\sqcup -q^\bbZ$ and $I_-=q^\bbZ\sqcup -q^\bbZ.$
\end{corollary}
\begin{proof}
For characteristic 0, it follows from Theorem \ref{cor:kacmoody} and Proposition \ref{Prop:I+-}. For positive characteristic, using modulo $\ell$ reduction, it follows from the result in characteristic 0.
\end{proof}
\subsubsection{Extra symmetries and colored weight functions}

In section \ref{Chap:extrasymmeties}, we study the extra symmetries of categorical  double quantum Heisenberg action on $RG_\bullet\mod.$ Now we will study the influence of extra symmetries on colored weight functions.

\begin{proposition}\label{Prop:extrasymm}
Let $\rho\in \Irr(G_n).$
\begin{itemize}
\item[$(a)$]If $G_n=\O_{2n+1}(q)$,
\begin{itemize}
\item $\bbO^{+}(u)(\sgn\cdot\rho)=\bbO^{+}(-u)(\rho)$ and $\bbO^{-}(v)(\sgn\cdot\rho)=\bbO^{-}(-v)(\rho),$
    \item $\bbO^{+}(u)(\sp\cdot\rho)=\bbO^{-}(u)(\rho)$ and $\bbO^{-}(v)(\sp\cdot\rho)=\bbO^{+}(v)(\rho).$
    \end{itemize}
\item[$(b)$] If $G_n=\Sp_{2n}(q)$,  $\bbO^{+}(u)(\rho^c)=\bbO^{+}(u)(\rho)$ and $\bbO^{-}(v)(\rho^c)=\bbO^{-}(-v)(\rho).$

\item[$(c)$]If $G_n=\O^{\pm}_{2n}(q)$,
\begin{itemize}
\item $\bbO^{+}(u)(\sgn\cdot\rho)=\bbO^{+}(-u)(\rho)$ and $\bbO^{-}(v)(\sgn\cdot\rho)=\bbO^{-}(-v)(\rho),$
    \item $\bbO^{+}(u)(\sp\cdot\rho)=\bbO^{-}(u)(\rho)$ and $\bbO^{-}(v)(\sp\cdot\rho)=\bbO^{+}(v)(\rho),$
    \item$\bbO^{+}(u)(\rho^c)=\bbO^{+}(u)(\rho)$ and $\bbO^{-}(v)(\rho^c)=\bbO^{-}(-v)(\rho).$
    \end{itemize}
\end{itemize}

\end{proposition}
\begin{proof}
We only prove the first formula of $(a)$, and the others are similar.  Indeed, let $\up\in\{\red{\up},\blue{\up}\},$ then for clockwise bubbles,
\begin{align*}
\mathord{\begin{tikzpicture}[baseline=0]
  \draw[<-,thin,darkblue] (0.3,0) to[out=90,in=0] (0,0.3);
  \draw[-,thin,darkblue] (0,0.3) to[out=180,in=90] (-.3,0);
\draw[-,thin,darkblue] (-.3,0) to[out=-90,in=180] (0,-0.3);
  \draw[-,thin,darkblue] (0,-0.3) to[out=0,in=-90] (0.3,0);
  \draw[-,thick,green, dashed] (0.6,-0.5) to (0.6,.6);
   \node at (-.6,0) {$\color{darkblue}\scriptstyle{x^r}$};
      \node at (-.3,0) {$\dott$};
\end{tikzpicture}
}
&\stackrel{(\ref{change})}{=}
\mathord{\begin{tikzpicture}[baseline=0]
  \draw[-,thick,green,dashed] (0.4,-0.5) to [out=135,in=-90](0.1,.05);
  \draw[-,thick,green,dashed] (0.1,0.05) to [out=90,in=-135] (0.4,.6);
  \draw[<-,thin,darkblue] (0.3,0) to[out=90,in=0] (0,0.3);
  \draw[-,thin,darkblue] (0,0.3) to[out=180,in=90] (-.3,0);
\draw[-,thin,darkblue] (-.3,0) to[out=-90,in=180] (0,-0.3);
  \draw[-,thin,darkblue] (0,-0.3) to[out=0,in=-90] (0.3,0);
  \node at (-.6,0) {$\color{darkblue}\scriptstyle{x^r}$};
      \node at (-.3,0) {$\dott$};
\end{tikzpicture}
}
\:\:\substack{(\ref{slide1})\\{\displaystyle =}\\(\ref{slide2})}\:
\mathord{\begin{tikzpicture}[baseline=0]
  \draw[-,thick,green,dashed] (-0.4,-0.5) to [out=45,in=-90](-.1,.05);
  \draw[-,thick,green,dashed] (-0.1,0.05) to [out=90,in=-45] (-0.4,.6);
  \draw[<-,thin,darkblue] (0.3,0) to[out=90,in=0] (0,0.3);
  \draw[-,thin,darkblue] (0,0.3) to[out=180,in=90] (-.3,0);
\draw[-,thin,darkblue] (-.3,0) to[out=-90,in=180] (0,-0.3);
  \draw[-,thin,darkblue] (0,-0.3) to[out=0,in=-90] (0.3,0);
  \node at (-.8,0) {$\color{darkblue}\scriptstyle{x^r}$};
      \node at (-.3,0) {$\dott$};
\end{tikzpicture}
}
\stackrel{(\ref{dotslide11})}{=}
\mathord{\begin{tikzpicture}[baseline=0]
  \draw[-,thick,green,dashed] (-0.4,-0.5) to [out=45,in=-90](-.1,.05);
  \draw[-,thick,green,dashed] (-0.1,0.05) to [out=90,in=-45] (-0.4,.6);
  \draw[<-,thin,darkblue] (0.3,0) to[out=90,in=0] (0,0.3);
  \draw[-,thin,darkblue] (0,0.3) to[out=180,in=90] (-.3,0);
\draw[-,thin,darkblue] (-.3,0) to[out=-90,in=180] (0,-0.3);
  \draw[-,thin,darkblue] (0,-0.3) to[out=0,in=-90] (0.3,0);
  \node at (.6,-.36) {$\color{darkblue}\scriptstyle{(-x)^r}$};
      \node at (-0.05,-.28) {$\dott$};
\end{tikzpicture}
}
\stackrel{(\ref{change})}{=}
\mathord{\begin{tikzpicture}[baseline=0]
  \draw[<-,thin,darkblue] (0.3,0) to[out=90,in=0] (0,0.3);
  \draw[-,thin,darkblue] (0,0.3) to[out=180,in=90] (-.3,0);
\draw[-,thin,darkblue] (-.3,0) to[out=-90,in=180] (0,-0.3);
  \draw[-,thin,darkblue] (0,-0.3) to[out=0,in=-90] (0.3,0);
  \draw[-,thick,green,dashed] (-1.25,-0.5) to (-1.25,.6);
   \node at (-.8,0) {$\color{darkblue}\scriptstyle{(-x)^r}$};
      \node at (-.3,0) {$\dott$};
\end{tikzpicture}
}\:
\:.
\end{align*}
The proof is similar for counter-clockwise bubbles. Then by definition, $(a)$ follows.
\end{proof}

\subsubsection{Unique parametrization by colored weight functions}
Based on the classification of irreducible modules which are both $F^+$-cuspidal and $F^-$-cuspidal and Proposition \ref{Prop:extrasymm}, we can determine a part of information of $\bbO^+(u)(-)$ and $\bbO^-(v)(-)$, which can help us to give the unique parametrization of irreducible modules which are both $F^+$-cuspidal and $F^-$-cuspidal.
\smallskip

For $G=\O_{2n+1}(q)$ case the parametrization for irreducible modules which are both $F^+$-cuspidal and $F^-$-cuspidal
 is completely determined.
\smallskip

 For $G=\Sp_{2n}(q)$, let $\rho=\rho_{t_+, t_-,\lambda_\star},$ then $\rho^c=\rho_{t_+, -t_-,\lambda_\star}$.
 By Proposition \ref{Prop:extrasymm}, we know that $\bbO^{-}(v)
 (\rho)
 =(v-q^{t_-})(v+q^{-t_-})$ or $(v-q^{-t_-})(v+q^{t_-})$.
 On the other hand, $$\bbO^{-}(v)
 (\rho)=\bbO^{-}(-v)(\rho^c).$$
 So one of $\bbO^{-}(v)
 (\rho)$ and $\bbO^{-}(v)(\rho^c)$ must be equal to $(v-q^{t_-})(v+q^{-t_-})$ and then the other is equal to $(v-q^{-t_-})(v+q^{t_-}).$
Hence we can relabel them such that $\bbO^{-}(v)
 (\rho_{t_+,t_-,\lambda_\star})=(v-q^{t_-})(v+q^{-t_-})$.
\smallskip

 Similarly, for $G=\O^\pm_{2n}(q)$, we relabel them such that  $\bbO^{+}(u)
 (\rho_{t_+,t_-,\lambda_\star})=(u-q^{t_+})(u+q^{-t_+})$ and
 $\bbO^{-}(v)
 (\rho_{t_+,t_-,\lambda_\star})=(v-q^{t_-})(v+q^{-t_-}).$ Now we get
 a unique parametrization for irreducible modules which are both $F^+$-cuspidal and $F^-$-cuspidal,
 see Table \ref{diag:label}.
 \begin{center}
\small{\begin{table}[h]
\caption{Unique parametrization of irreducible modules which are both $F^+$-cuspidal and $F^-$-cuspidal}\label{diag:label}
\begin{tabular}{cccc}
 \hline
  &$\O_{2\bullet+1}(q)$&$\Sp_{2\bullet}(q)$ &$\O^{\pm}_{2\bullet}(q)$\\\hline
  &$\rho_{t_+,t_-,\lambda_\star, \varepsilon}$&
$\rho_{t_+,t_-,\lambda_\star}$&
$\rho_{t_+,t_-,\lambda_\star}$\\
  &$t_\pm\in \bbN ,\varepsilon\in\{\pm1\}$&
$t_+\in\bbN ,t_-\in \bbZ$&
$ t_\pm\in \bbZ$\\\hline
  $\bbO^+(u)(-):$&$-$& $-$  &$(u- q^{t_+})(u+ q^{-t_+})$\\\hline
   $\bbO^-(v)(-):$&$-$& $(v- q^{t_-})(v+ q^{-t_-})$  &$(v- q^{t_-})(v+ q^{-t_-})$\\\hline
\end{tabular}
\end{table}}
\end{center}
\begin{remark}
The colored weight functions in the three blanks of Table \ref{diag:label}
  will be determined in the next section, see Table \ref{diag:final}.
\end{remark}

\begin{proposition}\label{prop:colorcuspidal}Let $\rho, \rho'$ be both $F^+$-cuspidal and $F^-$-cuspidal and $\rho^\sharp=\rho'^\sharp$.
Then $\rho=\rho'$ if and only if
$$\mathbb{O}^+(u)(\rho)=\mathbb{O}^+(u)(\rho')~~\text{  and }~~ \mathbb{O}^-(v)(\rho)=\mathbb{O}^-(v)(\rho').$$
\end{proposition}
\begin{proof}
%We know that $\rho=\rho_{\Lambda_+,\Lambda_-,\lambda_\star}$ (or $\rho=\rho_{\Lambda_+,\Lambda_-,\lambda_\star,\epsilon}$) is $F^+$- and $F^-$-cuspidal if and only if $\Lambda_+$ and $\Lambda_-$ are both cuspidal Lusztig symbols.
%There is only one cuspidal module if and only if the isolated labelling $\Lambda_+\times\Lambda_-$ is degenerated.
%Using the Jordan decomposition, we know that $\Lambda_\pm$ is degenerated if and only if $t_\pm=0$, i.e., $n^\pm_M(u)=(u-1)(u+1)=n^\pm_M(-u)$ , which is equivalent to $\mathbb{O}^\pm_{M}(u)=\mathbb{O}^\pm_{M'}(u)$  by the Lemma\cite{}.
By Theorem \ref{Thm:Panuniformproj},
\begin{itemize}
\item for $G=\O_{2n+1}(q)$, $\rho^\sharp=\rho'^\sharp$ if and only if $\rho'\in\{\rho,\rho\cdot \sgn\}$;
\item for $G=\Sp_{2n}(q)$, $\rho^\sharp=\rho'^\sharp$ if and only if $\rho'\in\{\rho,\rho^c\}$;
\item for $G=\O^\pm_{2n}(q)$, $\rho^\sharp=\rho'^\sharp$ if and only if $\rho'\in\{\rho,\rho^c,\rho\cdot \sgn,\rho^c\cdot \sgn\}$.
\end{itemize}

Note that $\rho$ is both  $F^+$-cuspidal and $F^-$-cuspidal if and only if $\rho^c$ (resp. $\rho\cdot\sgn$, $\rho^c\cdot\sgn$) is both $F^+$-cuspidal and $F^-$-cuspidal.
\smallskip

On the other hand, by Proposition \ref{Prop:extrasymm}, $\bbO^{+}(u)(\sgn\cdot\rho)=\bbO^{+}(-u)(\rho)$, $\bbO^{-}(v)(\sgn\cdot\rho)=\bbO^{-}(-v)(\rho),$ $\bbO^{+}(u)(\rho^c)=\bbO^{+}(u)(\rho)$ and $\bbO^{-}(v)(\rho^c)=\bbO^{-}(-v)(\rho).$
Then the lemma follows by the classification of modules which are both $F^+$-cuspidal and $F^-$-cuspidal (see Lemma \ref{Lem:cuspidal}) and Corollary \ref{cor:minimalpoly}.
\end{proof}
\medskip

\subsection{$\fraks\frakl'_{I_+}\oplus \fraks\frakl'_{I_-}$-categorification and complete invariants of $KG_\bullet\mod$}\label{Sec:allcharacters}\hfill\\

By Theorem \ref{ThmB}, we know that the set $I_+$ and $I_-$
are both equal to $q^\bbZ \sqcup -q^\bbZ$. So in characteristic 0, we have $\fraks\frakl'_{I_+}=\fraks\frakl_{\infty}\oplus\fraks\frakl_{\infty}$, $\fraks\frakl'_{I_-}=\fraks\frakl_{\infty}\oplus\fraks\frakl_{\infty}$, and so $\frakg=\fraks\frakl'_{I_+}\oplus \fraks\frakl'_{I_-}$ is isomorphic to four copies of Kac-Moody Lie algebra $\fraks\frakl_{\infty}.$
\smallskip
%\begin{proposition}
%
%Let $M\in \Irr(G_n)$ and assume that $M$ is a submodule of $F^{\varepsilon}N$ for some  $N\in \Irr(G_n)$, Let $G\in\{\rm{Det}, \rm{Spin}, \rm{Diag}\}$, then
%\begin{itemize}
%\item $\mathbb{O}^+_{M}(u)=\mathbb{O}^+_{G(M)}(u)$ if and only if $\mathbb{O}^+_{N}(u)=\mathbb{O}^+_{G(N)}(u).$
%
%    \item  $\mathbb{O}^-_{M}(v)=\mathbb{O}^-_{G(M)}(v)$ if and only if $\mathbb{O}^-_{N}(u)=\mathbb{O}^-_{G(N)}(u).$
%\end{itemize}
%
%\end{proposition}
Let $E^+=\bigoplus_{i\in I_+}E^+_i$ and  $F^+=\bigoplus_{i\in I_+}F^+_i$ (resp. $E^-=\bigoplus_{i'\in I_-}E^-_{i'}$ and $F^-=\bigoplus_{i'\in I_-}F^-_{i'}$) be the decomposition of the functors into generalized
$i$-eigenspaces for $X^+$ (resp. generalized
$i'$-eigenspaces for $X^-$) on $E^+$ and $F^+$ (resp. $E^-$ and $F^-$). Then $$\{[E^+_i], [F^+_i],[E^-_{i'}], [F^-_{i'}]\}_{i\in I_+,i'\in I_-}$$ act as the Chevalley generators $\{e^+_i,f^+_i,e^-_{i'},f^-_{i'}\}_{i\in I_+,i'\in I_-}$ of $\frakg=\fraks\frakl'_{I_+}\oplus \fraks\frakl'_{I_-}$ on
the Grothendieck group $[
KG_\bullet\mod]$ of $KG_\bullet\mod$ and many problems on $KG_\bullet\mod$ have a Lie-theoretic
counterpart:
\begin{itemize}
\item irreducible modules which are both $F^+$-cuspidal and $F^-$-cuspidal correspond to highest weight vectors of $\frakg$-module $[KG_\bullet\mod]$;
\item the decomposition of $[KG_\bullet\mod]$ as $\frakg$-module into
 a direct sum of irreducible highest weight
modules are parametrized by irreducible modules which are both $F^+$-cuspidal and $F^-$-cuspidal;
\item the colored weight functions $\bbO^+(u)(\sigma),$ $\bbO^-(v)(\sigma)$  and the weight $\wt(\sigma)$ for any irreducible $\sigma\in KG_\bullet\mod$ are determined mutually.
\end{itemize}
So once we can determine $\wt(\rho)$ the weight of irreducible modules $\rho$ which are both $F^+$-cuspidal and $F^-$-cuspidal, we can compute the weights of all irreducible modules.
In this section, we will show that the colored weight functions $\bbO^+(u)(\rho)$, $\bbO^-(v)(\rho)$ and uniform projection can distinguish all irreducible modules of $ KG_\bullet\mod$. This  extends Proposition \ref{prop:colorcuspidal} which only concerns the case where the modules are  both $F^+$-cuspidal and $F^-$-cuspidal.
\smallskip

\subsubsection{Residues and contents}\label{sub:contents}
We fix $\tuple \xi = (\xi_1,\ldots,\xi_l) \in \scrI^l$, and assume that $\scrI$ is stable under the multiplication by
$q$ and $q^{-1}$.
Let $A\in Y(\tuple\lambda)$ be the box which lies in
the $i$-th row and $j$-th column of the diagram of
$\lambda^p$.
The \emph{$(\tuple \xi,q)$-shifted residue} of the node $A$ is the element of $\scrI$ given by
$q\text{-}\!\res^{\tuple \xi}(A)=q^{j-i}\xi_{p}.$  If $\tuple\lambda$, $\tuple\mu$ are $l$-partitions such
that $|\tuple\mu|=|\tuple\lambda|+1$ we write $q\text{-}\!\res^{\tuple\xi}(\tuple\mu-\tuple\lambda)=i$ if $Y(\tuple\mu)$ is obtained by
adding a node of $(\tuple\xi,q)$-shifted residue $i$ to $Y(\tuple \lambda)$.  We call the node a removable $i$-node of $\tuple \mu$ and an addable node of $\tuple \lambda$.
A \emph{charge} of the tuple $\tuple \xi= (\xi_1,\ldots,\xi_l)$ is an $l$-tuple of integers
$\tuple s = (s_1,\ldots,s_l)$ such that $\xi_p = q^{s_p}$ for all $p = 1,\ldots,l$.
Conversely, given $\scrI \subset R^\times$ and $q \in R^\times$, any
$l$-tuple of integers $\tuple s \in \bbZ^l$ defines a tuple $\tuple \xi = (q^{s_1},\ldots,q^{s_l})$ with
charge $\tuple s$.
The \emph{$\tuple s$-shifted content}
of the box $A=(i,j,p)$ is the integer $\text{cont}^{\tuple s}(A)=s_p+j-i$. Similarly, if $\tuple\lambda$, $\tuple\mu$ are $l$-partitions such
that $|\tuple\mu|=|\tuple\lambda|+1$ we write $\text{cont}^{\tuple s}(\tuple\mu-\tuple\lambda)=i$ if $Y(\tuple\mu)$ is obtained by
adding a node of $\tuple s$-shifted content  $i$ to $Y(\lambda)$.
It is related to the residue of $A$ by the formula $q\text{-}\!\res^{\tuple \xi}(A)=q^{\text{cont}^{\tuple s}(A)}$.
\smallskip

\subsubsection{Fock spaces}

Recall that a \emph{Fock space} $\bfF(\tuple\xi)$
 is the $\bbC$-vector
space with basis
$\{|\tuple\lambda,\tuple\xi\rangle_I\,|\,\tuple\lambda\in\scrP^l\}$  and the action of $e_i, f_i$ for all $i \in \scrI$ given by
$$
  f_i(|\tuple\lambda,\tuple \xi\rangle_{\scrI})=\sum\limits_{q\text{-}\!\res^{\tuple\xi}(\tuple\mu-\tuple\lambda)=i}|\tuple\mu,\tuple\xi\rangle_{\scrI},$$$$
  e_i(|\tuple\mu,\tuple\xi\rangle_{\scrI})=\sum\limits_{q\text{-}\!\res^{\tuple\xi}(\tuple\mu-\tuple\lambda)=i}|\tuple\lambda,\tuple\xi\rangle_{\scrI}.
$$
For brevity, we will omit the subscript $\scrI$ if there is no confusion. The basis
$\{|\tuple\lambda,\tuple\xi\rangle\,|\,\tuple\lambda\in\scrP^l\}$ is
called  the
\emph{standard monomial basis}, where every element is a weight vector.
\smallskip

%If $\scrI= A_\infty$ then $\bfF(\tuple \xi)=\bfL(\Lambda_{\tuple \xi})$.
%In general,
%the $\frakg'$-submodule of $\bfF(\tuple\xi)$ generated by $|\tuple\emptyset,\tuple\xi\rangle$ is
% isomorphic to $\bfL(\Lambda_{\tuple \xi})$.

 Now we describe the action of $\fraks\frakl'_{\scrI}$ on $\bfF(\tuple \xi)$ in the case where $q$ has infinite
order, i.e., $\fraks\frakl'_{\scrI}$ is isomorphic to a direct sum of $\fraks\frakl_{\infty}$.
For an $l$-partition $\tuple \lambda$ and $i\in\scrI$ we define
\begin{align*}
&N_i(\tuple\lambda |\tuple \xi) = \sharp\{ \text{ addable $i$-nodes of $\tuple\lambda$}\}
 - \sharp\{ \text{ removable $i$-nodes of $\tuple\lambda$} \}.
\end{align*}
Let $\tuple \xi =(\xi_1,\dots,\xi_l) \in I^l$.
Now we can state the following theorem.
\begin{theorem}[\cite{U}]
The following formulas define on $\bfF(\tuple \xi)$ the structure of an integrable $\fraks\frakl'_{\scrI}$-module:
\begin{align*}
&f_i |\tuple \lambda,\tuple \xi\rangle  = \sum_{q\text{-}\!\res^{\tuple\xi}(\tuple \mu-\tuple \lambda) = i}  |\tuple \mu,\tuple \xi\rangle, \\
&e_i |\tuple \mu,\tuple \xi\rangle  = \sum_{q\text{-}\!\res^{\tuple\xi}(\tuple \mu-\tuple \lambda) = i}|\tuple \lambda,\tuple \xi\rangle,  \\
%\end{alignat*}
%\begin{align*}
&\alpha^\vee_i |\tuple \lambda,\tuple \xi\rangle  = {N_i(\tuple \lambda |\tuple \xi)} |\tuple \lambda,\tuple \xi\rangle.
\end{align*}
For this action the weight of a standard basis element is
\begin{equation}\label{eq:weight}
  \mathrm{wt}(|\tuple\lambda,\tuple \xi\rangle)= \sum_{i\in I}N_i(\tuple \lambda |\tuple \xi)\Lambda_i.
\end{equation}
\end{theorem}
%\subsubsection{Minimal categorical representations}
%\label{subsec:mini}
%
%Let $\Lambda \in \X_{\scrI}^{+}$ be a dominant integral weight, and
%$\lambda=\Lambda-\beta$ for a given $\beta\in \Q_{\scrI}^+$.
%We define $$\scrL(\Lambda):=\bigoplus_{\beta\in Q^+}\bfH_{\beta}^\Lambda(Q)\mod,
%  \scrL(\Lambda)_{\lambda}:=\bfH_{\beta}^\Lambda(Q)\mod,$$
%and the following data on
%$\scrL(\Lambda)$ :
%\begin{itemize}[leftmargin=8mm]
%\item $E_s1_\lambda=E^{\Lambda}_s1_{\beta}$,
%\item $F_s1_\lambda=F^{\Lambda}_s1_{\beta}$,
%\item $x_s1_\lambda\in\Hom(F^{\Lambda}_s1_\beta,F^{\Lambda}_s1_\beta)$
%is represented by the right multiplication by
%$x_{n+1}$ on $\bfH_{\beta+\alpha_s}^\Lambda(Q)e(\beta,s),$
%\item $\tau_{st}1_\lambda\in\Hom(F^{\Lambda}_sF^{\Lambda}_t1_\beta,F^{\Lambda}_tF^{\Lambda}_s1_\beta)$
%is represented by the right multiplication by
%$\tau_{n+1}$ on $\bfH_{\beta+\alpha_s+\alpha_t }^\Lambda(Q)e(\beta,ts)$ where $e(\beta,ts)=\sum_{\tuple k \in I^{\beta }}
%e(\tuple k ts)$.
%\end{itemize}
%
%
%
%\begin{theorem}[\cite{KK12}, \cite{K12}]\label{thm:minimalcat}
%  The endofunctors $E^{\Lambda}_s$ and $F^{\Lambda}_s$ of $\scrL(\Lambda)$ are bi-adjoint for all $s\in \scrI.$ The tuple $(\{E^{\Lambda}_s\}_{s\in I},\{F^{\Lambda}_s\}_{s\in I},\{x_s\}_{s\in I},\{\tau_{s,t}\}_{s,t\in I})$ and
%   the decomposition $\scrL(\Lambda)=
%  \bigoplus_{\omega\in \X} \scrL(\Lambda)_{\,\omega}$ form a minimal $\frakA(\frakg)$-categorical representation
%  of $\bfL(\Lambda)$.\qed
%\end{theorem}
%
%\smallskip
\subsubsection{Content function and $(q,\pm)$-residue function}\label{sub:contentfunction}
Let $|\lambda, s\rangle \in \scrP\times \bbZ$ be a charged partition,
then we define its content function by
\begin{align}\label{coloredweight1}\bbO_{|\lambda,  s \rangle}(u)=\frac{\prod\limits_{\text{$b$  addable }}(u-\text{cont}^{ s}(b))}{\prod\limits_{\text{$b$  removable }}(u-\text{cont}^{ s}(b))}
	\end{align}
and $( q,\epsilon)$-residue function by
\begin{align}\label{coloredweight2}\bbO^{(q,\epsilon)}_{|\lambda,  s \rangle}(u)=\frac{\prod\limits_{\text{$b$  addable }}(u-\epsilon\cdot q^{\text{cont}^{ s}(b)})}{\prod\limits_{\text{$b$  removable }}(u-\epsilon \cdot q^{\text{cont}^{ s}(b)})}	\end{align}
where $b$ runs over the boxes of the Young diagram $Y(\lambda)$ of $\lambda$ and $\epsilon\in \{\pm \}.$

\begin{example}
Let $\lambda=(4,1,1)$ and $s=1.$
The Young diagram of $\lambda$ with its $ s$-shifted contents is
  $$
\begin{array}{|c|c|c|c|}
  \hline
  1& 2  &3 &4   \\
  \hline
  0 \\
  \cline{1-1}
  -1 \\
  \cline{1-1}
\end{array}.$$
Then we have  $\bbO_{|\lambda, s \rangle}(u)=\frac
{(u-5)(u-1)(u+2)}{(u-4)(u+1)},$ $\bbO^{(q,+)}_{|\lambda, s \rangle}(u)=\frac
{(u-q^5)(u-q)(u-q^{-2})}{(u-q^4)(u-q^{-1})}$ and $\bbO^{(q,-)}_{|\lambda, s \rangle}(u)=\frac
{(u+q^5)(u+q)(u+q^{-2})}{(u+q^4)(u+q^{-1})}.$
\end{example}
\subsubsection{Interlacing integer sequences and charged partitions}\label{sub:intersequence}

\begin{definition}[\cite{Ker}]
Two integer sets (or increasing integer sequences) $y_1,\,\ldots,y_{n-1}\in \bbZ$ and
$x_1,\,\ldots,x_{n-1},x_n\in\bbZ$ are said to be \emph{interlacing}, if and only if
$$
x_1 < y_1 < x_2 < \ldots < x_{n-1} < y_{n-1} < x_n.
$$
The number $c=\sum x_k-\sum y_k$ is called the \emph{center} of
interlacing integer sequences.
\end{definition}

An interlacing integer sequences is identified with a function $\omega$ on $\bbR$ such that
\begin{enumerate}
    \item $\omega$ is piecewise linear with slopes being either $1$ or $-1$,
    \item local maxima and minima are located on $\bbZ$,
    \item $\omega (x)=x-c$ when $x\gg 0$ and $\omega (x)=c-x$ when $x\ll 0$ for some $c\in \bbZ$.
\end{enumerate}
The sequence can be recovered from its local minima $x_{1}<x_{2}<\cdots<x_{n}$
and local maxima $y_{1}<y_{2}<\cdots <y_{n-1}$. Note that the center is $c.$
\smallskip

Each charged partition $|\mu,c\rangle\in \scrP\times \bbZ$ uniquely defines two  interlacing integer sequences as follows:
\begin{itemize}
\item $x_1<x_2<\cdots< x_n$ is the ordered list of all $c$-shifted content $\text{cont}^c(b)$ of addable box $b$ with respect to the Young diagram of $\mu$,
\item  $y_1<y_2<\cdots< y_{n-1}$ is the ordered list of all $c$-shifted content $\text{cont}^c(b)$ of removeable box $b$ with respect to the Young diagram of $\mu$.
\end{itemize}
Conversely,  the region
$$S_\omega= \{(u,v) \in \bbR^2: |u-c| \leq v\leq  \omega(u)\}$$
between the graphs of functions $v=\omega(u)$ and $v=|u-c|$ corresponds to the shape of a
Young diagram, which determines a partition $\mu$ (see Example \ref{interlacing}). Together with the center $c=\sum x_k-\sum y_k$, we finally get the charged partition $|\mu, c\rangle$.
\begin{lemma}[\cite{Ker}]\label{Lem:interlacingseq} The above two maps are inverse to each other, so they give a bijection between the space of interlacing integer sequences  and charged partitions $\scrP\times\bbZ.$
\end{lemma}

\begin{example} \label{interlacing}Let $\red{x_{1}}<\red{x_{2}}<\red{x_3}<\red{x_{4}}$
and $\color{green}{y_{1}}<\color{green}{y_{2}}<\color{green}{y_{3}}$ be
the interlacing integer sequences
\begin{equation*}
{\color{red}{-4}} < {\color{red}{-2}} < {\color{red}{0}} < {\color{red}{3}} \quad\quad \text{and} \quad\quad {\color{green}{-3}} < {\color{green}{-1}} < {\color{green}{2}}.
\end{equation*}
The corresponding function $\omega$  is represented graphically by the boldened trajectory in the following diagram.
Note that center of this interlacing integer sequences is $c=-1$
and we can get the partition $\mu = (4,2,1)$. Finally, we get the charged partition $|\mu, c\rangle.$

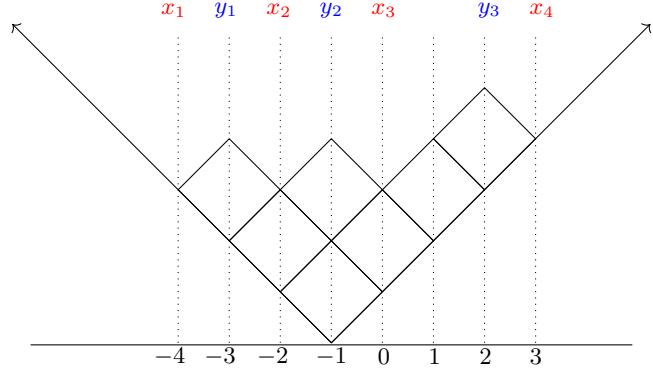
\begin{center}
\begin{tikzpicture}

\vspace{3mm}

%labelled diagonals (above)
\node at (2.8,2.3) {\color{red}{\scriptsize{$x_4$}}};
\node at (2.1,2.3) {\color{green}{\scriptsize{$y_3$}}};
\node at (.7,2.3) {\color{red}{\scriptsize{$x_3$}}};
\node at (0,2.3) {\color{green}{\scriptsize{$y_2$}}};
\node at (-.7,2.3) {\color{red}{\scriptsize{$x_2$}}};
\node at (-1.4,2.3) {\color{green}{\scriptsize{$y_1$}}};
\node at (-2.1,2.3) {\color{red}{\scriptsize{$x_1$}}};

\draw (-4,-2.15) -- (4,-2.15);

%labels on diagonals (below)
\node at (2.72,-2.3) {\scriptsize{$3$}};
\node at (2.05,-2.3) {\scriptsize{$2$}};
\node at (1.37,-2.3) {\scriptsize{$1$}};
\node at (.7,-2.3) {\scriptsize{$0$}};
\node at (0,-2.3) {\scriptsize{$-1$}};
\node at (-.78,-2.3) {\scriptsize{$-2$}};
\node at (-1.48,-2.3) {\scriptsize{$-3$}};
\node at (-2.15,-2.3) {\scriptsize{$-4$}};

\node at (0,0) {\begin{tikzpicture}[rotate = 45, scale = 1.2]

%first row diagram
\draw (0,0) rectangle (.8,.8);
\draw (.8,0) rectangle (1.6,.8);
\draw (1.6,0) rectangle (2.4,.8);
\draw (2.4,0) rectangle (3.2,.8);

%second row diagram
\draw (0,.8) rectangle (.8,1.6);
\draw (.8,.8) rectangle (1.6,1.6);

%third row diagram
\draw (0,1.6) rectangle (.8,2.4);

%coordinate axises
\draw[-> ] (0,0) -- (0,5);
\draw[->] (0,0) -- (5,0);
\draw[line width=1.1pt][->] (3.2,0) -- (5,0);
\draw[line width=1.1pt][->] (0,2.4) -- (0,5);
\draw[line width=1.1pt][-] (0,2.4) -- (0.8,2.4);
\draw[line width=1.1pt][-] (0.8,2.4) -- (0.8,1.6);
\draw[line width=1.1pt][-] (0.8,1.6) -- (1.6,1.6);
\draw[line width=1.1pt][-] (1.6,1.6) -- (1.6,0.8);
\draw[line width=1.1pt][-] (1.6,0.8) -- (3.2,0.8);
\draw[line width=1.1pt][-] (3.2,0.8)-- (3.2,0);

%diagonals above y = x
\draw[dotted] (0,0) -- (2.4,2.4);
\draw[dotted] (-.4,.4) -- (2,2.8);
\draw[dotted] (-.8,.8) -- (1.6,3.2);
\draw[dotted] (-1.2,1.2) -- (1.2,3.6);

%diagonals below y = x
\draw[dotted] (.4,-.4) -- (2.8,2);
\draw[dotted] (.8,-.8) -- (3.2,1.6);
\draw[dotted] (1.2,-1.2) -- (3.6,1.2);
\draw[dotted] (1.6,-1.6) -- (4,.8);

\end{tikzpicture}};
\end{tikzpicture}
\end{center}
\end{example}

Given a charged partition $|\mu,c\rangle\in \scrP\times\bbZ$, let $x_{1}<x_{2}<\cdots<x_{n}$
and $y_{1}<y_{2}<\cdots <y_{n-1}$ be the corresponding interlacing integer sequences then by the definition of content function, we have
$$\bbO_{|\mu, c\rangle}(u)=\frac{\prod_{i=1}^n(u-x_i)}{\prod_{j=1}^{n-1}(u-y_j)}.$$

\begin{corollary}\label{cor:injective}

\begin{itemize}
\item[$(a)$]
The content function $\bbO_{\bullet}(u): \scrP\times\bbZ\to K(u)$ is an injective map.
\item[$(b)$]Let $\epsilon\in \{\pm\}.$
The $(q,\epsilon)$-residue function $\bbO^{(q,\epsilon)}_{\bullet}(u): \scrP\times\bbZ\to K(u)$ is also an injective map.
\end{itemize}
\end{corollary}
\begin{proof}
$(a)$ Note that the content functions for charged partitions and interlacing sequences determines each other. The corollary is deduced from Lemma \ref{Lem:interlacingseq}.
\smallskip

$(b)$ Note that the content function $\bbO_{\bullet}(u)$ are 1-1 correspondent to $(q,\epsilon)$-residue function $\bbO^{(q,\epsilon)}_{\bullet}(u)$ since $q\in \bbR_{>1}$.
\end{proof}

\smallskip

\begin{remark}\label{bubbleempty}
It is straightforward to verify that $\bbO_{|\mu,c\rangle}(u)$ (resp. $\bbO^{(q,\epsilon)}_{|\mu,c\rangle}(u)$) is a polynomial (i.e., an element of $K[u]$ rather than merely $K(u)$) if and only if $\mu=\emptyset.$ Indeed, if $\mu\neq \emptyset,$ then the removable boxes of the Young diagram $Y(\mu)$ is non-empty. Since all addable and removable boxes possess distinct contents, it follows that $\bbO_{|\mu,c\rangle}(u)$ (resp. $\bbO^{(q,\epsilon)}_{|\mu,c\rangle}(u)$) can not be a polynomial in this case.
\end{remark}
\smallskip
%\begin{remark}
%For quantum Heisenberg categorification on $\calL(\tuple \xi)$, we have $$\bbO(u)(S_{|\tuple\lambda,\tuple \xi\rangle})=\prod\limits_{i\in I}(u-i)^{N_i(\tuple \lambda |\tuple \xi)},$$ which means that $\wt(\bullet)$ and $\bbO(u)(\bullet)$ determines each other. That is the reason why we call $\bbO^{+}(u)(\bullet)$ and $\bbO^-(v)(\bullet)$ colored weight functions.
%\end{remark}

\subsubsection{Complete invariants of $KG_\bullet\mod$}\label{sec:completeinvriants}

The main task of this subsection is to show that the colored weight functions $\bbO^+(u)(\rho)$, $\bbO^-(v)(\rho)$ and uniform projection can distinguish  all irreducible modules of $ KG_\bullet\mod$.
\smallskip

For short, we will write $\rho_{t_+\,t_-,\lambda_\star,(\epsilon)}=\rho_{t_+\,t_-,\lambda_\star}$ if $G=\Sp_{2n}(q)$ or $\O^\pm_{2n}(q)$ and $\rho_{t_+\,t_-,\lambda_\star,(\epsilon)}=\rho_{t_+\,t_-,\lambda_\star,\epsilon}$ if $G=\O_{2n+1}(q).$
We know that the irreducible modules which are both $F^+$-cuspidal and $F^-$-cuspidal have the form $\rho_{t_+\,t_-,\lambda_\star,(\epsilon)}$.
Now let $(KG_\bullet,\rho_{t_+,t_-, \lambda_\star,(\epsilon)})\mod$ be the Serre subcategory of $KG_\bullet\mod$ generated by the Kac-Moody 2-category $\frakU(\fraks\frakl'_{I_+}\oplus\fraks\frakl'_{I_-})$ acting on  $\rho_{t_+,t_-, \lambda_\star,(\epsilon)}$. Since the category $KG_\bullet\mod$ is semisimple,
it can actually be generated by
 the modules $(F^-)^{m_-}(F^+)^{m_+}(\rho_{t_+,t_-,\lambda_\star,(\epsilon)})$ for all $m_+,m_-\in \bbN$. %If $\lambda=-$, then $(KG,\rho_{t_+,t_-, \lambda,(\epsilon)})\mod$ becomes the category $\scrQU_{K,t_+,t_-}$ defined in \cite{LLZ}.
%If $\lambda=-,t_-=0$ and $G=\Sp_{2n}$, then $(KG,\rho_{t_+,t_-, \lambda})\mod$ becomes the category $\scrU_{K,t_+}$ defined in \cite{DVV2}.
Then we have that $$[KG_\bullet\mod]=\bigoplus_{t_+,t_-,\lambda_\star,(\epsilon)}[(KG_\bullet,\rho_{t_+,t_-, \lambda_\star,(\epsilon)})\mod]$$ is the decomposition of $[KG_\bullet\mod]$ as $\fraks\frakl'_{I_+}\oplus\fraks\frakl'_{I_-}$-module into
 a direct sum of irreducible highest weight
modules, and the image of $\rho_{t_+,t_-, \lambda_\star,(\epsilon)}$ in $[KG_\bullet\mod]$ is a highest weight vector. In particular, the highest weight of $[\rho_{t_+,t_-,\lambda_\star,(\epsilon)}]$ is determined by $\bbO^+(u)(\rho_{t_+,t_-,\lambda_\star,(\epsilon)})$ and $\bbO^-(v)(\rho_{t_+,t_-,\lambda_\star,(\epsilon)})$.
\smallskip

We have the following theorem.
\begin{theorem}\label{TheoremC}

Let $\rho=\rho_{t_+,t_-, \lambda_\star,(\epsilon)}$ be both $F^+$-cuspidal and $F^-$-cuspidal. We assume that $\bbO^+(u)(\rho)=(u-q^{s_1^{+}})(u+q^{s_2^{+}})$ and $\bbO^-(v)(\rho)=(v-q^{s_1^{-}})(v+q^{s_2^{-}}).$
 \begin{itemize}[leftmargin=8mm]
  \item[$\mathrm{(a)}$]
  The operators $[F^+_i],[E^+_i], [F^-_{i'}],[E^-_{i'}]$ for $i\in\scrI_+$, $i'\in\scrI_-$ yield a representation of
  $\fraks\frakl'_{I_+}\oplus \fraks\frakl'_{I_-}$ on $[(KG_\bullet,\rho)\mod]$.
  \item[$\mathrm{(b)}$] We write $\Upsilon(\Lambda_+)=(\mu^+_1,\mu^+_2)$
  and $\Upsilon(\Lambda_-)=(\mu^-_1,\mu^-_2).$
  Then the map
  $$|(\mu^{+}_1,\mu^+_2),(q^{s_1^{+}},-q^{s^+_2})\rangle
  \otimes|(\mu^-_1,\mu^-_2),(q^{s^-_1},-q^{s^-_2})\rangle\otimes
  \mapsto [\rho_{\Lambda_{+},
  \Lambda_{-},\lambda_\star,(\epsilon')}]$$ gives an
  $\fraks\frakl'_{I_+}\oplus \fraks\frakl'_{I_-}$-module isomorphism $$
  \bfF(q^{s^+_1},-q^{s^+_2})\otimes
  \bfF(q^{s^-_1},-q^{s^-_2})\simto [(KG_\bullet,\rho)\mod],$$
  where $\epsilon'=(-1)^{\mathrm{rank}(\Lambda_-)}\epsilon$ for $G_n=\O_{2n+1}(q)$ case and for other cases $\epsilon'=\emptyset$.

  \item[$\mathrm{(c)}$]
  For $G_n=\O_{2n+1}(q)$ case let $\epsilon'=(-1)^{\mathrm{rank}(\Lambda_-)}\epsilon$ and for other cases let $\epsilon'=\emptyset.$ We have

  \begin{itemize}
  \item[$(c1)$] $\bbO^+(u)(\rho_{\Lambda_{+},
  \Lambda_{-},\lambda_\star,(\epsilon')})= \bbO^{(q,+)}_{|\mu_1^+,s_1^+\rangle}(u)\cdot \bbO^{(q,-)}_{|\mu_2^+,s_2^+\rangle}(u)$,
  \item[$(c2)$] $\bbO^-(v)(\rho_{\Lambda_{+},
  \Lambda_{-},\lambda_\star,(\epsilon')})=\bbO^{(q,+)}_{|\mu_1^-,s_1^-\rangle}(v)\cdot \bbO^{(q,-)}_{|\mu_2^-,s^-_2\rangle}(v).$
  \end{itemize}

  \end{itemize}
\end{theorem}
\begin{proof}
The proofs of $(a)$ and $(b)$ are similar to \cite[Theorem 5.23]{LLZ}, see also \cite[Proposition 6.9]{DVV2}. In particular, for the $G=\O_{2n+1}(q)$ case, the sign $\epsilon'\in\{\pm 1\}$ is because of (\ref{O2n+1(1)}) and $(\ref{O2n+1(2)}).$
For $(c)$, note that under this isomorphism, we can compute the weight of any irreducible modules, hence colored weight functions $\bbO^{+}(u)(-)$ and $\bbO^-(v)(-)$ by (\ref{eq:weight}) and (\ref{chickens}).
\end{proof}

Now we can prove Theorem C in the Introduction, and we restate it as follows.
\begin{theorem} Let $G_n=\O_{2n+1}(q), \Sp_{2n}(q)$ or $\O_{2n}^{\pm}(q)$ and $\rho,\rho'\in \Irr(G_n).$
\begin{itemize}
\item[$(a)$]
We have $\rho=\rho'$ if and only if the following hold:
\begin{itemize}
\item[$(1)$] $\rho^\sharp=\rho'^\sharp,$
\item[$(2)$] $\mathbb{O}^+(u)(\rho)=\mathbb{O}^+(u)(\rho')$  and  $\mathbb{O}^-(v)(\rho)=\mathbb{O}^-(v)(\rho').$
    \end{itemize}
\item[$(b)$] If $\rho$ and $\rho'$ are quadratic unipotent, then $\rho=\rho'$ if and only if
$\mathbb{O}^+(u)(\rho)=\mathbb{O}^+(u)(\rho')$  and  $\mathbb{O}^-(v)(\rho)=\mathbb{O}^-(v)(\rho').$
\item[$(c)$] If $\rho$ and $\rho'$ are unipotent, then $\rho=\rho'$ if and only if
$\mathbb{O}^+(u)(\rho)=\mathbb{O}^+(u)(\rho').$
\end{itemize}
\end{theorem}

\begin{proof}

$(a)$
We assume that $\rho=\rho_{\Lambda_+,\Lambda_-,\lambda_\star}\in (KG_\bullet, \rho_{t_+,t_-,\lambda_\star})\mod$ (or $\rho=\rho_{\Lambda_+,\Lambda_-,\lambda_\star,\epsilon'}\in (KG_\bullet, \rho_{t_+,t_-,\lambda_\star,\epsilon})\mod$).
Then by Theorem \ref{TheoremC}, we know that $$\bbO^+(u)(\rho)=\bbO^{(q,+)}_{|\mu_1^+,s_1^+\rangle}(u)\cdot \bbO^{(q,-)}_{|\mu_2^+,s_2^+\rangle}(u)~\text{and}~\bbO^-(v) (\rho)=\bbO^{(q,+)}_{|\mu_1^-,s_1^-\rangle}(v)\cdot \bbO^{(q,-)}_{|\mu_2^-,s^-_2\rangle}(v).$$

By Theorem \ref{Thm:Panuniformproj}, we have
\begin{itemize}
\item for $G=\O_{2n+1}(q)$, $\rho^\sharp=\rho'^\sharp$ if and only if $\rho'\in\{\rho,\rho\cdot \sgn\}$;
\item for $G=\Sp_{2n}(q)$, $\rho^\sharp=\rho'^\sharp$ if and only if $\rho'\in\{\rho,\rho^c\}$;
\item for $G=\O^\pm_{2n}(q)$, $\rho^\sharp=\rho'^\sharp$ if and only if $\rho'\in\{\rho,\rho^c,\rho\cdot \sgn,\rho^c\cdot \sgn\}$.
\end{itemize}

We only prove the case $G=\Sp_{2n}(q),$ and the proofs of other cases are similar.
For $G=\Sp_{2n},$ we have $\bbO^+(u)(\rho^c)=\bbO^{(q,+)}_{|\mu_1^+,s_1^+\rangle}(u)
\cdot \bbO^{(q,-)}_{|\mu_2^+,s_2^+\rangle}(u)$  and
$$\bbO^-(v) (\rho^c)=\bbO^{(q,+)}_{|\mu_1^-,s_1^-\rangle}(-v)\cdot \bbO^{(q,-)}_{|\mu_2^-,s^-_2\rangle}(-v)=\bbO^{(q,+)}_{|\mu_2^-,s^-_2\rangle}(v)\cdot \bbO^{(q,-)}_{|\mu_1^-,s_1^-\rangle}(v).$$
By Proposition \ref{cor:injective},  we know that $\bbO^{(q,\epsilon)}_{\bullet}(v)$ is injective from $\scrP\times\bbZ$ to $K(v)$ for $\epsilon\in\{\pm\}$. Note that the  poles and zeros of $\bbO^{(q,+)}_{|\mu,s\rangle}(v)$ belong to $q^\bbZ$ and the  poles and zeros of $\bbO^{(q,-)}_{|\mu,s\rangle}(v)$ belong to $-q^\bbZ$. Hence $\bbO^-(v)(\rho^c)=\bbO^-(v)(\rho)$ if and only if $|\mu_1^-,s_1^-\rangle=|\mu_2^-,s^-_2\rangle$, i.e., $\Lambda_-=\Lambda^t_-,$
 which is equivalent to
$\rho=\rho^c.$
\smallskip

$(b)$
Assume that $\rho$ and $\rho'$ are quadratic unipotent, then $\rho=\rho_{\Lambda_+,\Lambda_-,(\epsilon)}$ and $\rho'=\rho_{\Lambda'_+,\Lambda'_-,(\epsilon')}.$
We have  $$\bbO^+(u)(\rho)=\bbO^{(q,+)}_{|\mu_1^+,s_1^+\rangle}(u)\cdot \bbO^{(q,-)}_{|\mu_2^+,s_2^+\rangle}(u)~\text{and}~\bbO^-(v) (\rho)=\bbO^{(q,+)}_{|\mu_1^-,s_1^-\rangle}(v)\cdot \bbO^{(q,-)}_{|\mu_2^-,s^-_2\rangle}(v).$$
Similarly, we have $$\bbO^+(u)(\rho')=\bbO^{(q,+)}_{|{\mu'}_1^+,{s'}_1^+\rangle}(u)\cdot \bbO^{(q,-)}_{|{\mu'}_2^+,{s'}_2^+\rangle}(u)~\text{and}~\bbO^-(v) (\rho')=\bbO^{(q,+)}_{|{\mu'}_1^-,{s'}_1^-\rangle}(v)\cdot \bbO^{(q,-)}_{|{\mu'}_2^-,{s'}^-_2\rangle}(v).$$
Then the conditions $\mathbb{O}^+(u)(\rho)=\mathbb{O}^+(u)(\rho')$  and  $\mathbb{O}^-(v)(\rho)=\mathbb{O}^-(v)(\rho')$ imply
$$\bbO^{(q,\pm)}_{|\mu_i^+,s_i^+\rangle}(u)=\bbO^{(q,\pm)}_{|{\mu'}_i^+,{s'}_i^+\rangle}(u)~\text{and}~
\bbO^{(q,\pm)}_{|\mu_i^-,s_i^-\rangle}(v)=\bbO^{(q,\pm)}_{|{\mu'}_i^-,{s'}_i^-\rangle}(v)$$ for $i=1,2$.
By the injectivity of $\bbO^{(q,\epsilon)}_{\bullet}(u)$, we must have $$|\mu_i^+,s_i^+\rangle=|{\mu'}_i^+,{s'}_i^+\rangle~\text{and}~|\mu_i^-,
s_i^-\rangle=|{\mu'}_i^-,{s'}_i^-\rangle$$ for $i=1,2$,
which imply that $\Lambda_+=\Lambda'_+$ and $\Lambda_-=\Lambda'_-$ (also $\epsilon=\epsilon'$ for $\O_{2n+1}(q)$ case). So $(b)$ follows.
\smallskip

The proof of $(c)$ is similar to $(b).$
\end{proof}
\smallskip

By Lemma \ref{Lem:cuspidal} and Remark \ref{bubbleempty},	we have the following corollary, which provides a characterization of $F^\pm$-cuspidal modules.
\begin{corollary} Let $\rho\in \Irr(G).$
Then 
\begin{itemize}
	\item[$(a)$] $\rho$ is $F^+$-cuspidal if and only if $\bbO^+(u)(\rho)$ is a polynomial;
	\item[$(b)$] $\rho$ is $F^-$-cuspidal if and only if $\bbO^-(v)(\rho)$ is a polynomial.
	\end{itemize} 
	In particular, for a unipotent irreducible module $\rho$, it is cuspidal if and only if $\bbO^+(u)(\rho)$ is a polynomial.
\end{corollary}

\medskip

\section{Theta correspondence and decategorification}\label{chp:thetaforclassical}
In this section, we will use the theta correspondence to
determine all colored weight functions for irreducible modules which are both $F^+$-cuspidal and $F^-$-cuspidal.
This gives a complete answer to the action of the Kac-Moody algebra $\frakg=\fraks\frakl'_{I_+}\oplus\fraks\frakl'_{I_-}$ on the Grothendieck group $[KG_\bullet\mod]$.
\smallskip

\subsection{Theta correspondence}\label{howe}\hfill\\

In this section, we will recall some basic results for the theta correspondence.
\smallskip

\subsubsection{Definition}\label{subsec:defthetacorr}
Let $V$ be a vector space over $\bbF_q$ equipped with a  non-degenerate quadratic from $\langle\,\,,\,\,\rangle_V:V\times V\to \bbF_q$ and let $G(V)$ be the isometry group of $\langle\,\,,\,\,\rangle_V.$
$V$ has a Witt decomposition $$V=V_a\oplus X\oplus Y$$
where $V_a$ is anisotropic and  its orthogonal complement $V_a^{\perp}=X\oplus Y$
consists of isotropic subspaces $X$ and $Y$. to organize such spaces into Witt towers, defined as
$$\calV:=\{V_a, V_a\oplus \bbH, V_a\oplus \bbH^{\oplus2}, \cdots\}$$
a collection of spaces formed by successively adjoining copies of the hyperbolic plane $\bbH$ to a fixed anisotropic space $V_a$. Clearly, $V$ itself is an element of $\calV$. Now, let $(V, V')$ be a dual pair. We consider the following types of related Witt towers:
\begin{itemize}
\item[$($I$)$] $V$ is an orthogonal space, and $\calV'$ is the Witt towers of symplectic spaces containing $V'$;
\item[$($II$)$] $V$ is a symplectic space, and $\calV'^+$ (resp. $\calV'^-$) is the Witt towers of even-dimensional orthogonal spaces with trivial (resp. two-dimensional) anisotropic kernel which contains $V'$;

\item[$($III$)$] $V$ is a symplectic space, and $\calV'$ is the Witt towers of odd-dimensional orthogonal spaces containing $V'$.
\end{itemize}

We fix a non-trivial character $\psi:(\bbF_q,+)\to \bbC^\times$.
Let $\omega_{G,G'}=\omega_{G,G'}^\psi$ be the Weil representation of $G(V\otimes V')$ constructed in \cite{Ger}.
It has a non-negative integral decomposition
\[
\omega_{G,G'}=\sum_{\rho\in\Irr(G), \rho'\in \Irr(G')}m_{\rho,\rho'}\rho\otimes\rho'
\]
where $\Irr(G)$ denotes the set of irreducible characters of $G$.
We say that $(\rho,\rho')$ \emph{occurs} in the \emph{theta correspondence} (or \emph{Howe correspondence}) if $m_{\rho,\rho'}\neq 0$,
i.e., there is a relation
\begin{align*}
\Theta_{G,G'}=\{\,(\rho,\rho')\in\Irr(G)\times\Irr(G')\mid m_{\rho,\rho'}\neq 0\,\}
\end{align*}
between $\Irr(G)$ and $\Irr(G')$.

\subsubsection{First occurrence}\label{sub:firstoccurrence}
%For any $V'^\dag\in \calV'$ (resp. $\calV'^+$ or $\calV'^-$), the pair $(G(V),G(V'^\dag))$ is a dual pair.
%Let $\sigma\in \Irr(G(V))$.
%For $($II$)$, let $n'^\epsilon_0$ be the minimal dimension of $V'^\dag\in \calV'^{\epsilon}$ such that $\sigma$ occurs  in the
%theta correspondence for the dual pair $(G(V),G(V'^\dag))$. For $($I$)$ and $($III$)$, let $n'_0(\sigma)$ be the minimal dimension of $V'^\dag\in \calV'$ such that $\sigma$ occurs  in the
%theta correspondence for the dual pair $(G(V),G(V'^\dag)).$
%The number $n'^\epsilon_0=n'^\epsilon_0(\sigma)$ (resp. $n'_0=n'_0(\sigma)$) is called the \emph{first occurrence index} of $\sigma$ with respect to the related Witt towers $\calV'^\epsilon$ (resp. $\calV'$).
%\smallskip

%For any quadratic space $V$, we take a basis $\{e_1,e_2,\cdots,e_n\}$ and define $$\disc(V):=(-1)^{\frac{n(n-1)}{2}}\det((\langle e_i,e_j\rangle_V)_{i,j})\in \bbF^{\times}_q/(\bbF^{\times}_q)^2.$$ The definition of $\disc$ is independent of choice of the basis.
%
%We define
%\begin{align}
%\varpi_{V'}=\begin{cases}
%(-1)^{\frac{1}{2}\dim(V')},& \text{for case $($I$)$}\\
%\zeta(\disc(V'))& \text{for case $($II$)$}\\
%\gamma_{\psi}(1)\zeta(\disc(V'))& \text{for case $($III$)$}
%\end{cases}
%\end{align}
Let $G_n$ (or $G'_n$) be one of $\O_{2n+1}(q),\Sp_{2n}(q)$ and $\O_{2n}^\pm(q)$. Assume that $(G_n, G'_{n'})$ is a dual pair and $\sigma\in \Irr(G_n)$ and $\sigma'\in \Irr(G'_{n'})$. Define
$$\Theta_{G'_{n'}(\sigma)}=\{\sigma'\in \Irr(G'_{n'})|(\sigma,\sigma')\,\,\text{occurs in } \Theta_{G_n,G'_{n'}}\}$$
and
$$\Theta_{G_{n}(\sigma')}=\{\sigma\in \Irr(G_{n})|(\sigma,\sigma')\,\,\text{occurs in } \Theta_{G_n,G'_{n'}}\}.$$
It is known that $\Theta_{G'_{n'}(\sigma)}\neq \emptyset$ implies $\Theta_{G'_{n''}(\sigma)}\neq \emptyset$ for all $n''>n'.$
We say that $(\sigma,\sigma')$ occurring in  $ \Theta_{G_n,G'_{n'}}$ is a \emph{first occurrence} if $\Theta_{G'_{n'-1}}(\sigma)=\emptyset$ and $\Theta_{G_{n-1}}(\sigma')=\emptyset.$

If $(G,G')=(\Sp_{2n}(q),\O^\epsilon_{2n}(q))$,
then the unipotent characters are preserved by $\Theta_{G,G'}$ (\cite{AMR} theorem 3.5),
i.e., we can write
\begin{align*}
\omega_{G,G',1}
&=\sum_{\rho\in\calE(G,1),\ \rho'\in\calE(G',1)}m_{\rho,\rho'}\rho\otimes\rho', \\
\Theta_{G,G',1} &=\Theta_{G,G'}\cap(\calE(G,1)\times\calE(G',1))
\end{align*}
where $\omega_{G,G',1}$ denotes the unipotent part of $\omega_{G,G'}$.
For the theta correspondence on unipotent characters,
the following theorem is well-known, due to Adams and Moy.

\begin{theorem}[\cite{AM}]\label{Thm:AM}

Let $\rho^{\Sp}_t$ be the unique unipotent cuspidal module of $\Sp_{2t(t+1)}(q)$ as defined in \S \ref{sub:unipotent} and $\pi_t^{\alpha}$ and $\pi_t^{\beta}$ (following the notation of \cite{AM}) are the unipotent cuspidal modules of $\O^{\epsilon_t}_{2t^2}(q)$, where $\epsilon_t=(-1)^t.$ Then we have the following:
\begin{itemize}

\item[$(a)$] $(\pi_t^{\alpha}, \rho_{t}^{\Sp})$ first occurs in the correspondence for the pair $(\O^{\epsilon_t}_{2t^2}(q),\Sp_{2t(t+1)}(q))$
 and
\item[$(b)$] $(\pi_t^{\beta},\rho^{\Sp}_{t-1})$ first occurs in the correspondence for the pair $(\O^{\epsilon_t}_{2t^2}(q),\Sp_{2t(t-1)}(q)).$

\end{itemize}
\end{theorem}

\begin{remark}
Recall that in \S \ref{sub:unipotent}, we label the two unipotent cuspidal modules of $\O^{\epsilon_t}_{2t^2}(q)$ by $\rho_t^{\O_{\mathrm{even}}}$ and $\rho_{-t}^{\O_{\mathrm{even}}}$. In Theorem \ref{Thm:unipotenttheta} (2), we will examine this parametrization in relation to the one introduced by Adams and Moy.
\end{remark}
%In \cite{AMR}, Aubert, Michel and Rouquier give an explicit description (in terms of partitions or bi-partitions) of the correspondence of unipotent characters for dual pairs of  $(\GL_n,\GL_{n'})$ and $(\GU_n,\GU_{n'})$ and provide a conjecture on the dual pair $(\Sp_{2n},\O^\epsilon_{2n})$. Pan proves in \cite{P1, P2} that the conjecture by Aubert, Michel and Rouquier is confirmed.
\smallskip

\subsubsection{Modified Lusztig's Jordan decomposition}\label{sub:Jordanfortheta}

Let $G$ be one of $\Sp_{2n}(q),$ $\SO_{2n+1}(q)$ and $\O^\epsilon_{2n}(q).$ For $s\in G^*,$ in \S \ref{sub:Jordan}, we define $G^{(+)}$, $G^{(-)}$ and $G^{(\star)}$ such that $C_{G^*}(s)\simeq  G^{(+)}\times G^{(-)}\times G^{(\star)}$.
We now define 
\begin{align}
\begin{split}
 \underline{G}^{(+)}
&=\begin{cases}
G^{(+)},\,&\text{ if $G$ is an orthogonal group;}\\
(G^{(+)})^*,\,&\text{ if $G$ is a symplectic group.}
\end{cases}
\end{split}
\end{align}

Then
\begin{equation}
(\underline{G}^{(+)},G^{(-)})=\begin{cases}
(\Sp_{2n^{(+)}}(q),\Sp_{2n^{(-)}}(q)), & \text{if $G=\SO_{2n+1}(q)$};\\
(\Sp_{2n^{(+)}}(q),\O^{\epsilon^{(-)}}_{2n^{(-)}}(q)), & \text{if $G=\Sp_{2n}(q)$};\\
(\O^{\epsilon^{(+)}}_{2n^{(+)}}(q),\O^{\epsilon^{(-)}}_{2n^{(-)}}(q)), & \text{if $G=\O^\epsilon_{2n}(q)$}
\end{cases}
\end{equation}
for some non-negative integers $n^{(+)},n^{(-)}$ depending on $s$,
and some $\epsilon^{(+)},\epsilon^{(-)}\in \{\pm1\}$ (see \S \ref{sub:Jordan} (\ref{equ:+-})).
\smallskip

Then a \emph{modified Lusztig's Jordan decomposition} defined in \cite{P3,P4}
\begin{equation}\label{modifiedjordan}
\underline{\mathcal{J}}_s\colon \mathcal{E}(G,s)\rightarrow\mathcal{E}(\underline{G}^{(+)}\times G^{(-)}\times G^{(\star)},1)
\end{equation}
can be written as
\begin{equation}
\underline{\mathcal{J}}_s(\rho)=\rho^{(+)}\otimes\rho^{(-)}\otimes\rho^{(\star)}
\end{equation}
where $\rho^{(\dag)}\in\mathcal{E}(G^{(\dag)},1)$ for $\dag\in \{-,\star\}$ or $\mathcal{E}(\underline{G}^{(\dag)},1)$ for $\dag=+$.
\smallskip

Similar to Lusztig correspondence, combining $\mathcal{L}_1$ for $\underline{G}^{(+)}\times G^{(-)}\times G^{(\star)}$
and the inverse of $\underline{\calJ}_s$ in (\ref{modifiedjordan}):
\[
\mathcal{S}_{\underline{G}^{(+)}}\times\mathcal{S}_{G^{(-)}}\times\mathcal{S}_{G^{(*)}}
\rightarrow \mathcal{E}(\underline{G}^{(+)}\times G^{(-)}\times G^{(\star)},1)
\rightarrow \mathcal{E}(G,s^{(+)}\times s^{(-)}\times s^{(\star)})
\]
we obtain a bijection
\begin{align}
	\begin{split}
		\underline{\mathcal{L}}_s\colon\mathcal{S}_{\underline{G}^{(+)}}\times
		\mathcal{S}_{G^{(-)}}\times\mathcal{S}_{G^{(\star)}} &\rightarrow\mathcal{E}(G,s) \\
		(\Lambda_+,\Lambda_-,\lambda_\star) &\mapsto  \rho_{\Lambda_+,\Lambda_-,\lambda_\star}.
	\end{split}
\end{align}
Such a bijection $\underline{\calL}_s$ is  called a \emph{modified Lusztig parametrization}. Since $\calE(\Sp_{2n}(q),1)$ and $\calE(\SO_{2n+1}(q),1)$ are parametrized by the same Lusztig symbols $\calS_{\SO_{2n+1}}=\calS_{\Sp_{2n}}$ via a unique way (see \cite[Proposition 4.9]{P4}),  the modified Lusztig parametrization coincides with Lusztig parametrization defined in \S \ref{sub:parametrization}.
\smallskip

\subsubsection{Pan's results}

Pan in \cite{P3}  proved the commutativity (up to a twist of the sign
character) between modified Lusztig's Jordan decomposition and theta correspondence.
\smallskip

For the symplectic-even orthogonal dual pairs, we have the following proposition.
\begin{proposition}[\cite{P3}]\label{prop:panevensmplec} \label{Prop:symmplecticeven}
Let $(G,G')=(\Sp_{2n}(q),\O^\epsilon_{2n'}(q))$,
and let $\rho\in\mathcal{E}(G,s)$ and $\rho'\in\mathcal{E}(G',s')$ for some semisimple elements $s\in G^*$
and $s'\in (G'^*)^0$.
Let $\underline{\calJ}_s,\underline{\calJ}_{s'}$ be any modified Lusztig's Jordan decomposition.
Write $\underline{\calJ}_s(\rho)=\rho^{(+)}\otimes\rho^{(-)}\otimes\rho^{(\star)}$ and
$\underline{\calJ}_{s'}(\rho')=\rho'^{(+)}\otimes\rho'^{(-)}\otimes\rho'^{(\star)}$.
Then $(\rho,\rho')$ occurs in $\Theta_{G,G'}$ if and only if the following conditions hold:
\begin{itemize}
\item $s^{(\star)}=s'^{(\star)}$ (up to conjugation), and $\rho^{(\star)}=\rho'^{(\star)}$;

\item $G^{(-)}\simeq G'^{(-)}$, and $\rho^{(-)}=\rho'^{(-)}$ or $\rho^{(-)}=\rho'^{(-)}\cdot \sgn$;

\item $(\rho^{(+)},\rho'^{(+)})$ or $(\rho^{(+)}, \rho'^{(+)}\cdot \sgn)$ occurs in $\Theta_{\underline{G}^{(+)},\underline{G}'^{(+)},1}$.
\end{itemize}
%Therefore, we have the commutative diagram
%\[
%\begin{CD}\label{comm}
%\mathcal{E}(G,s) @> \Theta^\psi_{G,G'} >> \mathcal{E}(G',s') \\
%@V \calJ_s VV @VV \calJ_{s'} V \\
%\mathcal{E}(G^{(+)}\times G^{(-)}\times G^{(\star)},1)
%@> \Theta_{\bfG^{(+)},\bfG'^{(+)},1}\otimes\id\otimes\id >> \mathcal{E}(G'^{(+)}\times G'^{(-)}\times G'^{(\star)},1)
%\end{CD}
%\]
%where $\Theta_{\bfG^{(+)},\bfG'^{(+)},1}$ denote the unipotent part of $\Theta_{\bfG^{(+)},\bfG'^{(+)}}.$
\end{proposition}

Now we  consider the  case of the symplectic-odd orthogonal dual pairs.
For $\rho'\in\Irr(\O_{2n'+1}(q))$,
since  $-\id$ is in the center of $\O_{2n'+1}(q)$,
we know that $\rho'(-\id)=\epsilon_{\rho'}\rho'(\id)$ where $\epsilon_{\rho'}=\pm1$.
It is clear that if $\rho'\in\mathcal{E}(G',s')_{\epsilon'}$, then $\epsilon_{\rho'}=\epsilon'$.
Similarly for $\rho\in\Irr(\Sp_{2n}(q))$,
we have $\rho(-\id)=\epsilon_\rho\rho(\id)$ where $\epsilon_\rho=\pm1$.

\begin{lemma}
Let $G=\Sp_{2n}(q)$ and $s\in G^*=\SO_{2n+1}(q)$ be a semi-simple element. If $\rho\in \calE (G,s)$, then $\rho(-\id)=\sp(s)\cdot \rho(\id),$ where $\sp$ is the spinor  character of $G^*.$
\end{lemma}
\begin{proof}
See the proof of \cite[Corollary 1.22]{DL} (see also \cite[Proposition 2.2.20]{GM}).
\end{proof}

\begin{proposition}[\cite{P3}]\label{prop:orthsym}
Let $(G,G')=(G(V),G(V'))=(\Sp_{2n}(q),\O_{2n'+1}(q))$,
and let $\rho\in\mathcal{E}(G,s)$ and $\rho'\in\mathcal{E}(G',s')_{\epsilon'}$ for some semisimple elements $s\in G^*$
and $s'\in (G'^0)^*$, and some $\epsilon'=\pm$.
Write $\underline{\calJ}_s(\rho)=\rho^{(+)}\otimes\rho^{(-)}\otimes\rho^{(\star)}$,
$\underline{\calJ}_{s'}(\rho'|_{G'^0})=\rho'^{(+)}\otimes\rho'^{(-)}\otimes\rho'^{(\star)}$
where $\underline{\calJ}_s,\underline{\calJ}_{s'}$ are the modified Lusztig's Jordan decompositions.
Then $(\rho,\rho')$ occurs in $\Theta_{G,G'}$ if and only if
\begin{itemize}
\item $s^{(\star)}=-s'^{(\star)}$ (up to conjugation), and $\rho^{(0)}=\rho'^{(0)}$;

\item $ \underline{G}^{(+)}\simeq G'^{(-)}$ and $\rho^{(+)}=\rho'^{(-)}$;

\item $(\rho^{(-)},\rho'^{(+)})$ or $(\rho^{(-)}\cdot\sgn ,\rho'^{(+)})$ occurs in $\Theta_{G^{(-)},\underline{G}'^{(+)},1}$;

\item $\sp(s)=\epsilon'$.
\end{itemize}

%Therefore, we have the commutative diagram
%\[
%\begin{CD}\label{comm}
%\mathcal{E}(G,s) @> \Theta^\psi_{G,G'} >> \mathcal{E}(G',s')_{\sp(s)} \\
%@V \calJ_s VV @VV \tau\circ\calJ_{s'}(\bullet|_{G'^0}) V \\
%\mathcal{E}(G^{(+)}\times G^{(-)}\times G^{(\star)},1)
%@> \id\otimes\Theta_{\bfG^{(-)},\bfG'^{(+)},1}\otimes\id >> \mathcal{E}( G'^{(-)}\times G'^{(+)}\times G'^{(\star)},1)
%\end{CD}
%\]
%where $\tau:\mathcal{E}( G'^{(+)}\times G'^{(-)}\times G'^{(\star)},1)\to \mathcal{E}( G'^{(-)}\times G'^{(+)}\times G'^{(\star)},1)$ is a bijection by flipping the first two components.
\end{proposition}
%We call the conjugacy classes $s$ and $s'$ are \emph{theta dual} to each orther if they satisfy the condition in above proposition.
\medskip

\subsection{Kac-Moody action on Grothendieck groups}\hfill\\

In this section, we will obtain a complete description of the action of Kac-Moody algebra $\fraks\frakl'_{I_+}\oplus \fraks\frakl'_{I_-}$ on the Grothendieck group of $KG_\bullet\mod$.
\smallskip

\subsubsection{Ma--Qiu--Zou's Lemma}
In \cite[\S 1.8]{MQZ}, Ma, Qiu and Zou introduced the concept of ``theta cuspidal.'' We now recall their definition. Let
\begin{align}
	\chi_{V,V'}=\begin{cases}  1& \text{ if in case $($I$)$ or $($II$)$;}\\
		\zeta& \text{ if in case $($III$)$.}
		\end{cases}
		\end{align}
For each $\sigma\in \Irr(G_n),$
define $$c(\sigma)=\max\{k\in \bbN|\Hom_{\GL_k}({^*}R^{G_n}_{L_{n-k,k}}(\sigma), \chi_{V,V'}\circ \sgn_{\GL_k})\neq 0\}$$
Clearly if $\sigma$ is a cuspidal representation of $G_n$, then $c(\sigma)=0$. In general, we call a representation $\sigma\in \Irr(G_n)$ \emph{theta cuspidal} (with respect to the dual pairs listed in \S \ref{subsec:defthetacorr})
when $c(\sigma) = 0.$ In particular, when $k=1$, we have the isomorphism of vector spaces \begin{align}\label{thetaformulaofE}
	\Hom_{\GL_1}({^*}R^{G_n}_{L_{n-1,1}}(\sigma), \chi_{V,V'})\cong E^{\epsilon}(\sigma)
	\end{align} which follows from the biadjointness of the functors $(E^\pm,F^\pm)$ and  isomorphisms (\ref{formulaofF}) in \S \ref{subsec:functors}, where $\epsilon=+$ in the case $($I$)$ and $($II$)$ and $\epsilon=-$ in the case $($III$).$
We will show in the following lemma that the notion of theta-cuspidal coincides with that of $F^+$-cuspidal or $F^-$-cuspidal.
\begin{lemma}\label{Lem:thetacuspidal}
\begin{itemize}
\item[$(a)$]For cases $(${\rm I}$)$ and $(${\rm II}$)$, $\sigma\in \Irr(G(V))$ is theta cuspidal if and only if it is $F^+$-cuspidal.
\item[$(b)$]
For case $(${\rm III}$)$, $\sigma\in \Irr(G(V))$ is theta cuspidal if and only if it is $F^-$-cuspidal.
\end{itemize}
\end{lemma}
\begin{proof}This follows from the definitions  and the fact that $\GL_n(q)$ has no irreducible quadratic unipotent cuspidal module unless $n=1$. Indeed, by isomorphisms (\ref{thetaformulaofE}), if $\sigma$ is theta cuspidal, then $c(\sigma)=0$, which implies $$\label{formulaofE}\Hom_{\GL_1}({^*}R^{G_n}_{L_{n-1,1}}(\sigma), \chi_{V,V'})\cong E^{\epsilon}(\sigma)=0,$$ hence $\sigma$ is $F^\epsilon$-cuspidal. Conversely, if $\sigma$ is $F^\epsilon$-cuspidal, meaning that  $E^\epsilon(\sigma)=0$, we have \begin{align*}
0=(E^\epsilon)^k(\sigma)&\cong\Hom_{\GL_1^{\times k}}({^*}R^{G_n}_{L_{n-k,1^k}}(\sigma), \chi_{V,V'}^{\otimes k})\\&\cong \Hom_{\GL_k}({^*}R^{G_n}_{L_{n-k,k}}(\sigma), R_{\GL_1^{\times k}}^{\GL_k}(\chi_{V,V'}^{\otimes k})),
\end{align*}
for all $k>0.$ Here, the last isomorphism follows from the transitivity of Harish-Chandra restriction functors $${^*}R^{G_n}_{L_{n-k,1^k}}\cong {^*}R^{L_{n-k,k}}_{L_{n-k,1^k}}\circ {^*}R^{G_n}_{L_{n-k,k}}$$ and Frobenius reciprocity.
On the other hand, since the quadratic unipotent character $\chi_{V,V'}\circ \sgn_{\GL_k}$ is a constituent of $R_{\GL_1^{\times k}}^{\GL_k}(\chi_{V,V'}^{\otimes k})$, this forces $$\Hom_{\GL_k}({^*}R^{G_n}_{L_{n-k,k}}(\sigma), \chi_{V,V'}\circ \sgn_{\GL_k})=0$$ for all $k> 0$.
Consequently, $c(\sigma)=0,$ i.e., $\sigma$ is theta cuspidal.
\end{proof}

\begin{lemma}[\cite{MQZ}]\label{Lemma:MQZ}
 Let $\sigma\in \Irr(G(V))$ be theta cuspidal for a relative Witt towers $\calV'$(or $\calV'^\pm$).
\begin{itemize}
\item[$(1)$] In case $(${\rm I}$)$,  assume that $\sigma'\in \Irr(G(V'))$ (resp.
$\widetilde{\sigma}'\in \Irr(G(\widetilde{V}')) $) is the first occurrence of $\sigma$
(resp. of $\sigma\cdot \sgn$) with respect to relative Witt towers $\calV'$, then
$X^+(\sigma)$ acts on $F^+(\sigma)$ by the following quadratic relation
$$(X^+(\sigma)-\epsilon_{V'}\cdot q^{\frac{\dim(V')}{2}-\lfloor\frac{\dim (V)}{2}\rfloor})
(X^+(\sigma)+\epsilon_{V'}\cdot q^{\frac{\dim(\widetilde{V}')}{2}-\lfloor\frac{\dim (V)}{2}\rfloor})=0,$$
where $\epsilon_{V'}=\zeta(-1)^{\frac{\dim(V')}{2}}.$
\item[$(2)$] In case $(${\rm II}$)$, assume that $\sigma'^+\in \Irr(G(V'^+))$
(resp. $\sigma'^-\in \Irr(G(V'^-)) $) is the first occurrence of $\sigma$
with respect to relative Witt towers $\calV'^+$ (resp. $\calV'^-$),
then $X^+(\sigma)$ acts on $F^+(\sigma)$ by the following quadratic relation
$$(X^+(\sigma)-q^{\frac{\dim (V'^+)-\dim (V)}{2}})
(X^+(\sigma)+q^{\frac{\dim (V'^-)-\dim (V)}{2}})=0.$$

\item[$(3)$] In case $(${\rm III}$)$,
let $\alpha=\gamma_\psi/\sqrt{\zeta(-1)q}$,
where $\gamma_\psi=\Sigma_{x\in \bbF_q}\psi(\frac{1}{2}x^2)$ is the Gauss sum.
Assume that $\sigma'\in \Irr(G(V'))$ (resp. $\widetilde{\sigma}'\in \Irr(G(\widetilde{V}')) $) is the first occurrence of $\sigma$ (resp. of $\sigma^c$) with respect to relative Witt towers $\calV'$, then $X^-(\sigma)$ acts on $F^-(\sigma)$ by the following quadratic relation
    $$(X^-(\sigma)-\alpha\cdot q^{\lfloor\frac{\dim (V')}{2}\rfloor-\frac{\dim(V)}{2}})
    (X^-(\sigma)+\alpha\cdot q^{\lfloor\frac{\dim (\widetilde{V}')}{2}\rfloor-\frac{\dim(V)}{2}})=0.$$

\end{itemize}
\end{lemma}
\begin{proof}

This lemma is exactly a special case of \cite[Lemmas 3.4, 3.5 or Proposition 3.6]{MQZ} with $l=1.$ Note that in \cite{MQZ}, they  use modified Weil representation $\omega_{V,V'}$ rather than $\omega^\psi_{G,G'}$ defined in \cite{Ger}:
\begin{align}
\omega_{V,V'}=\begin{cases}
((\zeta\circ\det_{G(V)})^{\frac{1}{2}\dim(V')}\otimes 1_{G(V')})\otimes \omega^\psi_{G,G'}& \text{in case $($I$)$;}\\
(1_{G(V)}\otimes (\zeta\circ\det_{G(V)})^{\frac{1}{2}\dim(V)})\otimes \omega^\psi_{G,G'}& \text{in cases $($II$)$ and $($III$)$,}
\end{cases}
\end{align}
so for case $(1)$, there is an extra sign $\epsilon_{V'}.$
\smallskip

Assume $G_n=G(V)$. For any $F^\epsilon$-cuspidal module $\sigma\in \Irr(KG_n)$, the evaluation map $\phi_{F^\epsilon}:\End(F^\epsilon)\to \End_{KG_\bullet\mod}(F^\epsilon(\sigma))=\End_{KG_{n+1}}(R_{L_{n,1}}^{G_{n+1}}(\sigma\otimes \zeta_{\epsilon}))$ is surjective and  $X^\epsilon\in \End(F^\epsilon)$ maps to $X^\epsilon(\sigma)\in \End_{KG_{n+1}}(R_{L_{n,1}}^{G_{n+1}}(\sigma\otimes \zeta_{\epsilon}))$, where $\zeta_{\epsilon}=1$ if $\epsilon=+$ and $\zeta_{\epsilon}=\zeta$ if $\epsilon=-$.
Then $X^\epsilon(\sigma)$ is identified with a scalar multiple of ``unnormalized Hecke operator'' $T_{l}$ defined in \cite[\S3.2]{MQZ} for the case $l=1$. In fact,
\begin{align}
X^\epsilon(\sigma)=\begin{cases} q^{-\lfloor\frac{1}{2}\dim(V)\rfloor}T_{l}& \text{in cases $($I$)$ and $($II$)$ with $\epsilon=+$; }\\
\sqrt{\zeta(-1)}\,q^{-\frac{\dim(V)+1}{2}}T_{l}& \text{in case $($III$)$ with $\epsilon=-$.}
\end{cases}
\end{align}
 Then the lemma follows from \cite[Proposition 3.6]{MQZ} immediately.
\end{proof}
\begin{remark}Note that the Gauss sum $\gamma_{\psi}=\pm \sqrt{\zeta(-1)q}$, and thus
$\alpha\in\{\pm 1\}$.
Consequently, Lemma \ref{Lemma:MQZ} and Proposition \ref{Prop:extrasymm} together imply that the sets $I_+$ and $I_-$ are both equal to $q^\bbZ \sqcup -q^\bbZ$, which aligns with the result stated in Proposition \ref{Prop:I+-}.
\end{remark}

\subsubsection{Unipotent case}

Before dealing with the general case, we first deal with the unipotent case.
\begin{theorem}\label{Thm:unipotenttheta}
\begin{itemize}
\item[$(a)$]
If $G_\bullet=\Sp_{2\bullet}(q),$ then
  $$\bbO^+(u)(\rho^{\Sp}_{t})=(u-(-q)^{t})(u-(-q)^{-1-t}).$$
\item[$(b)$]
Let $G_\bullet=\O^\pm_{2\bullet}(q)$ and $t\in \bbN $, 
then
 $$\rho^{\O_{\mathrm{even}}}_{t}=\pi_{t}^{\beta}\quad \text{and}\,\,\quad \rho^{\O_{\mathrm{even}}}_{-t}=\pi_{t}^{\alpha}.$$

\end{itemize}
\end{theorem}
\begin{proof}$(a)$
We consider the symplectic-even orthogonal dual pair. Write $\epsilon_k=(-1)^k$. Let $\sigma=\rho^{\Sp}_{t}\in \Irr(\Sp_{2(t^2+t)}(q))$ be the unique irreducible unipotent cuspidal module.
By Theorem \ref{Thm:AM}, we know that $\sigma$ first occurs in the correspondences
for both dual pairs $$(\Sp_{2(t^2+t)}(q), \O^{\epsilon_t}_{2t^2}(q))\quad \text{and}\,\,
\quad
(\Sp_{2(t^2+t)}(q), \O^{\epsilon_{t+1}}_{2(t+1)^2}(q)).$$
By Lemma \ref{Lemma:MQZ}$(2)$, we know that $X^+(\sigma)$ acts on $F^+(\sigma)$ satisfying
 $$(X^+(\sigma)-(-q)^{t})(X^+(\sigma)-(-q)^{-1-t})=0.$$ Then by Lemma \ref{Lemma:minimalpoly},
 $(a)$ follows.

 \smallskip

$(b)$  We consider even orthogonal-symplectic dual pair. Let $$\sigma=
\pi^{\beta}_{t}\in \Irr(\O^{\epsilon_t}_{2t^2}(q)),$$
then $\sigma\cdot \sgn=\pi^{\alpha}_{t}$.
We know that $\sigma=\pi^{\beta}_{t}$ first occurs in the correspondence
for  dual pair $(\O^{\epsilon_t}_{2t^2}(q),\Sp_{2t(t-1)}(q))$
and $\sigma\cdot\sgn=\pi^{\alpha}_{t}$ first occurs in the correspondence
for  dual pair $(\O^{\epsilon_t}_{2t^2}(q),\Sp_{2t(t+1)}(q)).$
Note that $\frac{\dim(V')}{2}$ is an even number, so $\epsilon_{V'}=\zeta(-1)^{\frac{\dim(V')}{2}}=1.$
By Lemma \ref{Lemma:MQZ} $(1)$, we know that $X^+(\sigma)$ acts on $F^+(\sigma)$ satisfying
 $$(X^+(\sigma)-q^{t})(X^+(\sigma)+q^{-t})=0.$$
 Then by Lemma \ref{Lemma:minimalpoly},
 we know $\pi_{t}^{\beta}=\rho^{\O_{\mathrm{even}}}_{t}$. A similar argument implies  $\pi_{t}^{\alpha}=\rho^{\O_{\mathrm{even}}}_{-t}$.
\end{proof}
\begin{remark}
Note that, for the case of unipotent cuspidal modules of $\Sp_{2n}(q)$,
this result coincides with \cite[Theorem 6.5]{DVV2}.
The above theorem gives a new proof of characteristic 0. Using unitary-unitary theta correspondence \cite[Proposition 3.6]{MQZ}, a similar argument also applies to the unitary groups $\GU_n(q)$ case, which gives a new proof of \cite[Theorem 4.12]{DVV}.
\end{remark}
\begin{remark}\label{connections}Conversely, if we know the colored weight function $\bbO^+(u)(\rho)$ for unipotent cuspidal modules $\rho$, then by Lemma \ref{Lemma:MQZ}, we can determine the two first occurrences of $\rho$ via $\bbO^+(u)(\rho)$. In particular, Dudas, Varagnolo and Vasserot \cite[Theorem 6.5]{DVV2} proved that for symplectic groups $\Sp_{2n}(q),$  we have $\bbO^+(u)(\rho^{\Sp}_{t})=(u-(-q)^{t})(u-(-q)^{-1-t})$,  then by Lemma \ref{Lemma:MQZ}, one can show that $\rho^{\Sp}_{t}$ first occurs in the theta correspondence for the dual pairs $(\Sp_{2t(t+1)}(q),\O^{\epsilon_{t+1}}_{2(t+1)^2}(q))$ and $(\Sp_{2t(t+1)}(q),\O^{\epsilon_t}_{2t^2}(q)),$ which gives a new proof of Theorem \ref{Thm:AM}. Moverover, by \cite[Theorem 1.3]{MQZ}, one can recover the symplectic-even orthogonal theta correspondence for unipotent characters.
\end{remark}

\subsubsection{General case}

Now we compute $\bbO^+(u)(-)$ and $\bbO^-(v)(-)$ for all irreducible modules which are both
$F^+$-cuspidal and $F^-$-cuspidal.% For $\rho\in\Irr(G)$ and $x\in G$, we define $\omega_\rho(x):=\frac{\rho(x)}{\rho(1)}.$
%\begin{lemma}
%Suppose that $(\rho, \rho')\in \Theta_{G,G'}$ then we have
%$$\omega_\rho(-1)=\omega_{\rho'}(-1).$$
%\end{lemma}
%\begin{proof}
%This lemma is just \cite[Proposition 8.7]{P4}.
%\end{proof}
%\begin{lemma}
%Let $G=\Sp_{2n}(q)$ and $s\in G^*=\SO_{2n+1}(q)$ be a semi-simple element.
%If $\rho\in \calE (G,s)$, then $\omega_\rho(-1)=\sp(s),$ where $\sp$ is the spinor norm character of $G^*.$
%\end{lemma}
\smallskip

We first introduce a lemma to compute the spinor character for semisimple elements of orthogonal groups. Let $G=G(V)$ be an orthogonal group and $s\in G$ be a semisimple element. Let $s=\prod\limits_{\Gamma\in \calF}s_\Gamma$ be a primary decomposition of $s$. Recall that we denote by $\eta_\Gamma=\eta_\Gamma(s)$ the type of $V_\Gamma=V_\Gamma(s)$, $\delta_\Gamma$ the reduced degree of $\Gamma$, and $m_\Gamma=m_\Gamma(s)$ the multiplicity of $(\Gamma)$ in $s_\Gamma$. Suppose that  $\alpha$ is a root of $\Gamma$. Let $\sigma:\calF_1\cup\calF_2\to \{\pm1\}$ be defined by
\begin{align*}
\sigma(\Gamma)=\begin{cases}1&\text{if $\alpha\in (\bbF^\times_{q^{\delta_\Gamma}})^2$ and $\Gamma\in \calF_1$};\\
1&\text{if $\alpha^{(q^{\delta_\Gamma}+1)/2}=1$ and $\Gamma\in \calF_2$};\\
-1&\text{otherwise}.
\end{cases}
\end{align*}
\begin{lemma}\label{spinor}Let $G=G(V)$ be an orthogonal group. If $s\in G^0,$ we have
\begin{align}
\sp(s)=\begin{cases}\zeta(-1)^{\frac{\dim (V_{x+1})}{2}}\cdot \prod\limits_{\Gamma\in \calF_1\cup\calF_2,\sigma(\Gamma)\neq 1}(-1)^{m_\Gamma(s)} &\text{if  $\eta_{x+1}(s)=1$};\\
-\zeta(-1)^{\frac{\dim (V_{x+1})}{2} }\cdot\prod\limits_{\Gamma\in \calF_1\cup\calF_2,\sigma(\Gamma)\neq 1}(-1)^{m_\Gamma(s)}&\text{if  $\eta_{x+1}(s)=-1$}.
\end{cases}
\end{align}
\end{lemma}
\begin{proof}This is \cite[Proposition 16.30]{CE04}.
\end{proof}

 Now we define some constants that occur in colored weight functions.
Let $\rho=\rho_{t_+,t_-,\lambda_\star,(\epsilon)}$ be both $F^+$-cuspidal and $F^-$-cuspidal, where $\lambda_\star=\prod\limits_{\Gamma\in \calF_1\cup \calF_2}\lambda_\Gamma.$
 For $G=\Sp_{2n}(q)$ and $\rho=\rho_{t_+,t_-,\lambda_\star},$ we define
 \begin{align}\eta(\rho)=\eta_{t_+,t_-,\lambda_\star}=
 (-1)^{t_++t_-}\cdot \prod\limits_{\Gamma\in \calF_1}(-1)^{|\lambda_\Gamma|}.
 \end{align}
For $G=\O_{2n+1}(q)$ and $\rho=\rho_{t_+,t_-,\lambda_\star,\epsilon}$, we define
\begin{align}\eta^+(\rho)=\eta^+_{t_+,t_-,\lambda_\star,\epsilon}=
\epsilon\cdot(-1)^{t_+}\cdot\prod_{\Gamma\in \calF_1\cup\calF_2,\sigma(\Gamma)\neq1}(-1)^{|\lambda_\Gamma|},
\end{align} and
\begin{align}
\eta^-(\rho)=\eta^-_{t_+,t_-,\lambda_\star,\epsilon}=
\epsilon\cdot (-1)^{t_-}\cdot\prod_{\Gamma\in \calF_1\cup\calF_2,\sigma(\Gamma)\neq 1}(-1)^{|\lambda_\Gamma|}.
\end{align}

Note that $\eta(\rho),\eta^\delta(\rho)\in \{\pm1\}$ for $\delta\in \{\pm\}.$ Then we have the following theorem.

\begin{theorem}\label{Thm:weightfunctions}
Let $\rho=\rho_{t_+,t_-,\lambda_\star,(\epsilon)}$ be both $F^+$-cuspidal and $F^-$-cuspidal.
\begin{itemize}

\item[$(a)$]
If $G_\bullet=\Sp_{2\bullet}(q),$ then
 \begin{itemize}
\item[$(a1)$] $\bbO^+(u)(\rho_{t_+,t_-,\lambda_\star})=(u-\eta_{t_+,t_-,\lambda_\star}\cdot q^{t_+})
    (u+\eta_{t_+,t_-,\lambda_\star}\cdot q^{-1-t_+}),$
\item[$(a2)$]$\bbO^-(v)(\rho_{t_+,t_-,\lambda_\star})=
(v-q^{t_-})(v+q^{-t_-}).$
\end{itemize}
\item[$(b)$]
If $G_\bullet=\O^\pm_{2\bullet}(q)$, then
\begin{itemize}
\item[$(b1)$] $\bbO^+(u)(\rho_{t_+,t_-,\lambda_\star})=(u-
    q^{t_+})(u+q^{-t_+}),$
\item[$(b2)$]$\bbO^-(v)(\rho_{t_+,t_-,\lambda_\star})=(v-q^{t_-})(v+q^{-t_-}).$
\end{itemize}

\item[$(c)$]If $G_\bullet=\O_{2\bullet+1}(q)$, then
\begin{itemize}
\item[$(c1)$] $\bbO^+(u)(\rho_{t_+,t_-,\lambda_\star,\epsilon})=
    (u-\eta^+_{t_+,t_-,\lambda_\star,\epsilon}\cdot q^{t_+})(u+ \eta^+_{t_+,t_-,\lambda_\star,\epsilon}\cdot q^{-1-t_+}),$
\item[$(c2)$]$\bbO^-(v)(\rho_{t_+,t_-,\lambda_\star,\epsilon})=
(v-\eta^-_{t_+,t_-,\lambda_\star,\epsilon}\cdot q^{t_-})(v+ \eta^-_{t_+,t_-,\lambda_\star,\epsilon}\cdot q^{-1-t_-}).$
\end{itemize}
\end{itemize}
\end{theorem}
\begin{proof}
We first prove $(a1)$
by using the theta correspondence of dual pair  $(G,G')=(\Sp_{2n}(q),\O^\pm_{2n'}(q)).$
 Let $(\rho, \rho'^+)$ (resp.  $(\rho,\rho'^-)$) be the first occurrence of dual pair $(G,G')=(G(V), G(V'^+))$ (resp. $(G(V), G(V'^-))$) with respect to $\calV'^+$ (resp. $\calV'^-$). Assume that $\rho\in \calE(G,s)$ and $\underline{\calJ}_s(\rho)=\rho^{(+)}\otimes \rho^{(-)}\otimes \rho^{(\star)}.$
 Then $s\in (G^*)^0=G_0(V^*)$ determines a decomposition $$V^*=V^*_{x-1}\oplus V^*_{x+1}\oplus\bigoplus_{\Gamma\in \calF_1 \cup \calF_2}V^*_{\Gamma}$$ such that $$(\underline{G}^{(+)})^*=G_0(V^*_{x-1}),\quad G^{(-)}=G(V^*_{x+1}),\quad G^{(\star)}=\prod\limits_{\Gamma\in \calF_1 \cup \calF_2}C_{G(V^*_{\Gamma})}(s_\Gamma),$$
 where $V^*_{x-1}$ and $V^*_{x+1}$ are both orthogonal spaces.
 \smallskip

 Let $V_{x-1}$ be a symplectic space such that $G(V_{x-1})=\underline{G}^{(+)}=G_0(V^*_{x-1})^*.$
Then we have $\dim(V_{x-1})=\dim(V^*_{x-1})-1.$
\smallskip

Similarly, for $\epsilon\in \{\pm\}$, assume that $\rho'^\epsilon\in \calE(G',s')$ and  $\underline{\calJ}_{s'}(\rho'^\epsilon)=(\rho'^\epsilon)^{(+)}\otimes (\rho'^\epsilon)^{(-)}\otimes (\rho'^\epsilon)^{(\star)}.$ Then $s' \in G'^*=G((V'^\epsilon)^*)$ determines the decomposition $$(V'^\epsilon)^*=(V'^\epsilon)^*_{x-1}\oplus (V'^\epsilon)^*_{x+1}\oplus\bigoplus_{\Gamma\in \calF_1 \cup \calF_2}(V'^\epsilon)^*_{\Gamma}$$
such that $$\underline{G}'^{(+)}=G((V'^\epsilon)^*_{x-1}),\quad G'^{(-)}=G((V'^\epsilon)^*_{x+1}),\quad G'^{(\star)}=\prod\limits_{\Gamma\in \calF_1 \cup \calF_2}C_{G((V'^\epsilon)^*_{\Gamma})}(s'_{\Gamma}),$$
where $(V'^\epsilon)_{x-1}^*$ and $(V'^\epsilon)_{x+1}^*$ are both orthogonal spaces.
\smallskip

In fact, by \cite[Lemma 3.15]{P4} and Theorem \ref{Thm:Panuniformproj},
we have $\underline{\calJ}_s(\rho)=\rho^{\Sp}_{t_+}\otimes\rho^{\O_{\mathrm{even}}}_{\pm t_-}\otimes \rho_{\lambda_{\star}}$, i.e., $\rho^{(+)}=\rho^{\Sp}_{t_+}$ and $\rho^{(-)}=\rho^{\O_{\mathrm{even}}}_{ t_-}$ or $\rho^{\O_{\mathrm{even}}}_{- t_-}$ and $\rho^{(\star)}=\rho_{\lambda_{\star}}.$
By the commutativity of theta correspondence and modified Lusztig's Jordan decomposition (see Proposition \ref{Prop:symmplecticeven}), we must have \begin{itemize}
\item  $V^*_{\Gamma}\cong (V'^\epsilon)^*_{\Gamma}$,  $s_\Gamma=s'_\Gamma$ for any $\Gamma\in \calF_1 \cup \calF_2$ and $\rho^{(\star)}=(\rho'^\epsilon)^{(\star)}$;

\item $V^*_{x+1}\cong (V'^\epsilon)^*_{x+1}$ and $\rho^{(-)}=(\rho'^\epsilon)^{(-)}$ or $\rho^{(-)}=(\rho'^\epsilon)^{(-)}\cdot \sgn$;

\item $\rho^{(+)}\otimes(\rho'^\epsilon)^{(+)}$ or $\rho^{(+)}\otimes ((\rho'^\epsilon)^{(+)}\cdot \sgn)$ first occurs in $\omega_{\underline{G}^{(+)},\underline{G}'^{(+)},1}$ for the symplectic-even orthogonal theta dual pair $(V_{x-1},(V'^\epsilon)^*_{x-1})$.
\end{itemize}
In other words, we must have
\begin{itemize}
\item $(\rho'^\epsilon)^{(\star)}=\rho_{\lambda_{\star}}$,
\item $(\rho'^\epsilon)^{(-)}=\rho^{\O_{\mathrm{even}}}_{\pm t_-}$,
\item if $\eta((V'^\epsilon)^*_{x-1})=\epsilon_{t_+}$, then $(\rho'^\epsilon)^{(+)}=\rho^{\O_{\mathrm{even}}}_{\pm t_+}$ and $(\rho'^{-\epsilon})^{(+)}=\rho^{\O_{\mathrm{even}}}_{\pm (t_+-1)}$ and if $\eta((V'^\epsilon)^*_{x-1})=\epsilon_{t_+-1}$, then
    $(\rho'^\epsilon)^{(+)}=\rho^{\O_{\mathrm{even}}}_{\pm (t_+-1)}$ and $(\rho'^{-\epsilon})^{(+)}=\rho^{\O_{\mathrm{even}}}_{\pm t_+}$,
    \end{itemize} where $\epsilon_{k}=(-1)^{k}$ for $k=t_+$ and $t_+-1$.
\smallskip

Note that $\rho^{\O_{\mathrm{even}}}_{\pm t_-}\in
\Irr(\O^{\epsilon_{t_-}}_{2(t_-)^2}),$
$\rho^{\O_{\mathrm{even}}}_{\pm t_+}\in
\Irr(\O^{\epsilon_{t_+}}_{2(t_+)^2})$ and $\rho^{\O_{\mathrm{even}}}_{\pm(t_+-1)}
\in \Irr(\O^{\epsilon_{t_+-1}}_{2(t_+-1)^2}).$
Morover, we have $\epsilon=\eta(V'^\epsilon)=\eta((V'^\epsilon)^*)$, which is equal to $$\eta((V'^\epsilon)^*_{x-1})\eta((V'^\epsilon)^*_{x+1})
\eta(\bigoplus_{\Gamma\in \calF_1\cup\calF_2}(V'^\epsilon)^*_\Gamma).$$
Since $\eta((V'^\epsilon)^*_{x-1})=\epsilon_{t_+}$ or $\epsilon_{t_+-1}$, $\eta((V'^\epsilon)^*_{x+1})=\epsilon_{t_-}=(-1)^{t_-}$ and $$\eta(\bigoplus_{\Gamma\in \calF_1\cup\calF_2}(V'^\epsilon)^*_\Gamma)=\prod\limits_{\Gamma\in \calF_1}(-1)^{m_\Gamma(s')}=\prod\limits_{\Gamma\in \calF_1}(-1)^{m_\Gamma(s)}=\prod\limits_{\Gamma\in \calF_1}(-1)^{|\lambda_\Gamma|},$$
 we know that if $\eta((V'^\epsilon)^*_{x-1})=\epsilon_{t_+}$, then $\epsilon=\eta_{t_+,t_-,\lambda_\star},$ and if $\eta((V'^\epsilon)^*_{x-1})=\epsilon_{t_+-1}$, then $\epsilon=\eta_{t_+-1,t_-,\lambda_\star}.$
\smallskip

 On the other hand, we know  that  $\dim(V^*_{x+1})=\dim((V'^\epsilon)^*_{x+1})$ and $\dim(\bigoplus\limits_{\Gamma\in \calF_1\cup\calF_2} V^*_{\Gamma})=\dim(\bigoplus\limits_{\Gamma\in \calF_1\cup\calF_2}(V'^\epsilon)^*_\Gamma)$, so
 \begin{align*}
\dim(V)-\dim(V'^\epsilon)&=\dim(V^*)-1-\dim((V'^\epsilon)^*)\\
 &=\big(\dim(V^*_{x-1})+\dim(V^*_{x+1})+\sum\limits_{\Gamma\in \calF_1\cup\calF_2}\dim(V^*_\Gamma)\big)-1\\
 &-\big(\dim((V'^\epsilon)^*_{x-1})+\dim((V'^\epsilon)^*_{x+1})+
 \sum\limits_{\Gamma\in \calF_1\cup\calF_2}\dim((V'^\epsilon)^*_\Gamma)\big)\\
 &=\dim(V^*_{x-1})-\dim((V'^\epsilon)^*_{x-1})-1\\
 &=\dim(V_{x-1})-\dim((V'^\epsilon)^*_{x-1}).
 \end{align*}
Then  $(a1)$ follows from Lemma \ref{Lemma:MQZ} $(2)$.
 \smallskip

Note that $(a2)$ and $(b)$  hold true by the definitions.
\smallskip

$(c)$  By extra symmetries (see Proposition \ref{Prop:extrasymm}),  we have
 $$\bbO^-(v)(\rho_{t_+,t_-,\lambda_\star,\epsilon})=
 \bbO^+(v)(\sp\cdot\rho_{t_+,t_-,\lambda_\star,\epsilon})=
 \bbO^+(v)(\rho_{t_-,t_+,(-\lambda)_\star, \sp(-\id)\cdot\epsilon}),$$
  where $(-\lambda)_{\star}$ is defined by $(-\lambda)_\star=\prod\limits_{\Gamma\in \calF_1\cup\calF_2}\lambda_{-\Gamma}$ and $-\Gamma$ is defined by $(-\Gamma)(x)=(-1)^{\deg(\Gamma)}\cdot \Gamma(-x)$.
  By definition, $$\eta^+_{t_-,t_+,(-\lambda)_{\star},\sp(-\id)\cdot \epsilon}
  =\sp(-\id)\cdot \epsilon\cdot (-1)^{t_-}\cdot\prod_{\Gamma\in \calF_1\cup\calF_2,\sigma(-\Gamma)\neq1}(-1)^{|\lambda_\Gamma|}.$$
Assume that $\rho\in \calE(G,s)$ and $\underline{\calJ}_s(\rho|_{G^0})=\rho^{(+)}\otimes \rho^{(-)}\otimes \rho^{(\star)}.$ Then $s\in (G^*)^0=G(V^*)$ determines a decomposition \begin{align}\label{decomp}V^*=V^*_{x-1}\oplus V^*_{x+1}\oplus\bigoplus_{\Gamma\in \calF_1 \cup \calF_2}V^*_{\Gamma}.\end{align}
We have that $\dim(V^*)=\dim(V^*_{x-1})+\dim(V^*_{x+1})+\sum\limits_{\Gamma\in \calF_1\cup\calF_2}\dim(V^*_{x+1})$, $\dim(V^*_{x-1})=2t_+(t_++1)$, $\dim(V^*_{x+1})=2t_-(t_-+1)$ and $\dim(V^*_\Gamma)=2\delta_\Gamma m_\Gamma(s)=2\delta_\Gamma|\lambda_\Gamma|,$
so
$$\sp(-\id)=\zeta(-1)^{\frac{\dim{V}-1}{2}}=
\zeta(-1)^{\frac{\dim{V^*}}{2}}=\prod_{\Gamma\in \calF_1\cup\calF_2}\zeta(-1)^{\delta_\Gamma |\lambda_\Gamma|}.$$
On the other hand, under the decomposition (\ref{decomp}), $s=s^{(+)}\times s^{(-)}\times s^{(\star)}$. We consider $s^{(\star)}\in \O^\pm((V^*)^{(\star)})$, where $(V^*)^{(\star)}=\bigoplus_{\Gamma\in \calF_1\cup\calF_2}V^*_{\Gamma}$,  then we have $\sp(s^{(\star)})=\sp(-s^{(\star)})\sp(-\id_{(V^*)^{(\star)}})$. So we have $$\prod_{\Gamma\in \calF_1\cup\calF_2,\sigma(-\Gamma)\neq 1}(-1)^{|\lambda_\Gamma|}=\prod_{\Gamma\in \calF_1\cup\calF_2,\sigma(\Gamma)\neq 1}(-1)^{|\lambda_\Gamma|}\cdot \prod_{\Gamma\in \calF_1\cup\calF_2}\zeta(-1)^{\delta_\Gamma |\lambda_\Gamma|},$$
i.e., $\eta^-_{t_+,t_-,\lambda_{\star}, \epsilon}=\eta^+_{t_-,t_+,(-\lambda)_{\star},\sp(-\id)\cdot \epsilon}.$
Hence if we can prove $(c1)$, then $(c2)$
 follows by the definition.
  \smallskip

Now we prove $(c1)$. We use the theta correspondence of dual pair $(G,G')=(\O_{2n+1}(q),\Sp_{2n'}(q)).$
Recall that $s\in (G^*)^0=G(V^*)$ determines a decomposition
$$V^*=V^*_{x-1}\oplus V^*_{x+1}\oplus\bigoplus_{\Gamma\in \calF_1\cup\calF_2}V^*_{\Gamma}$$
such that $$G^{(+)}=G(V^*_{x-1}),\quad G^{(-)}=G(V^*_{x+1}),\quad G^{(\star)}=\prod\limits_{\Gamma\in \calF_1\cup\calF_2}C_{G(V^*_{\Gamma})}(s_\Gamma),$$
where both $V^*_{x-1}$ and $V^*_{x+1}$ are symplectic spaces.
\smallskip

Suppose that $(\rho,\rho')$ is the first occurrence of dual pair $(G,G')=(G(V),G(V'))$ with respect to $\calV'.$ Similarly, assume that $\rho'\in \calE(G',s')$ and  $\underline{\calJ}_{s'}(\rho')=\rho'^{(+)}\otimes \rho'^{(-)}\otimes \rho'^{(\star)}.$ Then $s' \in G'^*=G_0(V'^*)$ determines the decomposition
$$V'^*=V'^*_{x-1}\oplus V'^*_{x+1}\oplus\bigoplus_{\Gamma\in \calF_1\cup\calF_2}V'^*_{\Gamma}$$
where $V'^*_{x-1}$ is an odd-dimensional orthogonal space and $V'^*_{x+1}$ is an even-dimensional orthogonal space.
Let $V'_{x-1}$ be the symplectic space such that $G(V'_{x-1})$ is the dual group of $G_0(V'^*_{x-1})$, i.e., $G_0(V'^*_{x-1})^*=G(V'_{x-1}).$
By proposition \ref{prop:orthsym},  we must have \begin{itemize}
\item  $V^*_{\Gamma}\cong V'^*_{-\Gamma}$ and $s_\Gamma=-s'_{-\Gamma}$ for any $\Gamma\in \calF_1 \cup \calF_2$ and $\rho^{(\star)}=\rho'^{(\star)}$ where $-\Gamma$ is defined by $(-\Gamma)(x)=(-1)^{\deg(\Gamma)}\cdot\Gamma(-x)$;

\item $V^*_{x+1}\cong V'_{x-1}$ and $\rho^{(-)}=\rho'^{(+)}$;

\item $\rho^{(+)}\otimes\rho'^{(-)}$ or $\rho^{(+)}\otimes \rho'^{(-)}\cdot \sgn$ first occurs in $\omega_{G^{(+)},G'^{(-)},1}$ for the symplectic-even orthogonal theta dual pair $(V^*_{x-1},V'^*_{x+1})$,
\item $\epsilon=\sp(s').$
\end{itemize}
Let $(\rho\cdot\sgn,\widetilde{\rho}')$ be the first occurrence of dual pair $(G(V),G(\widetilde{V}'))$ with respect to $\calV'.$ So we must have
\begin{itemize}
\item $\rho'^{(\star)}=\widetilde{\rho}'^{(\star)}=\rho_{(-\lambda)_{\star}}$, where $(-\lambda)_{\star}$ is defined by $(-\lambda)_\star=\prod\limits_{\Gamma\in \calF_1\cup\calF_2}\lambda_{-\Gamma},$
\item $\rho'^{(+)}=\widetilde{\rho}'^{(+)}=\rho^{\Sp}_{t_-}$,
\item if $\eta(V'^*_{x+1})=\epsilon_{t_+}$, then $\rho'^{(-)}=\rho^{\O_{\mathrm{even}}}_{\pm t_+}$
    and $\widetilde{\rho}'^{(-)}=\rho^{\O_{\mathrm{even}}}_{\pm (t_+-1)}$
    and if $\eta(V'^*_{x+1})=\epsilon_{t_+-1}$, then
    $\rho'^{(-)}=\rho^{\O_{\mathrm{even}}}_{\pm (t_+-1)}$
    and $\widetilde{\rho}'^{(-)}=\rho^{\O_{\mathrm{even}}}_{\pm t_+}$
    ,

\item $\epsilon=\sp(s').$
    \end{itemize}

Moreover, by Lemma \ref{spinor} and $|\lambda_\Gamma|=m_\Gamma(s)=m_{-\Gamma}(s')$,  if $\eta(V'^*_{x+1})=\epsilon_{t_+}$, then $$\epsilon=\sp(s')=\epsilon_{t_+}\cdot\zeta(-1)^{t_+}\cdot\prod_{\Gamma\in \calF_1\cup\calF_2,\sigma(-\Gamma)\neq 1}(-1)^{|\lambda_\Gamma|}$$ and if $\eta(V'^*_{x+1})=\epsilon_{t_+-1}$, then
 $$\epsilon=\sp(s')=\epsilon_{t_+-1}\cdot \zeta(-1)^{t_+}\cdot \prod_{\Gamma\in \calF_1\cup\calF_2,\sigma(-\Gamma)\neq 1}(-1)^{|\lambda_\Gamma|}.$$

We consider $s'^{(\star)}\in \O^\pm((V'^*)^{(\star)})$, where $(V'^*)^{(\star)}=\bigoplus_{\Gamma\in \calF_1\cup\calF_2}V'^*_{\Gamma}$,  then we have $\sp(s'^{(\star)})=\sp(-s'^{(\star)})\sp(-\id_{(V'^*)^{(\star)}})$. So we have $$\prod_{\Gamma\in \calF_1\cup\calF_2,\sigma(-\Gamma)\neq 1}(-1)^{|\lambda_\Gamma|}=\prod_{\Gamma\in \calF_1\cup\calF_2,\sigma(\Gamma)\neq 1}(-1)^{|\lambda_\Gamma|}\cdot \prod_{\Gamma\in \calF_1\cup\calF_2}\zeta(-1)^{\delta_\Gamma |\lambda_\Gamma|}.$$

On the other hand, we know $\dim(V^*_{x+1})=\dim(V'_{x-1})$ and $\dim(\bigoplus\limits_{\Gamma\in \calF_1\cup\calF_2} V^*_{\Gamma})=\dim(\bigoplus\limits_{\Gamma\in \calF_1\cup\calF_2}V'^*_\Gamma)$, so by a similar argument as in $(a)$, we have $$\frac{\dim(V')}{2}-\lfloor\frac{\dim(V)}{2}\rfloor=\frac{\dim(V')}{2}-\frac{\dim(V)-1}{2}
=\frac{\dim(V'^*_{x+1})}{2}-\frac{\dim(V^*_{x-1})}{2}.$$
\smallskip

Furthermore, $\epsilon_{V'}=\zeta(-1)^{\dim{V'}/2}=\zeta(-1)^{t_+}\cdot\prod\limits_{\Gamma\in \calF_1\cup\calF_2}\zeta(-1)^{\delta_\Gamma |\lambda_\Gamma|}.$
By Lemma \ref{Lemma:MQZ} $(1)$, $(c1)$ follows.
\end{proof}

Finally, we get the following Table \ref{diag:final}, which is a summary of Theorem \ref{Thm:weightfunctions}.
 \begin{center}
\tiny{\begin{table}[h]\caption{Colored weight functions of  modules $\rho$ which are both $F^+$-cuspidal
and $F^-$-cuspidal}\label{diag:final}
\begin{tabular}{cccc}
 \hline
  &$\O_{2\bullet+1}(q)$&$\Sp_{2\bullet}(q)$ &$\O^{\pm}_{2\bullet}(q)$\\\hline
  \text{  }&$\rho=\rho_{t_+,t_-,\lambda_\star, \varepsilon}$&
$\rho=\rho_{t_+,t_-,\lambda_\star}$&
$\rho=\rho_{t_+,t_-,\lambda_\star}$\\
  \text{ }&$t_\pm\in \bbN,\varepsilon\in\{\pm1\}$&
$t_+\in\bbN,t_-\in \bbZ$&
$ t_\pm\in \bbZ$\\\hline
  $\bbO^+(u)(\rho):$&$(u-\eta^+(\rho)q^{t_+})(u+\eta^+(\rho)q^{-1-t_+})$& $(u- \eta(\rho)q^{t_+})(u+\eta(\rho)q^{-1-t_+})$  &$(u- q^{t_+})(u+ q^{-t_+})$\\\hline
   $\bbO^-(v)(\rho):$&$(v-\eta^-(\rho)q^{t_-})
   (v+ \eta^-(\rho)q^{-1-t_-})$& $(v- q^{t_-})(v+ q^{-t_-})$
    &$(v- q^{t_-})(v+ q^{-t_-})$\\\hline
\end{tabular}
\end{table}}
\end{center}
\begin{remark}Specifically, when considering quadratic unipotent modules (i.e., $\lambda_{\star}=\emptyset$), the corresponding signs appear notably simpler. For the symplectic group  $\Sp_{2n}(q)$, we have $\eta(\rho_{t_+,t_-})=(-1)^{t_++t_-}.$
 For the odd orthogonal group $\O_{2n+1}(q)$, we have $\eta^+(\rho_{t_+,t_-,\epsilon})=\epsilon\cdot(-1)^{t_+}$ and  $\eta^-(\rho_{t_+,t_-,\epsilon})=\epsilon\cdot (-1)^{t_-}$.
Via the restriction functor $\Res^{\O_{n+1}}_{\SO_{2n+1}}$, one can show that this gives a new proof of \cite[Theorem 5.17]{LLZ} for quadratic unipotent modules of $\SO_{2n+1}(q)$.
\end{remark}

\begin{remark}
Note that the theta correspondence depends on the choices of the non-trivial character $\psi:(\bbF_q,+)\to \bbC^\times.$  However, the colored weight functions are independent of choices of $\psi$ and the theta correspondence.
\end{remark}

\smallskip

\subsubsection{Kac-Moody action on  Grothendieck groups}
The
ambiguity in the determination of colored weight functions for all irreducible modules that are simultaneously
$F^+$-cuspidal and $F^-$-cuspidal, as originally stated in Theorem \ref{TheoremC}, has now been successfully resolved.
As a consequence, we have the following theorem.
 \begin{theorem}\label{thm:final}
 The pairs $(s^+_1,s^+_2)$ and $(s^-_1,s^-_2)$ in Theorem \ref{TheoremC} are given in Table \ref{diag:final} by  $\bbO^+(u)(\rho)=(u-q^{s^+_1})(u+q^{s^+_2})$ and  $\bbO^-(v)(\rho)=(v-q^{s^-_1})(v+q^{s^-_2})$.
\end{theorem}
 Hence colored weight functions of all irreducible modules are determined explicitly and the action of Kac-Moody algebra $\frakg=\fraks\frakl'_{I_+}\oplus\fraks\frakl'_{I_-}$
on the Grothendieck group $[KG_\bullet\mod]$ is completely determined.
\smallskip

%\begin{remark}
%
%In fact, for quadratic unipotent modules, one can generalize the proof of $\SO_{2n+1}(q)$ in \cite{LLZ}
%to other finite classical groups.
%However, it would be very complicated and not scalable to the whole category.
%\end{remark}

%\begin{remark}
%For $G=\Sp_{2n}(q)$ and $\O^{\pm}_{2n}(q),$ there exists a unique modified Lusztig's
%Jordan decomposition $$\underline{\calJ}_s:\calE(G,s^{(+)}\times s^{(-)}\times s^{(\star)} )
%\to \calE(\underline{G}^{(+)}\times G^{(-)}\times G^{(\star)},1),\,\rho\mapsto \rho^{(+)}
%\otimes\rho^{(-)}\otimes \rho^{(\star)} $$ such that $\bbO^+(u)(\rho)=
%\bbO^+_{\underline{G}^{(+)}}(u)(\rho^{(+)})$ and $\bbO^-(v)(\rho)=\bbO^+_{G^{(-)}}(v)(\rho^{(-)}).$
%In a sense, this is the compatibility of Kac-Moody
%categorification and Jordan decomposition.
%\end{remark}

Question: Under the parametrization via our complete invariants, how can the finite theta correspondence be explicitly described? For unipotent modules within the symplectic-even orthogonal dual pair, this is addressed by Theorem \ref{Thm:unipotenttheta}.   Pan \cite{P4}  demonstrated that the modified Lusztig's Jordan decomposition for a classical group can be rendered unique provided it is required to be compatible with parabolic induction and a specific finite theta correspondence with a fixed non-trivial character $\psi.$ Wang \cite{Wa} showed that, when $q$ is sufficiently large, there exists a unique choice of the modified Lusztig correspondence that is compatible with parabolic induction and the theta correspondence for all dual pairs. This question is essentially equivalent to exploring the relationship between colored weight functions and modified Lusztig's Jordan decompositions defined by Pan \cite[\S8,\S9]{P4} and  Wang \cite[Theorem 6.1]{Wa}.

\subsection*{Acknowledgements}
We are deeply grateful to Olivier Dudas, Xin Huang, Yanjun Liu, Sergio David C\'{i}a Luvecce, Jiajun Ma and Zhicheng Wang
for inspiring discussions. We also thank Rapha\"{e}l Rouquier for his helpful comments. The authors express their gratitude to the anonymous referee for their thorough review of the manuscript and for providing valuable feedback that significantly improved its readability.
\smallskip

P.L. and P.S. acknowledge the support by NSFC (No. 12225108). P.S. also acknowledge the support of the New Cornerstone Science Foundation through the Xplorer Prize.
J.Z. acknowledges the support by National Key R$\&$D Program of China (No. 2020YFE0204200) and NSFC (No. 11631001).

\end{document}